\documentclass[11pt,a4paper,reqno]{amsart}  
\usepackage{ifthen,amssymb,mathrsfs,a4wide,color,mathtools}
\let\spacecal=\mathscr
\usepackage[all]{xypic}
\usepackage[backref=page]{hyperref}
\DeclareMathAlphabet{\mathpzc}{OT1}{pzc}{m}{it}
\hypersetup{
 colorlinks,
 citecolor=green,
 linkcolor=blue,
 urlcolor=Blue}
\definecolor{maroon}{rgb}{0.5, 0.0, 0.0}

\newcommand{\red}[1]{\textcolor{red}{#1}}
\let\tempone\itemize
\let\temptwo\enditemize
\let\temponenum\enumerate
\let\temptwonum\endenumerate
\renewenvironment{itemize}{\tempone\addtolength{\itemsep}{0.2\baselineskip}}{\temptwo}
\renewenvironment{enumerate}{\temponenum\addtolength{\itemsep}{0.2\baselineskip}}{\temptwonum}
\let\tempsec\section
\renewcommand{\section}{\par\medskip\tempsec}

\newcommand{\qee}{\parbox{5pt}{\hfill}\hfill $\triangle$}
\newcommand{\et}{\mbox{\it \footnotesize \!(\'et)}}

\newenvironment{rem}{\begin{remqee}}{\qee\end{remqee}}
\newtheorem{thm}{Theorem}[section] 
\newtheorem{corol}[thm]{Corollary}
\newtheorem{lemma}[thm]{Lemma} 
\newtheorem{prop}[thm]{Proposition}
\newtheorem{defin}[thm]{Definition}
\theoremstyle{definition}

\newtheorem{strategy}[thm]{Strategy}
\theoremstyle{remark}
\newtheorem{remqee}[thm]{Remark}

\newtheorem{example}[thm]{Example}
\newenvironment{remark}{\begin{remqee}}{\qee\end{remqee}}
\numberwithin{equation}{section}

\newcommand\rk{\operatorname{rk}}
\newcommand\Id{\operatorname{Id}}
\newcommand\Ima{\operatorname{Im}}
\newcommand\coker{\operatorname{coker}}
\newcommand\Spec{\operatorname{Spec}}
\newcommand\Proj{\operatorname{Proj}}
\newcommand\Hom{\operatorname{Hom}}
\newcommand\Tor{\operatorname{Tor}}
\newcommand\Torsh{\mathcal{T}{or}}
\newcommand\iso{\kern.35em{\raise3pt\hbox{$\sim$}\kern-1.1em\to}\kern.3em}
\newcommand\Isom{\operatorname{Isom}}

\newcommand\Pic{\operatorname{Pic}}

\newcommand\rest[2]{#1_{\vert #2}}

\newcommand\F{{\mathcal F}}
\newcommand\Oc{{\mathcal O}}
\newcommand\Pc{{\mathcal P}}
\newcommand\M{{\mathcal M}}
\newcommand\Z{{\mathbb Z}}
\newcommand\N{{\mathbb N}}

\newcommand\U{{\mathcal U}}
\newcommand\Dcal{{\mathcal D}}
\newcommand\Nc{{\mathcal N}}
\newcommand\Kc{{\mathcal K}}
\newcommand\Lcl{{\mathcal L}}
\newcommand\Ec{{\mathcal E}}
\newcommand\Ac{{\mathcal A}}

\newcommand\Ps{{\mathbb P}}

\newcommand\Gc{{\mathcal G}}
\newcommand\Xcal{{\spacecal X}}
\newcommand\Ycal{{\spacecal Y}}
\newcommand\Sc{{\spacecal S}}
\newcommand\Tc{{\spacecal T}}
\newcommand\Zc{{\spacecal Z}}
\newcommand\Vc{{\spacecal V}}
\newcommand\Wc{{\spacecal W}}
\newcommand\Hcal{{\spacecal H}}
\newcommand\Gsc{{\spacecal G}}

\newcommand\Hc{{\mathcal H}}
\newcommand\Ucal{{\spacecal U}}
\newcommand\Rcal{{\spacecal R}}
\newcommand\Dsc{{\spacecal D}}

\newcommand\Psc{{\spacecal P}}
\newcommand\Qc{{\mathcal Q}}
\newcommand\Ic{{\mathcal I}}
\newcommand\Jc{{\mathcal J}}
\newcommand\Cc{{\mathcal C}}

\newcommand\bH{{\mathbf H}}
\newcommand\bP{{\mathbf P}}
\newcommand\bQ{{\mathbf Q}}

\newcommand\bR{{\mathbf R}}
\newcommand\bL{{\mathbf L}}

\newcommand\Ber{{{\mathcal B}er}}
\newcommand\Homsh{{{\mathcal H}om}}
\newcommand\Div{{{\mathcal D}iv}}

\newcommand\pf{{\mathfrak p}}
\newcommand\qf{{\mathfrak q}}
\newcommand\mf{{\mathfrak m}}
\newcommand\Mf{{\mathfrak M}}
\newcommand\As{{\mathbb A}}
\newcommand\Bs{{\mathbb B}}
\newcommand\Vs{{\mathbb V}}
\newcommand\Ss{{\mathbb S}}
\newcommand\SSpec{\mathbb{S}\mathrm{pec}\,}
\newcommand\SProj{\mathbb{P}\mathrm{roj}\,}
\newcommand\SSym{\mathbb{S}\mathrm{ym}\,}
\newcommand\SHilbf{\mathbb{S}\mathcal{H}{ilb}\,}
\newcommand\SPicf{\mathbb{S}\mathcal{P}{ic}\,}
\newcommand\Sgrassf{\mathbb{S}\mathcal{G}rass\,}
\newcommand\grass{\operatorname{Grass}\,}
\newcommand\Sgrass{\mathbb{S}\mathrm{Grass}\,}
\newcommand\SG{\mathbb{S}\mathrm{G}\,}
\newcommand\Hs{{\mathbb H}}
\newcommand\SHilb{\mathbb{SH}\mathrm{ilb}\,}
\newcommand\SPic{\mathbb{SP}\mathrm{ic}\,}
\newcommand\SPicc{\mathbb{P}\mathrm{ic}\,}
\newcommand\SDiv{\mathbb{SD}\mathrm{iv}\,}
\newcommand\Ab{\operatorname{Ab}}

\newcommand\SHom{\mathbb{H}\mathrm{om}\,}
\newcommand\Sisom{\mathbb{I}\mathrm{som}\,}
\newcommand\Morf{\mathcal{M}or\,}

\newcommand{\rra}{\rightrightarrows}
\newcommand\Picf{\mathcal{P}{ic}\,}
\newcommand{\pcal}{\mathpzc{p}}
\newcommand\Extsh{{{\mathcal E}xt}}

\newcommand\SFlagf{\mathbb{S}\mathcal{F}lag\,}
\newcommand\SFlag{\mathbb{S}\mathrm{Flag}\,}

\newcommand{\bos}{\operatorname{bos}}
\renewcommand{\det}{\operatorname{det}}
\newcommand{\wt}{\widetilde}
\newcommand{\sub}{\subset}
\newcommand{\om}{\omega}
\newcommand{\Ga}{\Gamma}
\newcommand{\La}{\Lambda}
\newcommand{\si}{\sigma}
\newcommand{\re}{\operatorname{red}}

\newcommand\Picb{\mathbf{Pic}}

\begin{document}
\title[Notes on Fundamental Algebraic Supergeometry]{NOTES ON FUNDAMENTAL ALGEBRAIC SUPERGEOMETRY.\\[4pt]
HILBERT AND PICARD SUPERSCHEMES}

\date{\today} 
 
\thanks{\P\ Research partly supported by Bolsa de Produtividade 313333/2020-3 of Brazilian CNPq, by GNSAGA-INdAM and by the PRIN project ``Geometria delle variet\`a algebriche.''}
\thanks{\ddag\ Research partly supported by grants  ``PID2021-128665NB-I00 funded by MCIN/AEI/ 10.13039/501100011033'', and  ``STAMGAD, SA106G19'' (Junta de Castilla y Le\'on)}
\thanks{$\star$\ Research partly supported by the NSF grant DMS-2001224, 
and within the framework of the HSE University Basic Research Program and by the Russian Academic Excellence Project `5-100'.}


\maketitle
\begin{center} {\sc Ugo Bruzzo,$^\P$  
Daniel Hern\'andez Ruip\'erez$^\ddag$  
 and Alexander Polishchuk$^\star$}

\smallskip \small
$^\P$ SISSA (Scuola Internazionale Superiore di Studi Avanzati),  Via Bonomea 265, 34136 Trieste, Italy;  \\
Departamento de Matem\'atica, Universidade Federal da Para\'iba,  Campus I, Jo\~ao Pessoa, PB, Brazil; \\
INFN (Istituto Nazionale di Fisica Nucleare), Sezione di Trieste; \\
IGAP (Institute for Geometry and Physics), Trieste;  \\
$^\ddag$ Departamento de Matem\'aticas and IUFFYM (Instituto Universitario de F\'{\i}sica Fundamental \\ y Matem\'a\-ticas),  Universidad de Salamanca, Plaza
de la Merced 1-4, 37008 Salamanca, Spain; \\
Real Academia de Ciencias Exactas, F\'{i}sicas y Naturales, Spain \\
$^\star$ Department of Mathematics, University of Oregon, Eugene, OR 97403, USA;  \\
 National Research
University Higher School of Economics, Moscow, Russia;  \\
  Korea Institute for Advanced Study, Seoul, Korea
\end{center}


\begin{abstract} 
These notes aim at providing a complete and systematic account of  some foundational aspects  of algebraic supergeometry,  namely, the extension to the geometry of superschemes of many classical notions, techniques and results that make up   the general backbone of algebraic geometry, most of them originating from Grothendieck's work. In particular, we extend to algebraic supergeometry such notions  as projective and proper morphisms, finiteness of the cohomology, vector and projective bundles,  cohomology base change, semicontinuity theorems,  relative duality, Castelnuovo-Mumford regularity, flattening,  Hilbert and Quot schemes, faithfully flat descent, quotient  \'etale relations (notably, Picard schemes), among others. Some  results may be found elsewhere, and, in particular, there is some overlap with \cite{MoZh19}. However, many techniques and constructions are presented here for the first time, notably, a first development of Grothendieck relative duality for   proper morphisms of superschemes,  the construction of the Hilbert superscheme in a more general situation than the one already known (which in particular allows one to treat the case of sub-superschemes of supergrassmannians), and a rigorous construction of the Picard superscheme for a  locally superprojective morphism of  noetherian superschemes with geometrically integral fibres. Moreover, some of the proofs given here are  new as well, even when restricted to ordinary schemes. In a final section we construct a period map   from an open substack of the moduli of proper and smooth supercurves to the moduli stack of principally polarized abelian {schemes}.
\end{abstract}


 \newpage

\tableofcontents

\section{Introduction}
The introduction of a geometry that encompasses both even (bosonic) and odd (fermionic) coordinates was  motivated by the supersymmetric
field theories that were formulated in the 1970s and 1980s.  After the  first notions introduced by physicists (see e.g.~\cite{SS}),   more mathematically sound theories were proposed, where   different kinds of ``supermanifolds'' were considered (see \cite{BBH91} and references therein). The Berezin-Le\u\i tes approach \cite{BL75}, also developed by  Kostant \cite{Kost75} and Manin \cite{Ma86,Ma87,Ma88,Ma88-2} among others, closer in spirit to algebraic geometry, has become the standard approach to superschemes. As in classical algebraic geometry, superschemes are locally ringed spaces which locally look  like spectra of rings, the only basic  difference being that the rings are $\Z_2$-graded commutative.

Many objects   have been introduced and studied in algebraic supergeometry, as, 
for instance, super Riemann surfaces,  or supersymmetric (SUSY) curves \cite{Fr86,BaSchw87,Ma86}, and their supermoduli spaces,
which are supposed to be relevant to the perturbative approach to superstring theories.
Supermoduli spaces were constructed by LeBrun and Rothstein \cite{LeRoth88} as superorbifolds (see also \cite {CraRab88}). In \cite{DoHeSa97} a supermoduli space for SUSY curves with Neveu-Schwarz (NS) punctures was constructed as an Artin algebraic superspace, that is, as a quotient of an \'etale equivalence relation of superschemes. 
The interest in this topic was revived by the works of  Witten \cite{Witten19}\footnote{This paper was published in 2019
but first appeared in 2012 as preprint {\tt arXiv:1209.2459}.}  and Donagi-Witten \cite{DoW13,DoW15}.
A relatively recent  paper is \cite{CodViv17}, which contains a construction of the supermoduli of SUSY curves as a Deligne-Mumford (DM) superstack, together  with a new  theory of stacks in algebraic supergeometry, or superstacks.
A supermoduli space of SUSY curves with both Neveu-Schwarz and Ramond-Ramond punctures  was constructed in \cite{BrHR19}, again as an Artin algebraic superspace, and in \cite{MoZh19} as a DM superstack. The last paper also  proves   the existence of a superstack which compatifies   the supermoduli of SUSY curves with any kind of punctures. Moreover,  substantial work has  been done about superperiods  and the generalization  of the Mumford formula  to the supermoduli of SUSY curves \cite{Witten19, Dir19, CodViv17, FKP19,FKP20} and about the relationship of supermodui problems with relevant aspects of string theory \cite{MoZh19}.
All these developments have required an increasing amount of tools in algebraic supergeometry, which, in most cases, have been introduced and studied by the corresponding authors according to their contingent needs.

\paragraph{\bf Goal of the paper.}
There exist some accounts of different aspects of algebraic supergeometry and of one of its basic ingredients, namely, the algebra of the $\Z_2$-commutative rings. Some elements were given in \cite{BBH91}; moreover, one  can mention Westra's dissertation \cite{We09}, Carmeli, Caston and Fioresi's work (eg.~\cite{CarCaFi11}) and  Cacciatori and Noja's paper \cite{CaNo18}. However, 
a systematic treatment of algebraic supergeometry  in the spirit of Grothen\-dieck's work, 
as laid down in  Grothendieck's FGA \cite{FGA}\footnote{See also \cite{FGA05} as a kind of update to FGA.} and in the EGA series, 
is still missing. That should include   superprojective and proper morphisms of superschemes, finiteness of   cohomology, supervector and projective superbundles, 
cohomology base change, semicontinuity theorems,  {relative duality,}  Castelnuovo-Mumford regularity, flattening,   Hilbert and Quot superschemes, faithfully flat descent, quotients of   \'etale relations, Picard superschemes, and more.
The scope of these notes is to give a contribution to the building  of such a  systematic approach. The extension of some classical results to supergeometry  is  sometimes very easy, while sometimes requires  quite a bit of work. In   other cases  it is better, or simply unavoidable, to develop the constructions from   scratch, often using the original ideas but exploiting new techniques. In doing so, some generalizations of the classical statements or simplifications of their proofs are achieved. Some results about the algebra of graded commutative rings are necessary, as are suitable extensions to the super case of theorems like the existence of the Grothendieck-Mumford complex or   the super Nakayama Lemma   for half exact functors.

While we were preparing these notes  the preprint \cite{MoZh19} by  Moosavian and Zhou appeared. It contains some results and constructions given here up to Section \ref{s:hilbert}, so there is  some overlap between the two works. Also, Jang in \cite{Jang20} constructed and studied the Hilbert superscheme of 0-dimensional subspaces of a $(1,1)$-dimensional supercurve.
We address  here many topics not covered in \cite{MoZh19}, and  some proofs and approaches are different (they are different also in the classical case). Moreover, we prove the representability of the super Hilbert functor under more general assumptions than in  \cite{MoZh19}. This added generality is needed to prove the representability of the super Picard functor, which is done here for the first time.

It is important to note here that the assumption of projectivity, which is quite natural in the classical algebraic geometry, is rather restrictive in supergeometry.  
As is well known, many interesting superschemes cannot be embedded into superprojective spaces, for instance,
supergrassmannians in general are not superprojective \cite{PenSkor85}. 
It seems however that all interesting superschemes (with projective bosonization) can be embedded into supergrassmannians. For example, we were informed by Vera Serganova that this is true for
all proper homogeneous superspaces. In this paper we make a useful observation that in some sense a superscheme embeddable into a supergrassmannian is very close to being superprojective. Namely, we show that there always exists a smooth surjective morphism of purely odd relative dimension from a superprojective scheme to such a scheme 
(see Corollary \ref{cor:grassproj}).
Because of this,
projective superschemes are still of great significance, and in particular, we first prove the existence of the Hilbert superscheme for them and then deduce it for more general
superschemes.

\paragraph{\bf Structure of the paper.}  In Section \ref{s:supergeom} we introduce and study projective superschemes  by means of the projective superspectrum $\SProj$ of a bigraded algebra  over a superring, that is, an algebra  which has two compatible gradings, one over $\Z$, as in the classical case, and another over $\Z_2$. Projective superschemes have been considered in other papers, such as \cite{CaNo18,MoZh19}, where they  were called ``superprojective superschemes.''   
In subsection \ref{ss:superproj} we compute the cohomology of the sheaves $\Oc(p)$ and prove the finiteness of the cohomology of coherent sheaves on   projective superspaces, as well as the super extension of some  classical facts about them, such as Serre's Theorem \ref{thm:serre}. Some of these results were already proved in \cite{CaNo18}, however our method does not rely on computations for the classical case, but rather we extend  from scratch Grothendieck's methods   to algebraic supergeometry.  We also introduce supervector bundles and projective superbundles, or more generally,  linear superspaces and projective superschemes associated to coherent sheaves as the superspectrum and the projective superspectrum, respectively, of the (graded) symmetric algebra of the corresponding module. Supergrassmannians are also introduced and the proof of their existence is referred to Manin's original treatment \cite{Ma88}.

In Section \ref{s:cohomprop} we extend to supergeometry the classical results on the relative cohomology of coherent sheaves with respect to proper morphisms of locally noetherian superschemes. The main statements are the  cohomology base change Theorem \ref{thm:cohombasechange}, the semicontinuity Theorem \ref{thm:semicontinuity} and Grauert's Theorem \ref{thm:grauert}. Our treatment is simpler than others available in the literature (e.g.~\cite{Hart77}) due to the use of the Nakayama Lemma for half exact functors, which we extend  to supergeometry in Subsection \ref{ss:nakayama}. In particular, the use of the Formal Function Theorem (\cite{Hart77} or \cite[Thm. 7.4]{MoZh19} in the super setting) is no more necessary. 
 {Section  \ref{s:cohomprop} also includes a treatment of Grothendieck's relative duality for superschemes. This differs in some parts from the treatments of Grothendieck's duality usually found in the literature.  We also include a proof that the relative dualizing complex of a smooth proper morphism is a shift of the relative Berezinian sheaf.}

Section \ref{s:hilbert} is devoted to the proof of the existence of the Hilbert superscheme, following  an analogous procedure as in the classical case (see e.g. the original exposition by Grothendieck \cite{FGA} or \cite{FGA05,Ni07}). So the  strategy is first to apply the super version of Castelnuovo-Mumford regularity (Subsection \ref{ss:CMregularity}) to embed the super Hilbert functor into the functor of points of a supergrassmannian, and then prove, using   flattening  (Subsection \ref{ss:genflatness}), that this embedding is representable by immersions, so that the Hilbert superscheme is a  sub-superscheme of a supergrassmannian, which is proved to be proper by the valuative criterion for properness. 
We first prove the existence of the Hilbert superscheme in the superprojective case 
(Theorem \ref{thm:hilbrepres}) and then we also give a more general version of this result (Theorem \ref{thm:hilbrepres2}), 
{which allows one to treat also the case of sub-superschemes of a supergrassmannian,} and 
 will be needed to prove
the existence of the Picard superscheme.
 As we discuss in a simple example, the ordinary scheme underlying a Hilbert superscheme has a richer structure than the usual Hilbert scheme.

As an application, in Section \ref{subsec:schemeofmorph} the superscheme of morphisms between two projective superschemes is constructed. The existence of this superscheme is necessary to prove that the supermoduli  of SUSY curves is not only a category fibred in groupoids, but it is  also a Deligne-Mumford stack \cite[7.3]{MoZh19} (this is also true in the presence of Neveu-Schwarz and Ramond-Ramond punctures, and for the compactified supermoduli).   This section is quite similar to \cite[7.1.8]{MoZh19}.
In Section \ref{sec:emb-supergrass} we prove that a superscheme is embeddable into some supergrassmannian if and only if it is a target of a smooth morphism of odd relative
dimension from a superprojective superscheme.

Section \ref{s:picard} offers the first  construction of the (relative) Picard superscheme for a  {superprojective} morphism of 
noetherian superschemes {that are cohomologically flat in dimension $0$} and whose bosonic fibres are geometrically integral schemes.
The total Picard superscheme is the disjoint union of the Picard superschemes   parameterizing the even and the odd line bundles on the fibres, and it 
{suffices} to construct  explicitly the even one. {The proof  mostly follows  Grothendieck's   construction     \cite{FGA}, or more precisely   Kleiman's version of it \cite{Kle05}, however, to overcome some difficulties  arising  from the fact that we are working with superschemes, sometimes we need to take a different route. We also take advantage of  a significant simplification, which can also be used when constructing the classical Picard scheme.} Our procedure consists in proving the existence of the open subfunctor of the even super Picard sheaf that parameterizes (relatively) acyclic even line bundles generated by their global sections, and the proof that translating them one  obtains a covering of the even Picard sheaf by representable open subfunctors, so that it is representable as well. In this  way one avoids the theory of $(b)$-sheaves used in the classical case to prove that the Picard functors with fixed Hilbert polynomials are representable. 

The final part of Section \ref{s:picard} discusses examples where the conditions for the existence of the Picard superscheme can be relaxed when the base superscheme is even affine as well as examples where the super Picard functor is not representable.

{In Section \ref{s:superperiod}, as an application, we construct a period map   from an open substack of the moduli of proper and smooth supercurves to the   moduli stack of principally polarized abelian {schemes}.} {We discuss the relation of this construction with the D'Hoker-Phong defintion of a superperiod matrix.}

Finally, we have gathered in the Appendix some auxiliary results, which are partly original.

We assume that the reader is familiar with the basic definitions and notions of algebraic supergeometry.  However, for convenience in Subsection \ref{ss:propermor} we recall some properties of morphisms of superschemes. 

Just to mention a few references that have been important to us in connection with     algebraic supergeometry, we would like to cite   Kostant's   paper \cite{Kost75}, Manin's book \cite{Ma88}, the early papers by Penkov and Skornyakov \cite{Penk83, PenSkor85}, the above mentioned paper by  Le Brun and Rothstein \cite{LeRoth88}, the  study of superprojective embeddings by  LeBrun, Poon and Wells  \cite{LePoWel90}, Deligne's letter \cite{Del87} and the already cited works of Donagi and Witten \cite{DoW13,DoW15,Witten19}. 
 We   also used some basic results in superalgebra and supergeometry   by Cacciatori and Noja \cite{CaNo18}, Carmeli, Caston and Fioresi \cite{CarCaFi11}, and Westra \cite{We09}.

 \smallskip
\paragraph{\bf Acknowledgement.}  We thank Carlos Sancho de Salas for introducing   us to a simplified approach to the construction of the classical Picard scheme. {We also thank Katherine Maxwell for asking about a possible relation between our super period map and the D'Hoker-Phong superperiod matrix.}

\section{Basic algebraic supergeometry}\label{s:supergeom}
In this section we collect  some basic definitions in supergeometry. 

We adopt the following convention: for every $\Z_2$-graded module $M$, we write $M_+$ for its even part and $M_-$ for its odd part, so that $M=M_+\oplus M_-$. 

\subsection{Superrings}
By a \emph{superring} $\As$ we mean a $\Z_2$-graded supercommutative ring such that one of the following equivalent conditions is satisfied, where
$J$ is the ideal generated by the odd elements:
\begin{enumerate}
\item $J$ is finitely generated;
\item $J^n=0$ for some $n>0$ and $J/J^2$ is a finitely generated module over $A=\As/J$.
\end{enumerate}

Note that under these conditions the graded ring $Gr_J(\As)=\bigoplus_{i\geq 0} J^i/J^{i+1}$ is finitely generated as an $A$-module.

We shall say that $A=\As/J$ is the \emph{bosonic reduction} of $\As$. We also say that $\As$ is \emph{split} if there exists a finitely generated projective $A$-module $M$ such that $\As\simeq {\bigwedge}_AM$. 

We define the \emph{odd dimension} of $\As$ as 
the smallest number $\operatorname{odd-dim}(\As)$ of generators of the ideal $J$, or, equivalently, the smallest number of generators of the $A$-module $J/J^2$.

There is actually another odd dimension parameter  one may consider, i.e., the smallest integer number $p$ such that $J^{p+1}=0$.  We write $\operatorname{ord}(\As)$ for this number. One easily sees that
$\operatorname{ord}(\As)\leq \operatorname{odd-dim}(\As)$ and the two numbers are coincide when $\As$ is split.

The notion of \emph{noetherian superring} generalizes the usual one, i.e., every ascending chain of $\mathbb Z_2$-graded ideals stabilizes \cite{We09}. 
For a noetherian local superring $\As$ with maximal ideal $\mf$, the odd dimension  $\operatorname{odd-dim}(\As)$
of $\As$ is the odd dimension of $\mf/\mf^2$ as a graded $\As/\mf$-vector space \cite[7.1.4]{We09}.

\subsection{Superschemes}

We recall the notion of superscheme.  
All   schemes and superschemes in these notes are \emph{locally noetherian} and all   superscheme morphisms are \emph{locally of finite type}.

\begin{defin}
A   locally ringed superspace is a pair $\Xcal=(X,\Oc_{\Xcal})$, where $X$ is a topological space, and $\Oc_{\Xcal}$ is a sheaf of $\Z_2$-graded commutative rings  such that for every point $x\in X$ the  stalk $\Oc_{\Xcal,x}$ is a local superring. 
\end{defin}

\begin{defin} \label{def:mor}
A morphism of locally ringed superspaces is a pair $(f,f^\sharp)$, where $f\colon X \to Y$ is  a continuous map, and  $f^\sharp\colon \Oc_{\Sc}\to f_\ast\Oc_{\Xcal}$ is a homogeneous morphism of graded commutative sheaves, such that for every point $x\in X$, the induced morphism of 
local superrings
$\Oc_{\Sc,f(x)}\to \Oc_{\Xcal,x}$ is local.  \end{defin}

Given a  locally ringed superspace $\Xcal=(X,\Oc_{\Xcal})$, we can consider the homogeneous ideal   {$\Jc=(\Oc_{\Xcal,-})^2\oplus\Oc_{\Xcal,-}$} generated by the odd elements. Then $\Oc_X:=\Oc_{\Xcal}/\Jc$ is a purely even sheaf of rings. We say that the locally ringed space $X=(X,\Oc_X)$ is the \emph{ordinary locally ringed space underlying} $\Xcal$.  
Sometimes we also call it the \emph{bosonic reduction of} $\Xcal$ and denote as $\Xcal_{\bos}$. 

There is a closed immersion of locally ringed superspaces
$$
i\colon X \hookrightarrow \Xcal
$$
induced by the epimorphism $\Oc_\Xcal \to \Oc_X$. $\Xcal$ is said to be \emph{projected} if there exists a morphism of locally ringed superspaces $\rho\colon \Xcal \to X$ such that $\rho\circ i=\Id$. As in the case of superrings, we said that $\Xcal$ is \emph{split} when $\Oc_\Xcal\simeq \bigwedge_{\Oc_X}\Ec$ for a locally free sheaf $\Ec$ on $X$ which generates the ideal $\Jc$, so that $\Ec\simeq\Jc/\Jc^2$. Analogously, $\Xcal$ is called \emph{locally split} if it can be covered by split open locally ringed superspaces.

The group $\Gamma=\{\pm1\}$ acts on $\Oc_\Xcal$ by $f^{(-1)}=f_+-f_-$ and defines another locally ringed superspace $\Xcal/\Gamma=(X, \Oc_{\Xcal,+})$, called the \emph{bosonic quotient} of $\Xcal$.
 A morphism $f\colon \Xcal \to \Ycal$ of locally ringed superspaces induces a morphism $f\colon X \to Y$ of the underlying locally ringed spaces and a morphism $f/\Gamma\colon \Xcal/\Gamma\to \Ycal/\Gamma$  between the bosonic quotients, so that 
$$
\Xcal=(X,\Oc_{\Xcal}) \mapsto  (X,\Oc_X)\,,\quad \Xcal\mapsto \Xcal/\Gamma
$$ 
are functors. In fact, it is easy to see that these functors are the right and left adjoint functors, respectively, 
to the natural inclusion from the category of (purely even) locally ringed spaces to that of locally ringed superspaces: for a locally ringed space $Y$, we have
functorial isomorphisms
$$\Morf(Y,\Xcal)\simeq \Morf(Y,X), \ \ \Morf(\Xcal,Y)\simeq \Morf(\Xcal/\Gamma,Y).$$

The sheaves $Gr^j(\Oc_{\Xcal})=\Jc^j/\Jc^{j+1}$ are annihilated by $\Jc$ so that they are $\Oc_X$-modules. Then we   can   consider the sheaf of $\Oc_X$-modules 
$$
Gr (\Oc_{\Xcal})=\bigoplus_{j\ge 0} Gr^j(\Oc_{\Xcal})=\bigoplus_{j\ge 0} \Jc^j/\Jc^{j+1}
$$
which comes with a natural $\Z_2$ grading.

The \emph{superspectrum} of a superring $\As$ is 
{the} {locally ringed superspace}  {$\SSpec\As=(X,\Oc)$}, where $X$ is the spectrum of the bosonic reduction $A$ of $\As$,
and $\Oc$ is a sheaf of $\Z_2$-graded commutative rings defined as follows:    any non-nilpotent element $f\in\As$ defines in the usual way a \emph{basic open subset}
$D(f)\subset X$, and one defines $\Oc(D(f))=\As_f$, the localization of $\As$ at the multiplicative subsystem defined by $f$.
The locally ringed superspaces of this form are called \emph{affine superschemes}.

\begin{defin}\label{def:superscheme}
A superscheme is a locally ringed superspace $\Xcal=(X,\Oc_{\Xcal})$ which is locally isomorphic to the superspectrum of a superring. 

A superscheme $\Xcal=(X,\Oc_{\Xcal})$ is \emph{noetherian} if $X$ has a finite open cover $\{U\}$ such that every restriction $\Xcal_{\vert U}$ is the superspectrum of a noetherian superring.  
\end{defin}

It is easy to see that 
\begin{enumerate}
\item the bosonic reduction and the bosonic quotient of a superscheme are usual schemes;
\item $Gr (\Oc_{\Xcal})$ is coherent as an $\Oc_X$-module.
\end{enumerate}

\begin{defin}\label{def:dimension}  The odd dimension of a superscheme $\Xcal$ is the supremum of the odd dimensions of the local superrings $\Oc_{\Xcal,x}$ for all the points $x\in X$.  The even dimension of $\Xcal$ is the dimension of the scheme $X$. Both dimensions may be infinite. The dimension of $\Xcal$ is the pair $\dim\Xcal=(\operatorname{even-dim} \Xcal, \operatorname{odd-dim}\Xcal)$.
We say that a morphism $f\colon \Xcal\to \Sc$ of superschemes has \emph{relative dimension $(m,n)$} if the fibres $\Xcal_s$ are superschemes of dimension $(m,n)$.
\end{defin}

For a locally split superscheme the odd dimension equals the rank of the locally free $\Oc_X$-module $\Jc/\Jc^2$.

A $\Z_2$-graded sheaf $\M$ of $\Oc_\Xcal$-modules on a superscheme $\Xcal=(X,\Oc_\Xcal)$  can be regarded as a sheaf of abelian groups on $X$, and the cohomology groups of $\M$ as an $\Oc_\Xcal$-modules are the same as its cohomology groups as an abelian sheaf, so that 
  the usual cohomology vanishing \cite{Tohoku,Hart77} also holds for $\Z_2$-graded sheaves of $\Oc_\Xcal$-modules.

\begin{prop}\label{prop:cohombound} If  $\Xcal=(X,\Oc_\Xcal)$ is a noetherian superscheme, and $\M$   a $\Z_2$-graded sheaf of $\Oc_\Xcal$-modules, then
$H^i(X,\M)=0$  for every  $i>\dim X$.
\qed
\end{prop}

\subsection{Projective superspaces and quasicoherent sheaves on them}

Let $\Ps_\Z^m$  be the projective  $m$-space over $\Z$, i.e., $\Ps_\Z^m=\Proj \Z[x_0,\dots,x_m]\to \Spec \Z$.

\begin{defin} The projective $(m,n)$-superspace over $\Z$ is the   split superscheme over $\Z$
$$
\Ps_\Z^{m,n}=(\Ps_\Z^m, \textstyle{\bigwedge}_{\Oc_{\Ps_{\Z^m}}} \Oc_{\Ps_\Z^m}(-1)^{\oplus n})\,.
$$
If $\Sc$ is a superscheme,  the  projective $(m,n)$-superspace over $\Sc$ is
$$
\Ps_\Sc^{m,n}=\Ps_\Z^{m,n}\times_{\Spec \Z} \Sc\,.
$$
\end{defin}

Note that $\Ps_\Z^{0,n}$ is just the super affine space with $n$ odd variables (the superspectrum of the Grassmann algebra). 

When $\Sc=\SSpec \As$ is the superspectrum of a superring $\As$, we also use the notation $\Ps^{m,n}_\As$. Its underlying scheme is the projective $m$-space $\Ps_A^m=\Proj A[x_0,\dots,x_m]$, with $A=\As/\Jc_\As$, where {$\Jc_\As$} is the ideal generated by the odd elements.

One can define a bigraded $\As$-algebra 
\begin{equation}\label{eq:polyalgebra}
\Bs(m,n)=\As[x_0,\dots,x_m,\theta_1,\dots,\theta_n]
\end{equation}
where the $x_i$ are even and the $\theta_J$ are odd, and all variables are free of $\Z$-degree 1. It has   both a $\Z$-grading and a $\Z_2$-grading. We call it the \emph{free polynomial superalgebra}.
Let $\bar\Bs= \Bs(m,n)/H=\As[\xi_0,\dots,\xi_m,\eta_1,\dots,\eta_n]$, where $H$  is an ideal, homogeneous for both gradings. Then, mimicking the construction of the projective spectrum, we can define a \emph{projective superspectrum} 
$\Xcal=\SProj_\As\bar\Bs=(X,\Oc)$ by letting
\begin{itemize}
\item $X=\Proj A[\xi_0,\dots,\xi_m]$ (it is a projective scheme over $A$).
\item
The structure sheaf $\Oc$ is defined by $\Z$-homogeneous localization: on the open set $U_i=X-(\xi_i)_0$, 
\begin{equation}\label{eq:homloc}
\Oc(U_i)=\left\{ \frac{P_q(\xi_j,\eta_J)}{\xi_i^q}\bigm\vert \text{$P_q$ is $\Z$-homogeneous of degree $q$} \right\}.
\end{equation}

\end{itemize}
The following properties are standard.
\begin{prop}
There is a natural morphism $\SProj _\As\bar\Bs\to \SSpec \As$ which is a projection in the case $\As$ is even. 
Moreover, if $\bar\Bs= \Bs(m,n)$, we recover the projective superspace $\SProj_\As \Bs(m,n) \simeq\Ps_\As^{m,n}$, which is split when $\As$ is even.
 \qed
\end{prop}

The  homogeneous localization can be used, as in the classical (nonsuper) case, to construct quasi-coherent sheaves.
 
\begin{defin}\label{def:homloc}Let $M$ be a bigraded $\bar\Bs$-module.
 The sheaf of $\Z$-homogeneous localizations of $M$ is the sheaf $\widetilde M^h$ associated to the presheaf whose sections on $U_i=X-(\bar\xi_i)_0$ as above are given 
by 
$$
M^h_{\bar\xi_i}:= \{ \frac{m_q}{\bar\xi_i^q}\,\vert\, \text{$m_q\in M$ is $\Z$-homogeneous of degree $q$} \}\,.
$$
It is a quasi-coherent sheaf of $\Oc$-modules.
\end{defin}

Given a bigraded module $M$, we denote by $M_q$ the homogeneous component of $\Z$-degree $q$. For any integer $r$ we define a new bigraded module $M(r)$ by shifting the $\Z$ grading by $-r$, that is, $M(r)_s=M_{s+r}$.
\begin{defin}
For any integer $r$ the sheaf $\Oc(r)$  {on $\Xcal=\SProj_\As \bar\Bs$} is defined as
$$
\Oc(r):= \widetilde{{\bar \Bs}(r)}^h\,.
$$
It is a line bundle of rank $(1,0)$. For any quasi-coherent $\Z_2$-graded sheaf $\M$ on $\Xcal$, we define
$$
\M(r)=\M\otimes_\Oc \Oc(r)\,.
$$
\end{defin}
Note that the restriction to $\Xcal\hookrightarrow \Ps=\Ps_\As^{m,n}$ of the sheaf $\Oc_\Ps(r)$ is the sheaf $\Oc_\Xcal(r)$.

All the results in the remainder of this subsection are proved as in the classical case.
\begin{prop}\label{prop:mder}
{\ }
\begin{enumerate}
\item If $M= {\bar\Bs}^{p,q}= {\bar\Bs}^p\oplus \Pi {\bar\Bs}^q$, then $\widetilde M^h=\Oc^{p,q}$.
\item For any bigraded module and any integer $r$, there is an isomorphism
 $\widetilde{M[r]}^h\simeq \widetilde{M}^h(r)$. 
 \item If $M_r=0$ for $r\gg 0$, then $\widetilde{M}^h=0$.
  \item Any bihomogeneous morphism of $\bar\Bs$-modules $N\to M$ induces a morphism $\widetilde N^h\to \widetilde M^h$ of $\Oc$-modules.
Moreover, if $0\to N\to M \to P \to 0$ is an exact sequence of bihomogeneous morphisms, the sequence $0\to \widetilde N^h\to \widetilde M^h \to \widetilde P^h \to 0$
is also exact.
\item If a bihomogeneous morphism  $N\to M$ of $\bar\Bs$-modules induces an isomorphism $N_r \iso M_r$ for $r\gg 0$, then $\widetilde N^h\to \widetilde M^h$ is an isomorphism of sheaves.
\end{enumerate}
Assume now that $\As$ is noetherian. Then one has:
\begin{enumerate}\setcounter{enumi}{5}
\item
If $M$ is finitely generated, then $\widetilde M^h$ is coherent.
\item If $M$ is finitely generated and $\widetilde M^h=0$, then $M_r=0$ for $r\gg0$.
\item  If a bihomogeneous morphism  $N\to M$ of $\bar\Bs$-modules induces an isomorphism of sheaves  $\widetilde N^h\iso \widetilde M^h$, then $M_r \iso N_r$ for $r\gg 0$.
\qed\end{enumerate}
\end{prop}

For every quasi-coherent sheaf $\M$ of ($\Z_2$-graded) $\Oc$-modules, one can define a bigraded $\As$-module by 
$$
\Gamma_\ast(\M)=\bigoplus_{r\ge 0} \Gamma(X, \M(r))\,.
$$
\begin{prop}\label{prop:mder2}
\begin{enumerate}
\item There is an isomorphism 
$$
\widetilde{\Gamma_\ast(\M)}^h \iso \M\,,
$$
 that is, every quasi-coherent sheaf on a  projective superspectrum $\Xcal$ can be obtained by $\Z$-homogeneous localization of a bigraded $\As$-module. 
 \item If $\Gamma(X,\M(r))=0$ for $r\gg 0$, then $\M=0$.
 \item Let $g\colon \M \to \Nc$ be a morphism of quasi-coherent sheaves on $\Xcal$. If the induced morphism $\Gamma(X, M(r))\to\Gamma(X, \Nc(r))$ is an isomorphism for $r\gg 0$, then $g$ is an isomorphism, $g\colon \M \iso\Nc$. (Note that different bigraded $ {\bar\Bs}$-modules can induce the same quasi-coherent sheaf.)
 \end{enumerate}
 If $\As$ is noetherian, and $\M$ is  a coherent sheaf on $\Xcal$, 
  one also has:
 \begin{enumerate}\setcounter{enumi}{3}
 \item for $r\gg 0$, the sheaf $\M(r)$ is generated by its global sections, i.e., the natural morphism $\Gamma(X,\M(r))\otimes_\As  {\Oc}\to \M(r)$ is an epimorphism.
 \item  $\M$ is the cokernel of a morphism of locally free sheaves; more precisely, there is an exact sequence
 $$
 \bigoplus_{1\le j\le M} {\Oc}(-r_j)\oplus\bigoplus_{1\le j'\le M'} \Pi{\Oc}(-r'_{j'})\to  
 \bigoplus_{1\le i\le N} {\Oc}(-r_i)\oplus  \bigoplus_{1\le i'\le N'}\Pi{\Oc}(-r'_{i'}) 
 \to \M \to 0\,.
 $$
 \end{enumerate}
\end{prop}

\subsection{Supervector bundles, projective superbundles and supergrassmannians}
Gro\-thendieck's  classical  definition of vector bundle associated with a locally free sheaf generalizes directly to superschemes.

Let $\M$ be a coherent sheaf on a superscheme $\Sc$. For every affine open  sub-superscheme $\Ucal=\SSpec \As_U$ of $\Sc$, one has the graded-commutative $\As_U$-algebra $\SSym_{\As_U} \M_U$, where $\M_U=\Gamma(U,\M)${,} and an affine superscheme $\SSpec (\SSym_{\As_U} \M_U)$ over $\As_U$. For a covering of $\Sc$ by affine open  sub-superschemes, the symmetric algebras glue together to give a sheaf $\SSym_{\Sc} \M$ of graded-commutative  $\Oc_\Sc$-algebras and the corresponding superspectra glue   to give a relatively affine superscheme
$\SSpec(\SSym_{\Sc} \M) \to \Sc$
over $\Sc$.

\begin{defin}\label{def:supervb} The \emph{linear superscheme} associated to $\M$ is the $\Sc$-superscheme 
$$
\pi\colon \Vs(\M)=\SSpec(\SSym_{\Sc} \M)\to\Sc\,.
$$
The \emph{supervector bundle} associated to a \emph{locally free sheaf} $\M$ is the linear superscheme
$$
\widetilde\Vs(\M):=\Vs(\M^\ast)\to \Sc\,,
$$
associated to its dual.
\end{defin}

For every $\Sc$ superscheme $\phi\colon \Tc \to \Sc$  one has
\begin{equation}\label{eq:ssymbasechang}
\SSym_\Tc(\M_\Tc)\simeq \SSym_\Sc(\M)_\Tc\,,
\end{equation} 
where we denote, as it is customary, the pull-back  by $\phi$  by a subscript $\Tc$. One then has:
\begin{equation}\label{eq:svbbasechang}
\Vs(\M)\times_\Sc\Tc=\Vs(\M)_\Tc\simeq \Vs(\M_\Tc)\,,
\end{equation} 
that is, \emph{the formation of the  linear superscheme $\Vs(\M)$ is compatible with base change}.

The functor of  points of the supervector bundle associated to a locally free sheaf can be easily described. Note that
$$
\widetilde\Vs(\M)^\bullet(\Tc)= \Hom_\Sc(\Tc, \Vs(\M^\ast))=\Gamma(\widetilde\Vs(\M)_\Tc/\Tc)\,,
$$
where $\Gamma(\widetilde\Vs(\M)_\Tc/\Tc)$ is the set of sections $\sigma$ of the projection $\pi_\Tc\colon \widetilde\Vs(\M)_\Tc\to \Tc$,
so that there is a diagram
$$
\xymatrix{
\widetilde\Vs(\M)_\Tc=\widetilde\Vs(\M)\times_\Sc\Tc \ar[r]^(.7)\phi \ar@<1ex>[d]^{\pi_\Tc} & \widetilde\Vs(\M)\ar[d]^\pi \\
\Tc\ar[r]^\phi   \ar@/^1pc/@{_{(}.>}[u]^\sigma& \Sc
}
$$
One has:
\begin{prop} If $\M$ is a locally free sheaf on $\Sc$, one has
$$
\widetilde\Vs(\M)^\bullet(\Tc)=\Gamma(\widetilde\Vs(\M)_\Tc/\Tc)\simeq \Gamma(\Tc, \M_{+\Tc})\,.
$$
In particular, for every point $s\in S$, the fibre of $\Vs(\M^\ast)_s\simeq 
\Vs(\M^\ast)\times_\Sc \kappa(s)$ over $s$ is given by 
$$\
\widetilde\Vs(\M)_s\simeq \M_{+s} \,,
$$
where $\M_{+s}=\M_+\otimes_{\Oc_\Sc}\kappa(s)$. Thus, the fibres of the supervector bundle associated with a locally free sheaf $\M$ are the fibres of $\M_+$.
As a consequence, $\M_s \simeq \widetilde\Vs(\M\oplus \Pi \M)_s$.
\end{prop} 
\begin{proof} It follows from
\begin{multline*}
\Hom_\Sc(\Tc, \Vs(\M^\ast))\simeq \Hom_{\Oc_\Sc-alg}(\SSym_{\Oc_\Sc}(\M^\ast),\Oc_\Tc) \\ \simeq
\Hom_{\Oc_\Sc}(\M^\ast,\Oc_\Tc)_+ \simeq \Hom_{\Oc_\Sc}(\Oc_\Tc, \M_+) = \Gamma(\Tc,\M_+)\,.
\end{multline*}

\end{proof}
This is the motivation for the above definition of supervector bundle. 

The functor of points of $\Vs(\M)$  is not so easily described when $\M$ is not locally free. However we shall see in Proposition \ref{prop:homrep} a result in that direction. 

For every coherent sheaf $\M$ on $\Sc$, there is a surjection $\SSym_{\Oc_\Sc}(\M)\xrightarrow{\rho}\Oc_\Sc\to 0$ of algebras, whose kernel is the ideal generated by $\M$.
\begin{defin}\label{def:zerosec}
The zero section of $\pi\colon \Vs(\M)\to \Sc$ is the closed immersion 
$$
\sigma_0\colon \Sc \hookrightarrow \Vs(\M)
$$
of $\Sc$-superschemes induced by $\rho$.
\end{defin}

One can also define projective superbundles or, more generally, projective superspaces associated to a coherent sheaf.
For this, we simply consider the projective superspectrum
$$
{\pi}\colon\SProj (\SSym(\M)) \to \Sc\,,
$$
defined by glueing the projective superspectra of the bigraded $\As_U$-algebras  $\SSym_{\As_U} \M_U$ of a cover of $\Sc$ by affine open  sub-superschemes.

\begin{defin}\label{def:spbundle} The projective superscheme associated to $\M$ is the $\Sc$-superscheme
$$
\pi\colon\Ps(\M):= \SProj (\SSym(\M)) \to \Sc\,.
$$
The projective superbundle associated to a locally free $\Oc_\Sc$-module $\Ec$ is the projective superscheme associated to $\Ec^\ast$,
$$
\widetilde\Ps(\Ec):=\Ps(\Ec^\ast)\to\Sc\,.
$$
\end{defin}
Proceeding as in Definition \ref{def:homloc} we can define a natural ``hyperplane'' line bundle $\Oc_{\Ps(\M)}(1)$ on $\Ps(\M)$.
If
$$
0\to \Nc \to \M \to \Pc \to 0
$$
is an exact sequence of coherent sheaves on $\Sc$, there is an epimorphism of bigraded $\Oc_\Sc$-algebras
$$
\SSym_{\Oc_\Sc}(\M) \to \SSym_{\Oc_\Sc}(\Pc)\to 0\,,
$$
whose kernel is the bihomogeneous ideal $\Nc\cdot \SSym_{\Oc_\Sc}(\M)$ generated by $\Nc$. This induces a closed immersion of $\Sc$-schemes
$$
\Ps(\Pc)\hookrightarrow \Ps(\M)\,,
$$
which is an isomorphism with the closed  sub-superscheme defined, on every  open affine sub-superscheme $\Ucal$ of $\Sc$, by the homogeneous localization of $\Nc\cdot \SSym_{\Oc_\Sc}(\M)$ on the projective superschemes $\Ps(\M_\Ucal)$. One has 
$$
\Oc_{\Ps(\Pc)}(1)\iso \rest{\Oc_{\Ps(\M)}(1)}{\Ps(\Pc)}\,.
$$

The following results are straigthforward.
\begin{prop}\label{prop:spbundle} {\ }
\begin{enumerate}
\item If $\M=\Oc_\Sc^{m+1,n}$ is free, then $\Ps(\M)\simeq \Ps_\Sc^{m,n}$. 
\item
If $\M$ is quotient of a free sheaf, $\Oc_\Sc^{m+1,n}\to\M\to 0$, there is a closed immersion  $\Ps(\M)\hookrightarrow \Ps_\Sc^{m,n}$ of superschemes over $\Sc$.
\item If $\Ec$ is locally free of rank $(m+1,n)$, the projective superbundle $\Ps(\Ec^\ast)\to\Sc$ is locally a projective superspace over the base, that is, there is a cover of $\Sc$ by affine open  sub-superschemes $\Ucal=\SSpec \As$ such that $\widetilde\Ps(\Ec)_\Ucal\simeq \Ps_{\As}^{m,n}$ as $\As$-superschemes.
\item  {For every coherent sheaf $\M$ there is a cover of $\Sc$ by affine open  sub-superschemes $\Ucal=\SSpec \As$ such that there is a closed immersion  $\Ps(\M)_\Ucal\hookrightarrow \Ps_\As^{m,n}$ of superschemes over $\As$.}
\end{enumerate}
\qed
\end{prop}

Let $\Ec$ be a locally free sheaf on $\Sc$ and let $\pi\colon \Ps(\Ec)\to \Sc$ be the corresponding projective superbundle.  The even associated line  bundle $\Oc(1)$ is relatively generated by its global sections $\pi_\ast\Oc(1)\simeq \Ec$, that is, there is an epimorphism
$$
\pi^\ast \Ec \to \Oc(1)\to 0\,.
$$
For every superscheme  $\phi\colon \Tc \to \Sc$ over $\Sc$, every morphim $f\colon \Tc \to \Ps(\Ec)$ of $\Sc$-superschmes gives rise to an even line bundle $\Lcl= f^\ast \Oc(1)$, which is a quotient
\begin{equation}\label{eq:quotlb}
\phi^\ast\Ec\simeq f^\ast \pi^\ast\Ec \xrightarrow{\varpi} \Lcl=f^\ast\Oc(1) \to 0\,,
\end{equation}
of $\phi^\ast\Ec$.
Conversely, if $\Lcl$ is an even line bundle on $\Tc$ which a quotient of $\phi^\ast\Ec$ as in Equation \eqref{eq:quotlb}, we have an epimorphism of bigraded $\Oc_\Sc$-algebras
$$
\SSym_{\Oc_\Sc}(\phi^\ast(\Ec))\xrightarrow{\varpi} \SSym_{\Oc_\Sc}(\Lcl)\to 0\,,
$$
and a morphism $f\colon \Tc \to \Ps(\Ec)$ of $\Sc$-superschemes
\begin{equation}\label{eq:projmor}
\xymatrix{
\Tc\simeq \SProj \SSym_{\Oc_\Sc}(\Lcl) \ar@{^(->}[r]\ar[rd]_\phi \ar@/^2pc/[rr]^f & \Ps(\phi^\ast\Ec)\simeq \Ps(\Ec)\times_\Sc\Tc\ar[d]\ar[r]& \Ps(\Ec)\ar[ld]^\pi \\
& \Sc & }
\end{equation}
such that $\Lcl\simeq f^\ast\Oc(1)$. Thus
\begin{prop}\label{prop:projmor} The functor $\Sgrassf(\Ec,(1,0))$ on $\Sc$-superschemes which associates to any morphism $\phi\colon\Tc \to \Sc$ the set of the quotients of $\phi^\ast\Ec\to \Lcl\to 0$, where $\Lcl$ is an even line bundle,  is \emph{representable} by the projective superbundle
$$
\widetilde\Ps(\Ec^\ast)=\Ps(\Ec)\,.
$$
\qed
\end{prop}

The above Proposition can be reformulated to say that the projective superbundles are supergrassmannians of a   particular kind.
More generally, for every pair $(p, q) $ of integer numbers, we have:
\begin{defin}\label{def:supergrass} Let $\Ec$ be a locally free sheaf on a superscheme $\Sc$. The supergrassmannian functor of locally free quotients of rank $(p,q)$ of $\Ec$ on $\Sc$ superschemes is the functor which associates to every superscheme $\phi\colon \Tc \to \Sc$ over $\Sc$ the family
$$
\Sgrassf(\Ec,(p,q))(\Tc)\,,
$$
of all the quotients $\phi^\ast\Ec \to \Ec_{p,q}\to 0$ of $\phi^\ast\Ec$ that are locally free of rank $(p,q)$.
\end{defin}

In the classical case, the construction of the grassmannian scheme, that is, the proof of the representability of the grassmannian functor, was first done by Grothendieck using ``big cells'', but the easiest way to do it is by means of the Pl\"ucker embedding into a projective space. 
In supergeometry, the proof of the existence of the supergrassmannian is due to Manin \cite[\S 3]{Ma88} using the super-analogue of the classical ``big cells''.  {His proof also shows   that the supergrassmannian is of finite type.} In this situation, the proof using an analogue of the  Pl\"ucker embedding is not available, since it is known that the supergrassmannians may fail to be superprojective  \cite{PenSkor85}. 
The precise statement is Theorem 1.3.10 of \cite{Ma88}, which we recall here for convenience:

\begin{prop}\label{prop:supergrass} Let $(p,q)$ be a pair of integer numbers and $\Ec$ a locally free sheaf of rank $(c+p,d+q)$ on a superscheme $\Sc$. 
 The supergrassmannian functor $\Sgrassf(\Ec,(p,q))$ is representable by an $\Sc$-superscheme $\Sgrass(\Ec,(p,q))$. Moreover, the natural morphism $\Sgrass(\Ec,(p,q))\to \Sc$ is proper of relative dimension $(cp+dq,cq+dp)$.
\end{prop}
\begin{proof} We need only to prove the properness. This follows from Proposition \ref{prop:proper} since $\Sgrass(\Ec,(p,q))$  {is of finite type over $\Sc$ and} the  underlying scheme to $\Sgrass(\Ec,(p,q))$ is the fibre product $\grass(\Ec_0,p,)\times_S \grass(\Ec_1,q)$ of two ordinary grassmannians.
\end{proof}

\begin{remark}\label{rem:subspacesgrass}  {Recall that a \emph{subbundle} of a locally free sheaf $\Ec$ is a locally free subsheaf of $\Ec$ such that $\Ec/\F$ is locally free. We shall freely use  this terminology through this paper.}
The
supergrassmanian of quotients $\Sgrass(\Ec,(p,q))$ can also be  seen as a supergrassmannian of  {subbundles} of $\Ec$, that is, it represents the functor $\Sgrassf((c,d),\Ec)$ which associates to every $\Sc$-superscheme $\Tc$ the family of all rank $(c,d)$   subbundles  of $\Ec_\Tc$. At times it may be  convenient to write
$$
\Sgrass((c,d),\Ec)\simeq\Sgrass(\Ec,(p,q))
$$
if we want to describe supergrassmannians as superschemes that parametrize  the rank $(c,d)$  subbundles of  of $\Ec$.
\end{remark}

\subsubsection{Flag superschemes} As in the ordinary case, one can construct \emph{flag superschemes}.
We equip $\Z\times\Z$  with the following partial order:
\begin{equation}\label{eq:order}
(h_0,h_1)\leq (h'_0,h'_1) \iff  \text{$h_0\leq h'_0$ and $h_1\leq h'_1$}
\end{equation}

Let $\Ec$ be a locally free sheaf of rank $(m,n)$ of a superscheme $\Sc$.
\begin{defin}\label{def:flag}
For every choice of an integrer $i$ with $1\leq i\leq m+n$ and of a  collection of dimensions $(0,0)\leq (m_1,n_1)\leq\dots\leq (m_i,n_i)\leq (m,n)$, the   superflag functor of $\Ec$ associates to every $\Sc$-superscheme $\Tc$ the family
$$
\SFlagf((m_1,n_1),\dots,(m_i,n_i),\Ec)(\Tc)
$$
of all the flags
$$
\F_1\subset\dots \F_i\subset \Ec_{\Tc}
$$
where $\F_j$ is a  rank $(m_j,n_j)$ subbundle of $\Ec_{\Tc}$.
\end{defin}

Thus, there is an immersion of functors on the category of $\Sc$-superschemes 
$$
\SFlagf((m_1,n_1),\dots,(m_i,n_i),\Ec) \hookrightarrow \Sgrassf((m_1,n_1),\Ec)\times \dots \times \Sgrassf((m_i,n_i),\Ec)\,.
$$
(cf.--Remark \ref{rem:subspacesgrass}). As in the ordinary case, one proves that this is representable by closed immersions. Then one has:

\begin{prop}\label{prop:flag} 
The superflag functor $\SFlagf((m_1,n_1),\dots,(m_i,n_i),\Ec)$ is representable by a closed sub-superscheme $\SFlag((m_1,n_1),\dots,(m_i,n_i),\Ec)$ (the flag superscheme) of a product of supergrassmanians of $\Ec$.
\qed
\end{prop}

\subsection{Superprojective morphisms}

Let $\bar\Bs= \Bs(m,n)/H=\As[\xi_0,\dots,\xi_m,\eta_1,\dots,\eta_n]$ as above, where $H$ is an  ideal, homogeneous for both gradings. 
The projection $\Bs(m,n)\to \bar\Bs= \Bs(m,n)/H$ induces a closed immersion of superschemes over $\As$
$$
\delta\colon  \SProj _\As\bar\Bs  \hookrightarrow \Ps_\As^{m,n}
$$
which identifies $\Xcal$ with the closed sub-superscheme defined by the ideal sheaf $\Hc$ obtained by $\Z$-homogeneous localization of the graded ideal $H$.
Conversely, if $\Xcal\hookrightarrow \Ps_\As^{m,n}$ is a closed immersion of superschemes defined by a quasi-coherent ideal $\Hc$ of $\Oc=\Oc_{\Ps_\As^{m,n}}$, the $\As$-module
$$
H=\Gamma_\ast(\Hc)=\bigoplus_{r\ge 0}\Gamma(X,\Hc(r))\,,
$$
is a homogeneous ideal for the two gradings. Since  $\widetilde{\Gamma_\ast(\Hc)}^h \iso \Hc$ by Proposition \ref{prop:mder2}, we see that 
$
\Xcal\simeq \SProj \bar\Bs
$. This motivates the following definition. 
\begin{defin}\label{def:superproj} A \emph{projective superscheme}\footnote{Superprojective superscheme in  \cite{CaNo18,MoZh19}.} over $\As$ is an $\As$-superscheme of the form
$$
\Xcal = \SProj _\As\bar\Bs \xrightarrow{f} \SSpec \As\,,
$$
for a bigraded $\As$-algebra $\bar\Bs= \Bs(m,n)/H=\As[\xi_0,\dots,\xi_m,\eta_1,\dots,\eta_n]$, with $H$ a homogeneous ideal of $\Bs(m,n)$ for the two gradings. 
Equivalently, a superscheme $\Xcal$  over $\As$ is projective if it is endowed with a closed immersion of $\As$-superschemes into a projective superspace $\Ps_\As^{m,n}$
 over $\As$.

A morphism $f\colon \Xcal \to \Sc$ of superschemes is \emph{locally superprojective} if there is a covering of the base $\Sc$ by affine superschemes $\Ucal=\SSpec \As$ such that each restriction $f_\Ucal\colon \Xcal_\Ucal \to \Ucal$ is a projective superscheme $\Xcal_\Ucal\simeq \SProj\bar\Bs$ over $\As$.
\end{defin}
Note that, if $f\colon \Xcal \to \SSpec \As$  is a locally superprojective morphism, $\Xcal$ may fail to be a projective superscheme over $\As$. 

A first example of a locally superprojective morphism is the projective superscheme\break  $\pi\colon \Ps(\M)\to \Sc$ (Definition \ref{def:spbundle}) associated to a coherent sheaf $\M$ on $\Sc$, as Proposition \ref{prop:spbundle} shows.

Then, following Grothendieck, we define:
\begin{defin}\label{def:superprojmor} A morphism $f\colon \Xcal\to\Sc$ of superschemes is superprojective if there is a coherent sheaf $\M$ on $\Sc$ such that $f$ factors through a closed immersion $\Xcal\hookrightarrow \Ps(\M)$ and the natural projection $\pi\colon \Ps(\M)\to \Sc$.

 A morphism $f\colon \Xcal\to\Sc$ of superschemes is quasi-superprojective if it is the composition of an open immersion $\Xcal\hookrightarrow \bar\Xcal$ and a superprojective morphism $\bar\Xcal \to\Sc$.
\end{defin}

\begin{remark}\label{rem:ample}
 {An} immersion $\Xcal\hookrightarrow \Ps(\M)$ gives rise to a  line bundle $$\Oc_\Xcal(1)=\rest{\Oc_{\Ps(\M)}(1)}{\Xcal}.$$ We then say that $\Oc_\Xcal(1)$ is a \emph{relatively very ample line bundle} on $f\colon\Xcal\to \Sc$. Since  there are different {immersions} in different projective superschemes $\Ps(\M')$, a {quasi-}superprojective morphism has many different relatively ample line bundles. Whenever we say that $f\colon \Xcal\to\Sc$ is a {quasi-}superprojective morphism with a relatively very ample line bundle $\Oc_\Xcal(1)$, we mean that we have chosen {an} immersion $\Xcal\hookrightarrow \Ps(\M)$ and  $\Oc_\Xcal(1)=\rest{\Oc_{\Ps(\M)}(1)}{\Xcal}$. 
When the base $\Sc$ is locally Noetherian then  for any {quasi-}projective  superscheme $(\Xcal,\Oc_\Xcal(1))$ over $\Sc$ there exists, locally over $\Sc$,  an embedding of $\Xcal$ into
a relative superprojective space inducing $\Oc_\Xcal(1)$. Indeed, it is enough to represent locally $\M$ as a quotient of $\Oc_{\Sc}^{(m,n)}$.
  
The classical theory of relatively ample and very ample line bundles and their relationship with embeddings in projective superschemes or superbundles can be extended to the super setting, but we shall not  develop that theory in these notes. 
\end{remark}

\begin{lemma}
Let {$\Oc_{\Xcal/\Sc}(1)$} be a relatively very ample line bundle on $\Xcal/\Sc$. Then for any $r>0$ the line bundle {$\Oc_{\Xcal/\Sc}(r)$}  is also relatively very ample.
\end{lemma}

\begin{proof} This corresponds to the natural embedding of $\Ps(\M)$ into $\Ps(\SSym^r(\M))$ associated with the surjection $\pi^*\SSym^r(\M)\to \Oc_{\Ps(\M)}(r)$. 
\end{proof}

\begin{defin}\label{def:superprojective}
A morphism $f\colon\Xcal \to \Sc$ is called strongly superprojective if $f$ factors through a closed immersion $\Xcal \hookrightarrow \widetilde\Ps(\Ec)$ for a projective superbundle 
$\widetilde\Ps(\Ec)\to \Sc$ associated with a locally free sheaf $\Ec$ on $\Sc$.
\end{defin}

In the case when $\Sc$ is affine (remember that we are assuming the noetherian condition) or superprojective over a graded-commutative ring ($\Z$ or a field $k$, for instance), then every coherent sheaf is the quotient of a free sheaf (Proposition \ref{prop:mder2}), and thus, strong superprojectivity is equivalent to superprojectivity in that case.
Also, as in the classical case, a flat superprojective morphism over a noetherian base is strongly superprojective (see Corollary \ref{cor:strongly-superproj}).

Let us say that a line bundle (of rank $(1,0)$) {$\Lcl$} over $\Xcal$ is \emph{strongly relatively ample over } $\Sc$ 
if there exists $n>0$, 
a locally free sheaf $\Ec$ over $\Sc$ and a closed {immersion} $\phi\colon\Xcal\hookrightarrow \Ps(\Ec)$ over {$\Sc$},
such that {$\Lcl^n\simeq\phi^\ast\Oc_{\Xcal/\Sc}(1)$}. One has the following useful criterion of strong superprojectivity of $\Xcal\to \Sc$
(extending a criterion for smooth families in \cite{LePoWel90}). Let $X\to S$ denote the
corresponding morphism between the bosonizations.

\begin{prop}\label{prop:strong-superproj} \cite[Prop.\ A.2]{FKP20}
Let $f\colon \Xcal\to \Sc$ be a flat morphism of noetherian superschemes. 
If a line bundle {$\Lcl$} on $\Xcal$ is such that $\rest{\Lcl}{X}$ is strongly relatively ample over $S$ then {$\Lcl$} is strongly relatively ample over $\Sc$.\qed 
\end{prop}

\subsection{Projectivity of some partial flag superschemes}
\label{ss:flag}

For further use we note that some flag superschemes, contrary to supergrassmannians,  are projective. 
Let $\Ec$ be a locally free sheaf of rank $(m,n)$ on a superscheme $\Sc$. 
For $a\le m, b\le n$, let us consider the relative partial flag superscheme $$\mathbb F:=\SFlag((a,0),(a,b);\Ec)$$
(see Proposition \ref{prop:flag}).

\begin{prop}\label{prop:superflag}
$\mathbb F$ is strongly superprojective over $\Sc$.
\end{prop}

\begin{proof}
It is easy to see that 
$$\mathbb F_{\bos}\simeq \grass(a;  E_+)\times_{S}\grass(b,  E_-),$$
where $E_+=(\rest{\Ec}{S})_+$ and $E_-=(\rest{\Ec}{S})_-$ are the components of the restrictions of $\Ec$ to the underlying scheme $S$ and
$\grass(a;  E_+)$ and $\grass(b,  E_-)$ are their relative grassmannians.
By Proposition \ref{prop:strong-superproj}, it is enough to prove that there exists an even  line bundle on $\mathbb F$   whose restriction
to the underlying scheme $\mathbb F_{\bos}$ is strongly relatively ample. 

Denote by $\pi\colon \mathbb F\to \Sc$ the natural projection. Let
$$
\Ec(a,0)\subset \Ec(a,b)\subset \pi^\ast \Ec$$
be the tautological subbundles over $\mathbb F$ and let
$$E_{+,a}\subset  E_+\,, \quad E_{-,b}\subset  E_-$$
be the tautological bundles on the respective bosonic grassmannians (see Remark \ref{rem:subspacesgrass}).
Then we have natural isomorphisms
$$
\Ec(a,0)_{\vert \mathbb F_{\bos}}\simeq p_1^\ast E_{+,a}\,, \quad
\Pi(\Ec(a,b)/\Ec(a,0))_{\vert \mathbb F_{\bos}}\simeq p_2^\ast E_{-,b}
$$
where $p_1$ and $p_2$ are the projections of $\mathbb F_{\bos}$ onto its factors.
But it is well known that the line bundles
$$
\Lcl_+:=\det^{-1}(E_{+,a}) \ \text{ and } \Lcl_-=\det^{-1}(E_{-,b})
$$
are strongly relatively ample on $\grass(a;E_+)$ and $\grass(b;E_-)$, respectively. 
Hence, $p_1^*\Lcl_+\otimes p_2^\ast\Lcl_-$ is strongly relatively ample on $\mathbb F_{\bos}$. It remains to observe that this extends to
the line bundle
$$\Ber^{-1}\Ec(a,0)\otimes \Ber^{-1}\Pi(\Ec(a,b))/\Ec(a,0))
$$
on $\mathbb F$.
\end{proof}

\begin{corol} \label{cor:grassproj}
Let $\Xcal$ be a closed sub-superscheme 
of a relative supergrassmannian over a base superscheme $\Sc$. Then there exists a smooth morphism
$\widetilde{\Xcal}\to \Xcal$ of relative dimension $(0,n)$ (in particular this morphism is finite) such that $\widetilde{\Xcal}$ is strongly superprojective 
over $\Sc$.
\end{corol}

\noindent (The notion of smooth superscheme morphism is given in Section \ref{s:smoothmorphisms}.)
\begin{proof} 
We just have to prove that for the relative partial flag superscheme $\mathbb F$ considered in Proposition \ref{prop:superflag} the natural projection
$$\mathbb F\to \Sgrass((a,b);\Ec)$$
is smooth of relative dimension $(0,n)$. But $\mathbb F$ can be identified with
the relative supergrassmannian of  rank $(a,0)$  subbundles  associated with the tautological subbundle $\Wc$ of rank $(a,b)$ on $\Sgrass((a,b),\Ec)$, so
we see that $\mathbb F$ is smooth of relative dimension $(0,ab)$ over $\Sgrass((a,b),\Ec)$.
\end{proof}

\subsection{Cohomology of even line bundles on the projective superspace}\label{ss:superproj}

Let $\As$ be a noetherian  superring and $\Bs= \As[x_0\,\dots,x_m,\theta_1,\dots, \theta_n]$. {In this Subsection we   denote by $\Xcal$ the  projective $(m,n)$-superspace $\Ps_\As^{m,n}=\SProj_\As \Bs$ over $\As$, so that the underlying scheme $X$ is the projective space $\Ps^m_A$.}

If we denote by $\Bs'$ the algebra $$\Bs'=\Bs(m-1,n)= \As[x_0,\dots,x_{m-1},\theta_1,\dots, \theta_n]$$  there is an exact sequence of bigraded $\Bs$-modules
$$
0 \to \Bs[-1]\xrightarrow{\cdot x_m}\Bs \to \Bs'\to 0
$$
which induces a closed immersion
$$
\delta\colon \Xcal':=\SProj_\As \Bs' \hookrightarrow \Xcal
$$ 
of the projective $(m-1,n)$-superspace $\Xcal'\simeq \Ps_\As^{m-1,n}$ as the hyperplane defined by the homogeneous ideal $(x_m)$.
Moreover, for every integer number $r$, there is an exact sequence of bigraded $\Bs$-modules
$$
0 \to \Bs[r-1]\xrightarrow{\cdot x_m}\Bs[r] \to \Bs'[r]\to 0
$$
which induces an exact sequence of sheaves
\begin{equation}\label{eq:odenseq}
0 \to \Oc(r-1) \xrightarrow{\cdot x_m} \Oc (r) \to \Oc'(r) \to 0\,,
\end{equation}
where we have written $\Oc=\Oc_\Xcal$ and $\Oc'=\delta_\ast \Oc_{\Xcal'}$.

\begin{lemma}\label{lem:odensec} {\ }
\begin{enumerate}
\item
If $m>0$, for every integer $r$ the natural morphism $\Bs_r\to H^0(X,\Oc(r))=\Gamma(X,\Oc(r))$ is an isomorphism. So, $H^0(X,\Oc(r))=0$ for $r<0$.
\item If $m=0$,  then  $H^0(X,\Oc(r))= x_0^r \cdot\bigwedge_\As E_\theta$ for any integer $r$, where $E_\theta$ is the free $\As$-module generated by $(\theta_1/x_0,\dots,\theta_n/x_0)$.
\item The sequence of global sections of 
{Equation \eqref{eq:odenseq}} is exact when either 
\begin{itemize}
\item $m>1$ and $r\ge 0$, or
\item $m=1$ and $r\ge n$. 
\end{itemize}

For $m=1$ and $0\leq r<n$, the image of $\Bs_r=H^0(X, \Oc(r)) \to H^0(X,\Oc'(r))$ is the free submodule  
$\bigoplus_{p=0}^r \bigwedge_\As^p E_\theta$.
\end{enumerate}
\end{lemma}
\begin{proof}
(1)  is proved as in the classical case. For (2), recall that when $m=0$, $\Xcal=\Ps_\As^{0,n}$ is the affine superscheme associated with $\bigwedge_\As E_\theta$. Then  the line bundle $\Oc(r)$ is free over $\As$ with basis $x_0^r\cdot(\theta_1/x_0,\dots,\theta_n/x_0)$ for any integer $r$ (positive or negative).

(3) The case $m>1$ and $r\ge 0$ follows directly from (1). 
When $m=1$, $\Xcal'=\Ps_\As^{0,n}$ so that $H^0(X,\Oc'(r))\simeq  x_0^r \cdot\bigwedge_\As E_\theta$. The image of   $\Bs_r=H^0(X, \Oc(r)) \to H^0(X,\Oc'(r))$ consists of the elements of the form 
\begin{align*}
P_r(x_0,\theta_1,\dots,\theta_n)&=\sum_{0\le p\le r}\sum_{i_1<\dots<i_p} a_{i_1,\dots,i_p} x_0^{r-p}\theta_{i_1}\cdot\dots\cdot\theta_{i_p}
\\
&= x_0^r \sum_{0\le p\le r}\sum_{i_1<\dots<i_p} a_{i_1,\dots,i_p} \frac{\theta_{i_1}}{x_0}\cdot\dots\cdot\frac{\theta_{i_p}}{x_0}\,.
\end{align*}
Then, the image equals $\bigoplus_{p=0}^r \bigwedge_\As^p E_\theta$ when $r<n$ and equals the total space $H^0(X,\Oc'(r))$ if $r\ge n$.

\end{proof}

Let us introduce the numbers
\begin{equation}\label{eq:dimsecoden}
\begin{aligned}
h_{(m,n)}(r)_0 &= \sum_{{\scriptsize \begin{aligned}0\le p&\le n, p \text{\ even} \\ 0\leq p&+s=r \end{aligned} }} \binom{m+s}{m}\binom np \\
h_{(m,n)}(r)_1& = \sum_{{\scriptsize\begin{aligned}0< p&\le n, p \text{\ odd} \\ 0\leq p&+s= r \end{aligned} }} \binom{m+s}{m}\binom np\,, \quad \text{and} \\
h_{(m,n)}(r)&=(h_{(m,n)}(r)_0,h_{(m,n)}(r)_1)\,.
\end{aligned}
\end{equation}
for $r\ge 0$. 
Note for future use that, for every value of the dimension $(m,n)$, $h_{(m,n)}(r)_0$ and $h_{(m,n)}(r)_1$ are \emph{polynomials in $r$ with rational coefficients}.
When $m>0$, these numbers give for each $r$ the rank of the component of $\Z$-degree $r$ of $\Bs=\Bs(m,n)$:
\begin{equation}\label{eq:dimsecoden2}
\rk_\As \Bs(m,n)_r=h_{(m,n)}(r)=(h_{(m,n)}(r)_0, h_{(m,n)}(r)_1)\,.
\end{equation}
In the case $m=0$ we get the rank of the image of the morphism $\Bs_r=H^0(X, \Oc(r)) \to H^0(X,\Oc'(r))$ (Lemma \ref{lem:odensec}):
\begin{equation}\label{eq:dimsecoden3}
\rk_\As \Ima (H^0(X, \Oc(r)) \to H^0(X,\Oc'(r))=(h_{(0,n)}(r)_0, h_{(0,n)}(r)_1)\,.
\end{equation}

Let $\Ucal=\Xcal-\Xcal '$ be the complementary open sub-superscheme of the hyperplane $x_m=0$. The injective morphism  $ \Oc \xrightarrow{\cdot x_m^r} \Oc(r)$ (Equation \eqref{eq:odenseq})  induces an isomorphism $\Oc_\Ucal\iso \Oc(r)_\Ucal$. Then we have injective morphisms $0\to \Oc(r) \to \iota_\ast \Oc_\Ucal$ where $\iota\colon \Ucal \to \Xcal$ is the open immersion of $\Ucal$ into $\Xcal$, consisting of the composition of the product by $x_m^{-r}$ and the restriction to $\Ucal$. Moreover, one has a commutative diagram
$$
\xymatrix{ \Oc(r) \ar@{^{(}->}[r] &\iota_\ast \Oc_\Ucal\\
\Oc(r-1)\ar[u]^{\cdot x_m} \ar@{^{(}->}[ur]&
}
$$
This gives rise to an injective morphism $\varinjlim_{r}\Oc(r) \to \iota_\ast\Oc_\Ucal$. As in the classical case one proves:
\begin{lemma}\label{lem:limoden} The above morphism induces an isomorphism
$\varinjlim_{r}\Oc(r) \iso \iota_\ast\Oc_\Ucal$.
\qed
\end{lemma}

We now get (see also \cite{CaNo18}):
\begin{prop}[Cohomology of the sheaves $\Oc(r)$]\label{prop:cohomoden}   
\begin{enumerate} 
\item\label{1} If $r\ge n-1$, the sheaf $\Oc(r)$ is acyclic and  $H^0(X,\Oc(r))$ is a free $\As$-module of rank $h_{(m,n)}(r)$ (See Equation \eqref{eq:dimsecoden}).
\item If $0\le r < n-1$, then $H^i(X,\Oc(r))=0$ for $i\neq 0,m$, and these groups are free $\As$-modules.
\item
If $r> 0$, $H^i(X,\Oc(-r))=0$ for $i\neq m$ and $H^m(X,\Oc(-r))$ is a free $\As$-module.
\end{enumerate}
\end{prop}
\begin{proof}
(1)  The statement about the rank was already proved. To prove the acyclicity, one proceeds by induction on the even dimension $m$, the case $m=0$ being trivial. Consider the exact sequence  
\begin{equation}\label{eq:odenseq2}
0 \to \Oc(r) \to \Oc (r+1) \to \Oc'(r+1) \to 0\,,
\end{equation}
(Equation \eqref{eq:odenseq}). Since $r+1\ge n$, the sequence of global sections is exact by Lemma \ref{lem:odensec}. Moreover, $\Oc'(r+1)$ is acyclic by induction on $m$, and we have isomorphisms
$$
H^i(X,\Oc(r))\iso H^i(X,\Oc(r+1))\,,\quad \text{for $i>0$}\,.
$$
Lemma \ref{lem:limoden} now gives
$$
H^i(X,\Oc(r))\iso H^i(X,\iota_\ast\Oc_\Ucal)=0\,,\quad \text{for $i>0$}\,,
$$
where the first equality is due to the fact that the cohomology commutes with direct limits, $X$ being a noetherian space, and the second follows because $\iota$ is an affine morphism and $\Ucal$ is affine.

(2) We also proceed by induction on $m\ge 1$.

Let us consider first the case $m=1$. We need only to prove that $H^1(X,\Oc(r))$ is a free $\As$-module for $0\leq r\le n-2$. By Lemma \ref{lem:odensec} we have an exact sequence of $\As$-modules
$$
0\to \bigoplus_{r+1\leq p\leq n}{\bigwedge}^p E_\theta \to H^1(X,\Oc(r))\to H^1(X,\Oc(r+1))\to 0\,.
$$
Since $H^1(X,\Oc(n-1))=0$ by (1), $H^1(X,\Oc(n-2))$ is free. By descending induction we get
$$
H^1(X,\Oc(r))\simeq \Bigl(\bigoplus_{r+1\leq p\leq n}{\bigwedge}^p E_\theta\Bigr)\oplus H^1(X,\Oc(r+1))
$$
for $0\leq r\le n-2$, and then
\begin{equation}\label{eq:h1odenP1}
\begin{aligned}
H^1(X,\Oc(r))\simeq &{\bigwedge}^{r+2}E_\theta\oplus({\bigwedge}^{r+3}E_\theta)^2\oplus\dots\oplus ({\bigwedge}^{n}E_\theta)^{n-r-1} \\
\simeq &\bigoplus_{r+1\le q\le n} ({\bigwedge}^q E_\theta)^{q-r-1} \,,\quad 0\leq r\leq n-2\,.
\end{aligned}
\end{equation}

For $m\ge 2$, by Lemma \ref{lem:odensec}  and induction on $m$, Equation \eqref{eq:odenseq} yields $H^i(X, \Oc(r))=0$ for $i\neq 0,m-1,m$ and an exact sequence
$$
\begin{aligned}
0\to H^{m-1}(X,\Oc(r))\to &H^{m-1}(X,\Oc(r+1))\to H^{m-1}(X,\Oc'(r+1))\to\\
\to &H^{m}(X,\Oc(r))\to H^{m}(X,\Oc(r+1))\to 0\,.
\end{aligned}
$$
Since $H^{m-1}(X,\Oc(n-1))=0$ by (\ref{1}), by descending induction we have $H^{m-1}(X,\Oc(r)=0$ for $0\leq r\le n-2$, and an isomorphism
$$
H^m(X,\Oc(r))\simeq H^{m-1}(X,\Oc'(r+1))\oplus H^m(X,\Oc(r+1))
$$
as $H^{m-1}(X,\Oc'(r+1))$ is free by induction. Moreover,  $$H^m(X,\Oc(n-1))=0\quad\text{and}\quad H^{m-1}(X,\Oc'(n-1))=0$$ by  (\ref{1}), so that (remember that $m\ge 2$):
\begin{equation}\label{eq:hmodenPm}
H^m(X,\Oc(r))\simeq \begin{cases} 0& r=n-2 \\H^{m-1}(X,\Oc'(r+1))& 0\leq r < n-2\,.
\end{cases}
\end{equation}

(3) Since $H^0(X,\Oc(-r))=0$  for $r>0$ by Lemma \ref{lem:odensec}, we need to prove that $H^i(X,\Oc(-r))$ $=0$ for $i\neq 0,m$ and $r\ge 0$; the inclusion of the case $r=0$, which is part of  (\ref{1}), makes the proof easier. We proceed by induction on $m$, the case $m=0$ being obvious. For $m\ge 1$, do induction on $r\ge 0$. For $r\ge 1$ we consider the sequence
$$
0\to \Oc(-r) \to \Oc(-(r-1))\to \Oc'(-(r-1))\to 0\,,
$$
By induction on $m$, $H^i(X, \Oc'(-(r-1)))=0$ for $i\neq m-1$ and $H^{m-1}(X,\Oc'(-(r-1)))$ is a free $\As$-module.

By induction on $r$, $H^i(X,\Oc(-(r-1)))=0$ for $i\neq m$ and $H^m(X,\Oc(-(r-1)))$ is a free $\As$-module. Thus, $H^i(X,\Oc(-r))=0$ for $i\neq m$ (the case $i=0$ this is part of  Lemma \ref{lem:odensec}), and there is an exact sequence of $\As$-modules
$$
0 \to H^{m-1}(X,\Oc'(-(r-1))) \to H^m(X,\Oc(-r)) \to H^m(X, \Oc(-(r-1))) \to 0\,.
$$
Since the first and the third modules are free, one has 
\begin{equation}\label{eq:hmodenPm2}
H^m(X,\Oc(-r))\simeq H^{m-1}(X,\Oc'(-(r-1)))\oplus H^m(X, \Oc(-(r-1)))
\end{equation}
and one finishes.
\end{proof}

\begin{remark}
The ranks of   the free $\As$-modules $H^i(X,\Oc(r))$ can be computed from Equations \eqref{eq:h1odenP1}, \eqref{eq:hmodenPm} and \eqref{eq:hmodenPm2}.
\end{remark}

\subsection{Cohomology of coherent sheaves on projective superschemes}

Let $f\colon \Xcal\to \Sc$ a locally superprojective morphism (Definition \ref{def:superproj}) over a locally noetherian superscheme $\Sc$.

The following result is the super analogue of a classical theorem by Serre.
\begin{thm}[Serre's Theorem]\label{thm:serre} Let $\M$ be a coherent sheaf on $\Xcal$.
\begin{enumerate}
\item The higher direct images $R^i f_\ast \M$ are coherent sheaves on $\Sc$.
\end{enumerate}
Moreover, if $\Sc$ is noetherian {and $f$ is superprojective}:
\begin{enumerate}\setcounter{enumi}{1}
\item there exists an integer $r_0$, depending only on $\M$, such that $R^i f_{\ast} \M(r)=0$   for every $i>0$ and $r\ge r_0$;
\item there exists an integer $r_0$, depending only on $\M$, such that $\M(r)$ is relatively generated by its global sections for every $r\ge r_0$, that is, the natural morphism $f^\ast f_\ast\M(r)\to \M(r)$ is surjective;
\item if $f_\ast \M(r)$ is locally free for $r\gg0$, then $\M$ is flat over $\Sc$.
\end{enumerate}
\end{thm}
\begin{proof} The first and fourth questions are local on the base, and for $\Sc$ noetherian, the second and the third too. We can then assume  that  $\Sc=\SSpec \As$ for a noetherian superring $\As$ and $\Xcal$ is a projective superscheme over $\As$, so that  there is a closed immersion $\delta \colon\Xcal\hookrightarrow \Ps_\As^{m,n}$ into a projective superspace.  Since $\delta_\ast$ preserves coherence and cohomology, we are reduced to the case $\Xcal =  \Ps_\As^{m,n}$. 

(1)  By Proposition \ref{prop:mder2}  there is an exact sequence
\begin{equation}\label{eq:seq}
0\to \Nc \to  \Ec \to \M \to 0 
\end{equation}
 where
 $\Ec =  \bigoplus_{1\le i\le N} {\Oc}(-r_i)\oplus  \bigoplus_{1\le i'\le N'}\Pi{\Oc}(-r'_{i'}) $ and $\Nc$ is coherent. Since all cohomology groups $H^i$ vanish for $i>m$ by cohomology vanishing (Proposition \ref{prop:cohombound}), we have an exact sequence
 $$
\dots \to H^m(X,\Nc)\to H^m(X,\Ec)\to H^m(X,\M)\to 0\,,
 $$
so that $H^m(X,\M)$ is finitely generated for every coherent sheaf by Proposition \ref{prop:cohomoden}. In particular, $H^m(X,\Nc)$ is finitely generated, and from
$$
H^{m-1}(X,\Ec) \to H^{m-1}(X,\M) \to H^m(X,\Nc) 
$$
one gets that $H^{m-1}(X,\M)$ is finitely generated for every coherent sheaf, because $\As$ is noetherian. Continuing by descending induction one proves the claim.

(2) From Equation \eqref{eq:seq} one gets $H^m(X,\M(r))=0$ for every $r$ such that $r\ge r_i$    and $r\ge r'_{i'}$  for  all $i$, $i'$. By descending induction as in (1) the result follows.

(3) This is (4) of Proposition \ref{prop:mder2}.

(4) One has   $\M\simeq \widetilde M^h$, where $M$ is the bigraded $\Bs(m,n)$- module $\bigoplus_{r\ge r_0} f_\ast \M(r)$. 
If $f_\ast \M(r)$ is locally free for $r\ge r_0$, then $M$ is flat over $\As$. It follows that   for every $i=0,\dots,m$ also the $\Z$-homogeneous localization  $\tilde M^h_{x_i}$ (Definition \ref{def:homloc}) is flat over $\As$, as it is a direct summand of the total localization $M_{x_i}$. But $M^h_{x_i}$ are the sections of $\M$ over the complementary open sub-superscheme  of the homogeneous zeroes of $x_i$, which proves that $\M$ is flat over $\As$. 
\end{proof}

Part (1) of this theorem immediately implies the finiteness of the cohomology of coherent sheaves.

\begin{corol} If $\Xcal$ is a projective superscheme over a field $k$ and $\M$ is a coherent sheaf on it, then all groups $H^i(X,\M) $ are finite-dimensional over $k$.
\label{cor:finiteness}
\end{corol}

\subsection{Filtrations associated to a sheaf}\label{subsec:filtrations}
To define the super Hilbert polynomial, and for other purposes, we shall need to consider some filtrations of
quasi-coherent sheaves. 
Let $f\colon \Xcal \to \Sc$ be a morphism of superschemes and consider the diagram
\begin{equation}\label{eq:diagram}
\xymatrix{
X \ar@{^{(}->}[r]
^j\ar[rd]_{f_{bos}}& \Xcal_S \ar@{^{(}->}[r] ^i\ar[d]^{f_S} & \Xcal\ar[d]^f\\
& S  \ar@{^{(}->}[r] ^i&\Sc
}
\end{equation}
For every quasi-coherent sheaf $\M$ of $\Oc_{\Xcal}$-modules we have two filtrations naturally associated with this diagram.
The first   is the \emph{base filtration}: let $\Jc$ be the ideal sheaf of $S$ in $\Sc$ and {$r$} the order of $\Jc$, that is, the first integer such that {$\Jc^{r+1}=0$}. The base filtration of $\M$ is 
\begin{equation}\label{eq:basefilt}
0\subset {\Jc^r}\M\subset \dots\subset \Jc\M\subset \M\,.
\end{equation}
The successive quotients $\M^{(p)}_{\Sc} = \Jc^p\M/\Jc^{p+1}\M$ are annihilated by $\Jc$, so that they are supported on $\Xcal_S$, $\M^{(p)}_{\Sc}=i_\ast(i^\ast \Jc^p\M)$, and one has
$$
0\to \Jc^{p+1}\M\to \Jc^p\M \to \M^{(p)}_{\Sc} \to 0\,.
$$

Analogously, there is a \emph{total filtration}. One considers the ideal $\Ic$ of $X$ as a closed super subscheme of $\Xcal$ and the filtration 
\begin{equation}\label{eq:totalfilt}
0\subset {\Ic^{n}}\M\subset \dots\subset \Ic\M\subset \M\,,
\end{equation}
{where $n$ is the order of $\Ic$ as above,}
whose successive quotients $\M^{(p)}= \Ic^p\M/\Ic^{p+1}\M$ are supported on $X$; one has $\M^{(p)}\simeq \iota_\ast(\iota^\ast \Ic^p\M)$, where $\iota=i\circ j$ is the immersion of $X$ into $\Xcal$, and there is an exact sequence
\begin{equation}\label{eq:total}
0\to \Ic^{p+1}\M\to \Ic^p\M \to \M^{(p)} \to 0\,.
\end{equation}

One has:
\begin{prop}\label{prop:filt} { \ }
\begin{enumerate}
\item
If $\M$ is coherent all the quotients of the base filtration (Equation \eqref{eq:basefilt}) are coherent sheaves on $\Xcal_S$.
Moreover, all the quotients of the total filtration (Equation \eqref{eq:totalfilt}) are coherent sheaves on $X$. 
\item
If $\M$ is flat over $\Sc$ {and  $\Oc_\Xcal=Gr_{\Ic}\Oc_\Xcal=\Oc_X\oplus \Ic/\Ic^2\oplus\dots$locally on $\Xcal$, then} all   successive quotients of the total filtration are flat over $S$. When $\Sc=S$ is an ordinary scheme, the converse is also true.
\end{enumerate}
\end{prop}
\begin{proof} (1) is straightforward. For (2), if $\M$ is flat over $\Sc$, then $\M_S$ is flat over $S$, so that we are reduced to the case when $\Sc=S$ is an ordinary scheme. Now, the question is local on $X$ so   we can assume   
{$\Oc_\Xcal=Gr_{\Ic}\Oc_\Xcal$}; hence
 there is a decomposition 
$\M= \bigoplus_{p\ge 0}\M^{(p)}$ of $\M$ as an $\Oc_X$-module. It follows that $\M$ is flat over $S$ if and only if all the sheaves $\M^{(p)}$ are flat over $S$.
\end{proof}

\subsection{The super Hilbert polynomial}

Let $\Xcal$ be a projective superscheme over a field $k$ {(Definition \ref{def:superproj})}. We are going to define the super Hilbert polynomial of a coherent sheaf $\M$ on $\Xcal$. It would be  possible to adapt the proof for ordinary projective schemes based on the Snapper polynomials, but we prefer to follow a simpler approach which relies on the existence of the Hilbert polynomial for a coherent sheaf on a projective scheme.

Let us consider the Euler characteristics of $\M$, defined as the pair of integer numbers
$$
\chi (X,\M)= \sum_{i\ge 0} (-1)^i \dim H^i(X,\M)\,,
$$
which is well defined due to the finiteness of the cohomology (Corollary \ref{cor:finiteness}) and the cohomology vanishing  (Proposition \ref{prop:cohombound}).

\begin{lemma}\label{lem:euler}
The Euler characteristic of a coherent sheaf $\M$ is the sum of the Euler characteristics of the quotients of its total filtration (Equation \eqref{eq:totalfilt}):
$$
\chi (X,\M)= {\sum_{0\leq p\leq n}} \chi(X, \M^{(p)})= {\sum_{0\leq p\leq n}}\left(\chi(X, \M^{(p)}_+), \chi(X, \M^{(p)}_-)\right)\,.
$$
\end{lemma}
\begin{proof} The result is proved by repeatedly applying
the additivity of the Euler characteristic and Equation \eqref{eq:total}.
\end{proof}

Since $\M^{(p)}_+$ and $\M^{(p)}_-$ are coherent sheaves on the ordinary  {projective} scheme $X$, they have Hilbert polynomials, that is, there exists polynomials $H(\M^{(p)}_+, r)$, $H(\M^{(p)}_-,r)$ with rational coefficients, such that
$$
H(\M^{(p)}_+, r) = \chi(X, \M^{(p)}_+(r))\,,\quad H(\M^{(p)}_-, r) = \chi(X, \M^{(p)}_-(r))
$$
for $r\gg 0$. So  there are polynomials
$$
H(\M,r)_+:=  {\sum_{0\leq p\leq n}} H(\M^{(p)}_+, r) \,, \quad
H(\M,r)_-:=  {\sum_{0\leq p\leq n}} H(\M^{(p)}_-, r)\,.
$$

\begin{defin}
The super Hilbert polynomial of a coherent sheaf $\M$ is the pair 
$$
\bH(\M,r)=(H(\M,r)_+,H(\M,r)_-)
$$
of polynomials with rational coefficients. 
\end{defin}

By Lemma \ref{lem:euler} one has
$$
\bH(\M,r)= \chi (X,\M(r)) \quad \text{for $r\gg 0$}.
$$
 
\begin{prop}  Let $f\colon \Xcal \to \Sc$ be a  {super}projective morphism of locally noetherian superschemes   with $S$ connected, and $\Oc_\Xcal(1)$ a relatively very ample line bundle (Remark \ref{rem:ample}). For every coherent sheaf $\M$ on $\Xcal$, flat over $\Sc$, the super Hilbert polynomials $\bH(\M_s,r)$ of the restrictions $\M_s=\M\otimes\kappa(s)$ of $\M$ to the fibres $f_s\colon \Xcal_s \to \Spec \kappa(s)$ of $f$ are independent of the choice of the point $s\in S$. In other words, the function
$$
s\in S \mapsto \bH(\M_s,r)
$$
is constant on $S$.
\end{prop}
\begin{proof} {The question is local so that we can assume that $\Xcal$ is a relative projective superspace over $\Sc$. Then the successive quotients $\M^{(p)}$ of the total filtration are flat over $S$ by Proposition \ref{prop:filt}, and then their odd and even parts are flat over $S$ as well}. We finish by the corresponding statement for the classical case.
\end{proof}

We can then define:

\begin{defin}\label{def:Hilbpol} Let $f\colon \Xcal \to \Sc$ be a  {super}projective morphism   of locally noetherian superschemes, $\Oc_\Xcal(1)$ a relatively very ample line bundle (Remark \ref{rem:ample}), and $\bP=(P_+,P_-)$ a pair of polynomials with rational coefficients. We say that a coherent sheaf $\M$ on $\Xcal$ has super Hilbert polynomial $\bP$ on the ``fibre'' of a ``point'' $s\colon\Tc\to \Sc$ if
\begin{enumerate}
\item $\M_\Tc$ is flat over $\Tc$.
\item $\M_\Tc$ has super Hilbert polynomial $\bP$ on every fibre of $f_\Tc\colon \Xcal_\Tc\to\Tc$, that is, $\bH(\M_t,r) $ $=\bP(r)$ for every point $t\in T$.
\end{enumerate}
\end{defin}

\section{Cohomology of proper morphisms}\label{s:cohomprop}

\subsection{Finiteness theorems and Grothendieck-Mumford complex}

In the classical case the finiteness of the cohomology for proper morphisms of locally noetherian superschemes is proved from Serre's Theorem \ref{thm:serre}, using an argument based on ``d\'evissage'' and on   Chow lemma.  Although a super version of the Chow lemma is available \cite[7.1.3]{MoZh19}, we prefer a different approach, using the filtrations of a sheaf defined in Subsection \ref{subsec:filtrations} and  the classical finiteness result. 

\begin{prop}\label{prop:ftecoh}
Let $f\colon \Xcal \to \Sc$ be a proper morphism of locally noetherian superschemes. For every coherent sheaf $\M$ on $\Xcal$  and for every $i\geq 0$, the higher direct images $R^if_\ast\M$ are coherent sheaves on $\Sc$.
\end{prop}
\begin{proof} Consider the total filtration of $\M$.  The exact sequence of    higher direct images of Equation \eqref{eq:total} gives 
$$
\begin{aligned}
&0 \to f_\ast(\Ic^{p+1}\M)\to f_\ast(\Ic^p\M) \to f_\ast\M^{(p)}\to  R ^1f_\ast(\Ic^{p+1}\M)\to R^1 f_\ast(\Ic^p\M) \to 
\\
& \to R^1f_\ast\M^{(p)}\to
R^2 f_\ast (\Ic^{p+1}\M)\to R^2 f_\ast(\Ic^p\M) \to  R^2f_\ast\M^{(p)}\to\dots
\end{aligned}
$$ 
By Proposition \ref{prop:filt}  $\iota^\ast \Ic^p\M$ is coherent on $X$. Due to 
$$
R^i f_\ast \M^{(p)}\iso i_\ast R^i f_{bos\ast}(\iota^\ast \Ic^p\M) 
$$
these sheaves are coherent as a consequence of the corresponding property for classical schemes. By descending induction on $p$, we can assume that the sheaves  $R^if_\ast(\Ic^{p+1}\M)$ are coherent. By local noetherianity  the sheaves $R^if_\ast(\Ic^p\M)$ are coherent as well.
\end{proof}

The finiteness theorem allow us to generalize straightforwardly many classical results. We reproduce the relevant statements offering proofs only when they are different from the classical ones.  Among the many references for those classical results we mention Grothendieck's original  treatment \cite{EGAIII-II}, Hartshorne \cite{Hart77} and Nitsure \cite{Ni07}.

The first result we would like to report on is the existence of a complex of finitely generated modules that computes the cohomology of a coherent sheaf on a proper superscheme.

\begin{prop}[Grothendieck-Mumford complex]\label{prop:GMcomplex} Let $\Xcal$ be a proper superscheme over a noetherian superring $\As$ and $\M$ a coherent sheaf on $\Xcal$ {flat over $\As$}. There exists a finite complex
$$
K^\bullet := 0\to K_0\xrightarrow{\partial_0} K_1 \xrightarrow{\partial_1}\dots\xrightarrow{\partial_{s-1}} K_s\to 0\,,
$$
of finitely generated {and locally free} $\Z_2$-graded $\As$-modules, and a functorial $\As$-linear isomorphism
$$
H^i(X, \M\otimes_\As N) \iso H^i(K^\bullet\otimes_\As N)\,,
$$
in  the category of all ($\Z_2$-graded) $\As$-modules $N$.
\end{prop}
\begin{proof} The same proof as in the classical case gives that the modules of the complex
are projective. By the  super Nakayama lemma  (\cite{BBH91,CarCaFi11} or \cite[6.4.5]{We09})\footnote{A ``super'' version of the Nakayama lemma first appeared in \cite{BBH91}.} this is equivalent to the local freeness.
\end{proof}

Let $\As$ be a noetherian superring and $f\colon \Xcal \to \Sc=\SSpec\As$ a proper morphism of superschemes.
For every coherent sheaf $\M$ of $\Oc_\Xcal$-modules, flat over $\Sc$, and every index $i\ge 0$ we consider the linear half  exact functor (Subsection \ref{ss:nakayama}) $T^i$ from the category of ($\Z_2$-graded) finitely generated $\As$-modules to itself given by
\begin{equation}\label{eq:functorT}
T^i(N):= \Gamma(S, R ^i f_\ast (\M\otimes_\As N))= H^i(X,\M\otimes_\As N)= H^i(K^\bullet\otimes_\As N)\,,
\end{equation}
where $K^\bullet$ is the Grothendieck-Mumford complex associated to $\M$ (Proposition \ref{prop:GMcomplex}).  The functors $T^i$ form a $\delta$-functor, that is, for every exact sequence $0\to N'\to N \to N''\to 0$ of finitely generated $\As$-modules, there is an exact sequence
\begin{equation}\label{eq:cohexactsq}
\dots T^{i-1}(N'')\xrightarrow{\delta} T^i(N') \to T^i(N) \to T^i(N'') \xrightarrow{\delta} T^{i+1}(N')\to\dots
\end{equation}

For every point $s\in \Sc$, denote by $T^i_{(s)}$ the functor $T^i$ for the morphism $\Xcal\times _\Sc \SSpec \As_s \to \SSpec \As_s$ and the sheaf $\M\otimes_\As\As_s$ obtained by the localization flat base change $\As \to \As_s$.

\begin{lemma} \label{lem:ti} {\ }
\begin{enumerate}
\item $T^i$ is left  exact if and only if $W_i:=\coker\partial_{i-1}$ is a locally free $\As$-module. Therefore:
\begin{enumerate}
\item $T^0$ is left  exact.
\item If $T^i_{(s_0)}$ is left  exact for a point $s_0$, $T^i_{(s)}$ is also left  exact for all the points $s$ in an open neighbourhood of $s_0$.
\end{enumerate}
\item $T^i$ is right  exact if and only if $T^{i+1}$ is left  exact.
\end{enumerate}
\end{lemma}
\begin{proof}
(1) If $ N'\to N$ is an injective morphism, 
as $\coker (K_{i-1}\otimes_\As P \to K_i\otimes_\As P)=W_i\otimes_\As P$ for every $A$-module $P$,  
one has a commutative diagram of exact rows and columns:
$$
\xymatrix{
& & & 0\ar[d] \\
0 \ar[r] & T^i(N') \ar[r]\ar[d]^\alpha & W_i\otimes_\As N' \ar[r]\ar[d]^\beta & K_{i+1}
\otimes_\As N'\ar[d] \\
0 \ar[r] & T^i(N) \ar[r] & W_i\otimes_\As N \ar[r] & K_{i+1}
\otimes_\As N
}
$$
So $\alpha$ is injective if and only if $\beta$ is injective. Thus, $T^i$ is left  exact if and only of $W_i$ is flat and then locally free.

Now (a) is a consequence of Proposition \ref{prop:GMcomplex} as $W_0=K_0$, and (b) follows from the fact that if a sheaf of modules is free at a point, then it is free in an open neighbourhood.

(2) follows from Equation \eqref{eq:cohexactsq}.
\end{proof}

\subsection{Cohomology base change and semicontinuity}\label{ss:cohbasechange}
In this Section we prove the base change and semicontinuity theorems by applying  the Nakayama Lemma \ref{lem:nakayama}  for half  exact functors 
and Proposition \ref{prop:nakayama}
to the functors $T^i$.

\begin{thm}[Cohomology base change]\label{thm:cohombasechange} Let $f\colon \Xcal \to \Sc$ be a proper morphism of locally noetherian superschemes and $\M$ a coherent sheaf on $\Xcal$ {flat over $\Sc$}. If $s\in S$ is a point of $S$, one has:
\begin{enumerate}
\item If for some $i$ the base change map $\varphi^i_s\colon (R^if_\ast\M)_s \to H^i(X_s,\M_s)$ is surjective, then it is an isomorphism and the same happens for all points in an open neighbourhood of $s$.
\item If (1) is true, there exists an open sub-superscheme $\Ucal$ of $\Sc$ containing $s$, such that for any quasi-coherent $\Oc_\Ucal$-module $\Nc$ the natural morphism
$$
(R^i f_{\Ucal\ast}\M_{\vert \Xcal_\Ucal})\otimes_{\Oc_\Ucal}\Nc \to R^i f_{U\ast} (\M_{\vert \Xcal_\Ucal}\otimes f_\Ucal^\ast\Nc)\,,
$$
is an isomorphism.
\item
If (1) is true for $i>0$, then $\varphi^{i-1}_s\colon (R^{i-1}f_\ast\M)_s \to H^{i-1}(X_s,\M_s)$ is surjective if and only of $R^if_\ast\M$ is locally free on an open sub-superscheme $\Ucal$ fo $\Sc$ with $s\in U$.
\item If $H^i(X_s,\M_s)=0$ for some   $i>0$, the morphism $\varphi^{i-1}\colon (R^{i-1}f_\ast\M)_s \to H^{i-1}(X_s,\M_s)$ is an isomorphism.
\end{enumerate}

\end{thm}
\begin{proof} The question is local, so that we can assume that $\Sc$ is affine, $\Sc=\SSpec \As$, and that the superring $\As$ is noetherian. After the flat base change $\As \to \As_s$ we can also assume that $\As$ is local. Applying Proposition \ref{prop:nakayama} (Nakayama's Lemma for half  exact functors) to the functors $T^i$ and Lemma \ref{lem:ti} we obtain the first three statements.

(4) $T^i=0$ by Lemma \ref{lem:nakayama}, and then $T^{i-1}$ is right  exact by Lemma \ref{lem:ti}. The statement now follows from Proposition \ref{prop:nakayama}.
\end{proof}
\begin{corol}\label{cor:cohombound1}
Let $f\colon \Xcal \to \Sc$ be a proper morphism of locally noetherian superschemes and $\M$ a coherent sheaf on $\Xcal$ {flat over $\Sc$}. If there is an index $j$ such that $R^i f_\ast\M=0$ for every $i\ge j$, then $H^i(X_s,\M_s)=0$ for every $i\ge j$ and every point $s\in S$.
\end{corol}
\begin{proof} It follows from Theorem \ref{thm:cohombasechange} by descending induction on   $j$, taking into account that $H^i(X_s,\M_s)=0$ for every $i\ge \dim X_s$ and every point $s\in S$ by Proposition \ref{prop:cohombound}.
\end{proof}

\begin{corol}[Boundedness of higher direct images]\label{cor:cohombound2} Let $f\colon \Xcal \to \Sc$ be a proper morphism of noetherian superschemes, $\M$ a coherent sheaf on $\Xcal$ {flat over $\Sc$}. Let $r$ denote the maximum of the dimensions $\dim X_s$ of the bosonic fibres, which is well defined as $\Sc$ is noetherian.  Then $R^if_\ast\M=0$ for $i>r$.
\end{corol}
\begin{proof} For every point $s$ one has $(R^i f_\ast\M)_s=0$ for $i>r$ by cohomology boundedness (Proposition \ref{prop:cohombound}) and by (4) of Theorem \ref{thm:cohombasechange}. Then $R^if_\ast\M=0$ for $i>r$ by super Nakayama (\cite{BBH91,CarCaFi11} or \cite[6.4.5]{We09}).
\end{proof}

One can also deduce the local freeness of the direct images of relatively acyclic coherent sheaves, assuming that the latter are \emph{flat over the base}:

\begin{prop}\label{prop:locfree}
Let $f\colon \Xcal \to \Sc$ be a proper morphism of locally noetherian superschemes and $\M$ a coherent sheaf on $\Xcal$ flat over $\Sc$.
\begin{enumerate}
\item 
 If either $R^if_\ast\M=0$ for every $i>0$, or $H^i(X_s,\M_s)=0$ for every $i>0$ and every point $s\in S$, then $f_\ast\M$ is locally free.  Moreover, for every morphism $\Tc\to \Sc$, the base change morphism $(f_\ast\M)_\Tc\to f_{\Tc\ast}(\M_\Tc)$ is an isomorphism.
 \item If $f$ is superprojective and $\Sc$ is noetherian then $f_\ast\M(r)$ is locally free for $r\gg0$.
\end{enumerate}
\end{prop}
\begin{proof} 
(1) We can now assume that $\Sc=\SSpec\As$ and $\As$ noetherian. Reasoning as in the proof of Corollaries \ref{cor:cohombound1} and \ref{cor:cohombound2}, one sees that the two conditions are equivalent and   imply $T^1=0$. Then $T^0$ is right  exact by Lemma \ref{lem:ti} and thus $T^0=f_\ast(\M)\otimes_\As$  by Proposition \ref{prop:nakayama}. Again by Lemma \ref{lem:ti}, $T^0$ is left  exact because $\M$ is flat over $\Sc$, so that $f_\ast(\M)$ is flat over $\As$ and $(f_\ast\M)_s\iso H^0(X_s,\M_s)$ for every point $s\in S$.  Since $\M_\Tc$ is flat over $\Tc$, we also have an isomorphism $(f_{\Tc\ast}\M)_s\iso H^0(X_s,\M_s)$ for every point $s\in T$, and   $f_{\Tc\ast}\M_\Tc$ is locally free. Now $(f_\ast\M)_\Tc\to f_{\Tc\ast}(\M_\Tc)$ is a morphism of locally free sheaves inducing isomorphisms $((f_\ast\M)_\Tc)_s\iso (f_{\Tc\ast}(\M_\Tc))_s$ for every point; then, it is an isomorphism by Nakayama's Lemma (Proposition \ref{prop:nakayama}).

(2) It follows from (1) and Serre's Theorem \ref{thm:serre}.
\end{proof}

\begin{corol}\label{cor:strongly-superproj}
Let $f\colon\Xcal \to \Sc$ be a flat superprojective morphism, where $\Sc$ is noetherian. Then $f$ is strongly superprojective (see {Definition} \ref{def:superprojective}).
\end{corol}

\begin{proof} Let {$\Oc_{\Xcal/\Sc}(1)$} be a relatively very ample line bundle for $\Xcal/\Sc$. By Proposition {\ref{prop:locfree}(2)} and Theorem \ref{thm:serre}, there exists $r>0$ such that
{$f_\ast\Oc_{\Xcal/\Sc}(r)$} is locally free and the natural map {$f^\ast(f_\ast\Oc_{\Xcal/\Sc}(r))\to \Oc_{\Xcal/\Sc}(r)$} is surjective. Hence, by Serre's Theorem \ref{thm:serre} applied to the kernel and the projection formula, there exists $n_0$ such that for $n\ge n_0$, the morphism
$$
f_\ast\Oc_{\Xcal/\Sc}(r)\otimes f_*\Oc_{\Xcal/\Sc}(n)\to f_*(\Oc_\Xcal/\Sc(r+n))$$
is surjective. This implies that
$$(f_*\Oc_{\Xcal/\Sc}(r))^{\otimes m}\otimes f_\ast\Oc_{\Xcal/\Sc}(n)\to f_\ast(\Oc_{\Xcal/\Sc}(rm+n))
$$
is surjective for $n\ge n_0$ and $m>0$. Hence,
$$(f_\ast\Oc_{\Xcal/\Sc}(n_0r))^{\otimes m}\to f_\ast\Oc_{\Xcal/\Sc}(mn_0r)
$$
is surjective for any $m>0$. Thus, we get a surjection of sheaves of graded algebras
$$\SSym(f_\ast\Oc_{\Xcal/\Sc}(d))\to \bigoplus_{m\ge 0}f_\ast\Oc_{\Xcal/\Sc}(md)\,,
$$
where $d=n_0r$. This induces a {closed immersion} of $\Xcal$ into the projectivization of the dual vector bundle to {$f_\ast\Oc_{\Xcal/\Sc}(d)$}.
\end{proof}

We can now state the super semicontinuity theorem, also given in \cite[Theorem 7.3]{MoZh19}, whose proof is the same as in the classical case (see, for instance, \cite[Thm.~III.12.8]{Hart77} or \cite[7.7.5 and 7.6.9]{EGAIII-II}).
Recall that we have equipped $\Z\times\Z$  with the natural partial order given by Equation \eqref{eq:order}.

\begin{thm}[Semicontinuity]\label{thm:semicontinuity}  Let $f\colon \Xcal \to \Sc$ be a proper morphism of locally noetherian superschemes, and $\M$ a coherent sheaf on $\Xcal$, flat over $\Sc$. One has:
\begin{enumerate}
\item
for every integer $i$ the function 
\begin{eqnarray} S &\to & \Z\times\Z \\
s&\mapsto &\dim_{\kappa(s)} H^i(X_s, \M_s)\,;
\end{eqnarray}
is upper semicontinuous on $S$, where $\Z\times\Z$ is equipped with the natural partial order.
\item the function $s\mapsto \sum_{i\ge 0} \dim_{\kappa(s)} H^i(X_s, \M_s)$ is  locally constant on $S$.
\end{enumerate}
\end{thm}
\begin{proof}
The question is local, so that we can assume $\Sc=\SSpec\As$ with $\As$ noetherian. 
{For every point $s\in S$ and every index $i$, we have an exact sequence
\begin{equation}\label{eq:eq}
 0 \to T^i(\kappa(s)) \to W_i\otimes_\As\kappa(s) \to  K_{i+1}\otimes_\As\kappa(s)\,,
 \end{equation}
 where $K_i$ are the terms of the Grothendieck-Mumford complex (Proposition \ref{prop:GMcomplex}) and $W_i=\coker\partial_{i-1}$ (as in Lemma \ref{lem:ti}). Moreover, we have and exact sequence 
$$
 W_i \to K_{i+1}\to  W_{i+1} \to 0\,,
$$
so that Equation \ref{eq:eq} can be completed to an exact sequence
\begin{equation}\label{eq:semicontinuity}
 0 \to T^i(\kappa(s)) \to W_i\otimes_\As\kappa(s) \to  K_{i+1}\otimes_\As\kappa(s)\to W_{i+1}\otimes_\As\kappa(s) \to 0\,.
\end{equation}
}

Since $\dim K_{i+1}\otimes_\As\kappa(s)$ is constant because $K_{i+1}$ is free, and, for every $(p,q)$, the set of points   $s$ such that  $\dim (W_i\oplus W_{i+1})\otimes_\As \kappa(s)> (p,q)$ is closed by super Nakayama (\cite{BBH91,CarCaFi11} or \cite[6.4.5]{We09}), one proves (1).
To prove (2) it is enough to take alternate sums in Equation \eqref{eq:semicontinuity}.
\end{proof}

\begin{lemma}\label{lem:red}
Let $\Sc$ be a noetherian superscheme with $S$ reduced, and $\Nc$ a coherent sheaf on $\Sc$ such that $\dim \Nc\otimes_{\Oc_\Sc}\kappa(s)=(p,q)$ for some integers $p$, $q$ and for every point $s\in S$. Then $\Nc$ is locally free of rank $(p,q)$. 
\end{lemma}
\begin{proof} If $(n_1,\dots,n_p,\eta_1,\dots, \eta_q)$ is a family of $p$ even and $q$ odd sections of $\Nc$ on an open neighbourhood $U$ of a point $s$ such that the classes $(\bar n_1,\dots,\bar n_p,\bar \eta_1,\dots, \bar \eta_q)$ in $\Nc\otimes_{\Oc_\Sc}\kappa(s)$ are a basis, then by super Nakayama (\cite{BBH91,CarCaFi11} or \cite[6.4.5]{We09}), there is an open sub-superscheme $\Vc$ with $s\in V\subseteq U$ such that $\Nc_\Vc$ is generated by $(n_1,\dots,n_p,\eta_1,\dots, \eta_q)$. Then, there is an exact sequence
$$
\Oc_\Vc^{p,q} \xrightarrow{\phi} \Nc_\Vc \to 0\,.
$$
Since  $\dim \Nc\otimes_{\Oc_\Sc}\kappa(s')=(p,q)$ for every point $s'$, by shrinking $\Vc$ if necessary, we have that $\ker \phi\subset \pf_{s'}\cdot \Oc_\Vc^{p,q}$ for every point point $s'\in V$. Then $\ker \phi=0$ because $V$ is reduced.
\end{proof}
\begin{thm}[Grauert]\label{thm:grauert} Under the same hypotheses of Theorem \ref{thm:semicontinuity}, and assuming that $S$ is reduced,  
the function $s\mapsto \dim_{\kappa(s)} H^i(X_s, \M_s)$ is locally constant for some $i\ge 0$  if and only if  $R^if_\ast\M$ is locally free  and $(R^if_\ast\M)_s \simeq H^i(X_s,\M_s)$ for every point $s$ in $S$. If these conditions are satisfied, the base change morphism
 $ (R^{i-1} f_\ast\M)_s \to H^{i-1}(X_s,\M_s)$ is an isomorphism as well.
\end{thm}
\begin{proof} We can assume that $\Sc=\SSpec\As$ with $\As$ noetherian. Suppose that the function $s\mapsto \dim_{\kappa(s)} H^i(X_s, \M_s)$ is locally constant. Taking dimensions in Equation \eqref{eq:semicontinuity} for every point $s\in S$, and applying Lemma \ref{lem:red}, we have that $W_i\oplus W_{i+1}$ is locally free, and then all summands are flat. By Lemma \ref{lem:ti}, $T^i$ and $T^{i+1}$ are left  exact, so that $T^i$ is also right  exact again by   Lemma \ref{lem:ti}. By Proposition \ref{prop:nakayama}, $T^i(\As)\otimes \iso T^i$ and $T^i(A)=R^if_\ast\M$ is a flat, and then locally free,  $\As$-module. Moreover, for every point $s$ 
we have $$(R^if_\ast\M)_s=T^i(\As)\otimes_\As \kappa(s)\iso H^i(X_s,\M_s)=T^i(\kappa(s)).$$ The rest of the statement follows from cohomology base change, Theorem \ref{thm:cohombasechange}. 
\end{proof}

\subsection{Cohomological flatness in dimension $0$}

In this section we give the definition of cohomological flatness {in dimension} $0$ for morphisms of superschemes, which is necessary in the proof of  the existence of the Picard superscheme, as we shall see in Section \ref{s:picard}.

\begin{defin}
Let $\As$ be a superring. Recall that $J\subset\As$ denotes the ideal generated by odd elements. We say that 
$\As$ is \emph{integral} if $A=\As/J$ is an integral ring (i.e., a domain) and every element of $\As\setminus J$ is not a zero divisor in $\As$. 
In other words, the superring $\As$ is integral if $J$ is the set of zero divisors of $\As$.
\end{defin}

For example, the superring $k[x,\theta]/(x\theta)$ is not integral, while $k[x,\theta_1,\theta_2]/(\theta_1\theta_2)$ is (where $\theta$ and $\theta_i$ are odd variables).
Note that if $A=\As/J$ is integral and all the $A$-modules $J^i/J^{i+1}$ are torsion free, then $\As$ is integral.

\begin{defin}\label{def:geomint} 
A supescheme $\Xcal=(X,\Oc_\Xcal)$ is called integral if $\Oc_\Xcal(U)$ is integral for every open $U$.
A superscheme over a field $k$ is called geometrically integral if for every field extension $k\hookrightarrow \bar k$, where $\bar k$ is algebraically closed, the base change superscheme $\bar\Xcal=\Xcal\times_{\Spec k}\Spec\bar k$ is an integral superscheme.
\end{defin}

Clearly for a integral superscheme $\Xcal$ its bosonization $X$ is an integral scheme. Given a superscheme $\Xcal$ with integral bosonization $X$,
let $i_\eta\colon\eta\to X$ denote the embedding of the general point.
It is easy to see that $\Xcal$ is integral if and only if in addition the natural morphism of sheaves
$$\Oc_\Xcal\to i_{\eta,\ast}i_\eta^\ast\Oc_\Xcal$$
is injective.

The following easy fact, which follows immediately from the definition, will be needed later on.
\begin{prop}\label{prop:xxx} Every  morphism $\Lcl\to\Nc$ of line bundles on an integral superscheme $\Xcal=(X,\Oc_\Xcal)$
 whose bosonic restriction $\rest{\Lcl}X\to\rest{\Nc}X$ to $X$ is nonzero, is injective.
\qed\end{prop}

\begin{defin}\label{def:cohomflat} A morphism $f\colon\Xcal\to\Sc$ of superschemes is cohomologically flat {in dimension} $0$ if the natural morphism
$$
\Oc_\Tc \to f_{\Tc\ast}\Oc_{\Xcal_\Tc}
$$
is an isomorphism for every base change $\Tc\to\Sc$. A superscheme $\Xcal$ over a superring $\As$ is  is cohomologically flat {in dimension} $0$ if so is the natural morphism $\Xcal\to\SSpec\As$.
\end{defin}

\begin{example}\label{ex:cohomflat} {\ }
\begin{enumerate}
\item Let $f\colon X\to S$ be a proper flat morphism of locally noetherian ordinary schemes. If the geometric fibres, that is, fibres over spectra of algebraically closed fields, are irreducible and reduced, then $f$ is cohomologically flat {in dimension} $0$ (see \cite[Proposition 7.8.6 and Corollary 7.8.8]{EGAIII-II}). 
\item If $f\colon \Xcal=\widetilde\Ps(\Ec) \to \Sc$ is the superprojective bundle associated to a locally free sheaf $\Ec$ of rank $(m,n)$ on $\Sc$ (Definition \ref{def:spbundle}) with $m\ge 1$, then $f$ is cohomologically flat {in dimension} $0$. Since $\Xcal_\Tc \simeq \widetilde\Ps(\Ec_\Tc)$, one has only to prove that $\Oc_\Sc \iso f_\ast\Oc_\Xcal$, and this follows from Lemma \ref{lem:odensec} because $\Sc$ can be covered by affine superschemes $\Vc=\SSpec\As$ such that $\Xcal_\Vc\simeq \Ps_\As^{m,n}$ (Proposition \ref{prop:spbundle}).
\item A split superscheme $\Xcal=(X, \bigwedge_{\Oc_X}(\Ec))$ over a field $k$  is cohomologically flat {in dimension} $0$, if and only if $\Gamma(X,\Oc_X)=k$ and  {$\Gamma(X,{\bigwedge}^p_{\Oc_X}(\Ec))=0$ for every $p>0$}.
\end{enumerate}
\end{example}

We are going to extend (1) of Example \ref{ex:cohomflat} of superschemes. The proof is similar to the one for ordinary schemes given in \cite{Kle05}.

Given a superscheme $\Xcal=(X,\Oc_\Xcal)$ over a field $k$ and a field extension $k\hookrightarrow \bar k$, we denote by $\bar\Xcal=(\bar X=X\times_{\Spec k}\Spec\bar k,\Oc_{\bar\Xcal}=\Oc_\Xcal\otimes_k\bar k)$ the base change of $\Xcal$ under $\Spec\bar k\to\Spec k$. 

\begin{prop}\label{prop:cohomflat} Let $f\colon\Xcal \to \Sc$  {be} a proper and flat morphism of locally noetherian superschemes whose fibres are geometrically integral.  Then 
$f$ is cohomologically flat in dimension $0$ if and only if the geometric fibres of $f$ satisfy 
\begin{equation}\label{eq:geom-fibers-H0}
H^0(\Xcal_{\bar s},\Oc_{\Xcal_{\bar s}})\simeq \kappa(\bar s).
\end{equation}
\end{prop}
  \begin{proof}  It is clear that if $f$ is cohomologically flat in dimension $0$ then its geometric fibres satisfy \eqref{eq:geom-fibers-H0}.
  For the converse, since our conditions are stable under base change, 
  it is enough to prove that the morphism $\Oc_\Sc \to f_\ast\Oc_\Xcal$ is an isomorphism. 

 Let us consider the left exact functor $T^0(N)=f_\ast(f^\ast N)$  defined on the category  of finitely-generated $\As$-modules and  associated to  $\M=\Oc_\Xcal$ as in Equation \eqref{eq:functorT}.

By the flat base change, the condition \eqref{eq:geom-fibers-H0} implies that for every point $s\in \Sc$ one has
$$T^0(\kappa(s))=H^0(\Xcal_s,\Oc_{\Xcal_s})\simeq \kappa(s).$$
Hence, the composition of $\kappa(s) \to f_\ast\Oc_\Xcal\otimes\kappa(s)=T^0(\As)\otimes \kappa(s)$ with  $T^0(\As)\otimes \kappa(s)\to T^0(\kappa(s))$ is an isomorphism, so that the second morphism is surjective.  Then by Proposition \ref{prop:nakayama} one has that
$T^0(\As)\otimes \kappa(s)\iso T^0(\kappa(s))$, $T^0=T^0(\As)\otimes - $ and that $T^0$ is right  exact. Since $T^0$ is also left  exact, we get that $T^0(\As)= f_\ast\Oc_\Xcal$ is free
of rank $(1,0)$. Since the morphism $\Oc_{\Sc}\to f_\ast\Oc_\Xcal$ induces an isomorphism of fibres at every point, it is an isomorphism. 
\end{proof}

\subsection{The superscheme of homomorphisms between coherent sheaves}

The following result on the representability of morphisms between coherent sheaves will be quite useful in the construction of the Hilbert superscheme. The proof is the same as in  the classical case (see, for instance, \cite[Thm. 3.5]{Ni07}).

Let $f\colon \Xcal \to \Sc$ be a  {superprojective} morphisms of superschemes, and let $\M$, $\Nc$ be coherent sheaves on $\Xcal$. 
Let us consider the  functor $\mathbf{Hom}_\Xcal(\M,\Nc)_+$ defined on the category on $\Sc$-superschemes $\Tc\to \Sc$  by letting
$$
\mathbf{Hom}_\Xcal(\M,\Nc)_+(\Tc)=\Hom_{\Xcal_\Tc}(\M_\Tc,\Nc_\Tc)_+\,.
$$

\begin{prop} {\rm \cite[7.7.8]{EGAIII-II}} \label{prop:homrep} If $\Nc$ is flat over $\Sc$, the above functor is representable by a linear superscheme $\Vs(\Qc)$ associated to a coherent sheaf $\Qc$ on $\Sc$, that is, there is a ``universal'' morphism of $\Oc_{\Xcal_{\Vs(\Qc)}}$-modules $\Phi\colon \M_{\Vs(\Qc)}\to \Nc_{\Vs(\Qc)}$ and a functorial isomorphism
\begin{align*}
\Hom_\Sc(\Tc, \Vs(\Qc)) &\iso \mathbf{Hom}_\Xcal(\M,\Nc)_+(\Tc)\\
\gamma& \mapsto  \gamma^\ast\Phi\,.
\end{align*}
\qed
\end{prop}

\begin{corol}\label{cor:zerolocus} The zero section $\Vs_0(\Qc)$ of $\Vs(\Qc)$ (Definition \ref{def:zerosec}) is the locus where the universal morphism $\Phi$ vanishes, that is, for every $\Sc$-superscheme $\Tc$ and every morphism $\phi\colon \M_\Tc \to \Nc_\Tc$, corresponding to a morphism $\gamma\colon \Tc \to \Vs(\Qc)$ of $\Sc$-superschemes ($\phi=\gamma^\ast\Phi$), the closed sub-superscheme 
$$
\gamma^{-1}(\Vs_0(\Qc))\hookrightarrow \Tc\,,
$$
has the following universal property: a morphism $\psi\colon \Zc\to \Tc$ satisfies that $\psi^\ast\phi=0$ if and only if it factors through $\gamma^{-1}(\Vs_0(\Qc))$.
\qed
\end{corol}

\subsection{Relative Grothendieck duality}\label{subsec:duality}
Relative Grothendieck duality can be developed also for morphisms of superschemes. In the mid eighties  a simple version
of it, namely, Serre duality for projective smooth superschemes, 
was proved in \cite{OgPe84}
in its simplest version,  Penkov later extended duality to complete smooth complex superschemes \cite{Penk83}. A remarkable  result in these papers is the fact  that the  Berezinian  sheaf is the dualizing sheaf for smooth complex complete superschemes.
Relative Grothendieck duality for   proper morphisms   was stated in the super analytic case in \cite[1.5.1]{Va88}; however no proof was provided there in addition to   a reference  to the ordinary analytic case. 

So a description of relative duality for superschemes seems to be missing in the literature, and the aim of this Subsection is to fill that gap.
We state the main facts, stressing which proofs are straightforward translations of the corresponding ones for schemes, and  what is new in the super setting. Albeit most likely this is not the simplest approach,  we use Brown's representability theorem as in Neeman \cite{Nee96} (see also the references therein for a more complete perspective of the topic).

For every superscheme $\Xcal$ we denote  by $D(\Xcal)$ the unbounded derived category of the category of quasi-coherent $\Oc_\Xcal$-modules. 
If $\mathfrak{Mod}(\Xcal))$ is the category of all $\Oc_\Xcal$-modules, there is a natural functor $D(\Xcal)\to D(\mathfrak{Mod}(\Xcal))$. If $\Xcal$ is quasi-compact and separated, this induces an equivalence of categories
between $D(\Xcal)$ and the full subcategory $D_{qc}(\mathfrak{Mod}(\Xcal))$ of complexes with quasi-coherent cohomology. This is proved by translating the formal proof given in \cite[Cor.\;5.5]{BN93} to our setting. 

Again as in the ordinary case, one can prove, following Spaltenstein \cite{Spal88}, that for every quasi-compact morphism $f\colon\Xcal\to\Sc$ of superschemes   the direct and inverse images 
$$\bR f_\ast\colon D(\mathfrak{Mod}(\Xcal))\to D(\mathfrak{Mod}(\Sc)), \quad \bL f^\ast\colon D(\mathfrak{Mod}(\Sc)) \to D(\mathfrak{Mod}(\Xcal))$$
are defined for unbounded complexes of quasi-coherent sheaves,  and the latter functor is a right adjoint to the former.
Moreover, if $f$ is quasi-compact and quasi-separated, there are also derived functors
$$\bR f_\ast\colon D(\Xcal)\to D(\Sc), \quad \bL f^\ast\colon D(\Sc) \to D(\Xcal).$$ 
Following \cite[3.9.2]{Lip09}  one can prove that the two derived direct images agree, that is,   there is a commutative diagram
$$
\xymatrix{D(\Xcal)\ar[d]\ar[r]^{\bR f_\ast}& D(\Sc)\ar[d]\\
D(\mathfrak{Mod}(\Xcal))\ar[r]^{\bR f_\ast}&D(\mathfrak{Mod}(\Sc))\,.
}
$$
From this one sees that   also the functors  $\bR f_\ast\colon D(\Xcal)\to D(\Sc)$ and $\bL f^\ast\colon D(\Sc) \to D(\Xcal)$ are adjoint to each other.

To simplify the exposition, in this Subsection we assume that \emph{all   superschemes and all  morphisms between them are quasi-compact and separated}.
We note that there is  a projection formula
\begin{equation}\label{eq:projformula}
\bR f_\ast (\M^\bullet\otimes^{\bL} \bL f^\ast \Nc^\bullet) \iso \bR f_\ast \M^\bullet\otimes^{\bL} \Nc^\bullet\,,
\end{equation}
which is proved as in \cite[Prop.\; 5.3]{Nee96}, 
and also that  
\begin{equation}\label{eq:tensorhom}
\bR \Homsh_{\Oc_\Xcal}(\M^\bullet, \Kc^\bullet\otimes^\bL \Gc^\bullet)\simeq  \bR \Homsh_{\Oc_\Xcal}(\M^\bullet, \Kc^\bullet)\otimes^\bL \Gc^\bullet
\end{equation}
for $\M^\bullet$, $\Kc^\bullet$ and $\Gc^\bullet$ in $D(\Xcal)$, whenever either   $\M^\bullet$ or $\Gc^\bullet$ has finite homological dimension. A simple proof that extends directly to our setting can be found in \cite[1.2]{HRLMS09}.

Another interesting formula concerns cohomology flat base change in the derived category. It says that if 
$$
\xymatrix{ \Xcal_\Tc\ar[r]^{\phi_\Xcal}\ar[d]^{f_\Tc}& \Xcal \ar[d]^f
\\
\Tc\ar[r]^\phi& \Sc
}
$$
is a cartesian diagram of morphisms of superschemes and $\phi$ is flat,
there is an isomorphism
\begin{equation}\label{eq:derivedbasechange}
\phi^\ast \bR f_\ast \iso \bR f_{\Tc\ast}\phi_\Xcal^\ast
\end{equation}
of functors from $D(\Xcal)$ to $D(\Tc)$. This can be proved as in \cite[3.9.5]{Lip09}.

Let $f\colon \Xcal\to\Sc$ be  a morphism of superschemes. We want to apply Brown's representability theorem \cite[Thm.\;3.1]{Nee96} to $\bR f_\ast\colon D(\Xcal)\to D(\Sc)$. To this end, one needs two preliminary results.
The first   is the following, which is proved as \cite[Proposition 2.5]{Nee96}: 
\begin{lemma}\label{lem:compgen} The derived category $D(\Xcal)$ is a compactly generated triangulated category, that is:
\begin{enumerate}
\item it contains all small coproducts,\footnote{This is why we need to work with the unbounded derived category.} and
\item there exists a small set $P$ of \red{compact} objects in $D(\Xcal)$ such that if $\Hom(\Pc^\bullet, \M^\bullet)=0$ for every $\Pc^\bullet\in P$ then $\M^\bullet=0$ in $D(\Xcal)$.
\end{enumerate} 
\qed\end{lemma}

The second result, whose proof follows again \cite[Lemma 1.4]{Nee96},  is:
\begin{lemma}\label{lem:coprodfstar}  If $f\colon \Xcal\to\Sc$ is a morphism of superschemes, the functor $\bR f_\ast\colon D(\Xcal)\to D(\Sc)$ preserves (small) coproducts.
\qed
\end{lemma}

One can then apply Brown's representability to obtain an extension to the super setting of  \cite[Example 4.2 and 6]{Nee96}:
\begin{prop}[Relative Grothendieck duality]\label{prop:duality} If $f\colon \Xcal\to\Sc$ is a  morphism of  superschemes, the functor $\bR f_\ast\colon D(\Xcal)\to D(\Sc)$ has a right adjoint $f^!\colon D(\Sc)\to D(\Xcal)$.   So there is a functorial isomorphism
$$
\Hom_{D(\Sc)}(\bR f_\ast \M^\bullet, \Nc^\bullet)\iso \Hom_{D(\Xcal)}(\M^\bullet, f^!\Nc^\bullet) 
$$
for $\M^\bullet$ in $D(\Xcal)$ and $\Nc^\bullet$ in $D(\Sc)$.
Moreover, if $f\colon \Xcal\to\Sc$ is  proper,  there is a sheaf version of the duality isomorphism, namely, there is a functorial isomorphism
$$
\bR \Homsh_{\Oc_\Sc}(\bR f_\ast \M^\bullet, \Nc^\bullet)\iso \bR f_\ast \bR\Homsh_{\Oc_\Xcal}(\M^\bullet, f^!\Nc^\bullet)\,.
$$
\qed
\end{prop}

\begin{defin}\label{def:dualizingcomplex}
The object $\Dcal^\bullet_f:=f^!\Oc_\Sc$ in $D(\Xcal)$ is called the dualizing complex of $f$.
\end{defin}
The dualizing complex determines in many cases the functor $f^!$.  
The following  particular case of \cite[Thm.\;5.4]{Nee96} will  suffice to our purposes.

\begin{prop}\label{prop:dualizing} Let $f\colon \Xcal \to \Sc$ be a proper morphism of superschemes. If $\Nc^\bullet$ is of finite homological dimension in $D(\Sc)$  there is an isomorphism
$$
f^!\Nc^\bullet\simeq  \bL f^\ast \Nc^\bullet\otimes^{\bL} f^!\Oc_\Sc\,.
$$
\end{prop}
\begin{proof} By Grothendieck duality and Equations \ref{eq:tensorhom} and \ref{eq:projformula} one has
\begin{align*}
 \bR f_\ast \bR & \Homsh_{\Oc_\Xcal}(\M^\bullet, \bL f^\ast\Nc^\bullet\otimes^{\bL} f^!\Oc_\Sc)
 \simeq  \bR f_\ast(\bR \Homsh_{\Oc_\Xcal}(\M^\bullet,  f^!\Oc_\Sc)\otimes^\bL \bL f^\ast \Nc^\bullet)\\[2pt]
&\simeq  \bR f_\ast\bR \Homsh_{\Oc_\Xcal}(\M^\bullet,  f^!\Oc_\Sc)\otimes^\bL \Nc^\bullet
 \simeq  \bR \Homsh_{\Oc_\Sc}(\bR f_\ast \M^\bullet, \Oc_\Sc)\otimes^\bL \Nc^\bullet 
\\[2pt]
&\simeq  \bR \Homsh_{\Oc_\Sc}(\bR f_\ast \M^\bullet, \Nc^\bullet)
\simeq  \bR f_\ast \bR\Homsh_{\Oc_\Xcal}(\M^\bullet, f^!\Nc^\bullet)
\end{align*}
for every $\M^\bullet$ in $D(\Xcal)$. Then one concludes.
\end{proof}

\subsubsection{Local properties}

Relative Grothendieck duality is   local   both on the base and on the source.
The locality on the base, that is to say, the base change property under open immersions, was proved for schemes by Neeman \cite[Lemma 6.1]{Nee96} as a preliminary lemma for the sheaf version of relative Grothendieck duality. His proof passes directly to the super setting, so that one has:

\begin{prop}\label{prop:openduality} Let $f\colon \Xcal\to\Sc$ be a  proper morphism of superschemes and let $j\colon \U\hookrightarrow \Sc$ be an open subsuperscheme. Consider the cartesian diagram
$$
\xymatrix{\Xcal_\U\ar@{^(->}[r]^{i}\ar[d]^{f_\U} & \Xcal\ar[d]^f
\\ 
\U\ar@{^(->}[r]^{j}& \Sc\,.
}
$$
There is an isomorphism 
$$
i^\ast f^!  \iso f_\U^! j^\ast 
$$
of functors  $D(\Sc)\to D(\Xcal_\U)$.
\qed
\end{prop}

Locality on the source is also known as \emph{independence of the compactification}. The statement is:

\begin{prop}\label{prop:independence} Let $f_1\colon \Xcal_1\to\Sc$, $f_2\colon \Xcal_2\to\Sc$ be proper morphism of superschemes, and let $j_1\colon \Vc\hookrightarrow \Xcal_1$, $j_2\colon \Vc\hookrightarrow \Xcal_2$  be open immersions such that the diagram
$$
\xymatrix@R=0.9pt{ & \Xcal_1\ar[rd]^{f_1} &
\\
\Vc\ar@{^(->}[ru]^{j_1}\ar@{^(->}[rd]_{j_2} & & \Sc
\\
 &  \Xcal_2\ar[ru]_{f_2} &
}
$$
commutes. There is a functorial isomorphism
$$
j_1^\ast f_1^! \iso j_2^\ast f_2^!\,.
$$
\end{prop}
\begin{proof} If $\bar\Xcal$ is the superscheme-theoretic closure of the immersion $\Vc\hookrightarrow \Xcal_1\times_\Sc\Xcal_2$ induced by $(j_1,j_2)$, we have a commutative diagram
$$
\xymatrix@R=11pt{ & \Xcal_1\ar[rd]^{f_1} & 
\\
\Vc \ar@{^(->}[ru]^{j_1} \ar@{^(->}[r]_{\bar \jmath } \ar@{^(->}[rd]_{j_2}& \bar\Xcal\ar[r]^g \ar[u]^{\pi_1}\ar[d]^{\pi_2}& \Sc
\\
 & \Xcal_2\ar[ru]_{f_2} & 
}
$$
where $\pi_1$, $\pi_2$ and $g$ are proper. One is then reduced to proving that $j_1^\ast f_1^! \iso \bar \jmath ^\ast g^!$ and $j_2^\ast f_2^! \iso \bar \jmath ^\ast g^!$. It follows that we can assume that there is a proper morphism $g\colon \Xcal_1\to\Xcal_2$ such that $f_1=f_2\circ g$ and $j_2=g\circ j_1$. Let us consider the cartesian diagram
$$
\xymatrix{\Xcal_1\times_{\Xcal_2}\Vc \ar@{^(->}[r]^(.6){i} \ar[d]^{g_\Vc} & \Xcal_1 \ar[d]^g
\\
\Vc \ar@{^(->}[ru]^{j_1}\ar@{^(->}[r]^{j_2}  \ar@/^/[u]^{s} & \Xcal_2
} 
$$
where $s=(j_1,\Id)\colon \Vc \hookrightarrow \Xcal_1\times_{\Xcal_2}\Vc$. Notice that $s$ is a closed immersion because it is a section of the proper morphism $g_\Vc$, and since it is also an open immersion, it has to be an isomorphism with a connected component of $\Xcal_1\times_{\Xcal_2}\Vc$. Then  $s^!=s^\ast$ and then applying  Proposition \ref{prop:openduality} one has
$$
j_1^\ast g^!\iso s^\ast i^\ast g^! \iso s^\ast h_{\Vc}^! j_2^\ast \iso s^! g_{\Vc}^! j_2^\ast \iso j_2^\ast\,.
$$
It follows that $j_2^\ast f_2^!\iso j_1^\ast g^! f_2^!\iso j_1^\ast f_1^!$.
\end{proof}

\subsubsection{Flat base change for duality}

Relative Grothendieck duality is sometimes compatible with   base change, in particular it is compatible with flat base changes. We shall give a simple proof in the particular case of flat \red{and proper} morphisms, since the proofs available for  general flat morphisms of schemes (see e.g.~\cite[Thm.\;4.4.1]{Lip09}) are not   easily translated to superschemes. 

Let us consider a cartesian diagram of morphisms of superschemes
$$
\xymatrix{ \Xcal_\Tc\ar[r]^{\phi_\Xcal}\ar[d]^{f_\Tc}& \Xcal \ar[d]^f
\\
\Tc\ar[r]^\phi& \Sc
}
$$
where $\phi$ is flat. Taking right adjoints in Equation \eqref{eq:derivedbasechange} one gets a functorial isomorphism
$$
\bR \phi_{\Xcal\ast} f_\Tc^!\iso f^! \bR \phi_\ast\,.
$$
Composing the inverse morphism with the adjunction morphism $\Id\to \bR \phi_{\ast}\circ \phi^\ast$ one has a morphism $f^! \to \bR \phi_{\Xcal\ast} f_\Tc^!\phi^\ast$ and then, by adjunction of $\bR\phi_{\Xcal\ast}$ and $\phi_\Xcal^\ast$, a base change morphism 
\begin{equation}\label{eq;dualitybasechange}
\phi_\Xcal^\ast f^! \to f_\Tc^! \phi^\ast\,.
\end{equation}

\begin{prop}\label{prop:dualitybasechange} If $f$ is \red{proper}, the flat base change morphism \ref{eq;dualitybasechange} yields a functorial isomorphism
$$
\phi_\Xcal^\ast f^! \Nc^\bullet\iso f_\Tc^! \phi^\ast\Nc^\bullet
$$
for every object $\Nc^\bullet$ in $D(\Sc)$ of finite homological dimension. In particular, one has an isomorphism
$$
\phi_\Xcal^\ast \Dcal^\bullet_f\iso  \Dcal^\bullet_{f_\Tc}\,.
$$
\end{prop}
\begin{proof} By Proposition \ref{prop:openduality} we can assume that $\Sc=\SSpec\As$ is affine. Then $\Xcal$ can be covered by affine open sub-superschemes of finite type over $\As$,   that   can be embedded as open sub-superschemes of a superprojective superscheme over $\As$. By the Proposition \ref{prop:independence} on the independence of the compactification, and again by Proposition \ref{prop:openduality}, we can assume that $\Tc$ is   affine too and that $f$ is superprojective. 

Take a relatively very ample line bundle $\Oc_\Xcal(1)$. If we write $g=\phi\circ f_\Tc=f\circ \phi_\Xcal$, by applying cohomology base change (Equation \eqref{eq:derivedbasechange}), relative Grothendieck duality for $f$ and $f_\Tc$ (Proposition \ref{prop:duality}), and Proposition \ref{prop:dualizing}, one gets for every pair $r$, $s$ of integer numbers and every object $\Nc^\bullet$ in $D(\Sc)$ of finite cohomological dimension, functorial isomorphisms:
\begin{align*}
\Hom_{D(\Xcal_\Tc)} &  (\Oc_{\Xcal_\Tc}(r)[s],\phi_\Xcal^\ast f^! \Nc^\bullet)
\simeq \Hom_{D(\Xcal_\Tc)}(\phi_\Xcal^\ast\Oc_{\Xcal}(r)[s],\phi_\Xcal^\ast f^! \Nc^\bullet)
\\[2pt]
& \simeq   \Hom_{D(\Xcal)}(\Oc_\Xcal(r)[s],\bR\phi_{\Xcal\ast}\phi_\Xcal^\ast f^! \Nc^\bullet)
 \\[2pt] &   \simeq   \Hom_{D(\Xcal)}(\Oc_\Xcal(r)[s],\bR\phi_{\Xcal\ast}\Oc_{\Xcal_\Tc}\otimes^\bL f^! \Nc^\bullet)
\\[2pt]
& 
\simeq   \Hom_{D(\Xcal)}(\Oc_\Xcal(r)[s],\bL f^\ast \bR \phi_\ast \Oc_\Tc \otimes^\bL f^! \Nc^\bullet)
\\[2pt]
& \simeq   \Hom_{D(\Xcal)}(\Oc_\Xcal(r)[s],f^! (\bR \phi_\ast \Oc_\Tc \otimes^\bL\Nc^\bullet))
\\[2pt]
& \simeq   \Hom_{D(\Sc)}(\bR f_\ast \Oc_\Xcal(r)[s],\bR \phi_\ast \Oc_\Tc \otimes^\bL\Nc^\bullet)
\\[2pt]
& \simeq   \Hom_{D(\Sc)}(\bR f_\ast \Oc_\Xcal(r)[s],\bR \phi_\ast \phi^\ast\Nc^\bullet)
 \simeq   \Hom_{D(\Tc)}(\phi^\ast \bR f_\ast \Oc_\Xcal(r)[s], \phi^\ast\Nc^\bullet)
\\[2pt]
&\simeq  \ \Hom_{D(\Tc)}( \bR f_{\Tc_\ast} \Oc_{\Xcal_\Tc}(r)[s], \phi^\ast\Nc^\bullet)
  \simeq   \Hom_{D(\Xcal_\Tc)}(\Oc_{\Xcal_\Tc}(r)[s], f_\Tc^!\phi^\ast\Nc^\bullet).
\end{align*}
Then $\Hom_{D(\Xcal_\Tc)}(\Oc_{\Xcal_\Tc}(r)[s],\Cc^\bullet)=0$ for every $r,s$, 
where $\Cc^\bullet$ is the cone of $\phi_\Xcal^\ast f^! \Nc^\bullet\to f_\Tc^! \phi^\ast\Nc^\bullet$. Proceeding as in \cite[Example 1.10]{Nee96}, one sees that the family of   objects $\Oc_{\Xcal_\Tc}(r)[s]$ is a compact generating set of $D(\Xcal_\Tc)$, so that one gets $\Cc^\bullet=0$ in $D(\Xcal_\Tc)$ and then $\phi_\Xcal^\ast f^! \Nc^\bullet\to f_\Tc^! \phi^\ast\Nc^\bullet$ is an isomorphism.
\end{proof}

From the base change property one deduces the compatibility of the dualizing complex with products.

\begin{corol}\label{cor:productdual} Let $f_1\colon\Xcal_1\to\Sc$, $f_2\colon\Xcal_2\to\Sc$ \red{be proper and flat} morphisms of superschemes. Assume that $\Dcal^\bullet_{f_2}$ is of finite homological dimension. Then the dualizing complex of the product $f_1\times f_2\colon \Xcal_1\times_\Sc\Xcal_2 \to \Sc$ is the (cartesian) product of the dualizing complexes of the factors, that is, there is an isomorphism
$$
\Dcal^\bullet_{f_1\times f_2}\simeq  \Dcal^\bullet_{f_1}\otimes_{\Oc_\Sc}^{\bL} \Dcal^\bullet_{f_2}:=  p_1^\ast\Dcal^\bullet_{f_1}\otimes^{\bL} p_2^\ast\Dcal^\bullet_{f_2}\,,
$$
where $p_i$ are the projections of $\Xcal_1\times_\Sc\Xcal_2$ onto its factors.
\end{corol}
\begin{proof} Consider the cartesian diagram
$$
\xymatrix{\Xcal_1\times_\Sc\Xcal_2 \ar[d]^{p_2}\ar[r]^(.6){p_1}\ar[rd]^{f_1\times f_2} & \Xcal_1 \ar[d]^{f_1}
\\
\Xcal_2 \ar[r]^{f_2} & \Sc
}
$$
Propositions \ref{prop:dualizing} and \ref{prop:dualitybasechange} give
\begin{align*}
 \Dcal^\bullet_{f_1\times f_2}&= (f_1\times f_2)^! \Oc_\Sc \simeq  p_2^! f_2^!\Oc_\Sc \simeq  p_2^\ast f_2^!\Oc_\Sc\otimes^{\bL} p_2^! \Oc_{\Xcal_2}\simeq  p_2^\ast f_2^!\Oc_\Sc\otimes^{\bL} p_1^\ast f_1^! \Oc_{\Xcal_2}
 \\
 &= p_1^\ast\Dcal^\bullet_{f_1}\otimes^{\bL} p_2^\ast\Dcal^\bullet_{f_2}\,.
\end{align*}
\end{proof}

\subsubsection{Duality for affine morphisms}

When $f\colon\Xcal \to \Sc$ is an affine morphism of superschemes, the derived functor $\bR f_\ast$ is isomorphic to $f_\ast$. One then has:
\begin{prop}\label{prop:affinedual} Let $f\colon\Xcal \to \Sc$ be a proper affine morphism of (separated) superschemes. For every object $\Nc^\bullet$ in $D(\Sc)$ one has a functorial isomorphism
$$
f_\ast f^!\Nc^\bullet\simeq  \bR\Homsh_{\Oc_\Sc}(f_\ast\Oc_\Xcal, \Nc^\bullet)\,.
$$
\qed\end{prop}

Since any quasi-coherent sheaf on $\Xcal$ is determined by its direct image, we can write
\begin{equation}\label{eq:affinedual}
f^!\Nc^\bullet\simeq  \bR\Homsh_{\Oc_\Sc}(\Oc_\Xcal, \Nc^\bullet)\,,
\end{equation}
where the second member is endowed with its natural structure of   $\Oc_\Xcal$-module.

An interesting case is when $f$ is a closed immersion $f\colon \Xcal\hookrightarrow \Sc$. Then for every quasi-coherent sheaf $\Nc$ on $\Sc$ Equation \eqref{eq:affinedual} gives the cohomology sheaves of $f^!\Nc$:
\begin{equation}\label{eq:closedimmdual}
\Hc^{i}(f^!\Nc)\simeq  \begin{cases}\Extsh_{\Oc_\Sc}^i(\Oc_\Xcal,\Nc)& i\geq 0\\
0 & i<0
\end{cases}
\end{equation}

When the closed immersion $f\colon \Xcal\hookrightarrow \Sc$ is locally regular, that is, when its ideal $\Ic$ is locally generated by a regular sequence $(a_1,\dots,a_p,\eta_1,\dots,\eta_q)$ with the $a_i$ even and the $\eta_j$ odd \cite{OgPe84}, the sheaves $\Hc^{-i}(f^!\Nc)$ can be computed in terms of the normal sheaf $\Nc_f:=(\Ic/\Ic^2)^\ast$, as we shall see in the following paragraphs.

If $\As$ is a superring, $(a_1,\dots,a_p,\eta_1,\dots,\eta_q)$ ($\vert a_i\vert=0$, $\vert\eta_j\vert=1$), the associated Koszul complex is defined as the graded symmetric algebra  $K(a,\eta)= \operatorname{Sym}_\As L^\Pi$, where $L$ is a free module of rank $(p,q)$ with homogeneous basis $(x_1,\dots,x_p,\theta_1,\dots,\theta_q)$ ($\vert x_i\vert=0$, $\vert\theta_j\vert=1$) equipped with the differential $d=\sum_i a_i \frac{\partial}{\partial x_i^\Pi}+\sum_j \eta_j \frac{\partial}{\partial \theta_j^\Pi}$.

\begin{prop}\label{prop:ber}\cite{OgPe84}
Let $I$ be an ideal generated by a regular sequence $(a,\eta)=(a_1,\dots,a_p,$ $ \eta_1,\dots,\eta_q)$ and $\bar\As=\As/I$. One has:
\begin{enumerate} \itemsep=5pt
\item  $H_i(K(a,\eta))=0$ for $i\neq 0$ and $H_0(K(a,\eta))=\bar\As$, that is, the Koszul complex is a free resolution of $\bar\As$ as an $\As$-module.
\item $H^i(K(a,\eta)^\ast)=0$ if $i\neq p$ and $H^p(K(a,\eta)^\ast)=\operatorname{Ber} (I/I^2)^\ast$, so that there is an isomorphism
$$
\operatorname{Ext}^i_\As(\bar\As, \As)\xrightarrow{\overset{\gamma_{(a,\eta)}}{\sim}} \begin{cases} 0 & i\neq p \\
\operatorname{Ber} (I/I^2)^\ast & i=p
\end{cases}
$$
\item Moreover, if $(a',\eta')=(a'_1,\dots,a'_p,\eta'_1,\dots,\eta'_q)$ is another regular sequence generating $I$, then $\gamma_{(a',\eta')}=\operatorname{Ber}(A)\gamma_{(a,\eta)}$, where $A$ is the matrix relating the bases induced by the two regular sequences. \end{enumerate}
\qed \end{prop}

\begin{prop}\label{prop:regularimmdual} Let $f\colon\Xcal\hookrightarrow\Sc$ be a closed immersion of codimension $(p,q)$ given by an ideal $\Ic$, and let  $\Lcl$ be a line bundle on $\Xcal$.  If $\U$ is the open subsuperscheme of $\Sc$ where $f_\U\colon\Xcal_\U \hookrightarrow \U$ is locally regular, there is an isomorphism
$$
(f^!\Lcl)_{\vert \Xcal_\U}\iso (f_\U^!\Lcl_{\U})\iso \Ber\Nc_{f\vert \U}\otimes f^\ast\Lcl_{\vert\Xcal_\U}[-p]\,,
$$
in the derived category $D(\Xcal_\U)$, where $\Nc_f=(\Ic/\Ic^2)^\ast$ is the normal sheaf to $f$.
\end{prop}
\begin{proof}  By Proposition \ref{prop:openduality} one can assume that $f$ is locally regular. The result follows from Proposition \ref{prop:ber} as the local isomorphisms $\operatorname{Ext}^i_\As(\bar\As, \As)\iso \operatorname{Ber} (I/I^2)^\ast$ glue because of property (3).
\end{proof}

\subsubsection{Duality for smooth morphisms}

We can now   compute the dualizing complex of a proper smooth morphism. The first step is the following result on the dualizing complex of   projective superspace.

\begin{lemma}\label{lem:projdual} Let $\pi\colon\Xcal=\Ps^{m,n}_\As\to \Sc=\SSpec\As$ be the projective superspace over a superring $\As$. The relative dualizing complex is of the form
$$
\Dcal^\bullet_\pi\simeq \Lcl[m]\,,
$$
where $\Lcl$ is a line bundle.
\end{lemma}
\begin{proof}  One has $\Xcal=\SProj \As[x_0,\dots,x_m,\theta_1,\dots,\theta_n]$,  and $\Xcal$  is covered by the affine open sub-superschemes $j_i\colon \U_i\hookrightarrow \Xcal$, complementary to the closed sub-superschemes defined by the ideals $\pf_i$ generated by homogeneous localization with respect to  $x_i$. Then,
$$
j_{i\ast}j_i^\ast \Dcal^\bullet_\pi \iso \varinjlim_{r}\bR\Homsh_{\Oc_\Xcal}( \pf_i^r,\Dcal^\bullet_\pi)\iso  \varinjlim_{r} \bR \Hom_\As(\bR f_\ast \pf_i^r,\As) 
$$
by duality. Since $\pf_i$ is isomorphic with $\Oc_\Xcal(-1)$, one has that $R^i f_\ast\pf_i^r=0$ for $i\neq m$ and that $R^m f_\ast \pf_i^r\simeq H^m(\Xcal,\Oc_\Xcal(-r))$ is a free ($\Z_2$-graded) $\As$-module  for $r\gg0$ (Proposition \ref{prop:cohomoden}). It follows that 
\begin{equation}\label{eq:localiso}
j_{i\ast}j_i^\ast \Dcal^\bullet_\pi \iso  \varinjlim_{r} \bR \Hom_\As(R^m f_\ast \pf_i^r [-m],\As)\iso (\varinjlim_{r} H^m(\Xcal,\Oc_\Xcal(-r))^\ast)[m] \,.
\end{equation}
This proves that $\Dcal^\bullet_\pi\simeq \Lcl[m]$ for a quasi-coherent sheaf $\Lcl$ on $\Xcal$ such that 
$$
\Gamma(U_i,\Lcl)\iso \varinjlim_{r} H^m(\Xcal,\Oc_\Xcal(-r))^\ast\  \quad \text{for every $i$}\,.
$$
Let us prove that $\Lcl$ is a line bundle. In the case of the ordinary projective space $X=\Ps_A^m$ one has 
$$
 \varinjlim_{r} H^m(X,\Oc_X(-r))^\ast\iso A[x_0/x_i,\dots,\widehat{x_i/x_i},\dots,x_m/x_i]\,.
$$
In our setting, from the expression $\Oc_\Xcal\simeq \Oc_X\oplus \bigoplus_{1\leq j_1<\dots<j_p\leq n}\Oc_X(-p) \theta_{i_1}\cdot\dots\cdot \theta_{i_p}$ we obtain isomorphisms (that do not glue on $U_i\cap U_j$) 
$$
\varinjlim_{r} H^m(\Xcal,\Oc_\Xcal(-r))^\ast \iso A[x_0/x_i,\dots,\widehat{x_i/x_i},\dots,x_m/x_i,\theta_1,\dots,\theta_n]\simeq \Gamma(U_i,\Oc_\Xcal)\,,
$$
and then $\Lcl$ is a line bundle by Equation \eqref{eq:localiso}.
\end{proof}

As in the classical case one proves that:
\begin{lemma}\label{lem:smoothreg} If $f\colon \Xcal\to \Sc$ is smooth and $j\colon\Ycal\hookrightarrow \Xcal$ is a closed immersion such that the composition $g=f\circ j\colon \Ycal\to \Sc$ is smooth, then $j$ is locally regular.
\qed
\end{lemma}

We now recover the computation of the dualizing sheaf of a smooth morphism \cite{OgPe84, Penk83}.
\begin{lemma}\label{lem:smoothdual}
Let $f\colon \Xcal\to \Sc$ be a proper smooth morphism of superschemes of relative even dimension $m$. The dualizing complex of $f$ has a unique cohomology sheaf in degree $-m$, that is, there is an isomorphism
$$
\Dcal^\bullet_f \iso \omega_f[m] 
$$
in the derived category $D(\Xcal)$, where $m$ is the relative even dimension of $f$. Moreover, $\omega_f$ is a line bundle.
\end{lemma}
\begin{proof} 
This is a local question so that we can   assume that $\Sc$ is affine, $\Sc=\SSpec\As$ (Proposition \ref{prop:openduality}). Then we can cover $\Xcal$ by affine superschemes $\Vc$ smooth of finite type over $\As$. We have a closed immersion $j\colon\Vc\hookrightarrow \U=\SSpec \As[x_1,\dots,x_p,\theta_1,\dots\theta_q]$, where $(x_i,\theta_j)$ are free variables with $\vert x_i\vert=0$, $\vert\theta_j\vert=1$. Moreover $j$ is locally regular by Lemma \ref{lem:smoothreg}.
If we denote by $\bar   \jmath \colon \bar\Vc\hookrightarrow \Ps^{p,q}_\As$ the induced closed immersion, by Proposition \ref{prop:independence} we have to prove that $\Dcal^\bullet_{\pi\circ \bar \jmath \vert \Vc} \iso \omega[m]$ for a line bundle $\omega$ on $\Vc$, where $\pi\colon \Ps^{p,q}_\As\to \Sc$ is the projection. But $\Dcal^\bullet_\pi\iso \Lcl[p]$ for a line bundle $\Lcl$ on $\Ps^{p,q}_\As$ by Lemma \ref{lem:projdual} and one  finishes by  Proposition \ref{prop:regularimmdual}.
\end{proof}

\begin{prop}\label{prop:smoothdual} In the situation of Lemma \ref{lem:smoothdual}, there is an isomorphism
$$
\omega_f\iso \Ber(\Xcal/\Sc)=\Ber(\Omega_{\Xcal/\Sc})\,.
$$
\end{prop}
\begin{proof} If $\delta\colon \Xcal\hookrightarrow \Xcal\times_\Sc\Xcal$ is the diagonal immersion, so that $f=p_1\circ (f\times f)$, one has
$$
\Dcal^\bullet_f=f^!\Oc_S \iso \delta^! \Dcal^\bullet_{f\times f}\,.
$$
By Lemma \ref{lem:smoothdual} $\Dcal^\bullet_f\iso \omega_f[m]$. where $\omega_f$ is a line bundle, which is of finite homological dimension. Then we can apply Corollary \ref{cor:productdual} to get $\Dcal^\bullet_{f\times f}\iso \Dcal^\bullet_f\otimes_{\Oc_\Sc}^{\bL} \Dcal^\bullet_f\iso \omega_f \otimes_{\Oc_\Sc} \omega_f [2m]$. Thus 
$$
\omega_f[m]\iso \Dcal^\bullet_f\iso \delta^! (\omega_f \otimes_{\Oc_\Sc} \omega_f)  [-2m]\iso \delta^\ast (\omega_f \otimes_{\Oc_\Sc} \omega_f)\otimes \Dcal^\bullet_\delta  [2m]\iso \omega_f \otimes\omega_f\otimes \Dcal^\bullet_\delta  [2m]\,.
$$
by  Proposition \ref{prop:dualizing}. Since $f$ is smooth, $\delta$ is locally regular by Lemma \ref{lem:smoothreg}; then Proposition \ref{prop:regularimmdual} yields $\Dcal^\bullet_\delta\iso \Ber(\Delta_f/\Delta_f^2)^\ast[-m]=\Ber(\Omega_{\Xcal/\Sc})^\ast[-m]$ and the result follows.
\end{proof}
\begin{corol} Let $\pi\colon\Xcal=\Ps^{m,n}_\Sc\to \Sc$ be a projective superspace over a locally noetherian superscheme $\Sc$. Then 
$$
\omega_\pi\iso \Oc_\Xcal(n-m-1)\,.
$$
\end{corol}
\begin{proof} One has $\Ber(\Ps^{m,n}_\Ss/\Sc)\iso \Oc_\Xcal(n-m-1)$ \cite{Ma88}.
\end{proof}

\section{The Hilbert superscheme}\label{s:hilbert}

\subsection{The super Hilbert functor}

All superschemes we consider are locally noetherian. 

Let $f\colon \Xcal \to \Sc$ be a  morphism of superschemes.
For every superscheme $\phi\colon \Tc \to \Sc$ over $\Sc$ consider the cartesian diagram
$$
\xymatrix{\Xcal_\Tc=\Xcal\times_\Sc \Tc \ar[r]^(.7){\phi_\Xcal} \ar[d]^{f_\Tc} & \Xcal\ar[d]^f \\
\Tc \ar[r]^{\phi}& \Sc\,.
}
$$

\begin{defin} The relative super Hilbert functor is the functor on the category of superschemes over $\Sc$ that  to every superscheme morphism $\Tc\to\Sc$ associates  the family 
$\SHilbf_{\Xcal/\Sc}(\Tc)$ of all the closed sub-superschemes 
$$
\xymatrix{ \Zc \ar@{^{(}->}[r]\ar[rd]_g & \Xcal_\Tc\ar[d]^{f_\Tc}\\
& \Tc
}
$$
that are {proper and} flat over $\Tc$.
\end{defin}

The super Hilbert functor is a sheaf for the Zariski topology of superschemes.

Assume that $f\colon \Xcal \to \Sc$ is superprojective morphism with a relative very ample line bundle $\Oc_\Xcal(1)$ (Remark \ref{rem:ample}).
We can decompose the super Hilbert functor into subfunctors parametrizing closed sub-superschemes of the fibres with fixed super Hilbert polynomial (Definition \ref{def:Hilbpol}):

\begin{defin}\label{def:superHilbert} Let $\bP=(P_+,P_-)$ be a pair  of polynomials with rational coefficients. 
The super Hilbert functor $\SHilbf_{\Xcal/\Sc}^{\bP}$ is the subsheaf of $\SHilbf_{\Xcal/\Sc}$ given by the relative closed sub-superschemes of $f_\Tc\colon\Xcal_\Tc\to \Tc$ whose super Hilbert polynomial is $\bP$.
\end{defin}

When $f\colon\Xcal\to\Sc$ is quasi-superprojective, so that $f=\bar f\circ j$, where $j\colon\Xcal\hookrightarrow\bar\Xcal$ is an open immersion and $\bar f$ is superprojective, as the objects of $\SHilbf_{\Xcal/\Sc}(\Tc)$ are also objects of $\SHilbf_{\bar\Xcal/\Sc}(\Tc)$ one can   decompose $\SHilbf_{\Xcal/\Sc}(\Tc)$ as the union of the subfunctors $\SHilbf_{\Xcal/\Sc}^{\bP}$ given by the relative closed sub-superschemes of $f_\Tc\colon\Xcal_\Tc\to \Tc$, proper over $\Tc$, whose super Hilbert polynomial is $\bP$.

\subsection{Statement of the existence  {theorems}}

Our aim is to prove the following representability  {theorems}. The first deals with the superprojective case, while the second
establishes the result under more general conditions. This generality will be needed to prove the representability of the super Picard functor.

\begin{thm}[Existence, superprojective case]\label{thm:hilbrepres} Let $f\colon \Xcal \to \Sc$ be a superprojective (quasi-superprojective) morphism with a relative very ample line bundle $\Oc_\Xcal(1)$.
\begin{enumerate}
\item
For any super Hilbert polynomial $\bP$, the super Hilbert functor $\SHilbf_{\Xcal/\Sc}^\bP$ is representable by an $\Sc$-superscheme $\SHilb^\bP (\Xcal/\Sc) \to \Sc$. Moreover,  $\SHilb^\bP (\Xcal/\Sc)$ is proper over $\Sc$ (is an open sub-superscheme of a proper superscheme over $\Sc$). Then, $\SHilb^\bP (\Xcal/\Sc) \to \Sc$ is of finite type and separated.
\item
As a consequence, the super Hilbert functor $\SHilbf_{\Xcal/\Sc}$  is representable by the disjoint union $\SHilb(\Xcal/\Sc)$ of the $\Sc$-superschemes $\SHilb^\bP (\Xcal/\Sc)$ corresponding to the various super Hilbert polynomials. 
\end{enumerate}
\qed
\end{thm}

\begin{thm}[Existence, general case]\label{thm:hilbrepres2} Let $\Sc$ be noetherian and let $f\colon \Xcal \to \Sc$ be a separated morphism of superschemes. Assume that there exists a proper,  faithfully flat morphism $\bar\Xcal \to \Xcal$ such that $\bar\Xcal \to \Sc$ is {quasi-superprojective. Then the super Hilbert functor of $\Xcal/\Sc$ is representable by a closed sub-superscheme  $\SHilb(\Xcal/\Sc)$} of the super-Hilbert scheme $\SHilb(\bar\Xcal/\Sc)$ of $\bar\Xcal/\Sc$. Moreover, $\SHilb(\Xcal/\Sc)$ is locally of finite type and separated over $\Sc$.
\end{thm}

\begin{rem} Notice that, due to Corollary \ref{cor:grassproj}, Theorem \ref{thm:hilbrepres2} can be applied taking for $\Xcal$  any sub-superscheme of a supergrassmannian over $\Sc$.
\end{rem}
\begin{proof} We prove Theorem \ref{thm:hilbrepres2} using Theorem \ref{thm:hilbrepres}, which will be in turn proved later on.
Set $\Ycal= \bar\Xcal\times_\Xcal \bar\Xcal$. Note that $\Ycal$ can be identified with the superschematic preimage of the relative diagonal under the morphism
$$
\bar\Xcal\times_\Sc \bar\Xcal \to \Xcal\times_\Sc \Xcal\,.
$$
Since $\Xcal$ is separated over $\Sc$, we deduce that $\Ycal$ is a closed sub-superscheme of $\bar\Xcal\times_\Sc \bar\Xcal$, and then it is {quasi-superprojective} over $\Sc$. 
Thus by Theorem \ref{thm:hilbrepres} we have a Hilbert superscheme  $\SHilb(\Ycal/\Sc)$ which is separated over $\Sc$.

Note that, as $\bar\Xcal$ is {quasi-superprojective} over $\Sc$, the morphism $\bar\Xcal \to \Xcal$ is quasi-compact. By the effective
descent for sub-superschemes (Proposition \ref{prop:subdescent}), to give a sub-superscheme of $\Xcal\times_\Sc \Tc$ (where $\Tc$ is a superscheme over $\Sc$) is
equivalent to giving a sub-superscheme $\Zc$ of $\bar\Xcal_\Tc$ such that $p_1^{-1}(\Zc) = p_2^{-1}(\Zc)$ in $\Ycal_\Tc$, where $p_1$, $p_2$ are the projections of $\Ycal_\Tc=\bar\Xcal_\Tc\times_{\Xcal_\Tc} \bar\Xcal_\Tc$ onto its factors. 

Thus, the superHilbert functor of $\Xcal/\Sc$ is represented by the closed sub-superscheme of $\SHilb(\bar\Xcal/\Sc)$ given as the superschematic preimage of the relative diagonal under the morphism
$$
\SHilb(\bar\Xcal/\Sc) \xrightarrow{(p_1^{-1},p_2^{-1})} \SHilb(\Ycal/\Sc)\times_\Sc \SHilb(\Ycal/\Sc)\,.
$$
\end{proof}

We now prove Theorem \ref{thm:hilbrepres}. This will be done in several steps and reducing to simpler situations. 

\begin{lemma}\label{lem:reductionsteps}
Let $f\colon\Xcal\to \Sc$ be a {quasi-superprojective} morphism.
\begin{enumerate}
\item If there is a covering of $\Sc$ by  open sub-superschemes $\Ucal$ such that the functors $\SHilbf_{\Xcal_\Ucal/\Ucal}^\bP$ are representable, then $\SHilbf_{\Xcal/\Sc}^\bP$ is representable as well. Moreover $\SHilb^\bP (\Xcal/\Sc) \to \Sc$ is proper if and only if all the local morphisms $$\SHilb^\bP (\Xcal_U/\Ucal) \to \Ucal$$ are proper.
\item If $j\colon \Ycal \hookrightarrow \Xcal$ is a closed (open) immersion of $\Sc$-schemes, then the functor morphism $j^\ast\colon \SHilbf_{\Ycal/\Sc}^\bP \to \SHilbf_{\Xcal/\Sc}^\bP$ is representable by closed (open)  immersions.
\end{enumerate}
\end{lemma}
\begin{proof}
(1) It follows from the sheaf condition for the relative super Hilbert functor.

(2) One has to show that given a superscheme morphism $\phi\colon \Tc \to \Sc$ and a closed (open) sub-superscheme $\delta\colon \Zc\hookrightarrow \Xcal_\Tc$ flat {and proper} over $\Tc$ with Hilbert polynomial $\bP$, there exists a closed (resp. open) sub-superscheme $\Tc'\hookrightarrow \Tc$ with the following universal property: for any $\Tc$-superscheme $\psi\colon \Ucal\to \Tc$, $\psi$ factors through $\Tc'\hookrightarrow \Tc$ if and only if $\Zc_\Ucal$ is the image of 
$j^\ast\colon \SHilbf_{\Ycal/\Sc}^\bP(\Ucal) \to \SHilbf_{\Xcal/\Sc}^\bP(\Ucal)$.

In the case of a closed immersion, this is equivalent to saying that the pull-back epimorphism  
$\Oc_{\Xcal_\Ucal}\xrightarrow{\psi^\ast(\delta^\ast)} \Oc_{\Zc_\Ucal}\to 0$
factors through $\psi^\ast(j^\ast)$:
$$
\xymatrix{
0\ar[r] & \ker \psi^\ast(j^\ast) \ar[r]& 
\Oc_{\Xcal_\Ucal} \ar[r]^{\psi^\ast(j^\ast)}  \ar@{->>}[d]_{\psi^\ast(\delta^\ast)} & \Oc_{\Ycal_\Ucal}\ar[r]  \ar@{.>>}[dl]& 0
\\
&&\Oc_{\Zc_\Ucal} &&\,.
}
$$
Then $\Tc'$ is  the zero locus of the composition morphism $\ker \psi^\ast(j^\ast) \to \Oc_{\Xcal_\Tc}\to \Oc_{\Zc_\Tc}$, which exists by Corollary \ref{cor:zerolocus} because $\Oc_{\Zc_\Tc}$ is flat over $\Tc$.

When $j\colon \Ycal \hookrightarrow \Xcal$ is an open immersion, if its enough to take $\Tc'$ as the complementary of sub-superscheme $f(\Xcal_\Tc-j(\Ycal_\Tc))$, which is closed because $f$ is proper.
\end{proof}

One then has:
\begin{prop}\label{prop:reduction1} It is enough to prove Theorem \ref{thm:hilbrepres} when $\Sc=\SSpec \As$ is affine and noetherian, $\Xcal=\Ps_\As^{n,m}$ is a projective superspace over $\As$, and $f$ is the natural projection $\Ps_\As^{n,m}\to\Sc$.
\qed
\end{prop}

\begin{strategy} Once we are reduced to the case of the projective superspace  over a superring, the proof of the existence of the Hilbert superscheme is similar to the strategy for the classical case:
\begin{enumerate}
\item Using cohomology base change (Subsection \ref{ss:cohbasechange}) and Castelnuovo-Mumford regularity (Subsection \ref{ss:CMregularity}), one proves that, for any pair $\bP=\bP(r)=(P_+(r),P_-(r))$ of polynomials with rational coefficients, and every closed sub-superscheme $\Zc$ with $\Ps^{m,n}_\As$ of super Hilbert polynomial $\bP$, there exists an integer $q$, depending only on the coefficients of $\bP$, such that a $f_\ast\Oc_\Zc(q)$ is a locally free quotient of rank $\bP$ of the locally free sheaf $f_\ast\Oc(q)$ on $\SSpec\As$ (where $\Oc=\Oc_{\Ps^{m,n}_\As}$). This gives an immersion of the super Hilbert functor into a supergrassmannian.
\item Using generic flatness and the  flattening   (Subsection \ref{ss:genflatness}), one proves that this morphism of functors is representable by immersions. This proves that the Hilbert superscheme exists as a sub-superscheme of that supergrassmannian.
\end{enumerate}
The same results are also contained in \cite[7.1.8, 7.19]{MoZh19} in the more general situation of Quot superschemes.
\end{strategy}

\begin{remark}
The existence of the Hilbert superscheme implies the existence of the superscheme of morphisms  between two  {projective} superschemes, see Subsection
\ref{subsec:schemeofmorph}.
\end{remark}

In the following Subsections we complete the proof of Theorem \ref{thm:hilbrepres}.

\subsection{$m$-regularity for projective superschemes}\label{ss:CMregularity}
In the classical case, the Castelnuovo-Mumford $m$-regularity is used to embed the Hilbert functor into a grassmannian functor.
This procedure can be adapted to our situation  (see \cite[7.1.6]{MoZh19}). 
Let $k$ be a field and $\Xcal=\Ps_k^{m,n}$ the projective superspace over $k$. We write $\Oc(r)=\Oc_\Xcal(r)$ for every integer $r$.

\begin{defin}\label{def:mreg} A coherent sheaf $\M$ on $\Xcal$ is $r$-regular if one has
$$
H^i(X,\M(r-i))=0 \quad \text{for $i>0$.}
$$
\end{defin}

We now state a super Castelnuovo Theorem, whose proof in view of Remark \ref{rem:assoprim} 
is the same as in the  classical case \cite{Mum86,Ni07}.

\begin{prop}\label{prop:castelnuovo} If $\M$ is a coherent $r$-regular sheaf on $\Xcal$, then for every $r'\ge r$ one has:
\begin{enumerate}
\item $\M$ is $r'$-regular;
\item the natural morphism $H^0(X,\M(r'))\otimes_k H^0(X,\Oc(1)) \to H^0(X,\M(r'+1))$ is surjective;
\item $\M(r')$ is generated by its global sections.
\end{enumerate}
\qed
\end{prop}

For the construction of the super  Hilbert scheme we shall need the following result 
on $r$-regularity,  which corresponds to Mumford's theorem \cite[Thm. 2.3]{Ni07}.

\begin{thm}\label{thm:mumford}
For every pair $(m,n)$ of non-negative integers, there exist universal polynomials $F^{m,n}_+$, $F^{m,n}_0$ with integer coefficients having the following property:

for any field $k$ and every coherent sheaf  $\pcal$ of ideals on $\Xcal=\Ps_k^{m,n}$ with super Hilbert polynomial
$$
\bH(X,\pcal)(r)=\left(\sum_{0\le p\le m} a_{p,0} \binom rp, \sum_{0\le p\le m} a_{p,1} \binom rp\right)\,,
$$
the sheaf $\pcal$ is $q$-regular, with $q=\max (F^{m,n}_+(a_{00},\dots,a_{m0}), F^{m,n}_-(a_{01},\dots,a_{m1}))$.
\end{thm}

We   state a preliminary Lemma. 
Let $\Xcal'\simeq \Ps_k^{m-1,n}\hookrightarrow \Xcal$ be a super-hypersurface which does not meet any of the points corresponding to the associated primes of $\M$, as in Remark \ref{rem:assoprim}, so that there is an exact sequence
$$
0 \to \M(-1)\to \M \to \M':=\M_{\vert \Xcal'} \to 0\,.
$$

\begin{lemma}\label{lem:mumford}  If $\M'$ is $r$-regular, then:
\begin{enumerate}
\item if $i\ge 2$, then $H^i(X,\M(p))=0$ for $p\ge r-i$;
\item $H^i(X,\M(p))=0$ for $p\ge r-1+q$, with $q=\max(h_1(r-1)_0,h_1(r-1)_1)$, where $(h_1(r-1)_0,h_1(r-1)_1)=\dim_k H^1(X,\M(r-1))$.
\end{enumerate}
Thus, $\M$ is $(r+q)$-regular.
\end{lemma}
\begin{proof}
(1) For $i\ge 2$, $p\ge r-(i-1)$, the exact sequence of cohomology together with (1) of Proposition \ref{prop:castelnuovo}, give isomorphisms
$$
H^i(X,\M(p-1))\iso H^i(X,\M(p)) \iso \dots\iso 0\,.
$$
(2) If $s\ge r-1$, one has an exact sequence
\begin{multline*}
0 \to H^0(X,\M(p-1)) \to H^0(X,\M(p)) \xrightarrow{\alpha_p} H^0(X,\M'(p))\to \\
\to H^1(X,\M(p-1))\to H^1(X,\M(p))\to 0\,.
\end{multline*}

If $\alpha_p$ is surjective, from
$$
\xymatrix{
H^0(X,\M(p))\otimes_k H^0(X,\Oc(1)) \ar[d]\ar[r]^{\alpha_p\otimes 1} & H^0(X,\M'(p))\otimes_k H^0(X,\Oc(1)) \ar[d]\ar[r] & 0
\\
H^0(X,\M(p+1))\ar[r]^{\alpha_{p+1}} & H^0(X,\M'(p+1))\ar[d] &
\\
& 0 &
}
$$
whose second column is exact by Proposition \ref{prop:castelnuovo}, it follows that $\alpha_{p+1}$ is surjective as well. Thus, $H^1(X,\M(p-1))\iso H^1(X,\M(p))\iso\dots\iso 0$.

If $\alpha_p$ is not surjective, then $\dim_k H^1(X,\M(p)) < \dim_k H^1(X,\M(p-1))$, (with respect to the order given by Equation \eqref{eq:order}. 
Then, $\dim H^1(X,\M(p))$ decreases until it is zero.
\end{proof}

\begin{proof}[Proof of Theorem \ref{thm:mumford}] Since   cohomology is compatible with the flat base change $\Spec K \to \Spec k$ induced by a field extension, we can assume that $k$ is infinite.
We   proceed by induction on $m$, the case $m=0$ being trivial. Since $k$ is infinite, there exists a hypersurface $\Ps_k^{m-1,n}=\Xcal'\hookrightarrow \Xcal$ which does not meet the points corresponding to the associated primes of $\Oc/\pcal$ (Remark \ref{rem:assoprim}). Then $\Torsh_1^{\Oc}(\Oc',\Oc/\pcal)=0$
(Equation \eqref{eq:torsh}) so that $\pcal':= \pcal\otimes_\Oc \Oc'=\rest{\pcal}{\Xcal'}$ is a coherent ideal on $\Xcal'$, and one has
$$
0 \to \pcal(r-1)\to \pcal(r) \to \pcal'(r)\to 0\,,
$$
for every integer $r$.

Since
\begin{align*}
\chi(X,\pcal'(r))&=\chi(X,\pcal(r))-\chi(X,\pcal(r-1))
\\
& = (\sum_{0\le p\le m} a_{p,0} \left[\binom rp-\binom{r-1}p\right], \sum_{0\le p\le m} a_{p,1} \left[\binom rp - \binom{r-1}p\right])\\
& = (\sum_{0\le p\le m} a_{p+1,0} \binom{r-1}p, \sum_{0\le p\le m} a_{p,1} \binom{r-1}p)\,,
\end{align*}
by induction on $m$, $\pcal'$ is $q'$-regular with $q'=\max(Q^{m,n}_0(a_{1,0},\dots,a_{m,0}), Q^{m,n}_1(a_{1,1},\dots,a_{m,1}))$ for certain universal polynomials $Q^{m,n}_0$, $Q^{m,n}_1$.
By Lemma \ref{lem:mumford}, $\pcal$ is $(q'+q)$-regular, where 
$$
q=\max(\dim_k H^1(X,\pcal(r-1))_0, \\ \dim_k H^1(X,\pcal(r-1))_1)\,,
$$
and $H^i(X,\pcal(r-1))=0$ for $i\ge 2$. Moreover,  by Proposition \ref{prop:cohomoden}
\begin{align*}
\dim_k H^1(X,\pcal(r-1))&= \dim_k H^0(X,\pcal(r-1)) - \chi (X,\pcal(r-1)) \\
& \leq \dim_k H^0(X,\Oc(r-1))-\bP(r-1)\\
&= h_{(m,n)}(r) -\bP(r-1) \\
&= (H_0^{m,n}(r, a_{1,0},\dots,a_{m,0}), H_1^{m,n}(r, a_{1,1},\dots,a_{m,1}))
\end{align*}
for certain universal polynomials $H^{m,n}_0$, $H^{m,n}_1$.   Then $\pcal$ is $(q+\bar q)$-regular, where
$$
\bar q= \max (q+H_0^{m,n}(r, a_{1,0},\dots,a_{m,0}), q+H_1^{m,n}(r, a_{1,1},\dots,a_{m,1}))\,.
$$
\end{proof}

\subsection{Generic flatness and flattening}\label{ss:genflatness}
We start this Subsection with a result about base change without any flatness conditions.

Let $\Sc$ be a superscheme and $\M$ a coherent sheaf on $\Ps_\Sc^{m,n}$. For every $\Sc$-superscheme $\phi\colon \Tc \to \Sc$ we have the base change cartesian diagram
$$
\xymatrix{
\Ps_\Tc^{m,n}\ar[r]^\phi\ar[d]^{f_\Tc} & \Ps_\Sc^{m,n}\ar[d]^f \\
\Tc \ar[r]^\phi& \Sc
}
$$
\begin{prop} \label{prop:basechangenoflat} If $\Sc$ and $\Tc$ are noetherian, there exists an integer $r_0$ (which may depend on $\phi$ and on $\M$)  such that the base change morphism
$$
\phi^\ast f_\ast \M(r) \to f_{\Tc\ast} \M_\Tc(r)
$$
is an isomorphism for every $r\ge r_0$.
\end{prop}
\begin{proof} Since $\Sc$ and $\Tc$ are noetherian, we can assume that they are affine, $\Sc=\SSpec\As$, $\Tc=\SSpec \As'$. The sheaf $\M$ is the sheaf associated by $\Z$-homogeneous localization to the bigraded $\As$-module $\Gamma_\ast(\M)=\bigoplus_{r\ge 0} f_\ast\M(r)$ (Proposition \ref{prop:mder2}). Then $\M_\Tc$ is the sheaf associated by $\Z$-homogeneous localization to the bigraded $\As'$-module $\Gamma_\ast(\M)_\Tc\simeq \Gamma_\ast(\M)\otimes_\As \As'$. On the other hand, $\M_\Tc$ is also the sheaf associated by $\Z$-homogeneous localization to $\Gamma_\ast(\M_\Tc)= \bigoplus_{r\ge 0} f_{\Tc\ast} \M_\Tc(r)$. Thus the morphism of bigraded $\As'$-modules $\Gamma_\ast(\M)_\Tc \to \Gamma_\ast(\M_\Tc)$, given by the base change morphisms, induces an isomorphism between the sheaves associated by $\Z$-homogeneous localization.  One concludes by Proposition  \ref{prop:mder}.
\end{proof}

We now prove a generic flatness result for the simple case   we need. The proof adapts to our setting that  for the classical case given in \cite[Thm.~ 3.7.3]{Bir11}, and   is similar to the one in \cite[7.1.7]{MoZh19}.

\begin{prop}[Generic flatness]\label{prop:genericflatness}  Let $S$ be a noetherian integral scheme   and $f\colon \Xcal \to S$ a superprojective morphism. For every coherent sheaf $\M$ on $\Xcal$ there is a nonempty affine open subscheme $U$ of $S$ such that $\M_U$ is flat over $U$.
\end{prop}
\begin{proof}
 
One can assume that $S=\Spec A$ is affine 
and that $\Xcal=\Ps_A^{m,n}$ is a projective superspace over $A$. We now proceed by induction on $m$. If $m=0$, then $\Xcal=\SSpec A[\theta_1,\dots,\theta_n]$ and $f\colon \Xcal \to S$ is finite. It follows that $\M$ is finitely generated over $A$ and we conclude by \cite[Thm.~ 6.4.4]{We09}.
 
Take now $m>0$. If $S$ has only a finite number of points, we can assume   it reduces to its generic point. In this case $A$ is a field and flatness is automatic.
If $S$ has an infinite number of points, there is a  hyperplane $\Xcal'\simeq \Ps_A^{m-1,n}$ which does not meet any     point of $\Xcal$ corresponding to the (finitely many) associated primes of $\M$, Then, for every integer $r$,  one has an exact sequence
$$
0 \to \M(r) \to \M(r+1) \to \M'(r+1)\to 0\,,
$$
where $\M'=\M_{\vert\Xcal'}$, as in Remark \ref{rem:assoprim}.
By induction on $m$, $\M'_V$ is flat over $V$ for an affine open subscheme $V$ of $S$. After restriction to that open, we can assume that $\M'$ is flat over $S$. 
Then, by Theorem \ref{thm:serre} and Proposition \ref{prop:locfree}, there exists $r_0$ such that for every $r\ge r_0$ one has $R^i f_\ast\M(r)=0$, $R^i f_\ast \M'(r+1)=0$ for $i>0$, and   $f_\ast \M'(r+1)$ is locally free. Moreover, by the case $m=0$, $n=0$,  there is a nonempty affine open subscheme $U$ of $S$ such that $R^i f_\ast\M(r_0)_\Ucal$ is flat, and then locally free. 
From the exact sequence
$$
0 \to f_\ast\M(r_0) \to f_\ast\M(r_0+1)\to f_\ast\M'(r_0+1)\to  0\,, 
$$
one obtains that $f_\ast\M(r_0+1)$ is locally free on $U$. By ascending induction one sees that $f_\ast\M(r)$ is locally free on $U$ for every $r\ge r_0$. By Theorem \ref{thm:serre} $\M_U$ is flat over $U$.
\end{proof}

The result on flattening  will need   a preliminary result. In its  proof we shall use the following notation: for every affine superscheme $\Sc=\SSpec\As$, $\Sc_{rd}$ will stand for the affine superscheme $\SSpec(\As\otimes_A A_{red})$. It is a closed sub-superscheme of $\Sc$ whose ordinary underlying scheme is the reduced scheme $S_{red}=\Spec A_{red}$.

\begin{lemma}\label{lem:flatstrat1}
Let $\Sc$ be a noetherian superscheme, $f\colon \Xcal \to \Sc$ a superprojective morphism of superschemes, $\Oc_\Xcal(1)$ a relatively very ample line bundle,  and $\M$ a coherent sheaf on $\Xcal$. 
\begin{enumerate}
\item Only finitely many pairs of polynomials occur as super Hilbert polynomials ${\bH}(X_s,\M_s)$ of the fibres of $\M$.
\item There exists an integer $r_0$ such that for every $\Sc$-superscheme $\Tc\to\Sc$ and every $r\geq r_0$ one has
\begin{enumerate}
\item $R^i f_{\Tc\ast} \M_\Tc(r)=0$ for $i>0$;
\item the base change morphism 
$$
((f_\ast\M(r))_{\Tc})_s \to H^0(X_s,\M_s(r)) 
$$
is an isomorphism for every point $s\in T$.
\end{enumerate}
\end{enumerate}
\end{lemma}

\begin{proof} 
We can assume that $\Sc$   is affine, $\Sc=\SSpec\As$, and that $\Xcal$ is the superprojective space $\Xcal=\Ps_\As^{m,n}$ over $\As$.
By Serre's Theorem \ref{thm:serre}  there exists $r_0$ such that $R^il_\ast(\M(r))=0$ for $i>0$ and $r\ge r_0$. Then $H^i(X_s,\M_s(r))=0$ for  $i>0$ and $r\ge r_0$ by Corollary \ref{cor:cohombound1}. If $\Tc\to \Sc$ is a superscheme morphism, the same formula holds for every point $s\in T$, and then $R^i f_{\Tc\ast} \M_\Tc(r)=0$ for $i>0$ and $r\ge r_0$ by 
(1) of the cohomology base change Theorem \ref{thm:cohombasechange}. This proves (2a). 

Notice now that   to prove (1) we can assume that $\Sc=S$ is an ordinary scheme.
Moreover, as $((f_\ast\M(r))_{\Tc})_s\simeq ((f_\ast\M(r))_{T})_s$, to prove (2b) we can also assume that  $\Tc=T$ is an ordinary scheme.
Let $S'$ be one of the finitely many irreducible components of $S$.
By Proposition \ref{prop:genericflatness} on generic flatness there is a nonempty open sub-superscheme $U$ of $S'$  such that $\M_U$ is flat over $U$. We can assume that $U$ is affine of the form $U = \Spec A'_f$ where $A'$ is the ring of $S'$ and $f\in A'$. Then, $Z=\Spec A'/(f)\hookrightarrow A'$ is the complementary closed sub-superscheme of $U$. We  now take $Z$ as a new base superscheme and repeat the procedure. By noetherian induction this recursive process ends up with a finite number of subschemes $V_\alpha\hookrightarrow S$ such that $\M_{V_\alpha}$ is flat over $V_\alpha$ for every $\alpha$ and the induced morphism $\coprod_\alpha V_\alpha\to S$ is topologically surjective.
Since $\M_{V_\alpha}$ is flat over $V_\alpha$, the super Hilbert polynomials are constant on the fibres, so that there are only a finite number of super Hilbert polynomials ${\bH}(X_s,\M_s)$ of the fibres of $\M$ over $\Sc$. This proves (1).

By Proposition \ref{prop:basechangenoflat} for every $\alpha$ there exists an integer $r_\alpha$ such that 
\begin{equation}\label{eq:bc1}
(f_\ast\M(r))_{V_\alpha} \iso f_{V_\alpha\ast}\M_{V_\alpha}(r)
\end{equation}
for every $r\ge r_\alpha$. Write $T_\alpha=V_\alpha	\times_S T\hookrightarrow T$. Since all   higher direct images $R^i f_{V_\alpha}\M_{V_\alpha}(r)$ vanish by (2a) and $\M_{V_\alpha}$ is flat over $V_\alpha$, Proposition \ref{prop:locfree} implies that 
\begin{equation}\label{eq:bc2}
\begin{aligned}
(f_{V_\alpha\ast}\M_{V_\alpha}(r))_{T_\alpha}& \iso f_{T_\alpha \ast}\M_{T_\alpha}(r)\,,
\\
(f_{T_\alpha \ast}\M_{T_\alpha}(r))_s& \iso H^0(X_s\M_s(r))
\end{aligned}
\end{equation}
for every  $s\in T_\alpha$ and 
 $r\ge r_\alpha$.
Combining    Equations \eqref{eq:bc1} and \eqref{eq:bc2} we obtain an isomorphism
$$
((f_\ast\M(r))_{T_\alpha})_s \iso H^0(X_s\M_s(r))
$$
for every $r\ge r_\alpha$ and every point $s\in T_\alpha$.
Corollary \ref{cor:cohombound1} and the cohomology base change Theorem \ref{thm:cohombasechange} imply that for every point $s\in T_\alpha$ and every $r\ge r_\alpha$  the base change morphism is an isomorphism
$$
(f_{T_\alpha\ast}\M_{T_\alpha}(r))_s \iso H^0(X_s, \M_{s}(r))\,,.
$$
 If $r_0$ is an integer greater than all the $r_\alpha$'s the composition of the two above isomorphisms gives
$$
((f_\ast\M(r))_{T_\alpha})_s \iso H^0(X_s, \M_{s}(r))
$$
for $r\ge r_0$ and   every $s\in T_\alpha$. It follows that 
$$
((f_\ast\M(r))_T)_s \iso H^0(X_s, \M_{s}(r))\,,
$$
for every $r\ge r_0$ and   $s\in T$.

\end{proof}

\begin{lemma}\label{lem:flatstrat2} Let $\Sc$ be a noetherian superscheme and $\M$ a coherent sheaf on $\Sc$. For any pair of integer numbers $(p,q)$ the geometric locus of the points $\Tc\to \Sc$ where $\M$ is locally free of rank $(p,q)$ is a sub-superscheme of $\Sc$. In other words, there exists a sub-superscheme $\Vc_{p,q}(\M)\hookrightarrow \Sc$ such that a morphism $\Tc\to \Sc$ of superschemes factors through $\Vc_{p,q}(\M)\hookrightarrow \Sc$ if and only if $\M_\Tc$ is locally free of rank $(p,q)$.
\end{lemma}
\begin{proof} The question is local, so that it is enough to prove that for every point $s\in S$
there is an open sub-superscheme $\Ucal$ with $s\in U$ where the Lemma is true.
If $\M$ has rank $(p,q)$ at $s\in S$  by the super Nakayama Lemma (\cite{BBH91,CarCaFi11},\cite[6.4.5]{We09})  there is an open sub-superscheme $\Ucal$ with $s\in U$ and an exact sequence
$$
 \Oc_\Ucal^{r,s}\xrightarrow{\Psi} \Oc_\Ucal^{p,q} \to \M_\Ucal \to 0\,.
$$
Then $\M$ has rank smaller or equal to $(p,q)$, with respect to the order given by Equation \eqref{eq:order}, at every point of $U$. Taking $\Vc_{p,q}(\M_\Ucal)$ as the closed sub-superscheme of $\Ucal$ defined by the ideal generated by the entries of the graded matrix of $\Psi$, one finishes.
\end{proof}

The following result is part of what in the classical case is the existence of a ``flattening stratification''  (see e.g.~\cite{Ni07}).

\begin{prop} \label{prop:flatstrat}
Let $\Sc$ be a noetherian superscheme, $f\colon \Xcal \to \Sc$ a superprojective morphism of superschemes, $\Oc_\Xcal(1)$ a relatively very ample line bundle, and $\M$ a coherent sheaf on $\Xcal$. For each pair of polynomials with rational coefficients $\bP(r)=(P_+(r),P_-(r))$, the geometric locus of the points $\Tc\to \Sc$ where $\M$ has super Hilbert polynomial $\bP$ (Definition \ref{def:superHilbert}) is a sub-superscheme $\Sc_\bP$ of $\Sc$. 
More precisely, there is a sub-superscheme $\Sc_\bP\hookrightarrow\Sc$ with the following universal property: given a morphism $\phi\colon \Tc\to \Sc$ of superschemes,   $\M_\Tc$ is flat over $\Tc$ with super Hilbert polynomial $\bP(r)$ if and only if $\phi$ factors through $\Sc_\bP\hookrightarrow \Sc$.
\end{prop}
\begin{proof} The question is local on the base, so  we can assume that $\Sc=\SSpec\As$ is affine. Moreover, we can assume that $\Xcal=\Ps_\As^{m.n}$ is a projective superspace over $\As$.

By Lemma \ref{lem:flatstrat1} there exists $r_0$ such that for every $\Sc$-superscheme $\Tc\to\Sc$ one has   $R^i f_{\Tc\ast} \M_\Tc(r)=0$ for $i>0$ and  $((f_\ast \M(r))_\Tc)_s \iso H^0(X_s,\M_s(r))$
for every $r\ge r_0$. If $\M_\Tc$ is flat over $\Tc$ with super Hilbert polynomial $\bP$, then  $(f_\ast \M(r))_\Tc\simeq f_{\Tc\ast}\M_\Tc(r)$ is locally free of rank $\bP(r)$ (Proposition \ref {prop:locfree}), and then, by Lemma \ref{lem:flatstrat2}, $\Tc \to \Sc$ factors through the sub-superscheme $\Vc_{\bP(r)}(f_\ast \M(r))$ associated to the sheaf $f_\ast \M(r)$, where  $\bP(r)=(P_+(r),P_-(r))$ for every $r\ge r_0$. 
Thus $\Tc \to \Sc$ factors through the intersection sub-superscheme $\bigcap_{r\ge r_0}\Vc_{\bP(r)}(f_\ast \M(r))$ of $\Sc$ if one proves that the intersection makes sense as a superscheme.

Conversely, if $\Tc$ factors through the sub-superscheme $\Vc_{\bP(r)}(f_\ast \M(r))$ for every $r\ge r_1$ for some value $r_1\ge r_0$, then $f_{\Tc\ast} \M_\Tc(r)$ is locally free of rank $\bP(r)$ for $r\ge r_1$ and then $\M_\Tc$ is flat over $\Tc$ by Theorem \ref{thm:serre}.

To complete the proof one has to show that for some value $r_1\ge r_0$,   the sub-superscheme $\bigcap_{r\ge r_1}\Vc_{\bP(r)}(f_\ast \M(r))$ of $\Sc$ exists. This follows from the following statement:
\begin{itemize}\renewcommand{\labelitemi}{$(\bullet)$}
\item
For every integer $N\ge r_0$  the  sub-superschemes $\Vc_N=\bigcap_{r_0\le r\le N}\Vc_{\bP(r)}(f_\ast \M(r))$ of $\Sc$ form a stationary chain as $N$ grows.
\end{itemize}
To prove $(\bullet)$,  increase first $r_0$ so that Proposition \ref{prop:basechangenoflat} is true for the immersion $S\hookrightarrow \Sc$. Then, if $f_\ast \M(r)$ is locally free of rank $\bP(r)$, the same happens after the base change $S\hookrightarrow \Sc$, so that  we can assume that the base is the ordinary scheme $S$. Consider then the sub-superschemes $V_\alpha\hookrightarrow  S$ such that $\M_{V_\alpha}$ is flat over $V_\alpha$ for every $\alpha$ constructed in the proof of Lemma \ref{lem:flatstrat1}.  Since the induced morphism $\coprod_\alpha V_\alpha\to S$ is surjective it is enough to prove $(\bullet)$ after the base change $\coprod_\alpha V_\alpha\to S$, that is, we can assume that $\M$ is flat over $S$. Now, the super Hilbert polynomials of $\M$ on the fibres of $\Xcal \to S$ are all equal to a pair $\bH=(H_+,H_-)$ of polynomials.

If $\bP\neq \bH$, one has $\bP(r)\neq \bH(r)$ for $r\gg r_0$ and then $\Vc_N=\emptyset$ for $N$ big enough.

If $\bP=\bH$, then $\Vc_N=\Vc_{r_0}$ for every $N\ge r_0$ and one finishes.
\end{proof}

\subsection{Embedding into supergrassmannians}

Let $f\colon \Xcal=\Ps_\Sc^{m,n} \to \Sc$ be the projective superspace over $\Sc$, with $\Sc$ noetherian, {$\Oc_\Xcal(1)$ a relatively very ample line bundle}, and let $\bP=\bP(r)=(P_0(r),P_1(r))$ be a pair of polynomials with rational coefficients. 
From Mumford's Theorem \ref{thm:mumford} on regularity and by Proposition \ref{prop:locfree}  one obtains:

\begin{prop}\label{prop:mreg} There exists an integer $r_0$,   depending only on the coefficients of the polynomials of $\bP$, such that for every closed sub-superscheme $\Zc\hookrightarrow \Xcal_\Tc$ in $\SHilbf_{\Xcal/\Sc}^{\bP}(\Tc)$, and every integer $r\ge r_0$, the following conditions hold:
\begin{enumerate}
\item $R ^if_{\Tc\ast}(\Oc_{\Zc}(r))=0$ for $i>0$ and  $f_{\Tc\ast}(\Oc_{\Zc}(r))$ is a locally free sheaf on $\Sc$ of rank $\bP(r)$.
\item If $\pcal_\Zc$ is the ideal of $\Zc$ in $\Xcal_\Tc$, then $R ^i f_{\Tc\ast}(\pcal_\Zc(r))=0$ for $i>0$, so that 
\begin{enumerate}
\item there is an exact sequence
$$
0 \to f_{\Tc\ast}(\pcal_\Zc(r))\to f_{\Tc\ast}(\Oc_{\Xcal_\Tc}(r))\to f_{\Tc\ast}(\Oc_\Zc(r))\to 0\,,
$$
of locally free sheaves of ranks $h_{(m,n)}(r)-\bP(r)$, $h_{(m,n)}(r)$ and  $\bP(r)$, respectively (Proposition \ref{prop:cohomoden});
\item the sheaves $\pcal_\Zc(r)$, $\Oc(r)$ and $\Oc_\Zc(r)$ are generated by their global sections, that is, there is a commutative diagram
$$
\xymatrix@R=12pt@C=12pt{
0 \ar[r] & f_\Tc^\ast f_{\Tc\ast} (\pcal_\Zc(r))\ar[r] \ar[d] & f_\Tc^\ast f_{\Tc\ast} (\Oc_{\Xcal_\Tc}(r)) \ar[r]\ar[d] & f_\Tc^\ast f_{\Tc\ast} (\Oc_{\Zc}(r)) \ar[r] \ar[d] & 0\\
0\ar[r]& \pcal_\Zc(r) \ar[r]\ar[d] & \Oc_{\Xcal_\Tc}(r)\ar[r] \ar[d]& \Oc_{\Zc}(r)\ar[r]\ar[d]& 0
\\
& 0& 0& 0&
}
$$
with exact rows and columns. 
\end{enumerate}
\end{enumerate}
\qed\end{prop}

Let us fix an integer $p\ge r_0$, and write $\Ec=f_{\ast}(\Oc_\Xcal(p))$.
Since $\Ec_\Tc\simeq f_{\Tc\ast}(\Oc_{\Xcal_\Tc}(p))$ for every $\Tc \to \Sc$, by (2) of the cohomology base change Theorem \ref{thm:cohombasechange}  we have  a functor morphism
$$
\SHilbf_{\Xcal/\Sc}^{\bP} \to \Sgrassf(\Ec,\bP(p))
$$
from the super Hilbert functor to the supergrassmannian functor of the locally free quotients of  rank $\bP(p)$ of  $\Ec$  (see Definition \ref{def:supergrass}) given by
\begin{align}
\SHilbf_{\Xcal/\Sc}^{\bP}(\Tc) & \to \Sgrassf(\Ec,\bP(p))(\Tc) \\
\Zc & \mapsto f_{\Tc\ast}(\Oc_{\Zc}(p))
\end{align}
for every $\Sc$-superscheme $\Tc\to\Sc$. 

\begin{corol}\label{cor:injectve}
The above functor morphism is injective, that is, one has a functorial immersion
 $$
\SHilbf_{\Xcal/\Sc}^{\bP} \hookrightarrow \SG^\bullet\,,
$$
into the (functor of the points of the) supergrassmannian $\pi\colon\SG:=\Sgrass(\Ec,\bP(p))\to \Sc$.
\end{corol}
\begin{proof}
Let us see that the quotient $q\colon \Oc_{\Xcal_\Tc}\to \Oc_\Zc\to 0$ can be recovered from $$f_\Tc^\ast(q(p))\colon f_\Tc^\ast\Ec \to f_\Tc^\ast f_{\Tc\ast}(\Oc_\Zc(p))\to 0\,.$$
On $\SG$  the sheaf $\Ec_{\SG}\simeq \pi^\ast\Ec\simeq f_{\SG,\ast}(\Oc_{\Xcal_{\SG}}(p))$  has  a universal locally free quotient 
\begin{equation}\label{eq:univgrass}
0 \to \Kc \xrightarrow{\mu} \Ec_{\SG} \xrightarrow{\rho} \Qc \to 0
\end{equation}
so that, if $v\colon \Tc\to \SG$ is the morphism   corresponding to the surjection $f_\Tc^\ast(q(p))$,   one has $f_\Tc^\ast(q(p))=v^\ast(\rho)$. This gives an exact sequence
$$
0 \to v^\ast\Kc\simeq f_{\Tc\ast} (\pcal_\Zc(p))\xrightarrow{v^\ast(\mu)}  f_{\Tc\ast} (\Oc_{\Xcal_\Tc}(p)) \xrightarrow{f_\Tc^\ast(q(p))=v^\ast(\rho)} v^\ast \Qc\simeq f_{\Tc\ast} (\Oc_{\Zc}(p)) \to 0\,.
$$
If $g\colon f_\Tc^\ast f_{\Tc\ast} (\pcal_\Zc(p)) \to \Oc_{\Xcal_\Tc}(p)$ is the composition of $$f_\Tc^\ast(v^\ast(\mu))\colon f_\Tc^\ast f_{\Tc\ast} (\pcal_\Zc(p)) \to f_\Tc^\ast f_{\Tc\ast} (\Oc_{\Xcal_\Tc}(p))$$ and $f_\Tc^\ast f_{\Tc\ast} (\Oc_{\Xcal_\Tc}(p)) \to \Oc_{\Xcal_\Tc}(p) $, the diagram of (2)-(b) in Proposition \ref{prop:mreg} gives an exact sequence
$$
f_\Tc^\ast f_{\Tc\ast} (\pcal_\Zc(p)) \xrightarrow{g} \Oc_{\Xcal_\Tc}(p) \xrightarrow{q(p)} \Oc_\Zc(p)\to 0\,.
$$
Thus, $q(p)$ can be recovered from $v^\ast(\mu)$, and so from $f_\Tc^\ast(q(p))$, as the cokernel of $g$. Twisting by $-p$ we recover $q\colon  \Oc_{\Xcal_\Tc} \to \Oc_\Zc$ as well.
\end{proof}

\subsection{Existence of the Hilbert superscheme}

Now to prove that $\SHilbf_{\Xcal/\Sc}^{\bP} $ is representable it is enough to show  that $\SHilbf_{\Xcal/\Sc}^{\bP} \hookrightarrow \SG$ is representable by immersions, that is,   there exists a sub-superscheme $\Wc$ of $\SG$ whose points $g\colon \Tc \to \Wc$ are precisely the points $g\colon \Tc \to \SG$   belonging to the image of $\SHilbf_{\Xcal/\Sc}^{\bP}(\Tc)$.

We use the notation of the universal Equation \eqref{eq:univgrass} of the supergrassmannian $\SG$.
Let $g\colon f_{\SG}^\ast \Kc \to \Oc_{\SG}(s)$ be the composition
 $$
 \xymatrix{
  f_{\SG}^\ast \Kc \ar[r]^(.25){f_{\SG}^\ast(\mu)} \ar@/_2pc/[rr]^g&   f_{\SG}^\ast\Ec\simeq f_{\SG}^\ast f_{\SG\ast}(\Oc_{\Xcal_\SG}(s)) \ar[r] & \Oc_{\Xcal_\SG}(s)
  }\,.
 $$
 and  $\psi\colon \Oc_{\Xcal_\SG}(s) \to \Nc$ its cokernel. 
By generic flatness and the  flattening  stratification (Proposition \ref{prop:flatstrat}) for the sheaf $\Nc$ and the super Hilbert polynomial $\bar\bP(r)=\bP(r+s)$, there exists a sub-superscheme $\Ss\Hs\hookrightarrow \SG$ such that a morphism $\Tc\to\SG$ of $\Sc$-superschemes factors through $\Ss\Hs$ if and only if $\Nc$ is flat over $\Tc$ with super Hilbert polynomial $\bar\bP(r)$ (so that $\Nc(-s)$ defines a $\Tc$-valued point of $\SHilbf_{\Xcal/\Sc}^{\bP}$).

The superscheme $\Ss\Hs$ clearly represents the super Hilbert functor. One then has the following result, which completes the proof of the existence Theorem \ref{thm:hilbrepres} {due to Proposition \ref{prop:reduction1}}.

 \begin{prop} Let $f\colon \Xcal=\Ps_\Sc^{m,n} \to \Sc$ be the projective superspace over $\Sc$, with $\Sc$ noetherian and let $\bP=\bP(r)=(P_0(r),P_1(r))$ be a pair of polynomials with rational coefficients. The super Hilbert functor $\SHilbf_{\Xcal/\Sc}^\bP$ is representable by a sub-superscheme 
 $$
 \SHilb^\bP (\Xcal/\Sc) \hookrightarrow \SG
 $$ 
 of a supergrassmannian $\SG=\Sgrass(\Ec,\bP(s))$ associated to a locally free sheaf $\Ec$ over $\Sc$. Moreover, $\SHilb^\bP (\Xcal/\Sc)\to \Sc$ is a proper morphism of superschemes. 
 \end{prop}
 \begin{proof}
By the above discussion  we have only to prove that $\SHilb^\bP (\Xcal/\Sc)\to \Sc$ is proper. This follows straigthforwardly from the valuative criterion for properness for a morphism of superschemes (see Corollary \ref{cor:valuative}).
 \end{proof}
 
  \begin{remark} (Relation with the Hilbert scheme $\operatorname{Hilb}(X/S)$.) If $f\colon \Xcal \to \Sc$ is a superprojective superscheme morphism, and
 $f_{bos} \colon X \to S$ is the underlying projective morphism of ordinary schemes, one can   ask   about the relation between the underlying ordinary scheme $\Ss H:= \SHilb(\Xcal/\Sc)_{bos}$ of the Hilbert superscheme and the ordinary Hilbert scheme $H:=\operatorname{Hilb}(X/S)$, and the  relation  between the super Hilbert polynomial in the first case and the usual  Hilbert polynomial in the second.
The scheme $\Ss H$ represents the restriction of the super Hilbert functor to the category of schemes. Since every morphism $T\to \Sc$ from an ordinary scheme $T$ factors through $S$, we see that $\Ss H$ is as well the underlying scheme to the Hilbert superscheme $\SHilb(\Xcal_S/S)$. 
 
 Now, every closed subscheme of $X\to S$ flat over $S$ is also
 a closed sub-superscheme of $\Xcal_S\to S$, so that  there is an immersion of schemes
$
\varpi\colon  H  \to  \Ss H
$
which maps $H^P$ to $\Ss H^{(P,0)}$. Moreover, $\varpi$ is an open imersion, as for every  scheme morphism $T\to S$ and every closed sub-superscheme $\Zc\hookrightarrow\Xcal_T$ flat over $T$,  the geometric locus $U\subseteq T$ where $\Zc_U$ is an ordinary closed subscheme of $X_U$ is the   open subset complementary to the image by the proper morphism $f$ of the support of the ideal $\Jc_\Zc$ defining $Z$ in $\Zc$.

 However, $\Ss H$ is
 usually bigger than $H$ as there are non-purely bosonic families of closed sub-superschemes of  $\Xcal_S\to S$ that only have even parameters.  This is entirely analogous to
 what happens with the supergrassmannian of a supervector bundle $\mathcal E = \mathcal E_0\oplus \mathcal E_1$, whose
 underlying ordinary scheme is strictly bigger than the grassmannian of $ \mathcal E_0$ (cf.~Proposition \ref{prop:supergrass}).
 
 Actually one can say something more: if $\Zc\hookrightarrow \Xcal_S$ is a closed superscheme flat over $S$, the bosonic reduction $Z$ is a closed subscheme of $X$ and it is flat over $S$  by Proposition \ref{prop:filt}, so that we have a scheme morphism
 $
 \rho\colon \Ss H \to H
 $
 such that $\rho\circ\varpi=\Id$.
  \end{remark}

\subsection{Hilbert superschemes of 0-cycles} Let $f\colon \Xcal \to \Sc$ be a superprojective morphism of superschemes. 
We call \emph{Hilbert superscheme of 0-cycles} of $\Xcal $
 a superscheme which represents the super Hilbert functor associated with a constant super Hilbert polynomial
 $\mathbf P = (p,q)$, where $p$, $q$ are two nonnegative integers. It turns out that the ordinary scheme underlying this
 Hilbert superscheme has a richer structure than the ordinary Hilbert scheme of 0-cycles. 
 
 Consider for instance a smooth superprojective superscheme $\Xcal \to S$ over an ordinary scheme $S$, of relative dimension $(m,1)$.
 Then $\Xcal$ is split, and moreover its structure sheaf $\Oc_\Xcal $ has the form $\Oc_\Xcal = \Oc_X \oplus\mathcal L$,
 where $X\to S$ is the underlying ordinary scheme, and $\mathcal L$ is a line bundle on $X$.  Now let $\Zc\subset\Xcal$ 
 be a closed sub-superscheme of $\Xcal$ with super Hilbert polynomial $(p,q)$, flat over $S$. Let $\Ic_\Zc$ be the ideal sheaf of $\Zc$; it has the structure
 $$ \Ic_\Zc = \Ic_0 \oplus \Ic_1 \cdot \mathcal L,$$
 where $\Ic_0$, $\Ic_1$ are ideals of $\Oc_X$, which correspond to 0-cycles $Z_0$ and $Z_1$ of $X$ of length $p$ and $q$, respectively. Actually for $\Ic_\Zc$ to be an ideal
 of $\Oc_\Xcal$ one also needs that $\Ic_0\subseteq \Ic_1$, i.e., $Z_1\subseteq Z_0$. 
 Therefore, when $ p \ge q$ the ordinary scheme underlying the Hilbert superscheme 
$ \SHilb^{\!(p,q)} (\Xcal/S)$ is a {\em(2-step) nested Hilbert scheme of 0-cycles of $X$}, parameterizing pairs of 0-cycles $(Z_0,Z_1)$ of $X$,
with $Z_1\subseteq Z_0$. When $p<q$ the  Hilbert superscheme 
$ \SHilb^{\!(p,q)} (\Xcal/S)$ is empty. For a formal definition of the (2-step) nested Hilbert scheme of 0-cycles see e.g.~\cite{GhoSheYau}.

Nested Hilbert schemes have been recently studied intensively, for instance in connection with enumerative invariants and moduli spaces of quiver representations. One may envision that   Hilbert superschemes of 0-cycles can find applications to study these problems.

\subsection{Superschemes of morphisms of superschemes} 
\label{subsec:schemeofmorph}
Our next aim is to prove that, under  quite reasonable hypotheses,  the space of morphisms between two superschemes can be given the structure of a superscheme.
We start with a preliminary result.

\begin{lemma}\label{lem:fibreflat}
Let $\Sc$ be a noetherian superscheme and let  $f\colon \Xcal \to \Sc$, $g\colon\Zc\to\Sc$ be  flat morphisms of superschemes, locally of finite type,  and let $\phi\colon\Zc\to\Xcal$ be  a morphism of $\Sc$-superschemes. Let $z\in Z$ be a point of $Z$ and write $x=\phi(z)$, $s=g(z)=f(x)$.  If the restriction $\phi_s\colon \Zc_s\to\Xcal_s$  to the fibre is flat at $z$, then $\phi$ is flat at $z$ as well.
\end{lemma}
\begin{proof}
Directly from Corollary \ref{cor:localflat}.
\end{proof}

\begin{lemma}\label{lem:openiso} Let $\Sc$ be a noetherian superscheme, let $f\colon \Xcal \to \Sc$, $g\colon\Zc\to\Sc$ be proper and flat morphisms   of $\Sc$-superschemes, and let  $\phi\colon\Zc\to\Xcal$ be a {superprojective} morphism.   Then, the loci of the points where $\phi$ is flat or an isomorphism are open. In other words, there exists open subsuperschemes $\Ucal\hookrightarrow \Vc$ of $\Sc$ with the following universal property: For every $\Sc$-superscheme $h\colon\Tc \to \Sc$, the base change morphism $\phi_\Tc\colon \Zc_\Tc \to \Xcal_\Tc$ is flat (respect. an isomorphism) if and only if $\phi$ factors through $\Vc\hookrightarrow \Sc$ (respect. $\Ucal\hookrightarrow \Sc$).
\end{lemma}
 
\begin{proof} 
 By Proposition \ref{prop:genericflatness} on generic flatness  there is an open sub-superscheme $\Xcal'$ of $\Xcal$ which is the locus of the points where $\phi$ is flat.
$\Xcal'$ can be covered by open sub-superschemes $\Xcal_i$ such that $\phi\colon \Zc_i=\phi^{-1}(\Xcal_i)\to \Xcal_i$ is flat and superprojective. Then there exist integers $r_i$ such that $R^i \phi_\ast(\Oc_{\Zc_i}(r))=0$, $\Oc_{\Zc_i}(r)$ is relatively generated by its global sections, and $\phi_\ast \Oc_{\Zc_i}(r)$ is locally free for $r\ge r_i$ (Theorem \ref{thm:serre} and Proposition \ref{prop:locfree}). Since $\Xcal$ is noetherian, we can choose an $r$ that satisfies   $R^i \phi_\ast(\Oc_\Zc(r))=0$, $\phi^\ast\phi_\ast\Oc_\Zc(r)\to \Oc_\Zc(r)$, and $\phi_\ast \Oc_\Zc(r)$ is locally free. Let $\Vc$ be the open sub-superscheme of $\Xcal$ where $\Lcl=\phi_\ast \Oc_\Zc(r)$ has rank $(1,0)$ and $\Zc_\Vc=\phi^{-1}(\Vc)$. By Proposition \ref{prop:projmor} $\Oc_{\Zc_\Vc}(r)$ induces a morphism 
$$
\Zc_\Vc\to \Ps(\Lcl_\Vc)\simeq \Vc
$$
of superschemes over $\Vc$ which is  an isomorphism. Now, $U=S-f(X-V)$ is an open subset of $S$ because $f$ is proper. If we give $U$ the induced structure $\Ucal$ of sub-superscheme of $\Sc$, one sees that $\Ucal$ satisfies the required universal property.
\end{proof}

Let $\Xcal \to \Sc$, $\Ycal\to\Sc$ be superschemes over $\Sc$. 
\begin{defin}\label{def:morfunct}
The functor of morphisms of $\Sc$-superschemes from $\Xcal $ to $\Ycal$ is the functor on the category of $\Sc$-superschemes that associates to every $\Sc$-superscheme $\Tc\to \Sc$ the family
$$
\Homsh_\Sc(\Xcal,\Ycal)(\Tc)=\Hom_{\Tc}(\Xcal_\Tc, \Ycal_\Tc)\,,
$$
of all the morphisms $\Xcal_\Tc \to \Ycal_\Tc$ of $\Tc$-superschemes.
\end{defin}

\begin{prop}\label{prop:morpsuper} Let $\Sc$ be a noetherian superscheme and let $\Xcal \to \Sc$, $\Ycal\to\Sc$ be  {superprojective} morphisms. If $\Xcal\to\Sc$ is flat,   the  functor of morphisms $\Homsh_\Sc(\Xcal,\Ycal)$ is representable by an open sub-superscheme $\SHom_\Sc(\Xcal,\Ycal)$ of the Hilbert superscheme $\SHilb(\Xcal\times_\Sc\Ycal/\Sc)$.
\end{prop}
\begin{proof} If $f\colon \Xcal_\Tc \to \Ycal_\Tc$ is a morphism of $\Tc$-superschemes, the graph morphism $$(1,f)\colon \Xcal_\Tc \to \Xcal_\Tc\times_\Tc \Ycal_\Tc\simeq (\Xcal\times_\Sc\Ycal)_\Tc$$ is a closed immersion as $f$ is separated. Then, its graph  $\Gamma(f)=\Ima(1\times f)$ is a closed sub-superscheme of $(\Xcal\times_\Sc\Ycal)_\Tc$. Moreover, the first projection $\Gamma(f)\to \Xcal_\Tc$ is an isomorphism as it is the inverse of $(1,f)\colon \Xcal_\Tc \iso \Gamma(f)$. Then $\Gamma(f)$ is flat over $\Tc$, that is, it defines a $\Tc$-valued point of the super Hilbert functor $\SHilbf(\Xcal\times_\Sc\Ycal/\Sc)$. 
This defines a morphism of functors
$$
\gamma\colon \Homsh_\Sc(\Xcal,\Ycal)\to \SHilbf(\Xcal\times_\Sc\Ycal/\Sc)\,,
$$
given, for every $\Sc$-superscheme $\Tc$, by 
\begin{align*}
\gamma(\Tc)\colon \Homsh_\Sc(\Xcal,\Ycal)(\Tc)&\to \SHilbf(\Xcal\times_\Sc\Ycal/\Sc)(\Tc)
\\
f & \mapsto \Gamma(f) \,.
\end{align*}
One needs only to prove that $\gamma$ is representable by open immersions. Take an $\Sc$-superscheme $\Tc$ and an element $\Zc \in \SHilbf(\Xcal\times_\Sc\Ycal/\Sc)(\Tc)$, that is, a closed sub-superscheme $\Zc\hookrightarrow (\Xcal\times_\Sc\Ycal)_\Tc$ flat over $\Sc$. Since $\Ycal_\Tc\to\Tc$ is locally superprojective, the same is true for the projection $(\Xcal\times_\Sc\Ycal)_\Tc\to\Xcal_\Tc$ and then the restriction of that projection to $\Zc$ is a locally superprojective morphism $\Zc\to \Xcal_\Tc$ as well. By Lemma \ref{lem:openiso} there is an open sub-superscheme $\Ucal$ of $\Tc$ with the following universal property: a morphism of superschemes $\Rcal\to \Tc$ factors through $\Ucal\hookrightarrow\Tc$ if and only if the base change map $\Zc_\Rcal\to \Xcal_\Rcal$ is an isomorphism, that is, if $\Zc$ is the graph of a well-determined morphism $\Xcal_\Rcal\to\Ycal_\Rcal$ of $\Rcal$-superschemes. This finishes the proof.
\end{proof}

\begin{corol}\label{cor:isosuper} Let $\Sc$ be a noetherian superscheme and let  $\Xcal \to \Sc$, $\Ycal\to\Sc$  be {superprojective} and flat morphisms. Then, the functor of isomorphisms of $\Sc$-superschemes from $\Xcal$ to $\Ycal$ is representable. That is, there exists an $\Sc$-superscheme
$\Sisom_\Sc(\Xcal,\Ycal)$ such that for every $\Sc$-superscheme $\Tc$  there is  a functorial isomorphism
$$
\Hom_\Sc(\Tc, \Sisom_\Sc(\Xcal,\Ycal))\simeq \Isom_\Tc(\Xcal_\Tc,\Ycal_\Tc)\,.
$$
\qed
\end{corol}

\subsection{Superschemes embeddable into supergrassmannians}\label{sec:emb-supergrass}

Using super Hilbert superschemes we get the following characterizations of superschemes embeddable into supergrassmannians.

\begin{prop}\label{prop:supergrassm-emb}
Let $\Xcal\to \Sc$ be a proper morphism of Noetherian superschemes. Then the following conditions are equivalent:
\begin{enumerate}
\item 
there exists a closed immersion of $\Xcal$ into the relative supergrassmannian of some supervector bundle over $\Sc$;
\item
there exists a smooth surjective morphism $\bar\Xcal\to \Xcal$ of relative dimension $(0,n)$ such that $\bar\Xcal$ is strongly superprojective over $\Sc$;
\item
there exists a proper, faithfully flat morphism $g:\bar\Xcal\to \Xcal$ such that $\bar\Xcal$ is strongly superprojective over $\Sc$.
\end{enumerate}
\end{prop}
\begin{proof}
(1)$\implies$(2). This follows immediately from Corollary \ref{cor:grassproj}.

(2)$\implies$(3). We just have to observe that a smooth surjective morphism is faithfully flat. 

(3)$\implies$(1). 
The idea is to show that $\Xcal$ embeds into the Hilbert superscheme component $\SHilb^\bP(\bar\Xcal/\Sc)$ for some $\bP$. 
Recall that by Theorem \ref{thm:hilbrepres2}, 
the Hilbert superscheme $\SHilb(\Xcal/\Sc)$  exists and we have a closed immersion
$$\SHilb(\Xcal/\Sc)\xrightarrow{g^{-1}}\SHilb(\bar\Xcal/\Sc).$$
Since $\Xcal$ is separated over $\Sc$, 
the diagonal $\delta:\Xcal\hookrightarrow \Xcal\times_{\Sc}\Xcal$ is a closed immersion, so viewing it as a flat family of subschemes in $\Xcal$ parametrized by $\Xcal$,
we get a morphism 
$$\Xcal\to\SHilb(\Xcal/\Sc)\to \SHilb(\bar\Xcal/\Sc).$$

Assume that $\Xcal$ is connected. Then the above morphism factors through  $\SHilb^\bP (\bar\Xcal/\Sc)$ for some super Hilbert polynomial $\bP$.
We claim that the obtained morphism
$$\Xcal\to \SHilb(\Xcal/\Sc)\times_{\SHilb(\bar\Xcal/\Sc)} \SHilb^\bP (\bar\Xcal/\Sc)$$ 
is an isomorphism. Indeed, let $\Tc$ be an $\Sc$-superscheme and suppose we have a flat family of sub-superschemes $\Zc\sub \Xcal_\Tc$ such that
$g^{-1}(\Zc)\sub \bar\Xcal_\Tc$ has super Hilbert polynomial $\bP$. It is enough to show that $\Zc$ is a graph of a morphism $\Tc\to \Xcal$ of $\Sc$-superschemes.
In other words, we need to show that the projection $\Zc\to \Tc$ is an isomorphism.
Since $\Zc$ is flat over $\Tc$, it suffices to show that for every geometric point $t:\Spec(K)\to\Tc$, the fibre $\Zc_t$ is isomorphic to $\Spec(K)$. We know that
$g^{-1}(\Zc_t)$ is a closed sub-superscheme in $\bar\Xcal_s$, where $s:\Spec(K)\to\Sc$ is the corresponding $K$-point of $\Sc$, with the same super Hilbert polynomial as
$g^{-1}(x)$, where $x$ is a $K$-point of $\Xcal$. This implies that $\Zc_t$ is nonempty, so picking a $K$-point $z$ of $\Zc_t$, we obtain that
the sub-superschemes $g^{-1}(\Zc_t)$ and $g^{-1}(z)$ of $\bar\Xcal_s$ have the same super Hilbert polynomial. Since one sub-superscheme contains the other, this is possible only if
$\Zc_t$ is the reduced point, as claimed.

Thus, we obtained a closed immersion of $\Xcal$ into $\SHilb^\bP (\bar\Xcal/\Sc)$, and hence, into a relative supergrassmannian.

In the case when $\Xcal$ is not connected, the above argument gives an immersion of each connected component of $\Xcal$ into a relative supergrassmannian.
It remains to observe that the disjoint union of two supergrassmannians $\Sgrass((c_i,d_i),\Ec_i)$, $i=1,2$, can be embedded into the supergrassmannian
$$\Sgrass((c_1+c_2,d_1+d_2),\Ec_1\oplus \Ec_2\oplus \bigoplus_{i=1}^{c_1+c_2}\Oc e_i\oplus \bigoplus_{j=1}^{d_1+d_2}\Pi \Oc f_j)\,.$$
Namely, a subbundle $\mathcal F_1\sub \Ec_1$ of rank $(c_1,d_1)$ ($\mathcal F_2\sub \Ec_2$ of rank $(c_2,d_2)$, resp.) 
is sent to 
$$\mathcal  F_1\oplus 0\oplus \bigoplus_{i=c_1+1}^{c_1+c_2}\Oc e_i\oplus \bigoplus_{j=d_1+1}^{d_1+d_2}\Pi \Oc f_j \ \  \ 
(0\oplus \mathcal  F_2\oplus \bigoplus_{i=1}^{c_1}\Oc e_i\oplus \bigoplus_{j=1}^{d_1}\Pi \Oc f_j \text{, resp.})\,.$$
\end{proof}

\section{The Picard superscheme}\label{s:picard}

All the superschemes considered in this Section are locally noetherian.

\subsection{The super Picard functors}

If $\Xcal$ is a superscheme, we can associate to it two Picard groups:  the {Picard} group ${\Pic}(\Xcal)$ of isomorphism classes of all even line bundles on $\Xcal$, and the {total Picard} group $\Picb(\Xcal):=\Pic(\Xcal)\coprod \Pic_-(\Xcal)$ of {all  line bundles, even and odd,} on $\Xcal$, that is, all  locally free sheaves of $\Oc_\Xcal$-modules of rank either $(1,0)$ or $(0,1)$. Moreover, ${\Pic}(\Xcal)$ is a subgroup of ${\Picb(\Xcal)}$ 
{and for an ordinary scheme the group $\Pic(X)$ so defined is the usual Picard group of $X$.}
If $\Oc_\Xcal^{\Pi}=\Oc_{\Xcal,1}\oplus\Oc_{\Xcal,0}$ is the {trivial odd line bundle}, the multiplication by $\Oc_\Xcal^{\Pi}$ gives a one-to-one correspondence between ${\Pic}(\Xcal)$ and $\Pic_-(\Xcal)$.

\begin{lemma}\label{lem:Pic-bos-quot}
One has {natural isomorphisms} 
$${\Pic}(\Xcal)\simeq H^1(\Xcal,\Oc_{\Xcal,+}^*)\simeq \Pic(\Xcal/\Gamma),$$
where $\Xcal/\Gamma$ is the bosonic quotient of $\Xcal$.
\end{lemma}

\begin{proof} The proof of the first isomorphism is the same as in the classical case: we just observe that transition functions of a line bundle are
even invertible functions. The isomorphism with the usual Picard group $\Pic(\Xcal/\Gamma)$ follows from this.
\end{proof}

Let $f\colon \Xcal\to\Sc$ be a morphism of superschemes. Then the inverse image of line bundles maps {$\Picb(\Sc)$ to $\Picb(\Xcal)$ preserving the parity}.

\begin{defin} {The relative Picard group of $f$ is the quotient group}
$$
{\Pic}(\Xcal/\Sc):={\Pic}(\Xcal)/f^\ast {\Pic}({\Sc})\,.
$$
\end{defin}

\begin{defin}\label{def:Picsh} The super Picard presheaf of $f$ is the group functor on the category of  (locally noetherian) superschemes over $\Sc$ that associates to every superscheme morphism $\Tc\to\Sc$ the relative Picard group
$$
\SPicf_{\Xcal/\Sc}(\Tc)=\Pic(\Xcal\times_\Sc\Tc/\Tc)\,.
$$
The super Picard functor, or super Picard sheaf of $f$,  is the sheaf 
$$
\SPicf_{\Xcal/\Sc\et}
$$
associated with $\SPicf_{\Xcal/\Sc}$ for the \'etale topology of superschemes. 

\end{defin}

The remainder of this Section is devoted to proving a  representability theorem for the super Picard functor, that is,  the existence,  under suitable hypotheses, of the Picard superscheme associated to $f$.

\begin{thm}\label{thm:picrepres} Let $\Sc$ be a  {locally} noetherian superscheme and $f\colon\Xcal\to \Sc$ a  {locally} superprojective flat morphism which has   geometrically integral fibres (Definition \ref{def:geomint}) and is cohomologically flat in dimension $0$ (Definition \ref{def:cohomflat}). 
The super Picard functor is representable by an $\Sc$-supergroup 
$$
\SPic(\Xcal/\Sc)\to \Sc
$$
which is locally of finite type over $\Sc$.
\end{thm}
In proving this Theorem we can assume that $\Sc$ is \emph{noetherian} 
as  the representability of the even super Picard sheaf is local on $\Sc$ in  the Zariski topology of superschemes.
The strategy for a  proof consists of the study of the super Abel morphism (or Abel contraction). This is the map from positive superdivisors to even line bundles that associates to a positive superdivisor the dual of the corresponding ideal sheaf. We shall see that   this map is a projective superbundle on an open part of the 
Picard sheaf, so that  that open part is actually representable. Since the super Picard functor is a group functor which can be covered by translates of that open subfunctor, this will imply that it is representable as well. 

The proof will be  actually given in the next Subsections, and will be completed by the end of Subsection \ref{ss:construction}.

\begin{remark} We might also define a total relative Picard group of $f\colon\Xcal\to\Sc$ by the formula 
$$
\Picb(\Xcal)/f^\ast\Picb(\Sc)\,.
$$
However, this group fails to have a natural $\Z_2$-grading as for any line bundle $\Lcl$ one has $\Lcl^\Pi=\Lcl\otimes\Oc_\Xcal^\Pi\simeq \Lcl\otimes f^\ast\Oc_\Sc^\Pi$ and then the classes of $\Lcl$ and $\Lcl^\Pi$ in $\Picb(\Xcal)/f^\ast\Picb(\Sc)$ coincide. A suitable definition of the total relative Picard group is rather
$$
\Picb(\Xcal/\Sc):=\Picb(\Xcal)/f^\ast\Pic(\Sc)\,,
$$
so that two line bundles $\Lcl$ and $\Lcl'$ on $\Xcal$ are identified in $\Picb(\Xcal/\Sc)$ when $\Lcl'\simeq\Lcl\otimes f^\ast\Nc$ for an even line bundle $\Nc$ on $\Sc$. There is a natural grading
$$
\Picb(\Xcal/\Sc)=\Pic(\Xcal/\Sc)\coprod \Pic_-(\Xcal/\Sc)
$$
where $\Pic_-(\Xcal/\Sc)$ denotes the set of   classes of   odd line bundles.  One defines the total Picard sheaf as in Definition \ref{def:Picsh}, and  Theorem \ref{thm:picrepres} now  implies that the total Picard sheaf is representable by 
the $\Sc$-supergroup 
$$
\SPic(\Xcal/\Sc)\times \Z_2
$$
where
the subgroup $\Z_2$ is generated by the class of $\Oc_\Xcal^\Pi$.
\end{remark}

We start the proof of Theorem \ref{thm:picrepres} by describing the
relation between the super Picard presheaf and its associated  \'etale sheaf.

\begin{lemma}\label{lem:injpic} Let $f\colon\Xcal\to\Sc$ be a morphism of superschemes which is 
   cohomologically flat in dimension $0$ (Definition \ref{def:cohomflat}).
For every base change $\Tc\to\Sc$ the morphism
$$
f_\Tc^\ast\colon\Pic(\Tc)\to \Pic(\Xcal_\Tc)
$$
is injective.
\end{lemma}
 
\begin{proof}  
 Let $\Nc$ be an even line bundle on $\Tc$ such that $f_\Tc^\ast\Nc\iso \Oc_{\Xcal_\Tc}$. Then $f_{\Tc\ast}\Oc_{\Xcal_\Tc}
 \iso \Nc\otimes f_{\Tc\ast}\Oc_{\Xcal_\Tc}$. {As} $f$ is cohomologically flat in dimension $0$, we have $\Nc\iso \Oc_\Sc$, which finishes the proof.\end{proof}

\begin{prop}\label{prop:asssheaf}  Let $f\colon \Xcal\to \Sc$ be  {proper and flat}.
\begin{enumerate}
\item  If 
 $f$ is cohomologically flat {in dimension} $0$ 
  the natural {morphism} of group functors
{
$
\SPicf_{\Xcal/\Sc} \to \SPicf_{\Xcal/\Sc\et}
$
is injective;}
\item if in addition
  $f\colon \Xcal\to\Sc$ has a section, then 
{$\SPicf_{\Xcal/\Sc} \iso \SPicf_{\Xcal/\Sc\et}$.}
\end{enumerate}

\end{prop} 
\begin{proof} 
(1) 
Let $\Tc\to\Sc$ be an $\Sc$-superscheme and $[\Lcl]$ be a class in $\SPicf_{\Xcal/\Sc}(\Tc)$ whose image in $\SPicf_{\Xcal/\Sc\et}(\Tc)$ is trivial. This means that there exist an \'etale covering $\phi\colon\Vc\to \Tc$ and an isomorphism $\alpha\colon\Lcl_{\Vc}\iso f_\Vc^\ast \Nc$ for some even line bundle $\Nc$ on $\Vc$. Let us write $\Rcal=\Vc\times_\Tc\Vc$ and $(p_1,p_2)\colon \Rcal \rra\Vc$ for the projections. Consider the cartesian diagram
$$
\xymatrix{\Xcal_\Rcal\iso\Xcal_\Vc\times_{\Xcal_\Tc}\Xcal_\Vc\ar@<2pt>[r]^(.75){(q_1,q_2)}\ar@<-3pt>[r] \ar[d]^{f_\Rcal} & \Xcal_\Vc \ar[r]^{\phi_\Xcal} \ar[d]^{f_\Vc}& \Xcal \ar[d]^{f_\Tc} \\
\Rcal \ar@<3pt>[r]^{(p_1,p_2)}\ar@<-3pt>[r] & \Vc \ar[r]^\phi & \Tc
}
$$
Since $\Lcl_\Vc=\phi_\Xcal^\ast\Lcl_\Tc$, there is an isomorphism $q_1^\ast f_\Vc^\ast \Nc \iso q_1^\ast\Lcl_\Vc\iso q_2^\ast\Lcl_\Vc = q_2^\ast f_\Vc^\ast \Nc$, that is, an isomorphism $f_\Rcal^\ast p_1^\ast\Nc \iso f_\Rcal^\ast p_2^\ast\Nc$. 

 {By Lemma \ref{lem:injpic}} there is an isomorphism $p_1^\ast\Nc \iso p_2^\ast\Nc$, that is, a descent data for $\Nc$. By faithfully flat descent (Proposition \ref{prop:moddescent}), there is an even line bundle $\Nc'$ on $\Tc$ such that $\Nc\iso\phi^\ast \Nc'$. Moreover, the isomorphism $\alpha\colon \Lcl_\Vc\iso f_\Vc^\ast\Nc$ satisfies $q_1^\ast\alpha=q_2^\ast\alpha$, so that, by faithfully flat descent for morphisms (Proposition \ref{prop:morphdescent}), $\alpha$ descends to an isomorphism $ \Lcl\iso f_\Tc^\ast\Nc'$. Then $[\Lcl]$ is trivial in $\SPicf_{\Xcal/\Sc}(\Tc)$.

(2) Let $\sigma\colon \Sc\hookrightarrow \Xcal$ be a section of $f\colon \Xcal\to\Sc$. If $\Tc\to\Sc$ is a superscheme over $\Sc$, a section $\xi$ of the super Picard  sheaf over $\Tc$ is given by an \'etale covering  $\phi\colon\Vc\to \Tc$ and a class $\xi=[\Lcl]$  of an even line bundle $\Lcl$ on $\Xcal_\Vc$ in the relative super Picard group. 
Let us consider the diagram
$$
\xymatrix{\Xcal_\Rcal\iso\Xcal_\Vc\times_{\Xcal_\Tc}\Xcal_\Vc\ar@<3pt>[r]^(.75){(q_1,q_2)}\ar@<-3pt>[r] \ar[d]^{f_\Rcal} & \Xcal_\Vc \ar[r]^{\phi_\Xcal} \ar[d]^{f_\Vc}& \Xcal \ar[d]^{f_\Tc} \\
\Rcal \ar@<3pt>[r]^{(p_1,p_2)}\ar@<-3pt>[r] \ar@/^1pc/[u]^{\sigma_\Rcal}  & \Vc \ar[r]^\phi \ar@/^1pc/[u]^{\sigma_\Vc} & \Tc \ar@/^1pc/[u]^{\sigma_\Tc}
}
$$
where $\sigma_\Tc$, $\sigma_\Vc$ and $\sigma_\Rcal$ are the sections of $f_\Tc$, $f_\Vc$ and $f_\Rcal$ induced by $\sigma$.
Since   $q_1^\ast\phi_{\Xcal}^\ast (\xi)= q_2^\ast\phi_{\Xcal}^\ast (\xi)$ in $\SPicf_{\Xcal/\Sc\et}(\Rcal)$, and the map $$\SPicf_{\Xcal/\Sc}(\Rcal)\to\SPicf_{\Xcal/\Sc\et}(\Rcal)$$ is injective 
by (1), one has    $q_1^\ast\Lcl\iso q_1^\ast\Lcl\otimes f_\Rcal^\ast\Nc$  for an even line bundle $\Nc$ on $\Rcal$. Take  $\tilde\Lcl=\Lcl\otimes f_\Vc^\ast\sigma_\Vc^\ast\Lcl^{-1}$. Then $[\tilde\Lcl]=[\Lcl]$ in the relative super Picard group,  and $q_1^\ast\tilde\Lcl\iso q_2^\ast \tilde\Lcl$. By faithfully flat descent (Proposition \ref{prop:moddescent}), there is an even line bundle $\Lcl'$ on $\Xcal_\Tc$ such that $\tilde\Lcl\iso \phi_\Xcal^\ast \Lcl'$, so that $\xi=[\Lcl']$. Then the immersion
$\SPicf_{\Xcal/\Sc}(\Rcal)\hookrightarrow \SPicf_{\Xcal/\Sc\et}(\Rcal)$ is also surjective.
\end{proof}

Let us also make some simple general observations about the Picard functor.

\begin{lemma}\label{lem:base-change-Pic}
{\ }
\begin{enumerate}
\item For a base change $\Sc'\to \Sc$, the restriction of the Picard functor $\SPicf_{\Xcal/\Sc\et}$ to the category of $\Sc'$-superschemes is naturally
isomorphic to $\SPicf_{\Xcal_{\Sc'}/{\Sc'}\et}$, where $\Xcal_{\Sc'}=\Xcal\times_{\Sc} \Sc'$. In particular, if $\SPicf_{\Xcal/\Sc\et}$
is representable by $\SPic(\Xcal/\Sc)$ then $\SPicf_{\Xcal_{\Sc'}/{\Sc'}\et}$ is representable by
$$\SPic(\Xcal_{\Sc'}/\Sc')\simeq \SPic(\Xcal/\Sc)\times_{\Sc} \Sc'\,.$$
\item
 Assume that $\Sc=S$ is a usual scheme. Then the restriction of {$\SPicf_{\Xcal/S\et}$ to the category of usual schemes over $S$ is naturally
isomorphic to the} Picard functor of the bosonic quotient $\Xcal/\Gamma$ over $S$. If the Picard functor of $\Xcal/S$ is representable by a superscheme
then the Picard functor of $(\Xcal/\Gamma)/S$ is also representable and we have
$$\SPicc((\Xcal/\Gamma)/S)\simeq {\SPic(\Xcal/S)_{\bos}}\,.$$
\end{enumerate}
\end{lemma}

\begin{proof}
Part (1) follows directly from the definitions. To prove (2) we use Lemma \ref{lem:Pic-bos-quot}. Namely, for any $S$-scheme $T$ we have
$${\Pic}(\Xcal_T/T)\simeq \Pic((\Xcal_T/\Gamma)/T)\simeq \Pic((\Xcal/\Gamma)_T/T).$$
It follows the \'etale sheafifications of these functors are also isomorphism.
Assuming that the Picard functor of $\Xcal/S$ is representable, we obtain for any $S$-scheme $T$,
$$\Morf(T, \SPic(\Xcal/S)_{\bos})\simeq \Morf(T,\SPic(\Xcal/S))\simeq \Picf_{(\Xcal/\Gamma)/S\et},$$
which shows that the Picard functor of $(\Xcal/\Gamma)/S$ is represented by {$\SPic(\Xcal/S)_{\bos}$}.
\end{proof}
 
\begin{remark} Lemma \ref{lem:base-change-Pic}(2) explains why we impose the superprojectivity assumption in order to prove the representability of the
super Picard functor. Even in the case when $S$ is even, the representability of super Picard functor of $\Xcal$ over $S$
implies the representability (by a scheme) of the usual Picard functor of the bosonic
quotient $\Xcal/\Gamma$ over $S$. But even for superschemes embeddable into a relative supergrassmannian, $\Xcal/\Gamma$ is not necessarily projective over
$S$, so its Picard functor may only be representable by an algebraic space.
\end{remark}
 
\subsection{Quotients of flat equivalence relations of superschemes}\label{ss:quots}

\begin{defin}\label{def:equivrel} 
Let $f\colon \Xcal \to \Sc$ be a morphism of superschemes. An equivalence relation of superschemes on $\Xcal\to \Sc$ is a closed sub-superscheme $\Rcal\hookrightarrow \Xcal\times_\Sc\Xcal$ that induces an equivalence relation for the functor of   points, that is,   for every $\Sc$-superscheme $\Tc$, the inclusion  $\Rcal^\bullet(\Tc)\hookrightarrow \Xcal^\bullet(\Tc)\times_{\Sc^\bullet(\Tc)}\Xcal^\bullet(\Tc)$ is an equivalence relation of sets.
\end{defin}

Here, as  customary, we write $\Ycal^\bullet(\Tc)=\Hom_\Sc(\Tc,\Ycal)$ for every $\Sc$-superscheme $\Ycal$.

Following \cite[I.5.1]{Knut71} we consider the notion of effective relation of superschemes, and the associated quotient superscheme.

\begin{defin}\label{def:effective} An equivalence relation of superschemes is effective if
\begin{enumerate}
\item there exists the cokernel $q\colon\Xcal\to\Zc$ of $(p_1,p_2)\colon \Rcal \rra \Xcal$, that is, $q\colon\Xcal\to\Zc$  is a categorical quotient;
\item   $\Rcal\simeq \Xcal\times_\Zc\Xcal$. 
\end{enumerate}
We then say that $\Zc$ is the \emph{quotient superscheme} of $\Rcal$.
\end{defin}

\begin{example}\label{ex:effective descent} Let $p\colon \Wc \to\Hcal$ be a morphism of  $\Sc$-superschemes, $\Xcal\hookrightarrow \Wc$ a sub-superscheme and $\Rcal=\Xcal\times_\Hcal\Xcal\hookrightarrow \Xcal\times_\Sc\Xcal$. Then $\Rcal$ is an equivalence relation of superschemes, and it is effective if and only if $\Xcal$ descends to a sub-superscheme $\Zc\hookrightarrow \Hcal$, that is, if and only if there exists a sub-superscheme $\Zc\hookrightarrow \Hcal$ such that $\Xcal=p^{-1}(\Zc)$.

\end{example}

Proposition \ref{prop:subdescent} can be stated as follows:

\begin{prop}[Grothendieck's effective descent for sub-superschemes]\label{prop:effective descent} With the notation of Example \ref{ex:effective descent}, assume  that $p\colon \Wc \to\Hcal$ is quasi-compact and faithfully flat. Then $\Rcal$ is effective if and only if $\Wc\times_\Hcal\Xcal\simeq\Xcal\times_\Hcal\Wc$ as sub-superschemes of $\Wc\times_\Hcal\Wc$.
\qed \end{prop}

\begin{prop}\label{prop:flatproperquot} Let $f\colon\Xcal \to \Sc$ be 
a morphism of superschemes such that  the  Hilbert superscheme  {$\Hcal=\SHilb(\Xcal/\Sc)$} exists.
An equivalence relation $\Rcal\hookrightarrow \Xcal\times_\Sc\Xcal$ such that the second projection $p_2\colon \Rcal\to\Xcal$ is flat and proper  is effective. The quotient  morphism $q\colon \Xcal\to\Zc$ is flat and proper,  and  {$\Zc$ is a closed sub-superscheme of the Hilbert superscheme $\Hcal$, so that $\Zc$ is of finite type and separated over $\Sc$ when $\Hcal$ is so.}
\end{prop}
\begin{proof} By   hypothesis $\Rcal$ is an $\Xcal$-valued point of the Hilbert superscheme  {$\Hcal$}, that is, it defines a morphism of superschemes $g\colon \Xcal \to \Hcal$ such that  $\Rcal\hookrightarrow \Xcal\times_\Sc\Xcal$ is the pull-back by $1\times g\colon \Xcal\times_\Sc\Xcal \to \Xcal \times_\Sc \Hcal$ of the  universal  closed sub-superscheme $\Wc\hookrightarrow \Xcal\times_\Sc\Hcal$. 

Notice that for every $\Sc$-superscheme $\Tc$  two $\Tc$-valued points $x_1,x_2\colon \Tc\to \Xcal$ are equivalent (i.e. $(x_1,x_2)\in \Rcal^\bullet(\Tc)$) if and only if $g(x_1)=g(x_2)$. Then $\Rcal\iso\Xcal\times_\Hcal \Xcal$. Moreover, as $x_1$, $x_2$ are equivalent if and only if $(x_1,g(x_2))\in \Wc^\bullet(\Tc)$, the graph $\Gamma_g\colon \Xcal \hookrightarrow  \Xcal \to \Hcal$ of $g$ gives a closed immersion $\Gamma_g\colon \Xcal \hookrightarrow  \Wc$. We conclude by applying Proposition \ref{prop:effective descent} to the natural morphism $\Wc\to \Hcal$, which is faithfully flat, and to the closed sub-superscheme $\Xcal\iso \Gamma_g(\Xcal)\hookrightarrow \Wc$. The required condition that $\Wc\times_\Hcal \Gamma_g(\Xcal)$ and $ \Gamma_g(\Xcal)\times_\Hcal\Xcal$ coincide in  $\Wc\times_\Hcal\Wc$ is readily verified.

 The quotient morphism $q\colon\Xcal\to\Zc$ is flat and proper as it is obtained from  $\Wc\to \Hcal$ by base change. Moreover,  $\Zc$ is a closed sub-superscheme of the Hilbert superscheme $\Hcal$  {and we conclude}.
\end{proof}

 The following lemma   corresponds to \cite[Exercise 9.4.9]{Kle05} and has an analogous proof.
\begin{lemma}\label{lem:quotrepres} Let $q\colon\Ss\F\to\Sc\Gc$ be a morphism of \'etale sheaves on $\Sc$-superschemes. Then $q$ is an epimorphism of sheaves if and only if it is the cokernel (coequalizer) of the two morphisms $(p_1,p_2)\colon \Ss\F\times_{\Ss\Gc}\Ss\F \rra \Ss\F$.\qed
\end{lemma}
 
 The following result will be necessary to get a  proof of the representability of the super Picard sheaf, and is a  {strengthened}  version for superschemes of \cite[Lemma 9.9]{Kle05}.
 
\begin{prop}\label{prop:quotsmooth}
Let $f\colon\Xcal \to \Sc$ be a 
morphism of superschemes  {such that the Hilbert superscheme $\SHilb(\Xcal/\Sc)$  exists}.
\begin{enumerate}
\item Let $\Rcal\hookrightarrow \Xcal\times_\Sc\Xcal$ be an equivalence relation on $f\colon\Xcal \to \Sc$ such that the second projection $p_2\colon \Rcal\to\Xcal$ is smooth and proper, so that it is effective by Proposition \ref{prop:flatproperquot}. Then, the quotient morphism $q\colon \Xcal\to\Zc$ is smooth and proper as well  and $\Zc$ is a closed sub-superscheme of $\SHilb(\Xcal/\Sc)$. Moreover $\Rcal^\bullet\hookrightarrow \Xcal^\bullet\times\Xcal^\bullet$ is an effective equivalence relation of \'etale sheaves on $\Sc$-superschemes and $q\colon \Xcal^\bullet\to\Zc^\bullet$ is the corresponding quotient morphism.
\item Let $\Ss\Gc$ be an \'etale sheaf on $\Sc$-superschemes and $\psi\colon\Xcal^\bullet\to \Ss\Gc$ an epimorphism of \'etale sheaves such that 
\begin{enumerate}
\item
$\Xcal^\bullet\times_{\Ss\Gc}\Xcal^\bullet$ is representable by a sub-superscheme $\Rcal\hookrightarrow \Xcal\times_\Sc\Xcal$.
\item $\Rcal$ is an equivalence relation of superschemes that satisfies the conditions of (1);
\end{enumerate}
Then,  $\Ss\Gc$ is representable by $\Zc$ and $\psi=q$.
 \end{enumerate}
\end{prop}
 
\begin{proof}
 (1). Since the horizontal $q$ in the cartesian diagram
$$
\xymatrix{ \Rcal\ar[r]^{p_1}\ar[d]^{p_a} & \Xcal\ar[d]^q \\
\Xcal\ar[r]^q & \Zc }
$$
is flat and surjective, and so faithfully flat, and $p_2$ is smooth, the   vertical $q$ is smooth. Since $\Rcal^\bullet\simeq \Xcal^\bullet\times{\Zc^\bullet}\Xcal^\bullet$ by Lemma \ref{lem:quotrepres} we have only to see that $q\colon \Xcal^\bullet\to\Zc^\bullet$ is a surjective morphisms of \'etale sheaves.

Let $\Tc\to\Sc$ be a $\Sc$-superscheme and $\xi\colon \Tc \to \Zc$ an element of $\Zc^\bullet(\Tc)$. Then, $q_\Tc\colon \Xcal_\Tc\to \Zc_\Tc$ is smooth, and thus  there exists an \'etale covering $\Vc \to \Tc$ such that $\Xcal_\Vc \to \Vc$ has a section $\sigma\colon \Vc\hookrightarrow \Xcal_\Vc$. The image of $\sigma$ maps to the element of $\Zc^\bullet(\Vc)$ induced by $\xi$, and we finish.

 (2). By (1) $q\colon\Xcal^\bullet\to \Zc^\bullet$ is the cokernel of the two morphisms $\Rcal^\bullet\rra\Xcal^\bullet$, and by Lemma \ref{lem:quotrepres} $\psi\colon \Xcal^\bullet\to\Ss\Gc$ is also the cokernel of the same arrows. Since the cokernel is unique, one has $\Ss\Gc\simeq \Zc$ and $\psi\simeq q$.
\end{proof}

\begin{remark}
If $\Rcal\hookrightarrow \Xcal\times_\Sc\Xcal$ is an effective \'etale equivalence relation of superschemes on $\Xcal\to \Sc$ (by this we mean that $p_1$ and $p_2$ are \'etale and proper), the quotient morphism $q\colon\Xcal\to\Zc$ is an \'etale morphism. This can be seen as in Proposition \ref{prop:quotsmooth} (see also \cite[I.5.6]{Knut71}). Again by Proposition \ref{prop:quotsmooth}, the induced equivalence categorical relation  $\Rcal^\bullet\hookrightarrow \Xcal^\bullet\times\Xcal^\bullet$ on the category of  \'etale sheaves on  $\Sc$-superschemes is \emph{effective} and   $q\colon\Xcal^\bullet\to\Zc^\bullet$ is a categorical quotient. The converse is not true, and this is what motivates the definition of Artin algebraic superspace \cite[Def. 2.19]{BrHR19}.
\end{remark}

\subsection{Relative positive divisors. The Abel morphism}\label{ss:abel}

Let $f\colon \Xcal\to\Sc$ be a morphism of superschemes.

\begin{defin}\label{def:uperdiv} A relative positive superdivisor in $\Xcal/\Sc$ (or in $f$) is a closed sub-superscheme $j\colon \Zc\hookrightarrow \Xcal$ flat over $\Sc$ (i.e. such that $f\circ i\colon\Zc\to\Sc$ is flat), whose ideal sheaf $\Ic_\Zc$ in $\Xcal$ is an even line bundle.
\end{defin}

In this situation, we write
$$
\Oc_\Xcal(-\Zc):=\Ic_\Zc\,,\quad \Oc_\Xcal(\Zc):=\Oc_\Xcal(-\Zc)^{-1}\,.
$$

\begin{defin}\label{def:uperdivf} The functor of relative positive superdivisors is the functor $\Div_{\Xcal/\Sc}$ on $\Sc$-superschemes that associates to an $\Sc$-superscheme $\Tc\to\Sc$ the family $\Div_{\Xcal/\Sc}(\Tc)$ of the relative positive superdivisors in $f_\Tc\colon\Xcal_\Tc\to\Tc$.
\end{defin}

Since every relative positive divisor is a closed sub-superscheme which is flat over the base, there is a natural functor immersion into the super Hilbert functor:
$$
\Div_{\Xcal/\Sc}\hookrightarrow \SHilbf_{\Xcal/\Sc}\,.
$$

\begin{prop}\label{prop:divisors} If $f$ is proper, the above immersion is representable by open immersions. Then, if  {the Hilbert superscheme $\SHilb(\Xcal/\Sc)$ exists (see Theorem \ref{thm:hilbrepres2})},  $\Div_{\Xcal/\Sc}$ is representable by a  superscheme $\SDiv(\Xcal/\Sc)\to \Sc$ which is an open sub-superscheme of $\SHilb(\Xcal/\Sc)$, and then it is  {locally} of finite type and separated over $\Sc$.
\end{prop}
\begin{proof} One has to prove that, for every $\Sc$-superscheme $\Tc\to\Sc$ and every closed sub-superscheme $\Zc\hookrightarrow \Xcal_\Tc$ flat over $\Tc$, there is an open sub-superscheme $\Ucal\hookrightarrow \Tc$ with the following universal property. A morphism of superschemes $\Vc\to \Tc$ factors through $\Ucal\hookrightarrow \Tc$ if and only if $\Zc_\Vc\hookrightarrow \Xcal_\Vc$ is a positive relative superdivisor. This is equivalent to impose that the ideal sheaf  $\Ic_{\Zc_\Vc}$ is a locally free $\Oc_{\Xcal_\Vc}$-module. 
Since $\Ic_\Zc$ is coherent, the locus of the points where it is free is an open subset $W$ of $X_T$. Then $f _T(X_T-W)$ is closed in $T$ as $f$ is proper. One has then to take $U=T-f_T(X_T-W)$ and endow it with the natural structure $\Ucal$ of an open sub-superscheme of $\Tc$. The second part follows from the existence Theorem \ref{thm:hilbrepres}.
\end{proof}

\begin{defin}\label{def:Abelmor} The Abel morphism is the morphism of functors
$$
\Ab\colon \Div_{\Xcal/\Sc} \to {\SPicf_{\Xcal/\Sc}}\,,
$$
where, for an $\Sc$-superscheme $\Tc$, the morphism $\Div_{\Xcal/\Sc}(\Tc) \to {\SPicf_{\Xcal/\Sc}(\Tc)}$ is defined by
$$
\Zc \mapsto [\Oc_{\Xcal_\Tc}(\Zc)]
$$
for every relative positive superdivisor $\Zc\hookrightarrow \Xcal_\Tc$ in  $\Xcal_\Tc/\Tc$.
\end{defin}

 It defines an equivalence relation in the functor $\Div_{\Xcal/\Sc}$, the \emph{relative linear equivalence},  by saying that two relative positive divisors are relatively linearly equivalent when they have the same image by the Abel morphism, that is, when the duals of the corresponding ideal sheaves are isomorphic up to the pull-back of a line bundle on the base. The relative linear equivalence is then given by the fibre product functor
 $$
 \Div_{\Xcal/\Sc} \times_{\SPicf_{\Xcal/\Sc}} \Div_{\Xcal/\Sc} \hookrightarrow \Div_{\Xcal/\Sc}\times_\Sc  \Div_{\Xcal/\Sc}\,.
 $$

 On our way to prove Theorem \ref{thm:picrepres}, we take a flat superprojective morphism $f\colon\Xcal\to\Sc$ where $\Sc$ is assumed to be noetherian and connected. 
Let us denote by $\Oc_\Xcal(1)$ the  fixed ample line bundle  for $f\colon \Xcal \to \Sc$.
Since $f$ is flat, every even  line bundle  on $\Xcal$ has constant super Hilbert polynomial  on the various fibres. This means that the 
super Picard presheaf decomposes as the disjoint union of the various 
super Picard presheaves ${\SPicf^{\bP}_{\Xcal/\Sc}}$ of the classes of even line bundles $\Lcl$ \emph{such that the dual $\Lcl^{-1}$ has  super Hilbert polynomial $\bP$}, that is,
$$
\bP(n)=\chi(X_s, \Lcl^{-1}_s(n))
$$
for every $s\in S$. The same happens for the coresponding \'etale sheaves  and for the superschemes   they represent.
Let $\bH$ be the super Hilbert polynomial of $\Oc_\Xcal$, and $\bQ(n)=\bH(n)-\bP(n)$. Then the pre-image of {$\SPicf^{\bP}_{\Xcal/\Sc}$} by the Abel morphism is the functor $\Div_{\Xcal/\Sc}^\bQ$ of relative superdivisors of super Hilbert polynomial $\bQ$ and  $\Div_{\Xcal/\Sc}$ is the disjoint union of the various $\Div_{\Xcal/\Sc}^\bQ$.

The Abel morphism decomposes as the union of the various
$$
\Ab\colon \Div_{\Xcal/\Sc}^\bQ \to {\SPicf^{\bP}_{\Xcal/\Sc}}\,,
$$
and there are relative linear equivalences
 $$
 \Div_{\Xcal/\Sc}^\bQ \times_{\SPicf_{\Xcal/\Sc}} \Div_{\Xcal/\Sc}^\bQ \hookrightarrow \Div_{\Xcal/\Sc}^\bQ\times_\Sc  \Div_{\Xcal/\Sc}^\bQ\,.
 $$
 Proceeding as in Proposition \ref{prop:divisors}, we have
 \begin{prop}\label{prop:divisorsQ} $\Div_{\Xcal/\Sc}^\bQ$ is representable by a  superscheme $\SDiv(\Xcal/\Sc)^\bQ\to \Sc$ which is an open sub-superscheme of the super  Hilbert scheme $\SHilb(\Xcal/\Sc)^\bQ$, and then it is a sub-superscheme of a supergrassmannian over $\Sc$. In particular, it is of finite type and separated over $\Sc$.
\end{prop}

\begin{remark} The existence of the superscheme of positive divisors $\SDiv(\Xcal/\Sc)^\bQ$, in the particular case of a smooth supercurve $f\colon\Xcal \to S$  (i.e., the relative dimension is $(1,1)$) over an ordinary scheme $S$, is known from the 90s \cite{DoHeSa93}.
Since $\Oc_\Xcal\simeq \Oc_X\oplus \Lcl
$
for a line bundle $\Lcl$ on $X$, one can define the \emph{dual supercurve} as the superscheme with the same bosonic reduction $X$ and with structure sheaf  $\Oc_{\Xcal^\vee}=\Oc_X\oplus (\kappa_{X/S}\otimes_{\Oc_X}\Lcl^{-1})$.
In \cite{DoHeSa93} it was proved that the functor of relative positive divisors of degree 1 is representable by $\Xcal^\vee$ and that the universal relative positive divisor in $\Xcal^\vee\times_S \Xcal \to \Xcal^\vee$
is the Manin superdiagonal. For higher degrees, the functor of relative positive divisors of degree $q$ is representable by the symmetric power $\mathrm{Sym}^q_S(\Xcal^\vee)$ of the dual supercurve, that is, $\SDiv(\Xcal/S)^q\simeq \mathrm{Sym}^q_S(\Xcal^\vee)$. This is a smooth superscheme over $S$ of relative dimension $(q,q)$.
\end{remark}

The  strategy for the proof of Theorem \ref{thm:picrepres} is the following.
\begin{strategy}\label{st:pic}
\ 
\begin{enumerate}
\item There is an open subfunctor of $ \Div_{\Xcal/\Sc}$, given by the divisors  such that the dual of their ideal sheaf is acyclic and generated by   global sections, which is representable by a superscheme $\Dsc_a$.
\item The relative linear equivalence for $\Dsc_a$ is given by a sub-superscheme   $\Rcal$ of $\Dsc_a\times_\Sc \Dsc_a$, that is, $\Rcal$ defines an equivalence relation of superschemes.  {Analogously, the relative linear equivalence for $\Dsc_a^\bQ$ is given by a sub-superscheme   $\Rcal^\bQ$ of $\Dsc_a^\bQ\times_\Sc \Dsc_a^\bQ$.}
\item Moreover, when  $f$ is flat and superprojective with geometrically integral fibres, and is cohomologically flat in dimension $0$,
then  for every super polynomial $\bQ$, $\Rcal^\bQ$ is proper and flat over  $\Dsc_a^\bQ$ w.r.t.~the second projection, so that  the quotient superscheme exists by Proposition \ref{prop:flatproperquot}.
Notice that Proposition \ref{prop:flatproperquot} can be applied here because by Proposition \ref{prop:divisorsQ} $\Dsc_a^\bQ$   is a sub-superscheme of a supergrassmannian, so that Corollary \ref{cor:grassproj} and Theorem \ref{thm:hilbrepres2} imply that the Hilbert superscheme of $\Dsc_a^\bQ/\Sc$ exists.
\item  This quotient superscheme represents the open subfuntor {$\SPicf^{\bP}_{a,\Xcal/\Sc\et}$} of the even super Picard sheaf associated to the classes of even line bundles that satisfy the following properties: they
 are relatively acyclic, are generated by their global sections, and their duals have super Hilbert polynomial $\bP$. The disjoint union of the various quotient superschemes for all the super Hilbert polynomial $\bP$ represents the open subfuntor {$\SPicf_{a,\Xcal/\Sc\et}$} of the 
 super Picard sheaf associated to the classes of even line bundles that relatively acyclic and  are generated by their global sections.
\item Finally, one proves that the 
super Picard sheaf and its open subsheaf {$\SPicf_{a,\Xcal/\Sc\et}$} satisfy the conditions of Proposition \ref{prop:opengoup}. This implies that 
 the even super Picard functor is representable, thus finishing the proof of Theorem \ref{thm:picrepres}.
\end{enumerate}
\end{strategy}

\subsection{Representability of acyclic divisors}\label{ss:acyclicdiv}

In this subsection we prove (1) of Strategy \ref{st:pic}, assuming that $f\colon\Xcal\to\Sc$ is proper  and flat.

Let {$\SPicf_{a,\Xcal/\Sc}$} be the subfunctor of {$\SPicf_{\Xcal/\Sc}$} given by the even line bundles that are relatively acyclic and are generated by their global sections, and let
$$
\Div_{a,\Xcal/\Sc}=\Ab^{-1}({\SPicf_{a,\Xcal/\Sc}})
$$
be the functor of the relative positive supervisors such that the duals of their ideal sheaf are relatively acyclic and generated by their global sections.

\begin{prop}\label{prop:acyclic} The functor morphisms
$$
{\SPicf_{a,\Xcal/\Sc}\hookrightarrow\SPicf_{\Xcal/\Sc}}\,,\quad \Div_{a,\Xcal/\Sc}\hookrightarrow \Div_{\Xcal/\Sc}
$$
are representable by open immersions. So, if $f\colon\Xcal\to\Sc$ is superprojective, $\Div_{a,\Xcal/\Sc}$ is representable by an open sub-superscheme $\SDiv(a,\Xcal/\Sc)\to \Sc$ of the superscheme $\SDiv(\Xcal/\Sc)$ $\to \Sc$ of relative positive superdivisors of $\Xcal/\Sc$.
 {Analogously, $\Div_{a,\Xcal/\Sc}^\bQ$ is representable by an open sub-superscheme $\SDiv(a,\Xcal/\Sc)^\bQ\to \Sc$ of $\SDiv(\Xcal/\Sc)^\bQ$.}
\end{prop}
\begin{proof} 
We prove the statement for the Picard functors, since the other is similar. Let $\Tc$ be a superscheme over $\Sc$ and $\Phi\colon\Tc \to {\SPicf_{\Xcal/\Sc}}$ a morphism of $\Sc$-superschemes. Then $\Phi$ is given by a class $[\Lcl]$ in the relative Picard group of $\Xcal_\Tc/\Tc$ of even line bundles on $\Xcal_\Tc$. By the cohomology base change Theorem \ref{thm:cohombasechange}, there is an open sub-superscheme $\Ucal$ of $\Tc$ such that a morphism $\Tc'\to \Tc$ of $\Sc$-superschemes factors through $\Ucal$ if and ony if $\Lcl_{\Tc'}$ is relatively acyclic  and generated by global sections with respect to $f_{\Tc'}\colon \Xcal_{\Tc'}\to \Tc'$. Let $\Ycal\hookrightarrow \Xcal_{\Tc'}$ be the support of the cokernel of $f_{\Tc}^\ast f_{\Tc\ast} \Lcl_{\Tc} \to \Lcl_{\Tc}$. Since $f_T$ is proper, $f_T(Y)$ is closed in $T$. If we endow $V=U-f_T(Y)$ with the structure $\Vc$ of an open sub-superscheme of $\Ucal$, we obtain an open sub-superscheme $\Vc$ of $\Tc$ with the following universal property: a morphism $\Tc'\to \Tc$ of $\Sc$-superschemes factors through $\Vc$ if and ony if $\Lcl_{\Tc'}$ is relatively acyclic with respect to $f_{\Tc'}\colon \Xcal_{\Tc'}\to \Tc'$ and it is relatively generated by its global sections.

The final part now follows from Proposition \ref{prop:divisors}. The statement about $\Div_{a,\Xcal/\Sc}^\bQ$ is proved in the same way using Proposition \ref{prop:divisorsQ}.
\end{proof}

To prove a similar statement for the Picard sheaves  we need a preliminary lemma.

\begin{lemma}\label{lem:acyclic} Let $\Ss\Gc\to\Ss\F$ be a morphism of presheaves on superschemes, and  let $\Ss\Gc_{\et}\to\Ss\F_{\et}$ be the induced morphisms between the associated \'etale sheaves. If  \  $\Ss\F\to\Ss\F_{\et}$ is injective and $\Ss\Gc\to\Ss\F$ is representable by open immersions, then $\Ss\Gc_{\et}\to\Ss\F_{\et}$ is representable by open immersions as well.
\end{lemma}
\begin{proof}
Let $\Tc\to\Sc$ be an $\Sc$-superscheme and $\lambda\colon \Tc\to \Ss\F_{\et}$ a morphism of functors on the category  of  $\Sc$-superschemes. There exists an open covering $\pi\colon \Tc'\to \Tc$ such that $\lambda\circ\pi\colon \Tc'\to \Ss\F_{\et}$ factors through a morphism $\lambda'\colon \Tc\to\Ss\F$ and the immersion $\iota\colon\Ss\F\hookrightarrow\Ss\F_{\et}$. Then  $\Ss\Gc_{\et}\times_{\Ss\F_{\et},\lambda\circ\pi}\Tc'$ is representable by an open sub-superscheme $\Vc\hookrightarrow \Tc'$. 
(in the notation of the fibre product we stress the second projection to avoid confusion.) In particular, $\Vc\simeq\Ss\Gc\times_{\Ss\F_{\et},\lambda\circ\pi}\Tc'$ is a sheaf, so that it coincides with its associated sheaf $\Ss\Gc_{\et}\times_{\Ss\F_{\et},\lambda\circ\pi}\Tc'$. Let us consider the projections $p_1,p_2$ of $\Tc'\times_\Tc\Tc'\rra\Tc'$. Then
 \begin{align*}
p_i^{-1}(\Vc)&=\Vc\times_\Tc\Tc'\simeq (\Ss\Gc_{\et}\times_{\Ss\F_{\et},\lambda\circ\pi}\Tc')\times_\Tc\Tc' \\
&\simeq \Ss\Gc_{\et}\times_{\Ss\F_{\et},\lambda'\circ p_1}(\Tc'\times_\Tc\Tc')\\
&\simeq \Ss\Gc_{\et}\times_{\Ss\F_{\et},\lambda'\circ p_2}(\Tc'\times_\Tc\Tc') \simeq \Tc'\times_\Tc \Vc\simeq p_2^{-1}(\Vc)\,,
\end{align*}
where we have used that $\lambda'\circ p_1=\lambda'\circ p_2$ as $\lambda\circ\pi\circ p_1=\lambda\circ\pi\circ p_2$ and $\iota\colon\Ss\F\hookrightarrow\Ss\F_{\et}$ is injective.
By Grothendieck effective descent for superschemes (Proposition \ref{prop:effective descent}) there exists an open sub-superscheme $\U\hookrightarrow \Tc$ such that $\Vc=\pi^{-1}(\U)$. One now sees that $\U\simeq \Ss\Gc_{\et}\times_{\Ss\F_{\et}}\Tc$.
\end{proof}

\begin{prop}\label{prop:acyclicsh}  If   $f$ is cohomologically flat {in dimension} $0$ 
the sheaf  {morphisms}
{
$$
\SPicf_{a,\Xcal/\Sc \et}\hookrightarrow\SPicf_{\Xcal/\Sc\et} {\,,\quad \SPicf^{\bP}_{a,\Xcal/\Sc \et}\hookrightarrow\SPicf^{\bP}_{\Xcal/\Sc\et}}
$$
}
 {are} representable by open immersions.
\end{prop}
\begin{proof}  This follows from Lemma \ref{lem:acyclic} and  Proposition \ref{prop:acyclic} taking also Proposition \ref{prop:asssheaf} into account.
\end{proof}

\subsection{Structure of the Abel morphism}\label{ss:abelstr}

We start by defining the complete linear series $\vert \Lcl\vert$ associated to an even line bundle $\mathcal L$ as the fibre of the Abel morphism (Definition \ref{def:Abelmor}) over 
the point of the Picard superscheme corresponding to $\mathcal L$. More precisely:
\begin{defin}\label{def:lineraseries} Let $\Lcl$ be an even line bundle on $\Xcal$. The complete linear series of $\Lcl$ is the subfunctor
$$
\vert\Lcl\vert=\Ab^{-1}(\Lcl)\hookrightarrow \Div_{\Xcal/\Sc}
$$
which associates to every superscheme $\Tc\to\Sc$ the set $\vert\Lcl\vert(\Tc):=\Ab^{-1}(\Lcl)(\Tc)$ of the relative positive superdivisors $\Zc\hookrightarrow\Xcal_\Tc$ such that
$$
[\Lcl_\Tc]=[\Oc_{\Xcal_\Tc}(\Zc)]
$$
in ${\SPicf_{\Xcal/\Sc}}(\Tc)$. This is equivalent to claiming that 
$$
\Oc_{\Xcal_\Tc}(\Zc)\iso \Lcl\otimes f_\Tc^\ast\Nc
$$
for an even line bundle $\Nc$ on $\Tc$.
\end{defin}

In general, the structure of the complete linear series of $\Lcl$ can be determined as in the classical case \cite{AlKl80}. However, we only need to consider the simple case of acyclic line bundles generated by their global sections.

Let $f\colon\Xcal\to \Sc$ a flat proper morphism of superschemes with $\Sc$ noetherian
and  let $\Lcl$ be an even line bundle on $\Xcal$ which is $f$-acyclic, that is, one has $R^if_\ast\Lcl=0$ for every $i>0$. Then $f_\ast\Lcl$ is locally free by Proposition \ref{prop:locfree}, and the formation of $f_\ast\Lcl$ commutes with arbitrary base change, that is, for every morphism $\Tc\to\Sc$ of superschemes one has
\begin{equation}\label{eq:basechange}
(f_\ast\Lcl)_\Tc \iso f_{\Tc\ast}\Lcl_\Tc\,.
\end{equation}

Assume moreover that $\Lcl$ is relatively generated by its global sections, i.e.,  the morphism $f^\ast f_\ast \Lcl\to\Lcl$ is surjective. Then $f_\ast \Lcl$ is nonzero.

If we consider the locally free sheaf $\Qc=(f_\ast\Lcl)^{-1}$, Equation \eqref{eq:basechange} implies that the formation of $\Qc$ is compatible with base change, that is, one has
\begin{equation}\label{eq:basechange2}
\Qc_\Tc \iso (f_{\Tc\ast}\Lcl_\Tc)^{-1}
\end{equation}
for every morphism $\Tc\to\Sc$ of superschemes, and there is a functorial isomorphism
\begin{equation}\label{eq:Qsheaf}
\Homsh_{\Oc_\Tc}(\Qc_\Tc, \Nc)\iso f_{\Tc\ast}(\Lcl_\Tc)\otimes \Nc\iso
 f_{\Tc\ast}(\Lcl_\Tc\otimes f_\Tc^\ast\Nc)
\end{equation}
for every quasi-coherent sheaf $\Nc$ on $\Sc$. In particular, for every base change $\Tc\to\Sc$ there is an isomorphism
\begin{equation}\label{eq:Qsections}
\gamma\colon\Hom_{\Tc}(\Qc_\Tc, \Nc)\iso H^0(\Tc, \Lcl_\Tc\otimes f_\Tc^\ast\Nc)\,.
\end{equation}

\begin{prop}\label{prop:Abel} {Let $\Sc$ be a noetherian superscheme and let }$f\colon\Xcal\to \Sc$ be  a flat proper morphism of superschemes  {which is cohomologically flat {in dimension} $0$ and} has  geometrically integral fibres. Let $\Lcl$ be an even line bundle on $\Xcal$ which is $f$-acyclic and relatively generated by its global sections. Then the complete linear series of $\Lcl$ is represented by the  projective superbundle
$\widetilde\Ps(f_\ast\Lcl){=\Ps(\Qc)}$.
\end{prop}
\begin{proof} {As we have seen,} $f_\ast\Lcl${, and then $\Qc$, are} locally free and nonzero. Given a superscheme $\Tc\to \Sc$ over $\Sc$, the relative positive divisors $\Zc$ in the linear series $\vert\Lcl\vert(\Tc)$ are identified with the exact sequences
$$
0\to \Lcl_\Tc^{-1}\otimes f_\Tc^\ast\Nc^{-1}{\xrightarrow{\sigma}} \Oc_{\Xcal_\Tc}{\xrightarrow{\rho}} \Oc_\Zc\to 0\,,
$$
where $\rho$ is the natural projection {and $\Nc$ is an even line bundle on $\Tc$}.  The sheaf $\Nc$ is uniquely determined by $\rho$ because if $ \Lcl_\Tc^{-1}\otimes f_\Tc^\ast\Nc^{-1}\simeq  \Lcl_\Tc^{-1}\otimes f_\Tc^\ast\Nc'^{-1}$, then $f_\Tc^\ast\Nc\simeq f_\Tc^\ast\Nc'$ so that $\Nc\simeq \Nc'$ by Lemma \ref{lem:injpic}.
The morphism $\sigma$ can be seen as a section of $\Lcl_\Tc\otimes f_\Tc^\ast\Nc$, and by Equation \eqref{eq:Qsections} it defines a morphism
$\gamma(\sigma)\colon \Qc_\Tc \to \Nc$ such that
\begin{equation}\label{eq:gc}
\gamma(\sigma)_{\Tc'}= \gamma_{\Tc'}(\sigma_{\Tc'})
\end{equation}
for every morphism $\Tc'\to \Tc$.
The condition that $\Oc_\Zc$ is flat over $\Tc$ is tantamount to the fact that for every $t\in T$ the restriction $\sigma_t$ of $\sigma$ to the fibre of  $f_\Tc\colon \Xcal_\Tc \to \Tc$ over $t$ is still injective. Since the fibres of $f$ are geometrically integral, Proposition \ref{prop:xxx} implies that the injectivity of $\sigma_t$ for every $t\in T$ is equivalent to its nonvanishing. Thus by Equation \eqref{eq:gc}
$$
\gamma(\sigma)_t\colon \Qc\otimes\kappa(t) \to \Nc\otimes \kappa(t)
$$
is nonzero for every $t\in T$, and then it is surjective as $\Nc$ is an even line bundle. By the super Nakayama Lemma  (\cite{BBH91,CarCaFi11} or \cite[6.4.5]{We09}), $\gamma(\sigma)\colon \Qc_\Tc \to \Nc$ is surjective, that is, it yields a $\Tc$-valued point of the supergrassmanian $\Sgrass(\Qc,(1,0))$.

Moreover, if $\sigma'$ is another section of $\Lcl_\Tc\otimes f_\Tc^\ast\Nc$ inducing the same positive relative divisor, then $\sigma'=\phi\circ \sigma$ where $\phi$ is an invertible section of $\Oc_{\Xcal_\Tc}$. Since $f$ is cohomologically flat in dimension $0$, so that $\Oc_\Tc\simeq f_{\Tc\ast}\Oc_{\Xcal_\Tc}$, $\phi$ is   an invertible section of $\Oc_\Tc$ and   $\gamma(\sigma')=\phi\circ \gamma(\sigma)$. Then, the two surjections $\gamma(\sigma)\colon \Qc_\Tc \to \Nc$ and $\gamma(\sigma')\colon \Qc_\Tc \to \Nc$ define the same $\Tc$-valued point of 
$\Sgrass(\Qc,(1,0))$.
It follows that there is a functor morphism $\vert\Lcl\vert \to \Sgrassf(\Qc,(1,0))$, which by the above discussion  is an isomorphism. This concludes the proof as $\Sgrass(\Qc,(1,0))\simeq \Ps(\Qc)$ by  Proposition \ref{prop:projmor}.
\end{proof}
\begin{remark} The requirements on $\Lcl$ in Proposition \ref{prop:Abel} are actually superfluous. There   always exists  a sheaf $\Qc$ fulfilling Equations \ref{eq:basechange2} and \ref{eq:Qsheaf} (see \cite[7.7.3]{EGAIII-II} for the corresponding classical statement), and one can follow the same proof of Proposition \ref{prop:Abel}  {to prove that $\vert\Lcl\vert\simeq \Ps(\Qc)$ (see \cite[Theorem 9.3.13]{Kle05}). However, the latter superscheme is not a superprojective bundle when $\Qc$ is not locally free.}
\end{remark}

\begin{remark}\label{rem:A01}

The hypothesis of cohomological flatness in dimension $0$ cannot be removed from Proposition \ref{prop:Abel}. Indeed, we can produces   examples of a morphism and a line bundle $\Lcl$ satisfying all the hypotheses of Proposition \ref{prop:Abel} except for the cohomological flatness and such that $\vert\Lcl\vert$ is not representable by $\widetilde\Ps(f_\ast\Lcl){=\Ps(\Qc)}$.

If $f\colon\Ps_\Sc^{0,1}= \As_\Sc^{0,1}\to\Sc$ is the affine superline over a superscheme $\Sc$, for every $\Sc$-superscheme $\Tc$  there are no nonzero relative effective divisors in $f_\Tc\colon \As_\Tc^{0,1}\to\Tc$ (as there are no codimension $(1,0)$ subschemes).
This means that for even line bundle $\Lcl$ on $\As_\Sc^{0,1}$ the complete linear system $\vert\Lcl\vert$ is 
$$
\vert\Lcl\vert(\Tc)=\begin{cases} \emptyset & \text{if $\Lcl_\Tc$ is nontrivial} \\
\{0\} \text{(a set with one element)} &  \text{if $\Lcl_\Tc$ is trivial}
\end{cases}
$$
Now take $\Lcl$ to be the trivial line bundle on $\As_k^{0,1}$, where $k$ is a field. Then the corresponding linear system is represented by the point,
but the corresponding projective superspace $\Ps((f_\ast\Lcl)^{-1})$ is $\Ps_k^{1,1}$, which has more than one $\Tc$-point whenever the $k$-superscheme $\Tc$ has nonzero odd functions.
\end{remark}

\subsection{Construction of the Picard superscheme}\label{ss:construction}
In this Subsection we prove (4) and (5) of Strategy \ref{st:pic}, thus finishing the proof of the existence Theorem \ref{thm:picrepres}.
Remember that we are assuming   that  $\Sc$ is noetherian  {and connected} and that $f\colon\Xcal\to\Sc$ is   \emph{flat and superprojective with geometrically integral fibres {and  is cohomologically flat in dimension $0$}}.

 For simplicity let us write $\Dsc=\SDiv(\Xcal/\Sc)\to \Sc$ and $\Oc=\Oc_{\Xcal\times_\Sc \Dsc}$. If $\Zc\hookrightarrow \Xcal\times_\Sc \Dsc$ is the relative universal divisor over $\Dsc$,   the open sub-superscheme $\SDiv(a,\Xcal/\Sc)\to \Sc$ (Proposition \ref{prop:acyclicsh}) is the locus $\Dsc_a$ of the points of $\Dsc$ where $\Oc(\Zc)$ is relatively acyclic and generated by its global sections over $\Dsc$.

In the rest of this paragraph we will confuse  superschemes with their functors of   points. Now, the Abel morphism
$$
\Ab\colon \Dsc_a \to {\SPicf_{a,\Xcal/\Sc}}
$$
is defined by the class of $[\Oc(\Zc)]$  
in the relative Picard group of $\Xcal\times_\Sc\Dsc_a/\Dsc_a$. Then in the cartesian diagram of morphisms of functors
$$
\xymatrix{
\Dsc_a\times_{{\SPicf_{a,\Xcal/\Sc}}}\Dsc_a \ar[d]^{p_2}\ar[r]^(.6){p_1} & \Dsc_a \ar[d]^{\Ab} \\
\Dsc_a \ar[r]^{\Ab} & {\SPicf_{a,\Xcal/\Sc}}
}
$$ the projection $p_2$ identifies $\Dsc_a\times_{{\SPicf_{a,\Xcal/\Sc}}}\Dsc_a$  with the fibre functor  of the (vertical) Abel morphism over the universal line bundle $\Oc(\Zc)$:
$$
 \Rcal_a:=\Dsc_a\times_{{\SPicf_{a,\Xcal/\Sc}}}\Dsc_a\iso \vert \Oc(\Zc)\vert =\Ab^{-1}([\Oc(\Zc)])\hookrightarrow \Dsc_a\times_\Sc\Dsc_a\,,
$$
so that the relative linear equivalence on $\Dsc_a$ is an \emph{equivalence relation of superschemes}, thus proving (2) of Strategy \ref{st:pic}.

If we fix the super Hilbert polynomials as in Subsection \ref{ss:abel} we have   similar formulas:
$$
\xymatrix{
\Dsc_a^\bQ\times_{{\SPicf_{a,\Xcal/\Sc}}}\Dsc_a^\bQ \ar[d]^{p_2^\bQ}\ar[r]^(.6){p_1^\bQ} & \Dsc_a^\bQ \ar[d]^{\Ab} \\
\Dsc_a^\bQ \ar[r]^{\Ab} & ~{\SPicf^{\bP}_{a,\Xcal/\Sc}}\,,
}
$$
and
$$
 \Rcal_a^\bQ:=\Dsc_a^\bQ\times_{{\SPicf_{a,\Xcal/\Sc}}}\Dsc_a^\bQ\iso \vert \Oc(\Zc)\vert =\Ab^{-1}([\Oc(\Zc^\bP)])\hookrightarrow \Dsc_a^\bP\times_\Sc\Dsc_a^\bQ\,,
$$
where  $\Oc(\Zc^\bP)$ is over the universal line bundle whose dual has super Hilbert polynomial $\bP$.
Since $\Sc$ is noetherian,  $\Dsc_a^\bQ$ is noetherian as well by Proposition \ref{prop:divisorsQ}. By Proposition \ref{prop:Abel} $p_2$ is a projective superbundle $p_2\colon \vert \Oc(\Zc)\vert\simeq \widetilde\Ps(f_{\Dsc_a,\ast}\Oc(\Zc))\to \Dsc_a$, so that it is proper and flat.  {Analogously, $p_2^\bQ$ is also a projective superbundle $p_2^\bQ\colon \vert \Oc(\Zc^\bP)\vert\simeq \widetilde\Ps(f_{\Dsc_a^\bQ,\ast}\Oc(\Zc^\bP))\to \Dsc_a^\bQ$.}

One has following result, which proves (3) of Strategy \ref{st:pic}.
\begin{prop} 
 The relative linear equivalence  $\Rcal_a^\bQ:=\Dsc_a^\bQ\times_{{\SPicf_{a,\Xcal/\Sc}}}\Dsc_a^\bQ$  on $\Dsc_a^\bQ\to\Sc$ is effective, so that the quotient $\Sc$-superscheme exists.  The relative linear equivalence  $\Rcal_a:=\Dsc_a\times_{{\SPicf_{a,\Xcal/\Sc}}}\Dsc_a$  on $\Dsc_a\to\Sc$ is also effective, so that the quotient $\Sc$-superscheme exists;  we denote those quotients by $q\colon \Dsc_a^\bQ\to \Dsc_a^\bQ/\sim$ and $q\colon \Dsc_a\to \Dsc_a/\sim$. Moreover, $\Dsc_a^\bQ/\sim$ is of finite type and separated over $\Sc$  for every $\bQ$, and then $\Dsc_a/\sim$ is of finite type and separated over $\Sc$ as well.
\end{prop}
\begin{proof} 
Since $\Dsc_a$ is the disjoint union of the various $\Dsc_a^\bQ$, it is enough to prove the first statement. By Proposition \ref{prop:divisorsQ}  $\Dsc_a^\bQ$ is a sub-superscheme of a supergrassmannian, so that Corollary \ref{cor:grassproj} and Theorem \ref{thm:hilbrepres2} imply that the Hilbert superscheme of $\Dsc_a^\bQ/\Sc$ exists. Then, Proposition \ref{prop:flatproperquot} implies that the relative linear equivalence  $\Rcal_a^\bQ$ is effective and that the quotient $\Dsc_a^\bQ/\sim$ is of finite type and separated over $\Sc$.
\end{proof}

 We can  now  prove that the quotient superscheme $\Dsc_a/\sim$ represents the Picard sheaf ${\SPicf_{a,\Xcal/\Sc(et)}}$  and that $\Dsc_a^\bQ/\sim$ represents the Picard sheaf ${\SPicf^{\bP}_{a,\Xcal/\Sc(et)}}$.
For simplicity we introduce the notation
$$
\Ss\Pc_a:={\SPicf_{a,\Xcal/\Ss(et)}}\,.
$$
Note  that if we compose the Abel morphism with the immersion ${\SPicf_{a,\Xcal/\Ss}}\hookrightarrow \Ss\Pc_a$ (Proposition \ref{prop:asssheaf}) we have another Abel morphism
$$
\Ab_{et}\colon \Dsc_a\to \Ss\Pc_a\,.
$$
Moreover one has an isomorphism
$$
\Rcal_a=\Dsc_a\times_{{\SPicf_{a,\Xcal/\Ss}}}\Dsc_a\simeq \Dsc_a\times_{\Ss\Pc_a}\Dsc_a\,.
$$
 
\begin{prop}\label{prop:divpic} 
Under the hypotheses of Theorem \ref{thm:picrepres}
there  {are isomorphisms} of \'etale sheaves 
$$
 {\Dsc_a^\bQ/\sim \iso {\SPicf^{\bP}_{a,\Xcal/\Sc\et}}\,,\quad}\Dsc_a/\sim \iso {\SPicf_{a,\Xcal/\Sc\et}}
$$
for the \'etale topology on $\Sc$-superschemes.  Thus, the Picard sheaf 
${\SPicf^{\bP}_{a,\Xcal/\Sc\et}}$ is representable by the superscheme ${\SPic^{\bP}_{a}(\Xcal/\Sc)}:=\Dsc_a^\bQ/\sim$, and the Picard sheaf 
${\SPicf_{a,\Xcal/\Sc\et}}$ is representable by the superscheme ${\SPic_{a}(\Xcal/\Sc)}:=\Dsc_a/\sim$. Both Picard superschemes are of finite type and separated over $\Sc$.
\end{prop}
\begin{proof} 
 It is enough to prove the first case. Since $\Rcal_a^\bQ=\Dsc_a^\bQ\times_{\Ss\Pc_a}\Dsc_a^\bQ$ and the projection $p_2^\bQ\colon\Rcal_a^\bQ\to \Dsc_a^\bQ$ is smooth and proper, if we prove that  $\Ab\colon\Dsc_a^\bQ \to {\SPicf^{\bP}_{a,\Xcal/\Sc(et)}}$ is an epimorphism of \'etale sheaves, the result will follow  from Proposition \ref{prop:quotsmooth}. Let $\Tc\to\Sc$ be a superscheme over $\Sc$. A $\Tc$-valued point $\xi\colon \Tc \to {\SPicf^{\bP}_{a,\Xcal/\Sc\et}}$ of the functor ${\SPicf^{\bP}_{a,\Xcal/\Sc\et}}$ is given by an \'etale covering $\phi\colon\Tc'\to \Tc$ together with the class $\xi'=[\Lcl]\in {\SPicf^{\bP}_{a,\Xcal/\Sc}}(\Tc')$ of an even line bundle $\Lcl$ on $\Xcal_{\Tc'}\to \Tc'$ that relatively acyclic and is  generated by its global sections. One has to prove that there exist an \'etale covering $\psi\colon\Vc\to\Tc'$ and a morphism $\sigma\colon \Vc \to \Dsc_a^\bQ$ of $\Sc$-superschemes such that $\Ab\circ \sigma=\xi'\circ\psi$.

By Proposition \ref{prop:Abel} and Proposition \ref{prop:asssheaf}  the fibre of the Abel morphism over $\xi'$ is a projective bundle $p\colon\widetilde\Ps(f_{\Tc'\ast}\Lcl)\to \Tc'$:
$$
\xymatrix{
\widetilde\Ps(f_{\Tc'\ast}\Lcl) \ar[d]^p\ar[r]^{\xi'_{\Dsc_a^\bQ}} & \Dsc_a^\bQ\ar[d]^{\Ab}\\
\Tc'\ar[r]^{\xi'}& {{\SPicf^{\bP}_{a,\Xcal/\Sc}}\,.}
}
$$
 Since $p$ is smooth, there exist an \'etale covering $\psi\colon\Vc\to\Tc'$  and  a morphism $\varpi\colon\Vc\to \widetilde\Ps(f_{\Tc'\ast}\Lcl)$ satisfying $\psi=p\circ\varpi$. Now one can take for  $\sigma$   the composition $\xi'_{\Dsc_a^\bQ}\circ \varpi\colon \Vc\to \Dsc_a^\bQ$.
\end{proof}

To finish the proof of Theorem \ref{thm:picrepres} we   check that the even Picard sheaf ${\SPicf_{\Xcal/\Sc\et}}$ and the open subsheaf ${\SPicf_{a,\Xcal/\Sc\et}}$ satisfy condition (1) of Proposition \ref{prop:opengoup}. 
This is the content of the following statements. In the proof we use the following notation: we write $\Oc=\Oc_\Xcal$ and $\Oc(n)=\Oc(1)^{\otimes n}$, where $\Oc(1)$ is the even  ample line bundle which makes   $f$ a superprojective morphism. 

\begin{lemma}\label{lem:localrep} Let us consider the set of the $\Sc$-valued points of the Picard sheaf defined by the classes $[\Oc(-n)]$ ($n>0$). One has:
$$
\bigcup_{n\in \N}{\SPicf_{a,\Xcal/\Sc\et}}\cdot [\Oc(-n)]={\SPicf_{\Xcal/\Sc\et}}\,,
$$
that is, the even super Picard sheaf is the union of the translated of the even Picard sheaf ${\SPicf_{a,\Xcal/\Sc\et}}$ by the classes $[\Oc(-n)]$.
\end{lemma}
\begin{proof} 
 Let us write $\Ss\Pc={\SPicf_{\Xcal/\Sc\et}}$ and $\Ss\Pc_a={\SPicf_{a,\Xcal/\Sc\et}}$ for simplicity.
One has to prove that for every $\Sc$-superscheme $\Tc\to\Sc$ and every section $\lambda\in \Ss\Pc(\Tc)$, that is, for every functor morphism $\lambda\colon \Tc \to \Ss\Pc$, the open sub-superschemes $\Tc_n(\lambda) = \Ss\Pc_a\cdot [\Oc(-n)]\times_{\Ss\Pc,\lambda}\Tc\hookrightarrow \Tc$ for the various $n\in\Z$ yield a covering of $\Tc$.
 Let $\pi\colon\Tc'\to\Tc$ be an \'etale covering such that $\pi^\ast\lambda=\lambda\circ\pi\colon \Tc'\to \Ss\Pc$ is defined by the class $[\Lcl]$ of an even line bundle $\Lcl$ on $\Xcal_{\Tc'}$. Then, it is enough to prove that the open sub-superschemes $\Tc'_n(\lambda) = \Ss\Pc_a\cdot [\Oc(-n)]\times_{\Ss\Pc,\pi^\ast\lambda}\Tc'\hookrightarrow \Tc'$ for the various $n\in\Z$ yield a covering of $\Tc'$. By Proposition \ref{prop:acyclic} $\Tc'_n(\lambda)$ is the open sub-superscheme of the points $t\in T'$ such that the restriction of the sheaf $\Lcl_t(n)$ to the fibre $\Xcal_t$ of $\Xcal_{\Tc'}\to\Tc'$ over $t$ is acyclic and generated by its global sections. 
 It follows then from Serre's Theorem \ref{thm:serre} that for every point $t\in\Tc'$ there is $n$ such that $t\in \Tc'_n(\lambda)$.
\end{proof}
Then one has:
\begin{prop}[End of the proof of Theorem \ref{thm:picrepres}]\label{prop:proof} {Under the hypotheses of Theorem \ref{thm:picrepres}, the} even super Picard sheaf ${\SPicf_{\Xcal/\Sc\et}}$ is representable by a superscheme ${\SPic(\Xcal/\Sc)}$ which can be covered by open sub-superschemes of the form ${\SPic_{a}}(\Xcal/\Sc)\cdot [\Oc(-n)]$.  {In particular,  the Picard   superscheme ${\SPic}(\Xcal/\Sc)$ is locally of finite type over $\Sc$.}
\end{prop} 
\begin{proof}
The representability follows from Propositions \ref{prop:opengoup} and \ref{prop:divpic} together with Lemma \ref{lem:localrep}.  Moreover, the superschemes ${\SPic_{a}}(\Xcal/\Sc)\cdot [\Oc(-n)]$ are of finite type over $\Sc$ by Proposition \ref{prop:divpic} and then  ${\SPic}(\Xcal/\Sc)$ is locally of finite type over $\Sc$.
\end{proof}

\subsection{Representability of the Picard functor over an even affine base}

\subsubsection{General criterion}

The conditions that according to Theorem \ref{thm:picrepres} ensure the existence of the Picard superscheme of a morphism $f\colon\Xcal \to \Sc$ and, in particular the requirement of projectivity and of cohomological flatness in dimension $0$ may look quite restrictive. 
Using the following technical criterion for the  representability of the Picard functor for proper superschemes over an even affine base, we will show that in some cases these conditions can
be relaxed.

Let us fix a Noetherian even ring $R$.

\begin{prop}\label{prop:even-ring-Pic-repr}
Let $\Xcal$ be a proper superscheme over $S=\Spec(R)$.
Assume that $(\Oc_{\Xcal,-})^{N+1}=0$ and $N!$ is invertible in $R$ for some $N$.
Then ${\Picf_{\Xcal/S\et}}$ is representable if and only if
\begin{itemize}
\item the usual Picard functor of the bosonic quotient $\Xcal/\Gamma$ over $S$ is representable by an $R$-scheme;
\item the functor $M\mapsto H^1(\Xcal/\Gamma,\Oc_{\Xcal,-}\otimes_R M)$ on the category of $R$-modules is left exact.
\end{itemize}
Furthermore, if this is the case then ${\SPic}(\Xcal/S)$ splits as follows:
$${\SPic}(\Xcal/R)\simeq \SPicc((\Xcal/\Gamma)/S)\times_S \Spec({\bigwedge}_R F),$$
for some finitely generated $R$-module $F$ (where we think of elements of $F$ as odd).
\end{prop}

\begin{proof} 
Let us set for brevity $X_0=\Xcal/\Gamma$.
For an $R$-superalgebra $\As$ the set of $\As$-points ${\SPicf_{\Xcal/S}}(\As)$ can be identified with
$$H^1(\Xcal, \Oc_{\Xcal_\As,+}^*)/\Pic(\Spec(\As_+)).$$
Furthermore, we have
\begin{equation}\label{even-tensor-product-component-decomposition}
\Oc_{\Xcal_\As,+}=\Oc_{\Xcal,+}\otimes_R \As_+\oplus \Oc_{\Xcal,-}\otimes_R \As_-.
\end{equation}
Hence, $(\Oc_{\Xcal,+}\otimes_R \As_+)^*$ is naturally a subgroup in $\Oc_{\Xcal_\As,+}^*$.
On the other hand, we claim that there is a well defined homomorphism of sheaves of groups,
$$\exp: \Oc_{\Xcal,-}\otimes_R \As_-\to \Oc_{\Xcal_\As,+}^*.$$
Indeed, this follows from the assumption that $(\Oc_{\Xcal,-})^{N+1}=0$ and $N!$ is invertible in $R$.

Next, we claim that the homomorphism
$$(\Oc_{\Xcal,+}\otimes_R \As_+)^*\times \Oc_{\Xcal,-}\otimes_R \As_-\to \Oc_{\Xcal_\As,+}^*:(x_0,x_1)\mapsto x_0\exp(x_1)$$ 
is an isomorphism. Indeed, the components of $x_0\exp(x_1)$ with respect to the decomposition 
\eqref{even-tensor-product-component-decomposition} are given by
$$f_0=x_0\cosh(x_1)=x_0\big(1+\frac{x_1^2}{2!}+\ldots\big), \ \ f_1=x_0\sinh(x_1)=x_0\big(x_1+\frac{x_1^3}{3!}+\ldots\big)\,.$$
Hence, $f_0^{-1}f_1=\tanh(x_1)$, and we can recover $(x_0,x_1)$ from $(f_0,f_1)$ by 
$$(f_0,f_1)\mapsto \big(\frac{f_0}{\cosh(\tanh^{-1}(f_0^{-1}f_1))},\tanh^{-1}(f_0^{-1}f_1)\big)\,,$$
where $\tanh^{-1}(t)=t+\frac{t^3}{3}+\frac{t^5}{5}+\ldots$ (note that the compatibility with the group law corresponds to the addition theorem for $\tanh$).

The above isomorphism of sheaves of groups induces an isomorphism of functors
$${\SPicf_{\Xcal/S}}(\As)\simeq \Picf_{X_0/S}(\As_+)\times H^1(X_0,\Oc_{\Xcal,-}\otimes_R \As_-).$$
We claim that this implies the isomorphism for the sheafified functors
$${\SPicf_{\Xcal/S\et}}(\As)\simeq \Picf_{X_0/S\et}(\As_+)\times H^1(X_0,\Oc_{\Xcal,-}\otimes_R \As_-).$$
Namely, the fact that the second factor does not change follows by taking odd components in the identification
$$H^1(X_0,\Oc_{\Xcal,-}\otimes_R \As)=\ker(H^1(X_0,\Oc_{\Xcal,-}\otimes_R \Bs)\to H^1(X_0,\Oc_{\Xcal,-}\otimes_R (\Bs\otimes_{\As}\Bs))),$$
for any faithfully flat extension $\As\to \Bs$. Indeed, this follows from the faithfully flat descent for modules using the isomorphism
$$H^1(X_0,\Oc_{\Xcal,-}\otimes_R \Bs)\simeq H^1(X_0,\Oc_{\Xcal,-}\otimes_R \As)\otimes_{\As} \Bs$$
(and a similar isomorphism for $\Bs\otimes_{\As} \Bs$) which follows from the flat base change.

Assume first that ${\SPicf_{\Xcal/S\et}}$ is representable. Then by Lemma \ref{lem:base-change-Pic}(ii), $\Picf_{X_0/S}$ is also representable. 
Now let us consider $R$-superalgebras of the form
$$\As=R\oplus M_-,$$
where $M_-$ is any $R$-module (the multiplication on $\As$ is such that $(M_-)^2=0$). Note that 
for such superalgebras we have
$$
{\SPicf_{\Xcal/S\et}}(R\oplus M_-)\simeq \Picf_{X_0/S\et}(R)\times T^1(M_-)\,,
$$
where
$$
T^1(M_-):=H^1(X_0,\Oc_{\Xcal,-}\otimes_R M_-)\,.
$$
Now let us consider any fibred product diagram of $R$-modules
$$
\xymatrix{
M_-\times_{P_-} N_- \ar[r] \ar[d]& N_-\ar[d]\\
M_-\ar[r]& P_-
}
$$
with $M_-\to P_-$ surjective.
We can form the corresponding diagram of $R$-superalgebras (by adding $R$ as the even part). Then 
the corresponding spectra form a fibred coproduct diagram in the category of $R$-superschemes (this is checked similarly to
the even case considered in \cite{Schwede}). Hence,
taking $N_-=0$ and applying the isomorphisms
$$\Morf(\Spec(R\oplus M_-),\SPic(\Xcal/S))\simeq \Picf_{X_0/S\et}(R)\times T^1(M_-)$$
we get the fibred product diagram of sets
$$
\xymatrix{
\Picf_{X_0/S\et}(R)\times T^1(\ker(M_-\to P_-)) \ar[r] \ar[d]& \Picf_{X_0/S\et}(R)\times\{0\}\ar[d]\\
\Picf_{X_0/S\et}(R)\times T^1(M_-)\ar[r]& \Picf_{X_0/S\et}(R)\times T^1(P_-)
}
$$
which implies that $T^1$ is left exact.

Conversely, assume that $\Picf_{X_0/S}$ is representable and $T^1$ is left exact. 
Then there exists a finitely generated $R$-module $F$ such that there is a functorial isomorphism
$$T^1(M_-)\simeq \Hom_R(F,M_-).$$
This means that an element of ${\SPicf_{\Xcal/S\et}}(\As)$ is given by a pair: a morphism $\Spec(\As)\to \SPicc(X_0/S)$ (which automatically factors
through $\Spec(\As_+)$) and a homomorphism of $R$-modules $F\to \As_-$. Giving such a homomorphism is equivalent to giving a homomorphism of $R$-superalgebras
${\bigwedge}_R F\to \As$, so we deduce that ${\SPicf_{\Xcal/S\et}}$ is represented by
$$\SPicc(X_0/S)\times_S \Spec({\bigwedge}_R F).$$
\end{proof}

\begin{remark}
We claim that if 
the morphism $\Xcal\to \Spec(R)$ is flat and cohomologically flat in dimension $0$ then the functor $M\mapsto H^1(\Xcal/\Gamma,\Oc_{\Xcal,-}\otimes_R M)$ is left exact.
Indeed, by Proposition \ref{prop:GMcomplex}, the functor $M\mapsto T^i(M)=H^i(\Xcal/\Gamma,\Oc_{\Xcal,-}\otimes_R M)$ can be calculated as 
$$T^i(M)=H^i(K^\bullet\otimes_R M),$$
where $K^\bullet$ is a complex of finitely generated projective $R$-modules, concentrated in degrees $[0,n]$. The condition of cohomological flatness in dimension $0$
applied to $R$-superalgebras of the form $R\oplus M$ implies that $T^0(M)=0$ for all $R$-modules $M$. This implies that the differential $\partial_0:K^0\to K^1$
is injective with projective quotient $C^1=\coker(\partial_0)$. It follows that 
$$T^1(M)=\ker(C^1\otimes_R M\to K^2\otimes_R M),$$
so this functor is left exact.
\end{remark}

\begin{corol}\label{cor:even-ring-Pic-repr-dim1} 
Under the assumptions of Proposition \ref{prop:even-ring-Pic-repr},
assume in addition that $\Xcal$ is a flat over $R$ and the relative dimension of $\Xcal/\Gamma$ over $R$ is $\le 1$. Then
${\Picf_{\Xcal/S\et}}$ is representable if and only if
\begin{itemize}
\item the usual Picard functor of the bosonic quotient $\Xcal/\Gamma$ over $S$ is representable by an $S$-scheme;
\item the $R$-module $H^1(\Xcal/\Gamma,\Oc_{\Xcal,-})$ is projective.
\end{itemize}
\end{corol}

\begin{proof}
We just observe that using the notation above, the functor $T^1$ associated with $\Oc_{\Xcal,-}$ (which is flat over $R$) is right exact due to the assumption on the relative
dimension. Hence, it is left exact if and only if it is exact if and only if the $R$-module
$H^1(\Xcal\Gamma,\Oc_{\Xcal,-})$ is projective (see \cite[Cor.\ 12.6]{Hart77}).
\end{proof}

\begin{corol}
Let $f\colon\Xcal\to S$ be a flat family of $(1,1)$-curves over $S=\Spec(R)$, given by $\Oc_{\Xcal}=\Oc_X\oplus \Lcl$, where $f_{\bos}\colon X\to S$ is a family of smooth projective
curves, and $\Lcl$ is a line bundle over $X$. Then ${\Picf_{\Xcal/S\et}}$ is representable if and only if $R^1f_{\bos,\ast}(\Lcl)$ is locally free.
\qed\end{corol}

For example, if $\Lcl=\omega_{X/S}$ in the last corollary then ${\Picf_{\Xcal/S\et}}$ is representable since $R^1 f_{\bos,\ast}(\omega_{X/S})$ is locally free but $f\colon\Xcal\to S$ is
not cohomologically flat in dimension $0$ if genus is $\ge 1$.

\subsubsection{Generic representability and representability over a field}

\begin{thm}\label{thm:generic-Pic-repr}
Let $S=\Spec(R)$ where $R$ is an (even) integral domain. Assume that $\Xcal$ is a proper superscheme over $S$ and that
$(\Oc_{\Xcal,-})^{N+1}=0$ and $N!$ is invertible in $R$, for some $N$.
Then there exists
a nonempty open subset $V\subset S$ such that $\SPicf_{\Xcal_V/V\et}$ is representable by 
$$\SPicc((\Xcal/\Gamma)_V/V)\times \As^{0,n}$$
for some $n$, where $\SPicc((\Xcal/\Gamma)_V/V)$ is a disjoint union of  quasiprojective schemes over $V$.
\end{thm}

\begin{proof}
We would like to apply Proposition \ref{prop:even-ring-Pic-repr} to $\Xcal_V$ over $V$ for some open affine subset $V\subset S$. 
First, we apply the classical result \cite[Thm.\ 9.4.18.2]{FGA05} that implies that there exists $V$ such that the Picard functor of $(\Xcal/\Gamma)_V$ over $V$ is representable 
by a disjoint union of  quasiprojective schemes.

Next we claim that after replacing $R$ with its nonzero localization, we can achieve that 
the functor $M\mapsto H^1(\Xcal/\Gamma,\Oc_{\Xcal,-}\otimes_R M)$ on the category of $R$-modules is left exact.
Indeed, let $K^0\to\ldots\to K^n$ be the complex of finitely generated projective $R$-modules such that
$$T^i(M)=H^i(\Xcal/\Gamma,\Oc_{\Xcal,-}\otimes_R M)\simeq H^i(K^\bullet\otimes_R M)$$
(it exists by Proposition \ref{prop:GMcomplex}, since $\Oc_{\Xcal}$ is flat). Furthermore, $T^1$ is left exact if and only if
$\coker(\partial_0:K^0\to K^1)$ is projective (Lemma \ref{lem:ti}). But the latter condition is satisfied after replacing $R$ with its localization.

Applying Proposition \ref{prop:even-ring-Pic-repr}, we get the representability of the Picard functor and an isomorphism
$${\SPic}(\Xcal/R)\simeq \SPicc((\Xcal/\Gamma)/R)\times_{\Spec R}\Spec({\bigwedge}_R F),$$
for some finitely generated $R$-module $F$. Localizing $R$ further we can achieve that $F$ is a free $R$-module.
\end{proof}

\begin{corol}\label{cor:field-Pic-repr} 
Let $\Xcal$ be a proper superscheme over a field $k$, such that for some $N$, one has $(\Oc_{\Xcal,-})^{N+1}=0$ and $N!$ does not divide the characteristic of $k$. Then $\SPicf_{\Xcal/k\et}$ is representable by 
$$\SPicc((\Xcal/\Gamma)/k)\times_{\Spec(k)} \As^{0,n},$$
where $n=\dim_k H^1(\Xcal/\Gamma, \Oc_{\Xcal,-})$.
\qed\end{corol}

\begin{example}
Using Corollary \ref{cor:field-Pic-repr} one can deduce that for the $\Pi$-projective space $\Ps^n_\Pi$ over a field $k$ of characteristic zero (see e.g., \cite{Noja18}) 
one has $\SPic(\Ps^n_{\Pi}/k)\simeq \As^{0,1}$. Indeed, the corresponding reduced space is the usual projective space $\Ps^n$ and it is well known that only the trivial
line bundle on $\Ps^n$ extends to a line bundle on $\Ps^n_\Pi$. Our assertion follows from the fact that this remains true over any base (even) ring.
\end{example}

\subsubsection{Example of non-representability of the super Picard functor}

Extending Corollary \ref{cor:even-ring-Pic-repr-dim1},
we are going to give an example of a projective morphism of superschemes of relative dimension $(1,1)$ for which not only super Picard functor is not representable, but even the corrresponding deformation functor is not pro-representable.

Let $X$ be a family of smooth projective curves over $\Spec(R)$, where $R=k[\![t]\!]$, and let $\Lcl$ be a line bundle over $X$.
Let us consider the split $(1,1)$-dimensional smooth superscheme $\Xcal$ over $\Spec(R)$, with the usual underlying scheme $X$,
given by $\Oc_\Xcal=\Oc_X\oplus \Lcl$.

Let $K=R[t^{-1}]=k(\!(t)\!)$. Let $\Lcl_0=\Lcl_{\vert X_0}$, where $X_0$ is the fibre of $X$ over $\Spec k\subset \Spec R $
and let $\Lcl_K=\Lcl_{\vert X_K}$, where $X_K=X\times_{\Spec(R)} \Spec(K)$.

\begin{prop} Assume that $H^1(X_0,\Lcl_0)\neq 0$ while $H^1(X_K,\Lcl_K)=0$. Then the relative even Picard functor  of $\Xcal/\Spec(R)$ is not
representable. Moreover, the corresponding functor on Artin super $R$-algebras (deforming the trivial line bundle over $\Xcal_0$) is not pro-representable. 
\end{prop}

Here is an example when the above situation can realize. Consider a fixed smooth projective curve $X_0$ of genus $g\ge 1$ over $k$,
and let $X=X_0\times \Spec(R)$. Let $\Lcl_0$ be a line bundle of degree $g-1$ with $H^1(X_0,\Lcl_0)\neq 0$, i.e., the corresponding point sits
on the theta divisor in $\Pic^{g-1}(X_0)$. Now consider any curve in $\Pic^{g-1}(X_0)$ passing through this point and not contained in the theta-divisor.
Then take the induced family of line bundles over $k[[t]]$.

Furthermore, by allowing to vary the curve as well, we can make sure that $\Lcl$ is a theta-characteristic on $X$, so $\Xcal$ would be a family of SUSY curves.

\begin{proof}[Proof of Proposition]
Let us consider the functor $\As\mapsto {\SPic}(\Xcal_\As/\As)$ on local Artin $R$-{super\-algebras} with the residue field $k$. 
Restricting this functor to local Artin $R$-algebras with $\As_+=k$ and arguing as in the proof of Proposition \ref{prop:even-ring-Pic-repr},
we see that if the deformation functor were pro-representable then the functor
$$M_-\mapsto H^1(X,\Lcl\otimes_R M_-)$$
on the category of finite-dimensional torsion $R$-modules
would be left exact.
Note that by the base change, 
$$
H^1(X,\Lcl\otimes_R M_-)\simeq H^1(X,\Lcl)\otimes_R M_-.
$$ 
But $H^1(X,\Lcl)$ is a finitely generated $R$-module with $H^1(X,\Lcl)\otimes_R k=H^1(X,\Lcl_0)$ and $H^1(X,\Lcl)\otimes_R K=0$ (still by the base change).
Hence, $H^1(X,\Lcl)$ is a nonzero torsion $R$-module, so it is a direct sum of $R$-modules of the form $R/(t^n)$. Thus, it is enough to observe that the functor
$$
M_-\mapsto R/(t^n)\otimes_R M_-
$$
is not left exact on finite-dimensional torsion $R$-modules.
Indeed, for the exact sequence 
$$
0\to R/(t)\xrightarrow{t^n} R/(t^{n+1})\to R/(t^n)\to 0
$$
its tensor product with $R/(t^n)$ fails to be left exact.
\end{proof}

\section{The super period map}\label{s:superperiod}

Our goal in this section is to construct a morphism from an open substack of the moduli of proper and smooth supercurves to the moduli stack of
principally polarized abelian {schemes}.

\subsection{Smoothness of the Picard scheme}
To begin with, we need a standard criterion of smoothness of the Picard superscheme.
Let $f\colon\Xcal\to \Sc$ be a flat proper morphism of superschemes, with $\Sc$ Noetherian, 
for which the even Picard superscheme ${\SPic}(\Xcal/\Sc)$ exists
and represents the super Picard sheaf in the \'etale topology. 

\begin{prop}\label{prop:smooth-Pic}
Assume that for some point $s\in S$ one has $H^2(X_s,\Oc_{\Xcal_s})=0$. Then ${\SPic}(\Xcal/\Sc)$ is smooth over an open neighbourhood of $s$.
\end{prop}

\begin{proof}
This is a standard argument similar to the classical one. By semicontinuity and base change replacing $\Sc$ by an affine neighbourhood of $s$
we get that $H^2(X_s,\Oc_{\Xcal_s})=0$ for every $s\in S$, and for every affine $\Sc$-superscheme $\Tc=\SSpec(\Bs)$ one has $H^2(X_T,\Oc_{\Xcal_\Bs})=0$.

Now let $\SSpec(\As)$ be an affine superscheme over $\Sc$ with a zero-square ideal $\Nc\subset \As$, and let $\Bs=\As/\Nc$.
We can view $\Xcal_\As=\Xcal\times_\Ss \Spec(\As)$ as an infinitesimal
thickening of $\Xcal_\Tc=\Xcal_\Bs$. The exact sequence
$$
0\to (f_\Bs^\ast \Nc)_+\to (\Oc_{\Xcal_\As +})^\ast\to (\Oc_{\Xcal_\Bs+})^\ast\to 0
$$
leads to a long exact sequence
$$
\ldots\to \Pic(\Xcal_\As)\to \Pic(\Xcal_\Bs)\to H^2(X_B,f_\Bs^\ast\Nc)_+\to\ldots$$
Thus, the obstruction to lifting a line bundle on $\Xcal_\Bs$ to a line bundle on $\Xcal_\As$ lies
in the even part of $H^2(X_T,f_\Bs^\ast\Nc)$. On the other hand,
$$
H^2(X_T,f_\Bs^\ast\Nc)\simeq H^0(T,R^2f_{\Bs\ast}(f_\Bs^\ast\Nc))\simeq
H^0(T,\Nc\otimes R^2f_{\Bs\ast}(\Oc_{\Xcal_\Bs}))=0
$$
as $R^2f_{\Bs\ast}(\Oc_{\Xcal_\Bs})=0$. It follows that ${\SPic}(\Xcal_\Bs/\Bs)$ is formally smooth 
(see Definition \ref{def:formallysmooth}). Since it is locally of finite type by Theorem \ref{thm:picrepres}, the conclusion follows (cf.~Proposition \ref{prop:loc-free-smooth}).
\end{proof}

\subsection{Abelian {schemes} associated with Picard superschemes}

In this section we construct abelian {schemes} which are naturally associated with Picard superschemes of relative supercurves.

\subsubsection{From the Picard superscheme to an abelian {scheme}}

Let $f\colon\Xcal\to\Sc$ be a proper and smooth morphism of relative dimension $(1,1)$. Proceeding as in \cite[Lemma 7.24]{MoZh19}, one can see that $f$ is locally superprojective (see also \cite{FKP20} or \cite{Cod14}).
Assume that $H^\bullet(X_s,\Oc_{\Xcal_s})_-=0$ for each $s$ in $\Sc$. Then $f$ is cohomologically flat in dimension $0$ by Proposition \ref{prop:cohomflat}, and Theorem \ref{thm:picrepres} ensures the existence of the 
Picard superscheme ${\SPic}(\Xcal/\Sc)$ of $f$.

\begin{thm}\label{thm:SUSY-Picard} In this situation there exists an open subgroup scheme ${\SPic}^0(\Xcal/\Sc)\hookrightarrow {\SPic}(\Xcal/\Sc)$ such that for every point $s\in \Sc$, one has
$$
{\SPic}^0(\Xcal/\Sc)\times_\Sc \{s\}=\SPicc^0(X_s)\,,
$$
where $\SPicc^0(X_s)$ is the connected component of zero of the Picard scheme $\SPicc(X_s)$. Furthermore, ${\SPic}^0(\Xcal/\Sc)$ is smooth and proper of even dimension over $\Sc$. Hence, there exists an abelian scheme $A$ over the bosonic quotient $\Sc/\Gamma$ 
such that ${\SPic}^0(\Xcal/\Sc)\simeq \Sc\times_{\Sc/\Gamma} A$.
\end{thm}

\begin{proof}
First, we observe that by Proposition \ref{prop:smooth-Pic} the superscheme ${\SPic}(\Xcal/\Sc)$ is smooth over $\Sc$.
Furthermore, for every $s\in \Sc$ the tangent space to ${\SPic}(\Xcal_s)$ at any point is isomorphic to $H^1(X_s,\Oc_{\Xcal_s})$.
By assumption, 
$H^1(X_s,\Oc_{\Xcal_s})_-=0$. Hence, we derive that  ${\SPic}(\Xcal/\Sc)$ is a disjoint union of schemes which are smooth of even relative dimension over $\Sc$.

Next, let us consider the case when $\Sc=S$ is purely even.
By Corollary \ref{cor:even}, in this case
${\SPic}(\Xcal/S)$ is also a purely even superscheme (since it is a disjoint union of smooth superschemes of even dimension over $S$).

Let $C=\Xcal_{\bos}=\Xcal/\Gamma$ be the underlying family of usual curves over $S$. From Lemma \ref{lem:base-change-Pic}(ii) we get
an isomorphism
\begin{equation}\label{eq:Pic-even-S-curve-isom}
\SPicc(C/S)\simeq {\SPic}(\Xcal/S)_{\bos}\simeq {\SPic}(\Xcal/S).
\end{equation}
Thus, in the case of purely even $\Sc=S$, using the isomorphism \eqref{eq:Pic-even-S-curve-isom} we can define 
${\SPic}^0(\Xcal/S)$ as the subgroup scheme of ${\SPic}(\Xcal/S)$ corresponding to $\SPicc^0(C/S)$, which is an open subgroup scheme of $\SPicc(C/S)$.
Hence, by the classical result on the Picard scheme of a family of curves (see \cite[n$^{\mbox{o}}$ 236, Thm.~2.1]{FGA}), the scheme
${\SPic}^0(\Xcal/S)$ is proper over $S$.

Now let us return to the case of a general base $\Sc$. Let us set $S=\Sc_{\bos}$ and let $\Xcal_S\to S$ be the base change of our family. 
By Lemma \ref{lem:base-change-Pic}(i) one has
$$
{\SPic}(\Xcal/\Sc)\times_\Sc S \simeq {\SPic}(\Xcal_S/S)\,.
$$
In particular, ${\SPic}(\Xcal/\Sc)$ has the same   underlying topological space of ${\SPic}(\Xcal_S/S)$, so we define
${\SPic}^0(\Xcal/\Sc)$ as the open sub-superscheme of ${\SPic}(\Xcal/\Sc)$ corresponding to the open subset ${\SPic^0}(\Xcal_S/S)$.
Note that the base change of the morphism
\begin{align*}
\alpha\colon{\SPic}^0(\Xcal/\Sc)\times_\Sc {\SPic}^0(\Xcal/\Sc)&\to {\SPic}(\Xcal/\Sc)\\(a,b)&\mapsto a\cdot b^{-1}
\end{align*}
from $\Sc$ to $S$ is a similar morphism for ${\SPic}^0(\Xcal_S/S)$. It follows that the image of $\alpha$ is contained in 
${\SPic}^0(\Xcal/\Sc)$, so ${\SPic}^0(\Xcal/\Sc)$ is a subgroup.

It follows from Theorem \ref{thm:picrepres} that ${\SPic}^0(\Xcal/\Sc)$ is locally of finite type. Since we also know that
${\SPic}^0(\Xcal_S/S)$ is proper over $S$, we deduce that ${\SPic}^0(\Xcal/\Sc)$ is proper over $\Sc$.

Finally, applying Proposition \ref{prop:even-smooth} we see that ${\SPic}^0(\Xcal/\Sc)$ is obtained by the base change from a smooth and proper scheme $A$ over $\Sc/\Gamma$. Furthermore,
by the same theorem, the group structure on ${\SPic}^0(\Xcal/\Sc)$ also comes from a group structure on $A$. Thus, $A$ is an abelian {scheme}.
\end{proof}

\begin{remark}
As we have seen in the proof, in the case when the base $\Sc$ is even, $\Xcal_{\bos}=\Xcal/\Gamma$ is a smooth proper curve over $\Sc$, and the abelian scheme $A$ in the above theorem is its relative Jacobian.
In general, there seems to be no family of usual curves over $\Sc/\Gamma$ whose relative Jacobian would give $A$. For example, 
the bosonic quotient $\Xcal/\Gamma$ is not smooth over $\Sc/\Gamma$ in general. 
\end{remark}

\subsubsection{From a family of SUSY-curves to a polarized abelian {scheme}}

We can apply Theorem \ref{thm:SUSY-Picard} to a family of SUSY curves $f\colon \Xcal\to\Sc$. Recall that this is a proper smooth morphism of relative dimension $(1,1)$ equipped
with a certain distribution of rank $(0,1)$ (see e.g., \cite{LeRoth88} for details). Usually, SUSY curves are considered in characteristic $0$, however, all features of the
theory of SUSY curves are valid away from characteristic $2$.
We claim that in this case the abelian {scheme} over $\Sc/\Gamma$ constructed in Theorem \ref{thm:SUSY-Picard}
carries a natural polarization.

\begin{thm}\label{superperiod-thm}
Let $f\colon \Xcal\to\Sc$ be a relative SUSY curve such that $H^\bullet(X_s,\Oc_{\Xcal_s})_-=0$ for each $s$ in $\Sc$.
In addition we assume that $2$ is invertible on $\Sc$.
Then there exists a natural principally polarized abelian {scheme} $A$ over $\Sc/\Gamma$ that extends the relative Jacobian of the
corresponding family of usual curves $f_{bos}\colon X\to S$. Furthermore, $A$ is equipped with a natural relative symmetric divisor inducing
the polarization.
\end{thm}

\begin{proof} Let us set for brevity  $\Psc={ \SPic}^0(\Xcal/\Sc)$. As we have seen in Theorem \ref{thm:SUSY-Picard}, $\Psc=\Sc\times_{\Sc/\Gamma} A$ for an abelian scheme $A$ over $\Sc/\Gamma$.

First, assume that $f$ admits a section $\sigma\colon\Sc\to \Xcal$. Recall (see \cite[Lemma 2.9]{FKP20} and also \cite{DoHeSa93}) that to every Neveu-Schwarz (NS) puncture on
a relative SUSY curve one can associate canonically a relative effective Cartier divisor.
Applying this construction to the section $\sigma$ we can construct a relative effective divisor $\Zc_{\sigma}\hookrightarrow \Xcal$
supported on $\sigma(\Sc)$. On the other hand, considering the diagonal $\delta$ as an NS puncture on the family $\Xcal\times_\Sc \Xcal\xrightarrow{p_2} \Xcal$, 
using the same construction, we get a relative effective divisor
$\Zc_{\delta}\hookrightarrow\Xcal\times_\Sc \Xcal$, so that we have a line bundle $\Oc_{\Xcal\times_\Sc\Xcal}(\Zc_{\delta}-p_1^\ast\Zc_{\sigma})$ on $\Xcal\times_\Sc\Xcal$ which restricts to the trivial line bundle over $\Xcal\times_\Sc\sigma(\Sc)$.
Hence, it defines a morphism of superschemes over $\Sc$,
$$
\alpha_{\sigma}\colon\Xcal\to \Psc
$$
depending on $\sigma$, which restricts to the standard Abel morphism on the underlying curve $X$ over $S$.

Set $\Psc^\vee:=\Sc\times_{\Sc/\Gamma} A^\vee$, where $A^{\vee}$ is the dual abelian scheme to $A$ over $\Sc/\Gamma$.
Note that the dual abelian scheme $A^{\vee}$ exists by a theorem of Raynaud (see \cite[Ch.\! I, Sec.\! 1]{CF90}).
Let $\Pc$ be the pull-back of the Poincar\'e line bundle from $A\times A^{\vee}$ to $\Psc\times \Psc^\vee$ (normalized along the zero sections).
Consider the line bundle 
$$
\Lcl_\sigma:=(\alpha_\sigma\times \Id)^\ast\Pc
$$
on $\Xcal\times \Psc^\vee$. Note that it is trivalized along $\sigma(\Sc)\times \Psc^\vee$ and along $\Xcal\times 0$. Hence, it defines a morphism
$$
\Psc^\vee\to \Psc\,,
$$
sending the zero section to the zero section.
By Proposition \ref{prop:even-smooth}, this morphism comes from a homomorphism of abelian schemes over $\Sc/\Gamma$
$$
\lambda_\sigma\colon A^\vee\to A
$$
(recall that any morphism between abelian schemes preserving the zero section is a homomorphism, see \cite[Cor.\! 6.4]{MF-GIT}).
Furthermore,
 the induced morphism of abelian schemes over $S$ corresponds to the standard principal polarization of $J(X)^\vee\simeq J(X)$ (where $X/S$
 is the corresponding usual family of curves).
Hence, $(A^\vee,\lambda)$ is a polarized abelian scheme over $\Sc/\Gamma$ extending the relative Jacobian of $X$ over $S$.

We claim that if $\sigma'\colon\Sc\to \Xcal$ is another section, then $\lambda_{\sigma'}=\lambda_\sigma$. Indeed,
first we observe that $\alpha_{\sigma'}$ differs from $\alpha_{\sigma}$ by the translation map $t_{\xi}\colon \Psc\to \Psc$,
corresponding to the $\Sc$-point $\xi\in \Psc(\Sc)$ coming from the relative divisor $\Zc_{\sigma'}-\Zc_{\sigma}$.
Then, using the standard isomorphism
$$
(x+y,z)^\ast \Pc\simeq (x,z)^\ast\Pc\otimes (y,z)^\ast\Pc\,,
$$ 
we get an isomorphism
$$
\Lcl_{\sigma'}\simeq \Lcl_{\sigma}\otimes p_2^\ast(\Pc_{\vert \xi\times \Psc^\vee})\,.
$$
Hence $\Lcl_{\sigma'}$ defines the same element in the relative Picard functor, so the corresponding homomorphisms are the same.

In general, $f\colon\Xcal\to\Sc$ has a section locally in the \'etale topology, so we define the homomorphism from $A^{\vee}$ to $A$ locally and then use descent (Proposition \ref{prop:morphdescent})
to show that it exists globally.

To prove the last statement we observe that so far we constructed a morphism
$$S\to \Ac_g,$$
where $\Ac_g$ is the stack of principally polarized abelian varieties. Let
$$\wt{\Ac}_g\to \Ac_g$$
be the natural covering corresponding to a choice of a symmetric divisor inducing the polarization.
Since we work over $\Z[1/2]$, this is an \'etale covering.
Our family of supercurves gives rise to a family of spin-structures over the corresponding usual family of curves $X\to S$. Hence,
the morphism $S\to \Ac_g$ has a natural lift to a morphism
$$S\to \wt{\Ac}_g$$
(here we use the classical construction of the symmetric divisor $\Theta_L$ in the Jacobian of a curve associated with a spin-structure $L$).
Since $\Sc/\Gamma$ is a nilpotent extension of $S$, there is a unique compatible morphism
$$\Sc/\Gamma\to \wt{\Ac}_g,$$
as claimed.
\end{proof}
 
\subsection{Relation with the superperiod matrix} {Here we establish the relation of our construction of a super period map with
the D'Hoker-Phong superperiod matrix \cite{D'Hoker-Phong}.}
We   switch to working in the complex analytic category.
Let $\pi:\Xcal\to \Sc$ be a family of SUSY curves equipped with a symplectic trivialization of the local system $R^1\pi_*\Z_X$.
Let us pick a decomposition $R^1\pi_*\Z_X=\La_\Z\oplus \La'_\Z$ into (trivial) Lagrangian local systems, and let
$$\Vc:=R^1\pi_*\pi^{-1}\Oc_{\Sc}=\La\oplus \La'$$
be the corresponding decomposition into (trivial) Lagrangian subbundles of rank $g$.
Assume that $H^*(X_s,\Oc_{\Xcal_s})_-=0$ for every $s$.
Then the exact sequence of sheaves
$$0\to \pi^{-1}\Oc_{\Sc}\to \Oc_{\Xcal}\to \om_{\Xcal/\Sc}\to 0$$
gives rise to an exact sequence
\begin{equation}
0\to \pi_*\om_{X/S}\to \Vc \to R^1\pi_*\Oc_{\Xcal}\to 0
\end{equation}
(see \cite[Sec.\ 2]{FKP19})
Furthermore, $\pi_*\om_{X/S}$ is a Lagrangian subbundle in the trivial symplectic bundle $\Vc$,
which is transversal to $\La$ and $\La'$.
Hence, $\pi_*\om_{X/S}$ is the graph of a symmetric morphism $\tau:\La\to \La'$.
Since the bundles $\La$ and $\La'\simeq \La^\vee$ are trivialized, $\tau$ is given by a symmetric matrix of even functions on $\Sc$,
which is called the {\it superperiod matrix}.

In this way we get a morphism of schemes $\tau:\Sc/\Gamma\to {\frak H}_g$, where ${\frak H}_g$ is the Siegel upper half-space.
The restriction of this morphism to $S\sub \Sc/\Gamma$ is the classical period matrix associated with the family of usual curves over
$S$ (and a trivialization of $R^1\pi_*\Z_X$).
As is well known, there is a natural morphism to the moduli stack of principally polarized abelian varieties,
$${\frak H}_g\to \Ac_g.$$ 

\begin{prop}\label{superperiod-matrix-prop} In the above situation
the composed map $\Sc/\Gamma\to {\frak H}_g\to \Ac_g$ coincides with the superperiod map constructed in
Theorem \ref{superperiod-thm}.
\end{prop}

We will need the following standard fact.

\begin{lemma}\label{analytic-spaces-lem} {Let $S$ be an analytic space and 
$f:X\to X'$ a smooth morphism of $S$-analytic spaces.} Let $S_{\re}$ be the reduced subspace, and assume that the induced morphism
$$f_{{\re}}:X_{S_{\re}}\to X'_{S_{\re}}$$ is an isomorphism. Then $f$ is an isomorphism.
\end{lemma}

\begin{proof}
It is enough to check this locally near every point $x\in X$. By a standard argument (see \cite[V.17]{Dem}), it is enough to check that the morphism
$f_*:T_x X\to T_{f(x)} X'$
of Zariski tangent spaces is an isomorphism. Let $\pi:X\to S$ (resp., $\pi':X'\to S$) denote the projection. The space $T_x X$ fits into an exact sequence
$$0\to T_s S\to T_x X\to T_x\pi^{-1}(s)\to 0$$
where $s=\pi(x)$. We have a similar sequence for $X'$. The fact that $f_{{\re}}$ is an isomorphism implies that
the map $T_x\pi^{-1}(s)\to T_{f(x)}(\pi')^{-1}(s)$ is an isomorphism. Hence, $f_*$ is an isomorphism.
\end{proof}

\begin{proof}[Proof of Proposition \ref{superperiod-matrix-prop}]
Let us set $X'=\Xcal/\Gamma$, $S'=\Sc/\Gamma$, and let $\pi':X'\to S'$ be the morphism induced by $\pi$.

Note that the category of vector bundles of even rank on a superscheme $\Ycal$ is equivalent to the category of vector bundles on the bosonic quotient $\Ycal/\Gamma$.
Namely, this equivalence sends a vector bundle $V$ on $\Ycal/\Gamma$ to its pull-back under the natural morphism $\Ycal\to \Ycal/\Gamma$. The inverse equivalence
sends a vector bundle of even rank $\Vc$ on $\Ycal$ to $\Vc^+$.
Thus, the embedding of vector bundles 
$\Wc:=\pi_*\om_{\Xcal/\Sc}\to \Vc$ corresponds to an embedding of vector bundles $\Wc'\to \Vc'$ on $S'$.
Also, the quotient $\Vc/\Wc\simeq R^1\pi_*\Oc_{\Xcal}$ corresponds to the vector bundle
$$H:=\Vc'/\Wc'\simeq (\Vc/\Wc)^+\simeq R^1\pi'_*\Oc_{X'}.$$
Furthermore, $\Vc'$ still comes from the trivial symplectic local system on $S'$ and is equipped with a $\Z$-system $\Vc'_\Z\sub \Vc'$ which projects
to a $\Z$-lattice $H_\Z$ in $H$.

The morphism $S'\to \Ac_g$ given by the superperiod matrix corresponds to the analytic abelian scheme over $S'$
$$A_{\tau}:=\Vc'/(\Wc'+\Vc'_\Z)=H/H_{\Z}.$$
Furthermore, $A_{\tau}$ is equipped with a natural principal polarization.

On the other hand, we have an abelian scheme $A$ over $S'$ such that $A_{\Sc}=\SPic^0(\Xcal/\Sc)$.
We want to construct an isomorphism of polarized abelian schemes over $S'$,
$$A_{\tau} \simeq A.$$

Assume first that the family $\Xcal\to\Sc$ has a section $\si:\Sc\to \Xcal$. We denote still by $\si$ the induced
section $S'\to X'$.
We have $H\simeq R^1\pi'_*\Oc_{X'}$.
Hence, there is a tautological global section $s$ of the pull-back of $R^1\pi'_*\Oc_{X'}$ to $H$.
This means that there exists a Stein open covering $(U_i)$ of $S'$, such that $s$ defines an element $s_i\in H^1(H_{U_i}\times_{U_i} X'_{U_i},\Oc)$.
Furthermore, we have $H^1(H_{U_i},\Oc)=0$, so the induced elements $\exp(2\pi s_i)\in H^1(H_{U_i}\times_{U_i} X'_{U_i},\Oc^*)$ have trivial restrictions
under the section $\si:H_{U_i}\to H_{U_i}\times_{U_i} X'_{U_i}$. Thus, they define line bundles $L_i$ over $H_{U_i}\times_{U_i} X'_{U_i}$ with trivial
pull-backs under $\si$. If we fix a trivialization of $L_i$ over the image of $\si$ then $L_i$ is defined uniquely up to a unique isomorphism.
Hence, these line bundles glue into a line bundle $L$ over $H\times_{S'} X'$. Note that the tautological section $s$ has trivial restriction to the zero section $0_{S'}$ in $H$.
Hence, $L$ has trivial restriction to $0_{S'}\times_{S'} X'$.

Let $\Lcl$ denote the pull-back of $L$ to $H\times_{S'} \Xcal\simeq (H\times_{S'}\Sc)\times_{\Sc} \Xcal$. Then $\Lcl$
gives a morphism $H\times_{S'}\Sc\to \SPic^0(\Xcal/\Sc)=A_{\Sc}$. Passing to bosonic truncations we get a morphism
$$H=(H\times_{S'}\Sc)/\Ga\to A_{\Sc}/\Ga=A.$$
Furthermore, locally the translation by a section of $H_{\Z}$ does not change the elements $\exp(2\pi s_i)$, hence, 
our morphism factors through a morphism $H/H_{\Z}\to A$.
As we have seen above, this morphism sends zero to zero, hence it is a homomorphism of abelian schemes over $S'$.
We know that it becomes an isomorphism if we replace $S'$ be the reduced subscheme. 
By Lemma \ref{analytic-spaces-lem}, we conclude that it is an isomorphism.

We claim that the morphism $A_{\tau}\to A$ constructed above does not depend on a section $\si$.
Indeed, suppose $\si':\Sc\to \Xcal$ is another section. The corresponding line bundles $L_i$ and $L'_i$ are isomorphism,
hence, the induced morphisms $H_{U_i}\to A$ are the same, which implies our claim.

Since the family $\Xcal\to \Sc$ admits a section \'etale locally, we can glue the morphism $H/H_{\Z}\to A$ out of local morphisms 
constructed as  above.

The fact that our isomorphism $A_{\tau}\simeq A$ agrees with polarizations follows from the fact
that the polarizations agree over every $s\in \Sc$, using the rigidity lemma (see \cite[Prop.\ 6.1]{MF-GIT}).
\end{proof} 
 
\appendix
\section{Miscellaneous results}

\subsection{Associated primes of a module}

The definition of associated prime ideal of a module over an arbitrary ring as well as the primary decomposition of an ideal are defined and studied in many places; see for instance \cite{Lam99}. 
For the case of a noetherian superring $\As$, if we restrict ourselves to $\Z_2$-graded ideals and modules, the theory is very much like the usual one in commutative algebra.

If $M$ is a ($\Z_2$-graded) $\As$-module, the associated ($\Z_2$-graded) primes of $\M$ are the prime ideals that occur among the annihilators of the elements of $M$ (see \cite[Coro.~4.3.6]{We09}). Since $\As$ is noetherian, each finitely generated module has associated primes and there are only finitely many of them. By the very definition one has  {\cite[Thm.~6.4.3]{We09}}:
\begin{prop} \label{prop:assprimes}
\begin{enumerate}
\item
A prime ideal $\pf$ of $\As$ is an associated prime to $M$ if and only if there is  {either an injective morphism of $\As$-modules $\As/\pf \hookrightarrow M$ or an injective morphism of $\As$-modules $(\As/\pf)^\Pi \hookrightarrow M$}.
\item If $M$ is finitely generated, there is a filtration
$$ 0=M_0\subset M_1\subset\dots\subset M_s=M
$$
such that  {either $M_j/M_{j-1}\simeq \As/\pf_j$ or $M_j/M_{j-1}\simeq (\As/\pf_j)^\Pi$,} where $\pf_j$ is a prime ideal for every index $j=1,\dots,s$.
\item If $M$ is finitely generated and $f\in\As$ is a homogeneous element that does not belong to any of the associated primes of $M$, then
$$
\Tor^\As_1(A/(f),M)=0\,,
$$
so that there is an exact sequence 
$$
0 \to f\cdot\As \otimes_\As M \to M \to M/f\cdot M \to 0\,.
$$
\end{enumerate}
\qed
\end{prop}

\begin{remark}\label{rem:assoprim} If  $\Xcal$ is a noetherian superscheme, the above notions can be globalized for any coherent sheaf $\M$ on $\Xcal$. There is a particular situation that is used in the extension to the ``super'' setting of Castelnuovo-Mumford properties of $m$-regularity (Subsection \ref{ss:CMregularity}).
Let $$\Bs=\Bs_k(m,n)=k[x_0,\dots,x_m,\theta_1,\dots,\theta_n]$$ be a  free polynomial $k$-algebra and $\Xcal=\Ps_k^{m,n}=\SProj \Bs$ the projective superspace over $k$.
Given a $\Z$-homogeneous polynomial $f\in k[x_0,\dots,x_m]$ of $\Z$-degree 1, we consider $f$ as an even element of $\Bs$.  If $\Bs'=\Bs/(f)$ and $\Xcal'=\SProj \Bs'\simeq \Ps_k^{m-1,m}$, there is a closed immersion
$$
\Ps_k^{m-1,m}\simeq\Xcal' \hookrightarrow \Xcal=\Ps_k^{m,n}\,,
$$
which identifies $\Xcal'$ with the closed super hyperplane defined by the ideal $\Jc\simeq \widetilde{f\cdot\Bs}^h\simeq \Oc_\Xcal(-1)$ generated by $f$.

Let $\M$ be a coherent sheaf on $\Xcal$. If $\Xcal'$ does not contain any of the points of $\Xcal$ corresponding to the associated primes of $\M$, then the 1-Tor sheaf vanishes:
\begin{equation}\label{eq:torsh}
\Torsh^{\Oc_\Xcal}_1(\Oc_{\Xcal'},\M)=0\,,
\end{equation}
and one has an exact sequence
\begin{equation}
0 \to \M(-1)\simeq \Jc \otimes_{\Oc_\Xcal} \M \to \M \to \M_{\vert\Xcal'}\to 0\,.
\end{equation}
\end{remark}

\subsection{Filtrations of a graded module and super Artin-Rees theorem}

Let $\As$ be a superring, $I$ a $\Z_2$-graded ideal and $M$ a $\Z_2$-graded $\As$-module.
\begin{defin}
A filtration
$$
M=M_0\supseteq M_1\supseteq \dots\supseteq M_n\supseteq \dots
$$
is an \emph{$I$-filtration} if $I\cdot M_n\subseteq M_{n+1}$ for every $n$, and   it is a \emph{stable filtration} if there exists an integer $s\ge 0$ such that 
$I\cdot M_n=M_{n+1}$ for every $n\ge s$.
\end{defin}
Let us consider the bigraded $\As$-algebra
$$
S_I(\As)=\oplus_{n\ge0} \lambda^n I^n\,,
$$
where $\lambda$ is an even indeterminate of $\Z$-degree 1 which we use to   keep track of the $\Z$-degree.  For any $I$-filtration $\{M_n\}$ of an $\As$-module $M$ we have a bigraded $S_I(\As)$-module
$$
S(M)=\oplus\lambda^n M_n\,.
$$
\begin{lemma}\label{lem:AR}  Assume that $\As$ is noetherian.
\begin{enumerate}
\item $S_I(\As)$ is an $\As$-algebra of finite type, so that it is noetherian.
\item If $M$ is a finitely generated $\As$-module, then a $I$-filtration $\{M_n\}$ of $M$ is stable if and only if the associated $S_I(\As)$-module $S(M)$ is finitely  generated.
\end{enumerate}
\end{lemma}
 
\begin{proof}
 {(1). If $I=(a_1,\dots,a_s)$, then $S_I(\As)=\As[a_1\lambda,\dots,a_s
\lambda]$.}

 (2). Every $M_s$ is finitely generated over $\As$, and  the same happens for $N_n=\oplus_{s=0}^n M_s$, so that it generates a submodule of 
$S(M)$ which is finitely generated over  $S_I(\As)$. This submodule is given by
$$
P_n=M_0\oplus\dots\oplus \lambda^n M_n\oplus\lambda^{n+1}I
M_n\oplus\dots\oplus\lambda^{n+s}I^s M_n\oplus\dots\,.
$$
One then has an ascendent chain
$$
P_0\subseteq\dots\subseteq P_n\subseteq\dots
$$
whose union is $S(M)$. Since $S_I(\As)$ is noetherian, 
$S(M)$ is finitely generated if and only there is $n_0$ such that $P_{n_0}=S(M)$. This is equivalent to 
$I^n M_{n_o}=M_{n+n_0}$ for every $n\in\mathbb N$, that is, to the stability of the filtration.
\end{proof}
 We can now prove the super version of the Artin-Rees Lemma (see \cite[Lemma 7.8]{MoZh19}. Our proof is based on \cite[Prop.~10.9]{AtMcD69}.

\begin{prop}[Artin-Rees]\label{prop:AR} Let $\As$ be a noetherian superring, $I$ an ideal in it,
$M$ a finitely generated $\As$-module, and $\{M_n\}$ a stable $I$-filtration of $M$. If $M'\hookrightarrow M$ is a submodule, the induced filtration $\{M'\cap M_n\}$ is stable.
\end{prop}
\begin{proof} Since $$I(M'\cap M_n)\subseteq It M'\cap I  M_n\subseteq
M'\cap M_{n+1},$$  $\{M'\cap M_n\}$ is an  
$I$-filtration. The associated bigraded $S_I(\As)$-module is a submodule of $S(M)$. We now apply  Lemma \ref{lem:AR} repeatedly:  $S(M)$ is finitely generated, and then$S(M')$ is finitely generated as well, since $S_I(\As)$ is noetherian, so that the $I$-filtration
$\{M'\cap M_n\}$ is stable.
\end{proof}
\begin{corol}[Artin-Rees Lemma]\label{cor:AR} There exists an integer  $s\ge 0$ such that
$$
(I^n M)\cap M'=I^{n-s}((I^s M)\cap M')
$$
for every $n\ge s$.
\qed\end{corol}

\subsection{Local criterion for flatness}

We now review the extension to superrings of the local criterion for flatness, which for the ordinary case
can be found for instance  in \cite{AlKl70,Mat89}.
We only consider a simpler version which   suffices for  these notes (see \cite[Lemma 7.7]{MoZh19}).
\begin{lemma}[Local criterion for flatness]\label{lem:localflat}
Let $\phi\colon\As\to\Bs$ be a morphism of finite type of local noetherian superrings and $I$ an ideal of $\As$. Then $\phi$ is flat if and only if the induced morphism $\bar\phi\colon \As/I\to\Bs/I\Bs$ is flat and $I\otimes_\As\Bs\iso I \Bs$.
\end{lemma}
\begin{proof} The proof of \cite[Thm. 22.3]{Mat89} applies straitghforwardly to our situation as $I\Bs$ is contained in the Jacobson radical of $\Bs$, since $\Bs$ is local, and we can use the super version of the Artin-Rees lemma (Corollary \ref{cor:AR}).
\end{proof}
\begin{corol}\label{cor:localflat} Let $\psi\colon \As\to\Bs$, $\phi\colon\Bs\to\Bs'$ be morphisms of finite type of local noetherian superrings and write $\psi'=\phi\circ\psi\colon \As\to\Bs'$ for the composition. Assume that $\psi$ and $\psi'$ are flat. Then $\phi$ is flat if and only if the induced morphism $\bar\phi\colon \Bs/\mf_\As\Bs\to\Bs'/\mf_\As\Bs'$ is flat, where $\mf_\As$ is the maximal ideal of $\As$.
\end{corol}
\begin{proof} One has $\mf_\As\otimes_\As\Bs\simeq \mf_\As \Bs$ and $\mf_\As\otimes_\As\Bs'\simeq \mf_\As \Bs'$ since $\psi$ and $\psi'$ are flat. Then, $(\mf_\As \Bs)\otimes_\Bs\Bs'\simeq  (\mf_\As\Bs) \Bs'$ and one concludes by Lemma \ref{lem:localflat}.
\end{proof}
 
\subsection{Nakayama's lemma for half  exact functors}\label{ss:nakayama}
We prove here a super extension of   Nakayama's lemma for half  exact functors given in \cite{OgBerg72}.
Let $\As$ be a local noetherian superring, $\mf$ its maximal ideal and $\kappa=\As/\mf$. 
Let $\Mf(\As)$ be the category of $\Z_2$-graded finitely generated $\As$-modules and $T$ a \emph{half  exact} functor from $\Mf(\As)$ to itself which is \emph{linear}, that is, such that for every pair of $\Z_2$-graded finitely generated $\As$-modules $M$, $N$ the natural map
$$
\Hom_\As(M,N)\to \Hom_\As(T(M),T(N))
$$
is a morphism of $\Z_2$-graded modules.

\begin{lemma}\label{lem:nakayama}
 If $T(\kappa)=0$, then $T=0$.
\end{lemma}
\begin{proof} We have to prove that $T(M)=0$ for every $\Z_2$-graded finitely generated $\As$-module $M$. By Proposition \ref{prop:assprimes} $M$ has a filtration whose successive quotients are of the form $\As/\pf$  for certain prime ideals $\pf$. Then, it is enough to prove that $T(\As/\pf)=0$ for every ($\Z_2$-graded) prime ideal $\pf$. We proceed by induction on the dimension of the ordinary ring $\As/\pf$. If $\dim \As/\pf=0$, then $\pf=\mf$ and this is our assumption. If $\pf\neq \mf$, take an even element $a\in \mf-\pf$. For every prime ideal $\qf$ containing $\pf+(a)$ one has $\dim \As/\qf < \dim \As/\pf$, and then $T(\As/\qf)=0$ by induction. Applying this to the successive quotients of  the  above mentioned filtration in the case $M=\As/\pf+(a)$ we obtain   $T(\As/\pf+(a))=0$.

Now, from the exact sequence
$$
0\to \As/\pf \xrightarrow{a} \As/\pf \to \As/\pf+(a) \to 0\,,
$$
we get $a\cdot T(\As/\pf)=0$, so that $\As/\pf=0$ by the super Nakayama Lemma  (\cite{BBH91,CarCaFi11} or \cite[6.4.5]{We09}).
\end{proof}

The natural morphism $M=\Hom_\As(\As,M) \to \Hom_\As(T(\As),T(M))$,
defined for every $\Z_2$-graded finitely generated $\As$-module $M$, induces an morphism of functors
$T(\As)\otimes_\As\to T$.

\begin{prop}\label{prop:nakayama} Under the above hypothesis, the following conditions are equivalent:
\begin{enumerate}
\item the above morphism of functors is an isomorphism,
$$
T(\As)\otimes_\As\iso T\,;
$$
\item the morphism $T(\As)\to T(\As/\mf)$ is surjective;
\item $T$ is right  exact.
\end{enumerate}
\end{prop}
\begin{proof} It is clear that (1) implies (2). 

(2)$\implies$(3).  Let $Q$ be the half  exact functor given by $Q(M)=T(M)/\Ima((T(\As)\otimes_\As M))$. Since $Q(\As/\mf)=0$, one has $Q=0$ by Lemma \ref{lem:nakayama} and then, $T(\As)\otimes_\As M \to T(M)$ is a surjection for every $M$. It follows that $T$ is right  exact.

(3)$\implies$(1). Since $\M$ is finitely generated and $\As$ is noetherian, there is an exact sequence
$$
\As^{r,s}\to\As^{p,q} \to M \to 0\,.
$$
Moreover, $T(\As^{r,s})\iso T(\As)\otimes_\As \As^{r,s}$ and $T(\As^{p,q})\iso T(\As)\otimes_\As \As^{p,q}$ as $T$ is half  exact, and   there is an exact sequence
$$
T(\As)\otimes_\As \As^{r,s} \to T(\As)\otimes_\As \As^{p,q} \to T(M) \to 0
$$
as $T$ is right  exact. This proves that $T(M)\iso T(\As)\otimes_\As M$.
\end{proof}

\subsection{Separated and proper  morphisms of superschemes}\label{ss:propermor}

For convenience we   mention a few of types of morphisms of superschemes.
\begin{defin}\label{def:proper}
A morphism $f\colon \Xcal\to\Sc$ of superschemes is:
\begin{enumerate}
\item affine, if for every affine open sub-superscheme $\Ucal\subset \Sc$ the inverse image $f^{-1}(\Ucal)$ is affine;
\item finite, if it is affine, and for $\Ucal = \SSpec \As\subset \Sc$, then $f^{-1}(\Ucal)=\SSpec \Bs$, where $\Bs$ is a finitely generated graded $\As$-module;
\item locally of finite type, if $\Sc$ has an affine open cover $\{\Ucal_i=\SSpec \As_i\}$ such that every inverse image
 $f^{-1}(\Ucal_i)$ has an open affine cover $\mathfrak V_i =\{\mathcal V_{ij} = \SSpec \Bs_{ij}\}$, where
 each $\Bs_{ij}$ is a finitely generated graded $\As_i$-algebra.
 \item $f$ is of finite type if in addition each $\mathfrak V_i$ can be taken to be finite. In other words, if it is locally of finite type and quasi-compact.
\item separated, if the diagonal morphism $\delta_f\colon \Xcal\to \Xcal\times_\Sc \Xcal$ is a closed immersion (actually it is enough to ask that it is a closed morphism);
\item proper, if it is separated, of finite type and universally closed;
\item flat, if for every point $x\in X$, $\Oc_{\Xcal,x}$ is flat over $\Oc_{\Sc,f(x)}$;
\item faithfully flat, if it is flat and surjective.
\end{enumerate}
\end{defin}

The following proposition is standard.

\begin{prop} The properties of being flat and faithfully flat are stable under  base change and composition.
\qed
\end{prop}

The separateness and properness of a morphism depend only on the associated morphism between the underlying ordinary schemes.

\begin{prop}\label{prop:proper}
A morphism $f\colon \Xcal\to\Sc$ of superschemes which is  {locally of finite type}  is proper (resp.~separated) if and only if the induced scheme morphism $f_{bos}\colon X \to S$ is proper (resp.~separated).
\end{prop}
\begin{proof} Since the bosonic reduction of the diagonal morphism of $f$ is the diagonal morphism of $f_{bos}$, $\delta_f$ is closed if and only if $\delta_{f_{bos}}$ is closed. This proves the separatedness. For the properness we have to see first that $f$ is of finite type if and only if $f_{bos}$ is of finite type, which is true because we are assuming that $f$ is locally of finite type, and secondly that $f$ is universally closed if and only $f_{bos}$ is universally closed, which is true as this is a topological question.
\end{proof}

Thus, one has:
\begin{corol}\label{cor:valuative} The valuative criteria for separatedness  and properness \cite[Thm.~II.4.3 and II.4.7]{Hart77} are still valid in the graded setting.
\qed
\end{corol}

We can then extend to superschemes many of the properties of proper and separated morphisms of schemes.

\begin{prop}\label{prop:sorite}
{\ }
\begin{enumerate}
\item The properties of being separated and proper are stable under  base change.
\item Every morphism of affine superschemes is separated.
\item Open immersions are separated and closed immersions are proper.
\item The composition of two separated (resp.~ proper) morphisms of superschemes is separated (resp.~ proper).
\item If $g\colon \Ycal \to \Xcal$ and $f\colon \Xcal \to \Sc$ are morphisms of superschemes and $f\circ g$ is proper and {$f$ is separated, then $g$ is proper. If  $f\circ g$ is separated, then $g$ is separated.} 
\end{enumerate}
\qed
\end{prop} 

\subsection{Smooth morphisms}\label{s:smoothmorphisms}

A definition of smooth morphism of superschemes was given in \cite[2.16]{BrHR19} for superschemes that are locally of finite type over an algebraically closed field.
 We would like to remove the requirement that the base field is algebraically closed. Let us review the standard definitions in the case of ordinary schemes.
 A scheme $X$ over  $k$ is smooth if it is locally of finite type and geometrically regular, that is, if for every extension $k\hookrightarrow K$ where $K$ is an algebraically closed field, the local rings of the scheme $X\times_{\Spec k}\Spec K$ are regular.
A scheme morphism $f\colon X\to S$ is smooth when it is locally of finite presentation (or simply locally of finite type if the schemes are locally noetherian), flat, and its fibres $X_s$ are smooth over the residue field $\kappa(s)$ for every point $s\in S$. By \cite[17.15.5]{EGAIV-IV}, a flat morphism  $f\colon X\to S$ locally of finite presentation is smooth if and only if the sheaf of relative differentials $\Omega_f$ is a locally free $\Oc_X$-module of rank equal to the relative dimension $\dim f$.
This suggests the following definition.

\begin{defin}\label{def:smooth}  A morphism $f\colon \Xcal\to\Sc$ of superschemes of relative dimension $(m,n)$ is smooth if:
\begin{enumerate}
\item $f$ is locally of finite presentation (if the superschemes are locally noetherian, it is enough to ask that $f$ is locally of finite type);
\item $f$ is flat;
\item the sheaf of relative differentials $\Omega_{\Xcal/\Sc}$ is locally free of rank $(m,n)$.
\end{enumerate}
A morphism $f\colon \Xcal\to\Sc$ of superschemes is \'etale if it is smooth of relative dimension $(0,0)$.
\end{defin}

When $\Sc=\Spec k$ is a single point one obtains the definitions of smooth or \'etale superscheme over a field.

One has the following criterion for smoothness over a field, which extends Proposition 2.14 in \cite{BrHR19}. 
\begin{prop}\label{prop:smoothsplit}
 A superscheme $\Xcal$  of dimension $(m,n)$,
locally of finite type over a field $k$,   is smooth if and only if
\begin{enumerate}
\item $X$ is a smooth scheme over $k$ of dimension $m$;
\item  the $\Oc_X$-module $\Ec=\Jc/\Jc^2$ is locally free of rank $n$ and the natural map ${\bigwedge}_{\Oc_X} \Ec\to Gr_J(\Oc_\Xcal)$ is an isomorphism.
\end{enumerate}
If $\Xcal$ is smooth of dimension $(m,n)$ then it is locally split,
and for every closed point $x\in X$ there exist \emph{graded local coordinates}, that is, $m$ even functions $(z_1,\dots,z_m)$ which generate the maximal ideal $\mathfrak m_x$ of $\Oc_{X,x}$ and $n$  odd functions $(\theta_1,\dots,\theta_n)$ generating $\Ec_x$, such that $(dz_1,\dots,dz_m,d\theta_1,\dots,d\theta_n)$ is a   basis for $\Omega_{\Xcal,x}$.
\end{prop}
 
\begin{proof} 
Assume that the two conditions hold. Since the question is local, we can assume that $\Xcal$ is affine and $\Ec$ is a free $\Oc_X$-module with
 generators $\bar\theta_1,\ldots,\bar\theta_n$. Let us choose some lifts $\theta_i\in \Jc$ of $\bar\theta_i$ for $i=1,\ldots,n$. 
 Since $X$ is smooth, there exists a splitting $\Oc_X\to \Oc_{\Xcal,+}$ of the nilpotent extension $\Oc_{\Xcal,+}\to \Oc_X$. Thus, we get a homomorphism
$$\Oc_X[\theta_1,\dots,\theta_n]\to \Oc_{\Xcal}.$$ 
Since it induces an isomorphism of the associative graded algebras, it is an isomorphism.
Hence, $\Xcal$ is locally split and we have $\Omega_{\Xcal/X}\simeq \Oc_X d\theta_1\oplus\ldots\oplus\Oc_X d\theta_n$, and the exact sequence of $\Oc_{\Xcal}$-modules
$$
\Omega_X\otimes_{\Oc_X}\Oc_{\Xcal}\to \Omega_{\Xcal}\to \Omega_{\Xcal/X}\to 0
$$
splits:
$$
\Omega_{\Xcal}\simeq (\Omega_X\otimes_{\Oc_X}\Oc_{\Xcal}) \oplus \Oc_X d\theta_1\oplus\ldots\oplus\Oc_X d\theta_n\,.
$$ 
Since $X$ is smooth, $\Omega_{X,x}$ is free of rank $m$ for every point $x$, and then $\Omega_{\Xcal}$ is locally free, so that $\Xcal$ is smooth.

For the converse, the question is again local, so we can also assume that $X$ is the spectrum of a local ring, which we still denote by $\Oc_X$. The exact sequence induced by $i\colon X\hookrightarrow \Xcal$ gives an exact sequence
$$
\Ec=\Jc/\Jc^2\xrightarrow{\delta} \Omega_{\Xcal}\otimes_{\Oc_{\Xcal}}\Oc_X \to \Omega_X \to 0
$$
where $\delta(\theta)$ is the class of $d\theta$ modulo $\Jc$. Taking  even and odd parts we get an isomorphism
$$\Omega_X  \simeq (\Omega_{\Xcal}\otimes_{\Oc_{\Xcal}}\Oc_X)_+ \,$$
and a surjection
\begin{equation}\label{eq:Ec-odd-Omega-surj}
\Ec\twoheadrightarrow (\Omega_{\Xcal}\otimes_{\Oc_{\Xcal}}\Oc_X)_- .
\end{equation}
Since $(\Omega_{\Xcal}\otimes_{\Oc_{\Xcal}}\Oc_X)_+$ is free of rank $m$ and $X$ has dimension $m$, we get that $X$ is smooth by \cite[17.15.5]{EGAIV-IV}.
On the other hand, by definition of odd dimension, there exists a surjection $\Oc_X^n\to \Ec$. Since $(\Omega_{\Xcal}\otimes_{\Oc_{\Xcal}}\Oc_X)_-$ is free of rank $n$,
this implies that \eqref{eq:Ec-odd-Omega-surj} is an isomorphism, so $\Ec$ is also free $\Oc_X$-modules of rank $n$. 

Since $X$ is smooth, as before, we can choose a splitting $\Oc_X\to \Oc_{\Xcal,+}$. Let also $(\theta_1,\dots,\theta_n)$ be sections of $\Jc$ projecting to a basis of $\Ec$,
and let $(z_1,\dots,z_m)$ be a set of minimal generators of the maximal ideal $\mathfrak m$ of $\Oc_X$. Then 
$(dz_1,\dots,dz_m,d\theta_1,\dots,d\theta_n)$ is a basis of $\Omega_{\Xcal}$.
It remains to prove that the natural epimorphism  
$$\rho:\Oc_X[t_1,\ldots,t_n]\to \Oc_{\Xcal}: t_i\mapsto\theta_i,$$ 
where $(t_i)$ are formal odd variables, is an isomorphism. We proceed by induction on $n$, the case $n=0$ being trivial. Let $\Sc$ be the closed sub-superscheme of $\Xcal$
defined by the ideal $(\theta_n)$. Then $\Omega_{\Sc}$ is a free $\Oc_{\Sc}$-module with the basis $(dz_1,\dots,dz_m,d\theta_1,\dots,d\theta_{n-1})$,
and by induction assumption we have an isomorphism
\begin{equation}\label{eq:Oc-X-t-Oc-mod-th-n}
\Oc_X[t_1,\ldots,t_{n-1}]\to \Oc_{\Sc}=\Oc_\Xcal/(\theta_n).
\end{equation}
Now, if an element $f\in \Oc_X[t_1,\ldots,t_n]$ is in the kernel of $\rho$, reducing modulo $\theta_n$ and applying the above isomorphism, we see that 
$$f=\sum a_{j_1\dots j_{i-1}}t_{j_1}\ldots t_{j_{i-1}} \cdot t_n,$$ 
where the sum runs over  $1\leq j_1<\dots<j_{i-1}\leq n-1$ and the coefficients are in $\Oc_X$. Let us take differentials in the identity
$$0=\rho(f)=\sum a_{j_1\dots,j_{i-1}}\theta_{j_1}\cdot\dots\cdot\theta_{j_{i-1}} \cdot\theta_n.$$ 
Since $(dz_1,\dots,dz_m,d\theta_1,\dots,d\theta_n)$ is a  basis of $\Omega_{\Xcal}$, looking at the coefficient of $d\theta_n$, we get 
$$0=\sum a_{j_1\dots,j_{i-1}}\theta_{j_1}\cdot\dots\cdot\theta_{j_{i-1}}.$$ 
Hence, using isomorphism \eqref{eq:Oc-X-t-Oc-mod-th-n} we get that all the coefficients should be zero, so $f=0$.
\end{proof}

\subsubsection{Formally smooth morphisms}
As in the commutative case, there are other possible approaches to defining smoothness of morphisms. We will prove the equivalence between smoothness and formal smoothness plus local finite type in the locally noetherian case.

\begin{defin}\label{def:formallysmooth}
A morphism  of superschemes $f\colon \Xcal\to \Sc$
is formally smooth if for every affine $\Sc$-superscheme $\SSpec(\Bs)$ and every nilpotent ideal $\Nc\subset \Bs$,
any $\Sc$-morphism $\SSpec(\Bs/\Nc)\to \Xcal$ extends to an $\Sc$-morphism $\SSpec(\Bs)\to \Xcal$. 
\end{defin}

\begin{example}\label{ex:smoothpoly} For any superring $\As$ the polynomial superalgebra $\Bs(m,n)$ defined in Equation \eqref{eq:polyalgebra} is formally smooth over $\As$, that is, $\SSpec\Bs\to\SSpec\As$ is formally smooth.
\end{example}

\begin{prop}\label{prop:loc-free-smooth} Let $f\colon \Xcal\to \Sc$ be a formally smooth morphism locally of finite type of relative dimension $(m,n)$ of locally noetherian superschemes.
\begin{enumerate}
\item $\Omega_{\Xcal/\Sc}$ is locally free of rank $(m,n)$.
\item Let $\Xcal\hookrightarrow \Ycal$ be a closed immersion of $\Sc$-superschemes locally of finite type, there is a natural exact sequence
of sheaves of $\Oc_\Xcal$-modules
$$
0\to \Ic/\Ic^2\to (\Omega_{\Ycal/\Sc})_{\vert \Xcal} \to \Omega_{\Xcal/\Sc}\to 0\,,
$$
where $\Ic$ is the ideal sheaf of $\Xcal$ in $\Ycal$.
\item
Assume that $\Xcal$ and $\Sc$ are affine, $\Xcal=\SSpec \Bs$, $\Sc=\SSpec\As$. Then there exists an affine open covering $\SSpec \Bs=\cup_i \,\SSpec \Bs_i$, where
each $\SSpec \Bs_i\to \SSpec \As$ is a \emph{standard smooth morphism}, i.e.,
$$
\Bs_i=\As[x_1,\dots,x_p,\theta_1,\dots,\theta_q]/(f_1,\dots,f_c,\phi_1,\dots,\phi_d)\,,
$$
where $x_i$ and $f_j$ are even, while $\theta_r$ and $\phi_s$ are odd,
such that the matrices $(\partial f_i/\partial x_j)_{i,j\le c}$ and $(\partial \phi_i/\partial \theta_j)_{i,j\le d}$
are invertible in $\As_i$.
\end{enumerate}
\end{prop}
\begin{proof} 
We prove (1) and (2) together. The statements are local, so we can assume $\Xcal=\Spec(\Bs)$, $\Ycal=\Spec(\As)$, where $\Bs$ is a finitely generated $\As$-algebra.
Let $\Bs'$ be another superalgebra over $\As$ such that $\Bs\simeq \Bs'/I$.
By smoothness of $\SSpec \Bs$ there exists a homomorphism $\Bs\to \Bs'/I^2$ lifting the isomorphism $\Bs\simeq \Bs'/I$.
Hence, we have a splitting $\Bs'/I^2=\Bs\oplus I/I^2$. Let $x\mapsto x_1$ denote the corresponding projection
$$
\Bs'\to \Bs/I^2\to I/I^2\,.
$$ 
It is easy to check that it is derivation. Hence it induces a well defined splitting 
$$
\Omega_{\Bs/\As}\otimes_{\Bs'} \Bs\to I/I^2\,: \quad dx\mapsto x_1
$$
of the standard exact sequence 
$$
I/I^2\to \Omega_{\Bs'/\As}\otimes_{\Bs'} \Bs\to \Omega_{\Bs/\As}\to 0\,.
$$
Applying this in the case when $\Bs'$ is a free polynomial superalgebra $\Bs(m',n')$ over $\As$ (Equation \eqref{eq:polyalgebra}), we deduce
that $\Omega_{\Bs/\As}$ is a direct summand of the free $\Bs$-module $\Omega_{\Bs'/\As}\otimes_{\Bs'} \Bs$. Therefore, $\Omega_{\Bs/\As}$ is locally free.

(3) is proved similarly to \cite[Lemma 10.136.10]{Stacks}.
\end{proof}

\begin{remark}
One can check as in the even case
that formal smoothness for finitely generated superalgebras $\Bs/\As$ is equivalent to smoothness understood in
terms of the naive cotangent complex $NL_{\Bs/\As}:=\tau_{\leq 1} L_{\Bs/\As}$, i.e., $H_1(L_{\Bs/\As})=0$ and $H_0(L_{\Bs/\As})$ is
a finitely generated projective module (see \cite[Prop.\ 10.137.8]{Stacks}). 
\end{remark}

\begin{example}\label{ex:smooth-even}
If $k$ is a field, and $(\Bs,\mf)$ is a local Noetherian superalgebra over $k$ with $\Bs/\mf=k$, obtained as a localization of a finitely generated
$k$-superalgebra,
the formal smoothness of $\Bs$ over $k$ implies that the completion $\hat{\Bs}$ is isomorphic to the formal power ring
over $k$ in even variables $x_1,\dots,x_m$ and odd variables $\theta_1,\dots,\theta_n$ (see also \cite{Fi08}).
In particular, if $\Omega_{\Bs/k}$ has rank $(m,0)$ then $\hat{\Bs}$ is purely even, hence, so is $\Bs$.
\end{example}

We are now going to prove the equivalence between smoothness and formal smoothness for  morphisms locally of finite type. Recall that all the superschemes are supposed to be locally noetherian.

\begin{lemma}\label{lem:rel-dim} 
Let $f\colon\Xcal\to \Sc$ be a smooth morphism of relative dimension $(m,n)$, and let $f$ be a function on $\Xcal$ such that $df$ generates a subbundle
(of rank $(1,0)$ or $(0,1)$) in $\Omega_{\Xcal/\Sc\vert \Ycal}$,\footnote{By a subbundle we mean a locally free subsheaf such that the corresponding quotient is locally free}  where $\Ycal$ is the closed sub-superscheme corresponding to the ideal sheaf $(f)$. Then $\Ycal$ is smooth of dimension $(m-1,n)$ or $(m,n-1)$ for
$f$ even or odd, respectively.
\end{lemma}

\begin{proof} First, let us check that $\Ycal\to \Sc$ is flat. This is a local question, so we can assume $\Ss=\SSpec(\As)$ and $\Xcal=\SSpec(\Bs)$ for some local noetherian
superrings, and that $f$ is in the maximal ideal  $\mf\subset \As$.  We know that $\Bs/\mf\Bs$ admits generators in the maximal ideal $\bar{z}_1,\dots,\bar{z}_m,\bar{\theta}_1,\dots,\bar{\theta}_n$ 
(with $\bar{z}_i$ even and $\bar{\theta}_j$ odd) such that 
 the image $\bar{f}$ of $f$ in $\Bs/ \mf\cdot \Bs$ is one of them (see Proposition \ref{prop:smoothsplit}), and the map 
$$
\As/\mf[t_1,\dots,t_m,\psi_1,\dots,\psi_n]\to \Bs/\mf \Bs\,, \quad t_i\mapsto \bar{z}_i\,, \quad \psi_j\mapsto \bar{\theta}_j
$$
from the free polynomial superring is \'etale.

Let us lift $(\bar{z}_i)$ and $(\bar{\theta}_j)$ to some elements $(z_i)$ and $(\theta_j)$ in the maximal ideal of $\Bs$. We claim that the morphism
$$
\As[t_1,\dots,t_m,\psi_1,\dots,\psi_n]\to \Bs\,:\quad  t_i\mapsto z_i\,, \quad \psi_j\mapsto \theta_j
$$
is \'etale. Indeed, it is clear that it is unramified (i.e., relative differentials vanish). Also, it is flat, as follows immediately from Corollary \ref{cor:localflat}.
Now let us consider $\Bs/(f)$. Note that $f$  is one of the coordinates $(t_1,\dots,t_m,\psi_1,\ldots,\psi_n)$.
Hence, the induced map
$$
\As[t_1,\dots,t_m,\psi_1,\ldots,\psi_n]/(\widetilde{f})\to \Bs/(f)
$$
is again flat. But $\As[t_1,\dots,t_m,\psi_1,\ldots,\psi_n]/(\widetilde{f})$ is still a polynomial superring, so it is flat over $\As$. Thus, $\Bs/(f)$ is flat over $\As$.

The exact sequence
$$
0\to \Ic_\Ycal/\Ic_\Ycal^2\to (\Omega_{\Xcal/\Sc})_{\vert \Ycal}\to \Omega_{\Ycal/\Sc}\to 0
$$
shows that $\Omega_{\Ycal/\Sc}$ is locally free of rank $(m-1,n)$ or $(m,n-1)$ for $f$ even or odd, respectively.

Thus, it remains to compute the dimension of the fibres of $\Ycal\to \Sc$. To this end we can work over a field $k$
and use the characterization of smooth superschemes $\Xcal$ of dimension $(m,n)$ as locally split superschemes with smooth bosonization of dimension $m$ and the
conormal bundle of $X$ in $\Xcal$ of rank $(0,n)$ (see Proposition \ref{prop:smoothsplit}). Note that $df$ generates a subbundle of $\Omega_{\Xcal/k}$ in an open neighbourhood of $\Ycal$.
If $f$ is even then $Y$ is a smooth divisor given by the function $\rest{f}{X}$ on $X$, while the odd dimension does not change.
If $f$ is odd then $Y=X$ and $f$ can be locally taken as one of the odd coordinates on $\Xcal$, so the odd dimension of $\Ycal$ is $n-1$.
\end{proof}

\begin{remark}
The standard argument for an analogous statement in the even case is based on the fact that an even function $f$ as in Lemma \ref{lem:rel-dim}
would not  be a  zero divisor. This is not true for odd functions, so the argument had to be modified.
\end{remark}

We now obtain the equivalence we were seeking for.
\begin{prop}\label{prop:smoothequiv} Let $f\colon\Xcal\to \Sc$ be a morphism of superschemes, locally of finite type. Then $f$ is formally smooth if and only if it is smooth.
\end{prop}
\begin{proof}
If $f$ is formally smooth then it is locally standard (Proposition \ref{prop:loc-free-smooth} (3)). Hence, we can assume that $\Xcal$ is a sub-superscheme in a relative affine space $P=\As^{p,q}_\Sc\to \Sc$
given by $(f_1,\dots,f_c)$, where $df_1,\dots,df_c$ generate a subbundle in $\Omega_{P/\Sc\vert\Xcal}$. Replacing $P$ by an open neighbourhood of $\Xcal$
we can assume that $df_1,\dots,df_c$ generate a subbundle of rank $(a,b)=(p-m,q-n)$ in $\Omega_{P/\Sc}$, where $a$ (resp., $b$) is the number of even (resp., odd)
functions among $(f_i)$. Now applying Lemma \ref{lem:rel-dim} we derive that $f$ is smooth.

Conversely, starting with a smooth morphism $f\colon\Xcal\to \Sc$, let us embed locally $\Xcal$ into a relative affine space $P$.
Locally we can choose functions $(f_1,\dots,f_c)\in \Ic_\Xcal$ such that the kernel
of $\Omega_{P/\Sc\vert\Xcal}\to \Omega_{\Xcal/\Sc}$ (which is locally free) is generated by $df_1,\dots,df_c$. The ideal $(f_1,\dots,f_c)$ defines a standard formally smooth closed sub-superscheme 
$\Xcal'\hookrightarrow \Ucal\subset P$ such that $\Xcal\hookrightarrow \Xcal'$ (here $\Ucal$ is some open neighbourhood of $\Xcal$ in $P$). 
By the first part of the proof, $\Xcal'$ is smooth over $\Sc$ of the same relative dimension as $\Xcal$.

Now by flatness, it is enough to check that $\Xcal_s=(\Ucal\cap \Xcal')_s$ for every $s\in S$, so we can work over a field.
Since $X$ and $X'$ are smooth of the same dimension, we have $X=X'$.
Furthermore, the conormal sheaves of $X$ in $\Xcal$ and $\Xcal'$ are locally free of the same rank, and one surjects 
onto the other, hence they are equal. Since both $\Xcal$ and $\Xcal'$ are smooth, they are locally split (see Prop.\ \ref{prop:smoothsplit}), 
and it follows that $\Xcal=\Xcal'$, so $\Xcal$ is formally smooth over $\Sc$.
\end{proof}

\subsubsection{Structure of the smooth morphisms of even relative dimension} All the superschemes are again  locally noetherian. We  shall freely use the equivalence between smoothness and formal smoothness for  morphisms locally of finite type (Proposition \ref{prop:smoothequiv}).

\begin{prop}\label{prop:even-smooth} 
Let $f\colon\Xcal\to \Sc$ be a 
smooth morphism of relative dimension $(m,0)$. Then the morphism $f/\Gamma\colon \Xcal/\Gamma\to \Sc/\Gamma$ induced between the bosonic quotients is a 
smooth morphism of relative dimension $(m,0)$, and
$f$ is obtained from $f/\Gamma$ by the base change with respect to $\Sc\to \Sc/\Gamma$. Furthermore,
the base change with respect to $\Sc\to \Sc/\Gamma$ yields an equivalence between the categories of smooth $\Sc/\Gamma$-schemes of relative dimension $m$ and smooth $\Sc$-superschemes of relative dimension $(m,0)$.
\end{prop}
\begin{proof} This is a local question, so we can consider a homomorphism of local superalgebras $f\colon\As\to \Bs$ instead.
We have the induced homomorphisms $f_+\colon\As_+\to \Bs_+$ and 
$$
\alpha\colon \Bs':=\Bs_+\otimes_{\As_+} \As\to \Bs\,.
$$
We would like to show that $\alpha$ is an isomorphism.

Let $\mf_\As\subset \As$ (resp., $\mf_\Bs\subset \Bs$) denote the maximal ideal.
First, we note that $\Bs/\mf_\As \Bs$ is smooth of even dimension over the field $k=\As/\mf_\As$, so it is purely even (see Example \ref{ex:smooth-even}).
Since $\As_-\subset \mf_\As$, this implies that $\Bs_-=\As_-\cdot\Bs_+$. Hence, the homomorphism $\alpha$ is surjective.
Let $I=\ker\alpha$. Then we have an exact sequence
\begin{equation}\label{eq:J-B-A-ex-seq}
0\to I/I^2\to \Bs\otimes_{\Bs'}\Omega_{\Bs'/\As}\to \Omega_{\Bs/\As}\to 0
\end{equation}
(see Proposition \ref{prop:loc-free-smooth} for  the exactness on the left due to smoothness of $\Bs$ over $\As$).
But $\Omega_{\Bs'/\As}\simeq \As\otimes_{\As_+} \Omega_{\Bs_+/\As_+}$. Hence, 
$\Omega_{\Bs'/\As}\otimes_\As k\simeq \Omega_{\Bs_+/\As_+}\otimes_{A_+} k$, so
tensoring  the composed map 
\begin{equation}\label{eq:Om-B-B'-A-comp}
\Omega_{\Bs'/\As}\to \Bs\otimes_{\Bs'}\Omega_{\Bs'/\As}\to \Omega_{\Bs/\As}
\end{equation}
by  $k$ over $\As$ one has 
$$
\As\otimes_{\As_+} \Omega_{\Bs_+/\As_+}\to \Omega_{\Bs/\As}\,,
$$
which is an isomorphism as $\Bs_+\otimes_{\As_+} k\to \Bs\otimes_\As k$.
Since both maps in Equation \eqref{eq:Om-B-B'-A-comp} are surjective,  they become isomorphisms upon tensoring with $k$ over $\As$.
Thus, tensoring the sequence \ref{eq:J-B-A-ex-seq} by $k$ and using the fact that $\Omega_{\Bs/\As}$ is flat over $\As$,
we deduce that $I/I^2\otimes_\As k=0$. Since $\mf_\As\subset \mf_\Bs$, this implies that $I/I^2=0$, so $I=0$. This proves the  claim that $\alpha$ is an isomorphism.

Now let us check that $\Bs_+$ is formally smooth over $\As_+$ (in the category of purely even commutative rings).
Given a square-zero extension $S_+\to R_+$ fitting into a commutative square
$$
\xymatrix{
\As_+\ar[r]\ar[d] & S_+\ar[d]
\\
\Bs_+\ar[r] & R_+
}\,,
$$
let us form the corresponding commutative square
$$
\xymatrix{
\As\ar[r]\ar[d] & S\ar[d]
\\
\Bs\ar[r] & R
}
$$
where $S:=\As\otimes_{\As_+}S_+$, $R=\Bs\otimes_{\Bs_+}R_+$. Note that taking even components we recover the original square.
By the formal smoothness of $\Bs$ over $\As$  there exists a lifting homomorphism $f\colon\Bs\to S$. Now the even component of $f$ is the required
lifting $\Bs_+\to S_+$.

For the last statement, it suffices to check that for any schemes $X_0$ and $Y_0$, smooth over $\Sc/\Gamma$, and for any morphism
$$
f\colon\Sc\times_{\Sc/\Gamma} X_0\to S\times_{\Sc/\Gamma} Y_0\,,
$$
one has $f=\Sc\times_{\Sc/\Gamma} (f/\Gamma)$. Again, this is a local statement, so we can assume that $\Sc=\SSpec(\As)$, $\Sc/\Gamma=\Spec(\As_+)$, and
$X_0$ and $Y_0$ correspond to some $\As_+$-algebras $B$ and $C$.
Then our statement is that any (even) homomorphism of $\Z_2$-graded $\As$-algebras
$$
\phi\colon B\otimes_{\As_+}A\to C\otimes_{\As_+}\As
$$
is determined by the corresponding homomorphism of $\As_+$-algebras
$\phi_+\colon B\to C$,
and this is clearly true.
\end{proof}

\begin{corol}\label{cor:even} 
If $\Xcal$ is smooth of dimension $(m,0)$ over $\Sc$ and $\Sc$ is even, then $\Xcal$ is even.
\end{corol}

\subsection{Faithfully flat descent for superschemes}

Grothendieck's   proof of the faithfully flat descent  can be straightforwardly extended to superschemes, 
indeed the key Lemma 1.4   in \cite[Exp.~ VIII]{SGA1}  is also true for superrings and $\Z_2$-graded modules. In particular, we have an analogue of  Grothendieck's   Theorem 1.1 (faithfully flat descent for quasi-coherent sheaves on superschemes), which has two parts, the analogues of Corollaries 1.2 and 1.3, respectively. We state here the results for reference.

\begin{prop}[Descent for homomorphisms]\label{prop:morphdescent} Let $p\colon \Tc\to \Sc$ be a faithfully flat quasi-compact morphism of superschemes, $\Rcal=\Tc\times_\Sc\Tc$ and $(p_1,p_2)\colon \Rcal\rra\Tc$ the projections. Let $\M$, $\Nc$ be (graded) quasi-coherent sheaves on $\Sc$, $\M'=p^\ast\M$, $\Nc'=p_\ast\Nc$ and $\M''=p_1^\ast\M'=p_2^\ast\M'$, $\Nc''=p_1^\ast\Nc'=p_2^\ast\Nc'$. The sequence 
$$
\Hom_\Sc(\M,\Nc) \to \Hom_\Tc(\M',\Nc') \rra \Hom_\Rcal(\M'',\Nc'')
$$
of (homogeneous) homomorphisms induced by the projections is exact, that is, it establishes a one-to-one correspondence between $\Hom_\Sc(\M,\Nc)$ and the coincidence locus  of the pair of arrows.
\qed \end{prop}

\begin{prop}[Descent for modules]\label{prop:moddescent} 
 Let $p\colon \Tc\to \Sc$ be a faithfully flat quasi-compact morphism of superschemes, and let  $\Rcal=\Tc\times_\Sc\Tc$ and $(p_1,p_2)\colon \Rcal\rra\Tc$ be  the projections. Let $\M$ be a graded quasi-coherent sheaf on $\Tc$ with  descent data, that is, an isomorphism  $\phi\colon p_1^\ast\M\iso p_2^\ast \M$ such that,
if $\U = \Tc\times_\Sc\Tc\times_\Sc\Tc$, and $p_{12}$, $p_{23}$, $p_{13}$ are the three projections onto $\Rcal$, the condition
$p_{23}^\ast \phi \circ p_{12}^\ast \phi = p_{13}^\ast\phi$
holds. Then, there is a graded quasi-coherent sheaf $\Nc$ on $\Sc$ such that $\M\simeq p^\ast\Nc$.
\qed \end{prop}

One also has a  statement corresponding  to Grothendieck's  Corollary 1.9. For every superscheme $\Xcal$ let us denote by $H(\Xcal)$ the set of all   closed sub-superschemes of $\Xcal$.
\begin{prop}[Descent for closed sub-superschemes]\label{prop:subdescent}
 Let $p\colon \Tc\to \Sc$ be a faithfully flat quasi-compact morphism of superschemes, The   sequence of sets
 $$
 H(\Sc)\to H(\Tc) \rra H(\Rcal)
 $$
 is exact.
\qed\end{prop}

\subsection{Local representability of a group functor}

Let $\Sc$ be a superscheme and $\iota\colon\Ss\Gc\hookrightarrow\Ss\F$ an  {\emph{open morphism of \'etale sheaves} on $\Sc$-superschemes}. We also say that $\Ss\Gc\hookrightarrow\Ss\F$ is \emph{representable by open immersions}. This means that for every $\Sc$-superscheme $\Tc\to\Sc$ and every sheaf morphism $\lambda\colon \Tc\to\Ss\F$ (or, equivalently, every $\Tc$-valued point $\lambda$ of $\Ss\F$, or every $\lambda\in\Ss\F(\Tc))$, the fibre product $\Ss\Gc\times_{\Ss\F,\lambda}\Tc$ is representable by a superscheme $\Tc_\lambda$ and the induced morphism $\Tc_{\Ss\Gc,\lambda}\to \Tc$ is an open immersion. We visualize this by the cartesian diagram
$$
\xymatrix{\Ss\Gc \ar@{^(->}[r]^\iota & \Ss\F \\
\Tc_{\Ss\Gc,\lambda}:=\Ss\Gc\times_{\Ss\F,\lambda}\Tc\ar[u]^{\lambda_\Tc}\ar@{^(->}[r]^(.7){\iota_\Tc} & \Tc\ar[u]^\lambda
}
$$

Assume that $\Ss\F$ is a \emph{sheaf in groups}.  For every $\Sc$-valued point $\xi\colon \Tc\to\Ss\F$, one  defines the translated open subfunctor $\iota_\xi\colon\Ss\Gc\cdot\xi\hookrightarrow\Ss\F$ by setting $\Ss\Gc\cdot\xi(\Tc)$ as the image of the composition $\Ss\Gc(\Tc)\xrightarrow{\Id\times\xi_\Tc}\Ss\Gc(\Tc)\times\Ss\F(\Tc) \xrightarrow{\cdot}\Ss\F(\Tc)$.
Notice that $\iota_\xi\colon\Ss\Gc\cdot\xi\hookrightarrow\Ss\F$ is also representable by open immersions, as for every $\Sc$-superscheme $\Tc\to\Sc$ and every sheaf morphism $\lambda\colon \Tc\to\Ss\F$ one has  
$$
\Tc_{\Ss\Gc\cdot\xi,\lambda}\simeq \Tc_{\Ss\Gc,\lambda\xi_\Tc^{-1}}
$$

Proceeding as in \cite[Chap. 0, Prop. 4.5.4]{EGA-I}, one obtains:
\begin{prop}\label{prop:opengoup} Under the above hypotheses, suppose  also that:
\begin{enumerate}
\item there exists a set $I$ of $\Sc$-valued points $\xi_i\in\Ss\F(\Sc)$ of $\Ss\F$ such that the open subfunctor $\bigcup_{i\in I}\Ss\Gc\cdot\xi_i\hookrightarrow\Ss\F$ equals $\F$. This is equivalent to saying that for every $\Sc$-superscheme $\Tc\to\Sc$ and every sheaf morphism $\lambda\colon \Tc\to\Ss\F$ the open sub-superschemes $\Tc_{\Ss\Gc\cdot\xi_i,\lambda}\hookrightarrow \Tc$, for $i\in I$, form an open covering of $\Tc$;
\item $\Ss\Gc$ is representable.
\end{enumerate}
Then $\Ss\F$ is representable as well.
\end{prop} 
\begin{proof} Let us denote by $\Gsc$ the $\Sc$-superscheme that represents $\Ss\Gc$. 
Since the translated functors $\Ss\Gc\cdot\xi_i$ are isomorphic to $\Ss\Gc$, they are representable by $\Sc$-superschemes, which  we denote by $\Gsc\cdot\xi_i$. Moreover, the condition $\bigcup_{i\in I}\Ss\Gc\cdot\xi_i\hookrightarrow\Ss\F$ yields glueing conditions for the superschemes $\Gsc\cdot\xi_i$, which therefore glue to yield an $\Sc$-superscheme $\bigcup_{i\in I)}\Gsc\cdot\xi_i$ whose functor of   points is $\bigcup_{i\in I}\Ss\Gc\cdot\xi_i=\Ss\F$.
\end{proof}


\begin{thebibliography}{10}

\bibitem{AlKl70}
{\sc A.~Altman and S.~Kleiman}, {\em Introduction to {G}rothendieck duality
  theory}, Lecture Notes in Mathematics, Vol. 146, Springer-Verlag, Berlin-New
  York, 1970.

\bibitem{AlKl80}
\leavevmode\vrule height 2pt depth -1.6pt width 23pt, {\em Compactifying the
  {P}icard scheme}, Adv. in Math., 35 (1980), pp.~50--112.

\bibitem{AtMcD69}
{\sc M.~F. Atiyah and I.~G. Macdonald}, {\em Introduction to commutative
  algebra}, Addison-Wesley Publishing Co., Reading, Mass.-London-Don Mills,
  Ont., 1969.

\bibitem{BaSchw87}
{\sc M.~A. Baranov and A.~S. Schwarz}, {\em On the multiloop contribution to
  the string theory}, Internat. J. Modern Phys. A, 2 (1987), pp.~1773--1796.

\bibitem{BBH91}
{\sc C.~Bartocci, U.~Bruzzo, and D.~Hern{\'a}ndez~Ruip{\'e}rez}, {\em The
  geometry of supermanifolds}, vol.~71 of Mathematics and its Applications,
  Kluwer Academic Publishers Group, Dordrecht, 1991.

\bibitem{BL75}
{\sc F.~A. Berezin and D.~A. Le{\u\i}tes}, {\em Supermanifolds}, Dokl. Akad.
  Nauk SSSR, 224 (1975), pp.~505--508.

\bibitem{Bir11}
{\sc C.~Birkar}, {\em Topics in algebraic geometry}.
\newblock {\tt arXiv:1104.5035 [math.AG]}.

\bibitem{BN93}
{\sc M.~B\"{o}kstedt and A.~Neeman}, {\em Homotopy limits in triangulated
  categories}, Compositio Math., 86 (1993), pp.~209--234.

\bibitem{BrHR19}
{\sc U.~Bruzzo and D.~Hern\'{a}ndez~Ruip\'{e}rez}, {\em The supermoduli of
  {SUSY} curves with {R}amond punctures}, Rev. R. Acad. Cienc. Exactas
  F\'{\i}s. Nat. Ser. A Mat. RACSAM, 115 (2021), p.~144.
\newblock also {\tt arXiv:1910.12236 [math.AG]}.

\bibitem{CaNo18}
{\sc S.~L. Cacciatori and S.~Noja}, {\em Projective superspaces in practice},
  J. Geom. Phys., 130 (2018), pp.~40--62.

\bibitem{CarCaFi11}
{\sc C.~Carmeli, L.~Caston, and R.~Fioresi}, {\em Mathematical foundations of
  supersymmetry}, EMS Series of Lectures in Mathematics, European Mathematical
  Society (EMS), Z\"{u}rich, 2011.

\bibitem{Cod14}
{\sc G.~Codogni}, {\em Pluri-canonical models of supersymmetric curves}.
\newblock {\tt arXiv:1402.4999v4 [math.AG]}.

\bibitem{CodViv17}
{\sc G.~Codogni and F.~Viviani}, {\em Moduli and periods of supersymmetric
  curves}, Adv. Theor. Math. Phys., 23 (2019), pp.~345--402.

\bibitem{CraRab88}
{\sc L.~Crane and J.~M. Rabin}, {\em Super {R}iemann surfaces: uniformization
  and {T}eichm\"uller theory}, Comm. Math. Phys., 113 (1988), pp.~601--623.

\bibitem{Del87}
{\sc P.~Deligne}, {\em Letter to {Y}.{I}. {M}anin}.
\newblock
  https://publications.ias.edu/sites/default/files/lettre-a-manin-1987-09-25.pdf,
  1987.

\bibitem{Dem}
{\sc J.-P. Demailly}, {\em Complex analytic and differential geometry}.
\newblock Pdf file available from {\tt https://
  www-fourier.ujf-grenoble.fr/$\sim$demailly/documents.html}.

\bibitem{D'Hoker-Phong}
{\sc E.~D'Hoker and D.~H. Phong}, {\em Conformal scalar fields and chiral
  splitting on super {R}iemann surfaces}, Comm. Math. Phys., 125 (1989),
  pp.~469--513.

\bibitem{Dir19}
{\sc D.~J. Diroff}, {\em On the super {M}umford form in the presence of
  {R}amond and {N}eveu-{S}chwarz punctures}, J. Geom. Phys., 144 (2019),
  pp.~273--293.

\bibitem{DoHeSa93}
{\sc J.~A. Dom{\'{\i}}nguez~P{\'e}rez, D.~Hern{\'a}ndez~Ruip{\'e}rez, and
  C.~Sancho~de Salas}, {\em The variety of positive superdivisors of a
  supercurve (supervortices)}, J. Geom. Phys., 12 (1993), pp.~183 -- 203.

\bibitem{DoHeSa97}
\leavevmode\vrule height 2pt depth -1.6pt width 23pt, {\em Global structures
  for the moduli of (punctured) super {R}iemann surfaces}, J. Geom. Phys., 21
  (1997), pp.~199--217.

\bibitem{DoW13}
{\sc R.~Donagi and E.~Witten}, {\em Super {A}tiyah classes and obstructions to
  splitting of supermoduli space}, Pure Appl. Math. Q., 9 (2013), pp.~739--788.

\bibitem{DoW15}
\leavevmode\vrule height 2pt depth -1.6pt width 23pt, {\em Supermoduli space is
  not projected}, in String-{M}ath 2012, vol.~90 of Proc. Sympos. Pure Math.,
  Amer. Math. Soc., Providence, RI, 2015, pp.~19--71.

\bibitem{CF90}
{\sc G.~Faltings and C.-L. Chai}, {\em Degeneration of abelian varieties},
  vol.~22 of Ergebnisse der Mathematik und ihrer Grenzgebiete (3) [Results in
  Mathematics and Related Areas (3)], Springer-Verlag, Berlin, 1990.
\newblock With an appendix by David Mumford.

\bibitem{FGA05}
{\sc B.~Fantechi, L.~G\"{o}ttsche, L.~Illusie, S.~L. Kleiman, N.~Nitsure, and
  A.~Vistoli}, {\em Fundamental algebraic geometry}, vol.~123 of Mathematical
  Surveys and Monographs, American Mathematical Society, Providence, RI, 2005.

\bibitem{FKP19}
{\sc G.~Felder, D.~Kazhdan, and A.~Polishchuk}, {\em Regularity of the
  superstring supermeasure and the superperiod map},
   Selecta Math. (N.S.), 28 (2022), no. 1, Paper No. 17, 64 pp. 

\bibitem{FKP20}
\leavevmode\vrule height 2pt depth -1.6pt width 23pt, {\em The moduli space of
  stable supercurves and its canonical line bundle}.
\newblock {\tt arXiv:2006.13271v1 [math.AG]}, 2020.

\bibitem{Fi08}
{\sc R.~Fioresi}, {\em Smoothness of algebraic supervarieties and supergroups},
  Pacific J. Math., 234 (2008), pp.~295--310.

\bibitem{Fr86}
{\sc D.~Friedan}, {\em Notes on string theory and two-dimensional conformal
  field theory}, in Workshop on unified string theories ({S}anta {B}arbara,
  {C}alif., 1985), World Sci. Publishing, Singapore, 1986, pp.~162--213.

\bibitem{GhoSheYau}
{\sc A.~Gholampour, A.~Sheshmani, and S.-T. Yau}, {\em Nested {H}ilbert schemes
  on surfaces: virtual fundamental class}, Adv. Math., 365 (2020), p.~107046.

\bibitem{Tohoku}
{\sc A.~Grothendieck}, {\em Sur quelques points d'alg\`ebre homologique},
  T\^ohoku Math. J.,  (1957), pp.~119--221.

\bibitem{FGA}
\leavevmode\vrule height 2pt depth -1.6pt width 23pt, {\em Fondements de la
  g\'eom\'etrie alg\'ebrique. [{E}xtraits du {S}\'eminaire {B}ourbaki,
  1957--1962.]}, Secr\'etariat math\'ematique, Paris, 1962.

\bibitem{EGAIII-II}
\leavevmode\vrule height 2pt depth -1.6pt width 23pt, {\em \'{E}l\'{e}ments de
  g\'{e}om\'{e}trie alg\'{e}brique. {III}. \'{E}tude cohomologique des
  faisceaux coh\'{e}rents. {II}}, Inst. Hautes \'{E}tudes Sci. Publ. Math.,
  (1963), p.~91.

\bibitem{EGAIV-IV}
\leavevmode\vrule height 2pt depth -1.6pt width 23pt, {\em \'{E}l\'{e}ments de
  g\'{e}om\'{e}trie alg\'{e}brique. {IV}. \'{E}tude locale des sch\'{e}mas et
  des morphismes de sch\'{e}mas {IV}}, Inst. Hautes \'{E}tudes Sci. Publ.
  Math.,  (1967), p.~361.

\bibitem{SGA1}
\leavevmode\vrule height 2pt depth -1.6pt width 23pt, {\em Rev\^{e}tements
  \'{e}tales et groupe fondamental}, Lecture Notes in Mathematics, vol. 224,
  Springer-Verlag, Berlin-Heildelberg-New York, 1971.
\newblock Also available as {\tt arXiv:0206203v2 [math.AG]}.

\bibitem{EGA-I}
{\sc A.~Grothendieck and J.~A. Dieudonn\'{e}}, {\em \'{E}l\'{e}ments de
  g\'{e}om\'{e}trie alg\'{e}brique. {I}}, vol.~166 of Grundlehren der
  Mathematischen Wissenschaften, Springer-Verlag, Berlin, 1971.

\bibitem{Hart77}
{\sc R.~Hartshorne}, {\em Algebraic geometry}, Graduate Texts in Mathematics,
  vol. 52, Springer-Verlag, New York, 1977.

\bibitem{HRLMS09}
{\sc D.~Hern\'{a}ndez~Ruip\'{e}rez, A.~C. L\'{o}pez~Mart\'{\i}n, and
  F.~Sancho~de Salas}, {\em Relative integral functors for singular fibrations
  and singular partners}, J. Eur. Math. Soc. (JEMS), 11 (2009), pp.~597--625.

\bibitem{Jang20}
{\sc M.~Y. Jang}, {\em Families of 0-dimensional subspaces on supercurves of
  dimension {$1\vert 1$}}, J. Pure Appl. Algebra, 224 (2020), pp.~106251, 16.

\bibitem{Kle05}
{\sc S.~Kleiman}, {\em The {P}icard scheme}, in Fundamental algebraic geometry,
  vol.~123 of Mathematical Surveys and Monographs, American Mathematical
  Society, Providence, RI, 2005, pp.~x+339.

\bibitem{Knut71}
{\sc D.~Knutson}, {\em Algebraic spaces}, Lecture Notes in Mathematics, Vol.
  203, Springer-Verlag, Berlin, 1971.

\bibitem{Kost75}
{\sc B.~Kostant}, {\em Graded manifolds, graded {L}ie theory, and
  prequantization}, in Differential geometrical methods in mathematical physics
  ({P}roc. {S}ympos., {U}niv. {B}onn, {B}onn, 1975), Springer, Berlin, 1977,
  pp.~177--306. Lecture Notes in Math., Vol. 570.

\bibitem{Lam99}
{\sc T.~Y. Lam}, {\em Lectures on modules and rings}, vol.~189 of Graduate
  Texts in Mathematics, Springer-Verlag, New York, 1999.

\bibitem{LePoWel90}
{\sc C.~LeBrun, Y.~S. Poon, and R.~O. Wells, Jr.}, {\em Projective embeddings
  of complex supermanifolds}, Comm. Math. Phys., 126 (1990), pp.~433--452.

\bibitem{LeRoth88}
{\sc C.~LeBrun and M.~Rothstein}, {\em Moduli of super {R}iemann surfaces},
  Comm. Math. Phys., 117 (1988), pp.~159--176.

\bibitem{Lip09}
{\sc J.~Lipman}, {\em Notes on derived functors and {G}rothendieck duality}, in
  Foundations of {G}rothendieck duality for diagrams of schemes, vol.~1960 of
  Lecture Notes in Math., Springer, Berlin, 2009, pp.~1--259.

\bibitem{Ma86}
{\sc Y.~I. Manin}, {\em Critical dimensions of string theories and a dualizing
  sheaf on the moduli space of (super) curves}, Funktsional. Anal. i
  Prilozhen., 20 (1986), pp.~88--89.

\bibitem{Ma87}
\leavevmode\vrule height 2pt depth -1.6pt width 23pt, {\em Quantum strings and
  algebraic curves}, in Proceedings of the {I}nternational {C}ongress of
  {M}athematicians, {V}ol. 1, 2 ({B}erkeley, {C}alif., 1986), Providence, RI,
  1987, Amer. Math. Soc., pp.~1286--1295.

\bibitem{Ma88}
\leavevmode\vrule height 2pt depth -1.6pt width 23pt, {\em Gauge field theory
  and complex geometry}, vol.~289 of Grundlehren der Mathematischen
  Wissenschaften, Springer-Verlag, Berlin, 1988.
\newblock Translated from the Russian by N. Koblitz and J. R. King, second
  edition 1997, with an appendix by S. Merkulov.

\bibitem{Ma88-2}
\leavevmode\vrule height 2pt depth -1.6pt width 23pt, {\em Neveu-{S}chwarz
  sheaves and differential equations for {M}umford superforms}, J. Geom. Phys.,
  5 (1988), pp.~161--181.

\bibitem{Mat89}
{\sc H.~Matsumura}, {\em Commutative ring theory}, vol.~8 of Cambridge Studies
  in Advanced Mathematics, Cambridge University Press, Cambridge, second~ed.,
  1989.
\newblock Translated from the Japanese by M. Reid.

\bibitem{MoZh19}
{\sc S.~F. Moosavian and Y.~Zhou}, {\em On the existence of heterotic-string
  and type-{II}-superstring field theory vertices}.
\newblock {\tt arXiv:1911.04343 [hep.th]}.

\bibitem{Mum86}
{\sc D.~Mumford}, {\em Lectures on curves on an algebraic surface}, With a
  section by G. M. Bergman. Annals of Mathematics Studies, No. 59, Princeton
  University Press, Princeton, N.J., 1966.

\bibitem{MF-GIT}
{\sc D.~Mumford, J.~Fogarty, and F.~Kirwan}, {\em Geometric invariant theory},
  vol.~34 of Ergebnisse der Mathematik und ihrer Grenzgebiete (2) [Results in
  Mathematics and Related Areas (2)], Springer-Verlag, Berlin, third~ed., 1994.

\bibitem{Nee96}
{\sc A.~Neeman}, {\em The {G}rothendieck duality theorem via {B}ousfield's
  techniques and {B}rown representability}, J. Amer. Math. Soc., 9 (1996),
  pp.~205--236.

\bibitem{Ni07}
{\sc N.~Nitsure}, {\em Construction of {H}ilbert and {Q}uot schemes}, in
  Fundamental algebraic geometry, vol.~123 of Math. Surveys Monogr., Amer.
  Math. Soc., Providence, RI, 2005, pp.~105--137.

\bibitem{Noja18}
{\sc S.~Noja}, {\em Supergeometry of {$\Pi$}-projective spaces}, J. Geom.
  Phys., 124 (2018), pp.~286--299.

\bibitem{OgPe84}
{\sc O.~V. Ogievetski\u{\i} and I.~B. Penkov}, {\em Serre duality for
  projective supermanifolds}, Funktsional. Anal. i Prilozhen., 18 (1984),
  pp.~78--79.

\bibitem{OgBerg72}
{\sc A.~Ogus and G.~Bergman}, {\em Nakayama's lemma for half-exact functors},
  Proc. Amer. Math. Soc., 31 (1972), pp.~67--74.

\bibitem{Penk83}
{\sc I.~B. Penkov}, {\em {${\mathcal D}$}-modules on supermanifolds}, Invent.
  Math., 71 (1983), pp.~501--512.

\bibitem{PenSkor85}
{\sc I.~B. Penkov and I.~A. Skornyakov}, {\em Projectivity and {${\spacecal
  D}$}-affineness of flag supermanifolds}, Uspekhi Mat. Nauk, 40 (1985),
  pp.~211--212.

\bibitem{SS}
{\sc A.~Salam and J.~Strathdee}, {\em Super-gauge transformations}, Nuclear
  Phys., B76 (1974), pp.~477--482.

\bibitem{Schwede}
{\sc K.~Schwede}, {\em Gluing schemes and a scheme without closed points}, in
  Recent progress in arithmetic and algebraic geometry, vol.~386 of Contemp.
  Math., Amer. Math. Soc., Providence, RI, 2005, pp.~157--172.

\bibitem{Spal88}
{\sc N.~Spaltenstein}, {\em Resolutions of unbounded complexes}, Compositio
  Math., 65 (1988), pp.~121--154.

\bibitem{Va88}
{\sc A.~Y. Va{\u\i}ntrob}, {\em Deformations of complex superspaces and of the
  coherent sheaves on them}, in Current problems in mathematics. {N}ewest
  results, {V}ol.\ 32, Itogi Nauki i Tekhniki, Akad. Nauk SSSR, Vsesoyuz. Inst.
  Nauchn. i Tekhn. Inform., Moscow, 1988, pp.~125--211.
\newblock J. Soviet Math. {{\bf{5}}1} (1990), no. 1, 2140--2188.

\bibitem{Stacks}
{\sc {Various authors}}, {\em The {S}tacks {P}roject}.
\newblock {\tt https://stacks.math.columbia.edu}.

\bibitem{We09}
{\sc D.~B. Westra}, {\em {S}uperrings and {S}upergroups}.
\newblock Dissertation, Dr.rer.nat., Universit\"at Wien, 2009.

\bibitem{Witten19}
{\sc E.~Witten}, {\em Notes on super {R}iemann surfaces and their moduli}, Pure
  Appl. Math. Q., 15 (2019), pp.~57--211.

\end{thebibliography}

\def\cprime{$'$}

\end{document}